\documentclass[12pt,reqno]{amsart}
\usepackage{amsmath}
\usepackage{latexsym,amssymb,eucal}
\usepackage{mathrsfs}
\usepackage{enumerate}
\usepackage{bm}
\allowdisplaybreaks[1]

\newtheorem{thm}{Theorem}[section]
\newtheorem{lem}[thm]{Lemma}
\newtheorem{claim}{Claim}
\newtheorem*{cla}{Claim}
\newtheorem{cor}[thm]{Corollary}
\newtheorem{prop}[thm]{Proposition}
\newtheorem{prob}[thm]{Problem}
\newtheorem{conj}[thm]{Conjecture}

\newtheorem{mthm}{Theorem}

\theoremstyle{definition}
\newtheorem{ass}[thm]{Assumption}
\newtheorem{rem}[thm]{Remark}
\newtheorem{defn}[thm]{Definition}

\newtheorem{ex}[thm]{Example}

\theoremstyle{remark}

\numberwithin{equation}{section}

\renewcommand{\labelenumi}{$(\arabic{enumi})$}

\newcommand{\C}{{\mathbb{C}}}
\newcommand{\R}{{\mathbb{R}}}
\newcommand{\Z}{{\mathbb{Z}}}
\newcommand{\N}{{\mathbb{N}}}
\newcommand{\T}{{\mathbb{T}}}
\renewcommand{\ker}{\mathrm{Ker}}

\DeclareMathOperator{\Ad}{Ad}
\DeclareMathOperator{\Tr}{Tr}

\DeclareMathOperator{\tr}{tr}
\DeclareMathOperator{\id}{id}
\DeclareMathOperator{\sgn}{sgn}

\DeclareMathOperator{\spa}{span}
\DeclareMathOperator{\ONB}{ONB}
\DeclareMathOperator{\Irr}{Irr}

\DeclareMathOperator{\Hom}{Hom}              

\DeclareMathOperator{\Mor}{Mor}
\DeclareMathOperator{\End}{End}
\DeclareMathOperator{\Sect}{Sect}
\DeclareMathOperator{\Aut}{Aut}
\DeclareMathOperator{\Out}{Out}

\DeclareMathOperator{\Int}{Int}
\def\oInt{\ovl{\Int}}
\DeclareMathOperator{\Rep}{Rep}
\DeclareMathOperator{\Hilb}{Hilb}
\DeclareMathOperator{\mo}{mod}

\DeclareMathOperator{\Ind}{Ind}
\DeclareMathOperator{\Inv}{Inv}

\DeclareMathOperator{\Obm}{Ob_m}

\def\mA{\mathcal{A}}
\def\mB{\mathcal{B}}
\def\cC{\mathcal{C}}
\def\mD{\mathcal{D}}
\def\cE{{\mathcal{E}}}
\def\mF{\mathcal{F}}
\def\mG{\mathcal{G}}

\def\cI{\mathcal{I}}
\def\mK{\mathcal{K}}
\def\mL{\mathcal{L}}
\def\cM{\mathcal{M}}

\def\mP{\mathcal{P}}

\def\mS{\mathcal{S}}
\def\mT{\mathcal{T}}

\def\mX{\mathcal{X}}
\def\mZ{\mathcal{Z}}

\def\meC{\mathscr{C}}

\def\meH{\mathscr{H}}

\def\meJ{\mathscr{J}}

\def\meR{\mathscr{R}}
\def\meS{\mathscr{S}}

\def\al{\alpha}
\def\be{\beta}
\def\ga{\gamma}
\def\de{\delta}
\def\ka{\kappa}
\def\la{\lambda}

\def\vep{\varepsilon}
\def\ph{{\phi}}
\def\ps{{\psi}}
\def\tps{\widetilde{\ps}}
\def\vph{\varphi}

\def\om{\omega}
\def\si{\sigma}
\def\ta{\tau}
\def\th{\theta}

\def\Si{\Sigma}
\def\Ph{\Phi}
\def\Ps{\Psi}
\def\Th{\Theta}
\def\Ga{\Gamma}
\def\De{\Delta}

\def\La{\Lambda}
\def\Om{\Omega}

\def\el{\ell}

\def\ovl{\overline}                 
\def\wdh{\widehat}
\def\wdt{\widetilde}
\def\opp{{\mathrm{opp}}}
\def\op{{\mathrm{op}}}

\def\cu{\check{u}}
\def\bu{\bar{u}}

\def\tM{\widetilde{M}}

\def\tP{\widetilde{\mathcal{P}}}

\def\hal{{\widehat{\al}}}

\def\hGa{{\widehat{\Ga}}}

\def\tal{{\wdt{\al}}}
\def\tbe{{\wdt{\be}}}
\def\tga{{\wdt{\ga}}}
\def\tde{{\wdt{\de}}}
\def\trho{\wdt{\rho}}
\def\tsi{\wdt{\si}}

\def\tN{\wdt{N}}
\def\tM{{\wdt{M}}}
\def\tP{\wdt{P}}

\def\opi{{\ovl{\pi}}}                
\def\orho{{\ovl{\rho}}}
\def\osi{{\ovl{\si}}}

\def\ops{{\ovl{\ps}}}

\def\subs{\subset}                   

\def\setm{\setminus}
\def\nin{\notin}
\def\dsp{\displaystyle}

\def\per{\perp}

\def\oti{\otimes}                    
\def\rti{\rtimes}                    

\def\col{\colon}
\def\ra{\rightarrow}

\def\btr{{\boldsymbol{1}}}                   
\def\bG{\mathbb{G}}
\def\bH{\mathbb{H}}
\def\bK{\mathbb{K}}

\def\bhG{{\wdh{\mathbb{G}}}}
\def\bhH{{\wdh{\mathbb{H}}}}
\def\bhK{{\wdh{\mathbb{K}}}}

\def\lhG{L^\infty(\bhG)}                    
\def\lhH{L^\infty(\bhH)}

\def\ltG{L^2(\bG)}

\def\CG{C(\bG)}

\def\lG{L^\infty(\bG)}                      

\def\lH{L^\infty(\bH)}

\def\CH{C(\bH)}
\def\CK{C(\bK)}

\def\IG{\Irr(\bG)}
\def\IH{\Irr(\bH)}

\author[T. Masuda]{Toshihiko Masuda$^1$}
\address{$^1$ 
Graduate School of Mathematics, Kyushu University,
Fukuoka\\ \indent \mbox{819-0395},
JAPAN}
\email{masuda@math.kyushu-u.ac.jp}

\author[R. Tomatsu]{Reiji Tomatsu$^2$}
\address{$^2$Department of Mathematics,
\hspace{-0.6mm}Hokkaido university,
\hspace{-0.6mm}Sapporo, Hokkaido\\ \indent \mbox{060-0810},
JAPAN}
\email{tomatsu@math.sci.hokudai.ac.jp}

\subjclass[2000]{Primary 46L65; Secondary 46L55}

\begin{document}
\title{Classification of actions of discrete Kac algebras
on injective factors}

\begin{abstract}
We will study two kinds of actions of a discrete amenable Kac algebra.
The first one is an action whose modular part is normal.
We will construct a new invariant which generalizes a characteristic
invariant for a discrete group action,
and we will present a complete classification.
The second is a centrally free action.
By constructing a Rohlin tower
in an asymptotic centralizer,
we will show that the Connes--Takesaki module is a complete invariant.
\end{abstract}

\maketitle

\section*{Introduction}

Since Connes' classification of automorphisms on the injective factor of
type II$_1$ \cite{Con-auto,Co-peri},
a classification of discrete amenable group actions on
injective factors has been developed by many researchers
\cite{JT,KtST,KST,Oc}, 
and now we have the complete classification \cite{KtST,M}.
On the contrary,
study of actions by a continuous group is more difficult.
Among them,
it is relatively easy to handle
a compact group
because its dual is a discrete object.
The point is that we still have the Takesaki-type duality
by the aid of Kac algebras,
and study of actions of compact groups
(or more generally, compact Kac algebras)
is essentially reduced
to those of discrete Kac algebras.

As in the case of discrete amenable groups,
it is crucial to construct a Rohlin tower.
However,
this argument always involves a technical difficulty
on treating an asymptotic centralizer.
Namely,
an action of a discrete Kac algebra means, roughly speaking,
a system of endomorphisms.
As well known, an endomorphism
does not preserve central sequences in general,
and
it is not so trivial to construct a Rohlin tower
for a given action.

Let us recall the classification results
obtained in our previous works \cite{MT1,MT2,MT3}.
In \cite{MT1},
we classified actions of
discrete amenable Kac algebras on the injective factor of type II$_1$.  
More precisely,
we showed the uniqueness of a free action of
a discrete amenable Kac algebra up to cocycle conjugacy.
By the duality argument as mentioned above,
the uniqueness of a minimal action of a compact group
on the injective type II$_1$ factor follows from this result.

The next problem is the classification of actions on
infinite injective factors.
By combining \cite{MT1}
and an analytic characterization of endomorphisms \cite{MT2},
we extended the result of \cite{MT1},
and showed the uniqueness of approximately inner
and centrally free actions
of a discrete amenable Kac algebra on injective factors of type II$_\infty$
and type III \cite{MT3}.
However for infinite factors,
there exist many free actions of
discrete Kac algebras, which are not approximately inner,
nor centrally free.
(In the type II$_1$ case, free actions are
automatically approximately inner and centrally free.) 
Thus, to classify actions,
we will introduce invariants of actions.
As in the case of discrete group actions,
we should use  
Connes--Takesaki modules,
and modular endomorphisms introduced in \cite{Iz}
in order to formulate invariants.
For approximately inner and centrally trivial actions,
invariants become trivial,
and thus the results in \cite{MT3} cover this case.

We will divide our problem into the following two cases
and treat them separately:

\begin{enumerate}
\item 
An action with non-trivial normal modular part.

\item
A centrally free action with Connes--Takesaki module.
\end{enumerate}

In the first case,
we will formulate cohomological invariants,
which can
be regarded as a generalization of a characteristic invariant
of group actions.
Then we introduce the notion of a modular kernel,
and our classification problem is reduced to that of modular kernels.
This enables us to construct the set $X(\ps,F^M)$ which can be regarded
as a generalization of the set of characteristic invariants.
Our main result concerning (1) is the following
(Theorem \ref{thm:non-triv-modular-main}):
\begin{mthm}
Let $\al$ and $\be$ be actions of a discrete amenable Kac algebra
$\bhG$ on a McDuff factor $M$.
Suppose that the following conditions hold:
\begin{itemize}
\item
$\al$ and $\be$ have
a common normal modular part;
\item
$\mo(\al_\pi)=\mo(\be_\pi)$ for all $\pi\in\IG$;
\item
$[(a^\al,c^\al)]=[(a^\be,c^\be)]$
in $X(\ps,F^M)$.
\end{itemize}
Then $\al$ and $\be$ are strongly cocycle conjugate.
\end{mthm}

In the second case,
if a given action has the trivial Connes--Takesaki module,
then our result corresponds to the one obtained in \cite{MT3}.
But even in this case, our proof presented
in this paper is completely different from that of \cite{MT3}.
Namely, in \cite{MT3}, we reduce
the problem to a classification of actions on the injective type
II$_\infty$ factor by using the continuous decomposition.
On the other hand,
our new proof is based on the intertwining argument.
This strategy
is the natural generalization of proof presented in \cite{MT1}
though
the absence of
a trace makes our argument technically difficult, and more subtle
estimations are required to obtain a classification.
Our second main result is the following
(Corollary \ref{cor:cent-free-inj}):

\begin{mthm}
Let $\al$ and $\be$ be centrally free actions of $\bhG$
on an injective factor $M$.
Then they have the Connes--Takesaki modules,
and if $\mo(\al_\pi)=\mo(\be_\pi)$ for all $\pi\in\IG$,
then $\al$ and $\be$ are cocycle conjugate.
\end{mthm}

Our results can be applied to simple compact,
connected Lie groups like $SU(n)$.
On the other hand,
readers should note that
there are many actions of finite
dimensional Kac algebras, to which we can not apply our results.

This paper is organized as follows.

In Section 1,
we will recall the notion of a modular endomorphism and a canonical extension
of an endomorphism due to Izumi \cite{Iz}.
Next, we will introduce an ultraproduct von Neumann algebra
and, in particular,
we will derive some useful inclusions of ultraproduct von Neumann algebras
arising from
a crossed product 
by using Ando--Haagerup theory \cite{AH}.

In Section 2,
We will give a certain decomposition of the canonical extension
of an irreducible endomorphism.
Next we will discuss how an endomorphism $\rho$ is decomposed
when $\rho\orho$ and $\orho\rho$ are modular
(Theorem \ref{thm:rho-orho}).

In Section 3,
we start to study an action of a compact or discrete Kac algebra.
In particular,
a dual 2-cocycle twisting and an application of Theorem \ref{thm:rho-orho}
are discussed.

In Section 4,
we will classify modular kernels.
The main ingredient is the strategy provided in \cite{M}
and the Bratteli--Elliott--Evans--Kishimoto intertwining argument.

In Section 5,
we will discuss a classification of actions with non-trivial modular parts.
On the basis of the classification results obtained
in Section 4,
a complete invariant of actions is introduced,
and we construct a model action which realize a given invariant.
As a special
case, we discuss the relationship between
modular actions and dual 2-cocycle twistings.

In Section 6,
we will classify centrally free actions on injective factors.
Our main tool is a refined Rohlin tower in the asymptotic centralizer.

In Section 7,
related unsolved problems and a plausible conjecture are paused.

In Section 8,
we will briefly summarize some basic results
about index theory of inclusion of von Neumann algebras
and a cocycle crossed product and a dual cocycle twisting of a Kac algebra.

Readers who want to follow the main results quickly,
they can skip some parts of this paper. 
For the proof of Theorem A (Theorem \ref{thm:non-triv-modular-main}),
Corollary \ref{cor:Vps} and Theorem \ref{thm:mod-ker} are used.
One can read only 
Section \ref{subsect:modularkernelclass} and skip other parts of Section 4
to follow the proof of Theorem \ref{thm:mod-ker}.
Thus one may directly access Section \ref{subsect:modularkernelclass}
and
Section 5 once he admits Corollary \ref{cor:Vps}.
For the proof of Theorem \ref{thm:free-class}
and
Theorem B (Corollary \ref{cor:cent-free-inj}),
one may directly
access Section 6 after reading Section 1.

\vspace{10pt}
\noindent
{\bf Acknowledgements.}
The first and second authors are supported in part by
JSPS KAKENHI Grant Number 23540246 and 24740095, respectively.
\tableofcontents

\section{Preliminary}

\subsection{Notation and terminology}
Our reference is \cite{Ta}.
Throughout this paper,
all von Neumann algebras have the separable preduals
except for ultraproduct von Neumann algebras.
Let $M$ be a von Neumann algebra.
For $\vph\in M_*$ and $a\in M$,
the functionals $\vph a$ and $a\vph$ in $M_*$ are defined
by $\vph a(x):=\vph(ax)$ and $a\vph(x):=\vph(xa)$ for all $x\in M$.
We denote by $U(M)$, $Z(M)$ and $W(M)$
the set of all unitary elements in $M$,
the center of $M$
and the set of faithful normal semifinite weights on $M$, respectively.
For a positive $\ph\in M_*$, we set $|x|_\ph:=\ph(|x|)$,
$\|x\|_\ph:=\ph(x^* x)^{1/2}$
and $\|x\|_\ph^\sharp:=2^{-1/2}(\vph(x^*x)+\vph(xx^*))^{1/2}$.
We denote by $M_\ph$ the centralizer of $\ph$, that is,
$x\in M_\ph$ if and only if $\ph(xy)=\ph(yx)$ for any $y\in M$.

Denote by $\Aut(M)$ the set of automorphisms on $M$.
For $\al\in \Aut(M)$ and $\vph\in M_*$,
we let $\al(\vph):=\vph\circ\al^{-1}$.
Note that $\Ad u(\vph)=u\vph u^*$ for a unitary $u$,
where $\Ad u(x):=uxu^*$ for $x\in M$.

Let $\meH\subs M$ be a subspace.
We say that $\meH$ is a \emph{Hilbert space} in $M$ if
$\meH\subs M$ is $\si$-weakly closed 
and $\eta^* \xi\in\C$ for all $\xi,\eta\in\meH$
\cite{Ro}.
Then the coupling $\langle \xi,\eta\rangle:=\eta^*\xi$
gives a complete inner product on $\meH$.
The smallest projection $e\in M$ such that $e\meH=\meH$ is
called the support of $\meH$.
If the support of $\meH$ equals 1,
then we have the endomorphism $\rho_\meH$ defined by
$\rho_\meH(x)=\sum_i v_i xv_i^*$,
where $\{v_i\}_i$ is an orthonormal basis.
We say that an endomorphism $\si$ is \emph{inner}
\index{inner endomorphism}
when $\si=\rho_\meH$ for some Hilbert space $\meH$
with support 1.

Next we recall theory of endomorphisms on a factor.
Our standard references are \cite{Iz-f,Iz,L1,L2}.
We denote by $\End(M)$ and $\Sect(M)$ the set of
normal endomorphisms and sectors \index{sector} on $M$,
that is, $\Sect(M)$ 
is the set of unitary equivalence classes of endomorphisms on $M$.
For two endomorphisms $\rho,\si\in \End(M)$,
we let
$(\rho,\si)=\{v\in M\mid v\rho(x)=\si(x)v\ \mbox{for all}\ x\in M\}$.
If $\rho$ is irreducible, then $(\rho,\si)$ is a Hilbert space 
with the inner product $(V,W)=W^* V$ for $V,W\in (\rho,\si)$.

For a factor $M$, we denote by $\End(M)_0$ the set of
endomorphisms with finite Jones--Kosaki index.
For $\rho\in \End(M)_0$, $E_\rho$
denotes the minimal expectation
from $M$ onto $\rho(M)$ in the sense of \cite{Hi}.
We define the \emph{standard left inverse}
\index{standard left inverse}
of $\rho$
by $\ph_\rho=\rho^{-1}\circ E_\rho$.
The statistical dimension $(\Ind E_\rho)^{1/2}$
is denoted by $d(\rho)$.

\subsection{Canonical extensions}
In \cite{Iz}, Izumi introduced
the canonical extension of an endomorphism,
which plays an important role in this paper.

Let $M$ be a factor and $\tM$ the core \index{core} of $M$
as defined in \cite[Definition 2.5]{FT}.
The core $\tM$
is the von Neumann algebra generated 
by $M$ and 
a one-parameter unitary group $\{\la^\vph(t)\}_{t\in\R}$, $\vph\in W(M)$, 
satisfying the following relations:
\[
\si_t^\vph(x)=\la^\vph(t)x\la^\vph(t)^*,
\quad
\la^\ps(t)=[D\ps:D\vph]_t \la^\vph(t)
\]
for all $x\in M$, $t\in\R$ and $\vph,\ps\in W(M)$. 
Then $\tM$ is naturally isomorphic to
the crossed product $M\rti_{\si^\vph}\R$.
Let $\th$ be the $\R$-action on $\tM$ satisfying
\[
\th_s(x)=x,\quad \th_s(\la^\vph(t))=e^{-ist}\la^\vph(t)
\quad
\mbox{for all}\ x\in M,\ s,t\in\R.
\]
The action $\th$ is called the \emph{dual action},
\index{dual action}
and the \emph{flow of weights}
\index{flow of weights}
of $M$
means the restriction of $\th$ on the center $Z(\tM)$.

\begin{defn}[Izumi]
Let $\rho$ be an endomorphism on a factor $M$ with finite index. 
Then the \emph{canonical extension}
\index{canonical extension}
$\trho$ of $\rho$ is
the endomorphism on $\tM$ defined by
\[
\trho(x)=\rho(x),
\quad
\trho(\la^\vph(t))=d(\rho)^{it}[D\vph\circ\ph_\rho:D\vph]_t\la^\vph(t)
\]
for all $x\in M$, $t\in \R$ and $\vph\in W(M)$. 
\end{defn}
Note that $\trho$ commutes with the dual action $\th$.
Conversely, if $\be\in\End(\tM)$ is commuting with $\th$,
then the restriction $\be|_M$ gives an endomorphism on $M$
because of $M=\tM^\th$, the fixed point algebra by $\th$.
In the following lemma,
we study a relationship between $\be$ and $\be|_M$.
On a probability index
\index{index!probability--}
for an inclusion of von Neumann algebras,
readers are referred to Section \ref{sect:index}.

\begin{lem}\label{lem:beta-can}
Let $\be\in \End(\tM)$.
Assume that there exists a faithful normal
conditional expectation from
$\tM$ onto $\be(\tM)$ such that
it has finite probability index and commutes with $\th$.
If $(\be(\tM)'\cap\tM)^\th=\C$,
then $\be|_M\in\End(M)$ is irreducible
and
there exists $s\in\R$ such that $\be=\th_s\widetilde{\be|_M}$.
\end{lem}
\begin{proof}
Let $E\col \tM\ra\be(\tM)$ be such an expectation
and $\si$ the restriction of $\be$ to $M$.
Then $E|_M\col M\ra \si(M)$ is a conditional expectation
with finite probability index.
Thus the statistical dimension $d(\si)$ is finite.
Take a dominant weight $\vph$ on $M$ so that
$(\vph,\ph_\si)$ is an invariant pair
\index{invariant pair}
\cite[Definition 2.2]{Iz},
that is, $d(\si)\vph\circ\ph_\si=\vph$.
Since $\be$ and $\th$ are commuting,
$\be(\la^\vph(t))\la^\vph(t)^*\in \tM^\th=M$ is
a $\si^\vph$-cocycle.
By Connes' theorem,
we have a faithful normal semifinite weight $\chi$ on $M$
such that
$\be(\la^\vph(t))=[D\chi:D\vph]_t\la^\vph(t)$.
For  $x\in M$, we have
\begin{align*}
\si_t^\vph(\si(x))
&=
\si_t^{\vph\circ\ph_\si}(\si(x))
=
\si(\si_t^\vph(x))
\\
&=
\be(\la^\vph(t)x\la^\vph(t)^*)
\\
&=
[D\chi:D\vph]_t
\la^\vph(t)\si(x)\la^\vph(t)^*
[D\chi:D\vph]_t^*
=
\si_t^\chi(\si(x)).
\end{align*}
This implies $\si_t^\chi(\si(M))=\si(M)$
and $[D\vph:D\chi]_t\in \si(M)'\cap M$.
Since $(\vph,\ph_\si)$ is an invariant pair,
the modular automorphism $\si_t^{\vph}$ is trivial
on $\si(M)'\cap M$.
Therefore,
\[
(\be(\tM)'\cap\tM)^\th
=
\be(\la^\vph(\R))'\cap (\si(M)'\cap M)
=
\{[D\chi:D\vph]_t\}_{t\in\R}'\cap  (\si(M)'\cap M).
\]
By our assumption, $(\be(\tM)'\cap\tM)^\th=\C$.
In particular, $\si(M)'\cap M$
is a finite dimensional type I factor.
However, the map $\R\ni t\mapsto [D\chi:D\vph]_t\in \si(M)'\cap M$
is a one-parameter unitary group,
and we must have $\si(M)'\cap M=\C$,
that is, $\si=\be|_M$ is irreducible.
Hence
$[D\chi:D\vph]_t=e^{-ist}$ for some $s\in\R$,
and we obtain
\begin{align*}
\be(\la^\vph(t))
&=
[D\chi:D\vph]_t\la^\vph(t)
=
e^{-ist}\la^\vph(t)
\\
&=
e^{-ist}\tsi(\la^\vph(t))
=
\th_s(\tsi(\la^\vph(t))).
\end{align*}
This implies $\be=\th_s\tsi$.
\end{proof}

\subsection{Modular endomorphisms}
\label{subsection:modular-endo}
Let $\al$ be an automorphism on a factor $M$.
If the canonical extension $\tal$ is an inner automorphism,
then we will say that
$\al$ is an \emph{extended modular automorphism},
\index{extended modular automorphism}
which is slightly different from the original definition of \cite{CT}.
The set of all extended modular automorphisms is denoted by
$\Aut(M)_{\rm m}$, which is a normal subgroup of $\Aut(M)$.
A generalization of this notion to an endomorphism is
presented in \cite[Definition 3.1]{Iz} as follows:
\begin{defn}[Izumi]
Let $\rho$ be an endomorphism on a factor $M$ with finite index.
Then $\rho$ is called a \emph{modular endomorphism}
\index{modular endomorphism}
\index{endomorphism!modular--}
when $\trho$ is inner.
\end{defn}

The set of modular endomorphisms is denoted by $\End(M)_{\rm m}$.
Let $\rho$ be a modular endomorphism.
Then we can find a family of isometries $\{V_i\}_{i=1}^n$ in $\tM$
such that $\trho=\sum_i V_i\cdot V_i^*$.
The commutativity of $\trho$ and $\th$
implies $c(t)_{ij}:=V_i^*\th_t(V_j)$ is central.
Then $c(t):=(c(t)_{ij})_{ij}\in Z(\tM)\oti B(\C^n)$
is a $\th\oti\id$-cocycle, that is, it satisfies
$c(s+t)=c(s)(\th_s\oti\id)(c(t))$.
This correspondence gives a bijection
between $\Sect(M)_{\rm m}$,
the sectors of modular endomorphisms,
and
the disjoint union of $H^1(Z(\tM)\oti B(\C^n))$
(see \cite[Theorem 3.3]{Iz})
for $n\in\N$.

\begin{lem}\label{lem:modular-can}
Let $\{V_i\}_{i=1}^n$ be isometries in $\tM$
such that $\sum_{i=1}^nV_iV_i^*=1$ and $V_i^*\th_t(V_j)\in Z(\tM)$
for all $i,j$.
Then there uniquely exists a modular endomorphism $\si$ on $M$
such that $\tsi=\sum_{i=1}^nV_i\cdot V_i^*$.
\end{lem}
\begin{proof}
Let $c(t)_{ij}:=V_i^*\th_t(V_j)$
and $c(t):=(c(t)_{ij})_{ij}\in Z(\tM)\oti B(\C^n)$,
which is a $\th\oti\id$-cocycle.
We set the endomorphism $\rho_V(\cdot):=\sum_{i=1}^nV_i\cdot V_i^*$
on $\tM$,
which is commuting with $\th$.
Thanks to \cite[Theorem 3.3 (1)]{Iz},
we have a family of isometries $\{W_i\}_{i=1}^n$ in $\tM$
such that $W_i^*\th_t(W_j)=c(t)_{ij}$
and
the endomorphism $\rho_W(\cdot):=\sum_{i=1}^n W_i\cdot W_i^*$
is the canonical extension of the modular endomorphism
$\si_1:=\rho_W|_M$.
Then the unitary $w:=\sum_{i=1}^nV_iW_i^*$ is fixed by $\th$,
and it is contained in $M$.
Since $\rho_V|_M=\Ad w\circ\si_1$,
$\rho_V$ is the canonical extension
of $\Ad w\circ \si_1$.
\end{proof}

\subsection{Ultraproduct von Neumann algebras}

Let us briefly recall
the notion of an ultraproduct von Neumann algebra.
Our references are \cite{AH,MT2,MT-Roh,Oc}.
Let $M$ be a von Neumann algebra.
Denote by $\ell^\infty(M)$ the C$^*$-algebra
which consists of norm-bounded sequences in $M$.
Let $\om$ be a free ultrafilter on $\N$.
Let $\cI_\om(M)$ be the set of $\om$-trivial sequences
\index{sequence!$\om$-trivial--}
in $M$,
that is,
$(x^\nu)_\nu\in \cI_\om(M)$
if and only if
$x^\nu\to0$ in the strong$*$ topology as $\nu\to\om$.

Then we let $\cM^\om(M)$ the normalizer of $\cI_\om(M)$ in $\ell^\infty(M)$,
and set the quotient C$^*$-algebra $M^\om:=\cM^\om(M)/\cI_\om(M)$,
which we will call the \emph{ultraproduct von Neumann algebra}.
\index{ultraproduct von Neumann algebra}
The equivalence class of $(x^\nu)_\nu$
is denoted by $(x^\nu)^\om$ for simplicity.

By mapping $M\ni x\mapsto (x)^\om\in M^\om$,
we can regard $M$ as a von Neumann subalgebra of $M^\om$.
We denote $(x)^\om$ by $x^\om$ for short.

An element $(x^\nu)_\nu\in \ell^\infty(M)$ is said to be
$\om$-central
\index{sequence!$\om$-central--}
if $\|x^\nu\vph-\vph x^\nu\|_{M_*}\to0$ as $\nu\to\om$
for all $\vph\in M_*$.
Denote by $\cC_\om(M)$ the set of $\om$-central sequences.
Then it is known that $\cI_\om(M)\subs\cC_\om(M)\subs\cM^\om(M)$.
The quotient C$^*$-algebra $M_\om:=\cC_\om(M)/\cI_\om(M)$
is called the \emph{asymptotic centralizer}.
\index{asymptotic centralizer}

We shortly write $\cI_\om,\cC_\om$ and $\cM^\om$
for $\cI_\om(M),\cC_\om(M)$ and $\cM^\om(M)$,
respectively
if there is no confusion.

In this subsection,
we would like to show the natural inclusion
$M^\omega\subs\tM^\om$.
To prove this,
we need Ando--Haagerup's result
which characterizing $\cM^\om$
in terms of a modular group \cite{AH}.
Let us restate their result in terms of $\om$-equicontinuity.

Let $\al$ be a flow on $M$,
that is, $\al$ is an $\R$-action on $M$.
Let us fix a faithful state $\vph\in M_*$.
An element $(x^\nu)_\nu\in\ell^\infty(M)$ is said to be
\emph{$(\al,\om)$-equicontinuous}
\index{equicontinuous!$(\al,\om)$--}
\cite[Definition 3.4]{MT-Roh}
if
for any $\vep>0$,
there exists $\de>0$ and $W\in\om$
such that
if $|t|<\de$ and $\nu\in W$,
then we have
$\|\al_t(x^\nu)-x^\nu\|_\vph^\sharp<\vep$.

\begin{thm}
\label{thm:AH-equicont}
Let $(x^\nu)_\nu\in\ell^\infty(M)$.
Let $\vph$ be a faithful normal state on $M$.
Then the following statements are equivalent:
\begin{enumerate}
\item
$(x^\nu)_\nu\in\cM^\om$;

\item
For any $\vep>0$,
there exists $a>0$ and $(y^\nu)_\nu\in\ell^\infty(M)$
such that
\begin{itemize}
\item
$\lim_{\nu\to\om}\|x^\nu-y^\nu\|_\vph^\sharp<\vep$;

\item
$y^\nu\in M(\si^\vph,[-a,a])$ for all $\nu\in\N$;
\end{itemize}

\item
$(x^\nu)_\nu$ is $(\si^\ps,\om)$-equicontinuous
for some faithful state $\ps\in M_*$;

\item
$(x^\nu)_\nu$ is $(\si^\ps,\om)$-equicontinuous
for all faithful normal weight $\ps$.
\end{enumerate}
\end{thm}
\begin{proof}
The equivalence (1)$\Leftrightarrow$(2)
is due to \cite[Proposition 4.11]{AH}.

(2)$\Rightarrow$(3).
Let $\vep>0$.
Take $a$ and $y^\nu$ as in (2).
Take $W\in \om$
so that
if $\nu\in W$,
then $\|x^\nu-y^\nu\|_\vph^\sharp<\vep$.
For all $t\in\R$ and $\nu\in W$,
we have
\begin{align*}
\|\si_t^\vph(x^\nu)-x^\nu\|_\vph^\sharp
&\leq
\|\si_t^\vph(x^\nu)-\si_t^\vph(y^\nu)\|_\vph^\sharp
+
\|\si_t^\vph(y^\nu)-y^\nu\|_\vph^\sharp
+
\|y^\nu-x^\nu\|_\vph^\sharp
\\
&=
2\|x^\nu-y^\nu\|^\sharp
+
\|\si_t^\vph(y^\nu)-y^\nu\|_\vph^\sharp
\\
&<
2\vep
+
\|\si_t^\vph(y^\nu)-y^\nu\|_\vph^\sharp.
\end{align*}

Take $f\in L^1(\R)$ such that $\hat{f}(x)=1$
if $|x|\leq a$.
Then $\si_f^\vph(y^\nu)=y^\nu$,
and
\[
\|\si_t^\vph(y^\nu)-y^\nu\|_\vph^\sharp
=
\|\si_{\la_t f-f}^\vph(y^\nu)\|_\vph^\sharp
\leq
\|\la_t f-f\|_1\|y^\nu\|_\vph^\sharp,
\]
where $\la_t$ denotes the left regular representation on $\R$.
Thus we have
\[
\|\si_t^\vph(x^\nu)-x^\nu\|_\vph^\sharp
<2\vep+\|\la_t f-f\|_1\|y^\nu\|_\vph^\sharp
\quad
\mbox{for all }
t\in\R,\ \nu\in W.
\]
It is obvious that we can take $\de>0$
as in the definition of
$(\si^\vph,\om)$-equicontinuity.

(3)$\Rightarrow$(1).
Let $(x^\nu)_\nu\in\ell^\infty(M)$ be a $(\si^\vph,\om)$-equicontinuous
sequence.
Let $\vep>0$ and take $\de>0$ and $W\in\om$
so that
if $|t|<\de$ and $\nu\in W$,
then we have
$\|\si_t^\vph(x^\nu)-x^\nu\|_\vph<\vep$.

For $r>0$,
we let $g_r(t):=\sqrt{1/\pi r}\,e^{-t^2/r}$ for $t\in\C$.
Then there exists an enough small $r$
such that
\[
\|\si_{g_r}^\vph(x^\nu)-x^\nu\|_\vph
\leq
\int_\R g_r(t)\|\si_t^\vph(x^\nu)-x^\nu\|_\vph\,dt
<2\vep
\quad
\mbox{for all }
\nu\in W.
\]

Now let $(y^\nu)_\nu\in\cI_\om$ with $\|y^\nu\|\leq1$
for all $\nu\in\N$.
Then for $\nu\in W$, we have
\begin{align*}
\|y^\nu x^\nu\|_\vph
&=
\|y^\nu x^\nu \xi_\vph\|
\leq
\|y^\nu (x^\nu-\si_{g_r}^\vph(x^\nu))\xi_\vph\|
+
\|y^\nu \si_{g_r}^\vph(x^\nu)\xi_\vph\|
\\
&=
\|y^\nu\|\|x^\nu-\si_{g_r}^\vph(x^\nu)\|_\vph
+
\|y^\nu \si_{g_r}^\vph(x^\nu)\xi_\vph\|
\\
&<
\vep
+
\|J_\vph \si_{i/2}^\vph(\si_{g_r}^\vph(x^\nu))^* J_\vph y^\nu\xi_\vph\|
\\
&\leq
\vep
+
\|g_r(\cdot-i/2)\|_1\|x^\nu\|\|y^\nu\|_\vph,
\end{align*}
where $\xi_\vph$ and $J_\vph$ denote the GNS vector of $\vph$
and the modular conjugation.
Then it follows that
$\lim_{\nu\to\om}\|y^\nu x^\nu\|_\vph\leq\vep$,
and we are done.

(3)$\Rightarrow$(4).
Since (1) and (3) are equivalent,
the sequence $(x^\nu)_\nu$ belongs to $\cM^\om$.
Then by \cite[Lemma 3.8 (2)]{MT-Roh},
it is $(\si^\ps,\om)$-equicontinuous
for any faithful normal semifinite weight $\ps$ on $M$.

(4)$\Rightarrow$(3).
This is a trivial implication.
\end{proof}

\begin{rem}
It is not possible to replace the phrase
in (4) of the previous theorem
with ``for some faithful normal semifinite weight''.
Otherwise,
normalizing sequences in $M=B(\ell^2)$ would be all of
norm bounded sequences.
However, it contradicts the fact that
for a type III von Neumann subalgebra $N\subs B(\ell^2)$,
$\cM^\om(N)$ is strictly smaller than $\ell^\infty(N)$.
\end{rem}

Now let $\al$ be a flow on a von Neumann algebra $M$.
Denote by $\cE_\al^\om(M)$ the set of $(\al,\om)$-equicontinuous
sequences.
The canonical embedding $\pi_\al\col M\to M\rti_\al\R$
induces
$\pi_\al\col \cE_\al^\om(M)\cap\cM^\om(M)\to \ell^\infty(M\rti_\al\R)$
by putting $\pi_\al((x^\nu)_\nu):=(\pi_\al(x^\nu))_\nu$.

\begin{lem}
If $(x^\nu)_\nu\in \cE_\al^\om(M)\cap\cM^\om(M)$,
then $\pi_\al((x^\nu)_\nu)\in \cM^\om(M\rti_\al\R)$.
\end{lem}
\begin{proof}
Let $(y^\nu)_\nu$ be an $\om$-trivial sequence in $M\rti_\al\R$
with $\|y^\nu\|\leq1$ for all $\nu$.
We will show that $(y^\nu \pi_\al(x^\nu))_\nu$ is $\om$-trivial.
It suffices to show that
$\|y^\nu \pi_\al(x^\nu)(\xi\oti f)\|\to0$
as $\nu\to\om$
for a vector $\xi\in H$ and $f\in L^2(\R)$
with ${\rm supp}(f)\subs[-R,R]$ with $R>0$.

Let $\vep>0$.
Since $(x^\nu)_\nu$ is $(\al,\om)$-equicontinuous,
there exists an enough large $N\in\N$ and $W\in \om$
such that if $|s-t|<2R/N$ and $\nu\in W$,
then$\|\al_s(x^\nu)\xi-\al_t(x^\nu)\xi\|<\vep$.
We let $t_j:=(2j-N)R/N$
and $K_j:=[t_j,t_{j+1}]$ for $j=0,\dots,N-1$.
Then for $\nu\in W$,
we obtain the following:
\begin{align*}
&\left
\|\Big{(}
\pi_\al(x^\nu)-\sum_{j=0}^{N-1}(\al_{-t_j}(x^\nu)\oti 1_{K_j})
\Big{)}(\xi\oti f)\right\|^2
\\
&=
\sum_{j=0}^{N-1}
\int_{K_j}
\|(\al_{-s}(x^\nu)-\al_{-t_j}(x^\nu))\xi\|^2
|f(s)|^2\,ds
\\
&<
\sum_{j=0}^{N-1}
\vep^2
\int_{K_j}
|f(s)|^2\,ds
=\vep^2\|f\|_2^2.
\end{align*}

Thus for all $\nu\in W$,
\[
\|y^\nu \pi_\al(x^\nu)(\xi\oti f)\|
\leq
\vep\|f\|_2
+
\left\|y^\nu \sum_{j=0}^{N-1}(\al_{-t_j}(x^\nu)\oti 1_{K_j})(\xi\oti f)
\right\|.
\]
In the last term,
we know that $(\al_{-t_j}(x^\nu)\oti 1_{K_j})_\nu$
belongs to $\cM^\om(M\oti B(L^2(\R)))$
by the proof of \cite[Lemma 2.8]{MT-Roh}.
In particular,
$y^\nu(\al_{-t_j}(x^\nu)\oti 1_{K_j})$ converges to 0
in the strong topology
as $\nu\to\om$.
Hence the above inequality implies that
\[
\lim_{\nu\to\om}\|y^\nu \pi_\al(x^\nu)(\xi\oti f)\|
\leq
\vep\|f\|_2.
\]
Thus we are done.
\end{proof}

The following result is immediately implied
by the previous lemma.

\begin{thm}
Let $\al$ be a flow on a von Neumann algebra $M$.
Then one has a normal embedding
$\pi_\al^\om$ of $M_\al^\om$ into $(M\rti_\al\R)^\om$,
where $\pi_\al^\om((x^\nu)^\om):=(\pi_\al(x^\nu))^\om$.
\end{thm}

It turns out from Theorem \ref{thm:AH-equicont}
that
$\cM^\om(M)=\cE_{\si^\vph}^\om(M)$.
Thus we have the following result.

\begin{cor}
\label{cor:core-ultra}
Let $M$ be a von Neumann algebra
and $\tM$ the core of $M$.
Then $M^\om\subs \tM^\om$ and $M_\om\subs\tM_\om$.
\end{cor}
\begin{proof}
The first inclusion is an immediate consequence
of the previous result
since $\tM\cong M\rti_{\si^\vph}\R$.
The second is proved in the proof of \cite[Lemma 4.11]{MT2}.
\end{proof}

We can strengthen the theorem above.
Let us denote by $\la^\al(t)$
the implementing one-parameter unitary group in $M\rti_\al\R$.

\begin{thm}
One has the canonical embedding
of $M_\al^\om\rti_{\al^\om}\R$
into $(M\rti_\al\R)^\om$
such that
it maps $(x^\nu)^\om$
and $\la^{\al^\om}(t)$
to $(\pi_\al(x^\nu))^\om$
and $\la^\al(t)^\om$,
respectively,
for all $(x^\nu)_\nu\in\cE_\al^\om\cap \cM^\om$
and $t\in\R$.
Moreover, there exists a faithful normal conditional expectation
from $(M\rti_\al\R)^\om$ onto $M_\al^\om\rti_{\al^\om}\R$.
\end{thm}
\begin{proof}
Put $N:=\pi_\al^\om(M_\al^\om)\vee \{\la^\al(t)^\om\mid t\in\R\}''$.
We will show that there exists a canonical isomorphism
from $M_\al^\om\rti_{\al^\om}\R$ onto $N$.
By considering a tensor product with $B(\ell^2)$
and a suitable perturbation of $\al$,
we may and do assume that $M$ is properly infinite
and a dominant weight $\vph$ on $M$ is $\al$-invariant.

Let $\ps$ be the dual weight of $\vph$ on $M\rti_\al\R$
and $\ps^\om$ the ultraproduct weight on $(M\rti_\al\R)^\om$.
Thanks to Ando--Haagerup theory \cite[Theorem 4.1]{AH},
we obtain $\si^{\ps^\om}=(\si^\ps)^\om$.
Thus $N$ is globally invariant under $\si^{\ps^\om}$.
It is obvious that $\ps^\om$ is semifinite on $N$
since
$N$ contains $M\rtimes_\alpha\R$.
By Takesaki's theorem,
there exists a faithful normal conditional expectation
from $(M\rti_\al\R)^\om$ onto $N$.

Let $\vph^\om$ be the ultraproduct weight on $M_\al^\om$,
which is semifinite since $M$ is contained in $M_\al^\om$.
Let $\chi$ be the dual weight of $\vph^\om$ on $P:=M_\al^\om\rti_{\al^\om}\R$.

We will compare the GNS Hilbert space $L^2(P,\chi)$
with the GNS space $L^2(N,\ps^\om)$.
The definition left ideals are denoted by $n_\chi$ and $n_{\ps^\om}$,
respectively.
Denote by $\La_\chi\col n_\chi\to L^2(P,\chi)$
and $\La_{\ps^\om}\col n_{\ps^\om}\to L^2(N,\ps^\om)$
the canonical embeddings.

Let us introduce a map $V$ which maps
$\La_\chi(\la^{\al^\om}(f)\pi_{\al^\om}(x))$
to
$\La_{\ps^\om}(\la^\al(f)^\om\pi_\al^\om(x))$
for $f\in C_c(\R)$ and $x\in M_\al^\om$.
Then by the naturality of the definition of the ultraproduct weights,
$V$ extends to an isometry from $L^2(P,\chi)$
into $L^2(N,\ps^\om)$.
We will check that $V$ is surjective.

Let us denote by $K$ the image of $V$.
Then $K$ is $N'$-invariant.
Indeed, for a $\si^{\vph^\om}$-analytic $y\in M_\al^\om$
and $t\in\R$,
we obtain
\[
J_{\ps^\om}\si_{i/2}^{\ps^\om}(\pi_\al^\om(y))^*J_{\ps^\om}
\La_{\ps^\om}(\la^\al(f)^\om\pi_\al^\om(x))
=
\La_{\ps^\om}(\la^\al(f)^\om\pi_\al^\om(xy))\in K,
\]
and
\[
J_{\ps^\om}\si_{i/2}^{\ps^\om}(\la^\al(t)^\om)^*J_{\ps^\om}
\La_{\ps^\om}(\la^\al(f)^\om\pi_\al^\om(x))
=
\La_{\ps^\om}(\la^\al(\la_t f)^\om\pi_\al^\om(\al_{-t}^\om(x)))
\in K
\]
for all $f\in C_c(\R)$ and $x\in M_\al^\om$.

Now let us take a $\si^{\ps^\om}$-analytic $y\in n_{\ps^\om}$.
Then
\[
J_{\ps^\om}\si_{i/2}^{\ps^\om}(y)^*J_{\ps^\om}
\La_{\ps^\om}(\la^\al(f)^\om\pi_\al^\om(x))
=
\la^\al(f)^\om\pi_\al^\om(x)\La_{\ps^\om}(y).
\]
This implies that
$\La_{\ps^\om}(y)$ is contained in the closure of $N'K$.
Since $N'K\subs K$,
$\La_{\ps^\om}(y)$ belongs to $K$.
Thus $K=L^2(N,\ps^\om)$.
It is trivial that $\Ad V$ provides us with a desired isomorphism.
\end{proof}

\section{Canonical extension of irreducible endomorphisms}
In this section,
we discuss the canonical extension of an endomorphism.

\subsection{Decomposition of the canonical extension of an endomorphism}
In general,
it is known that
even if a given endomorphism $\rho$ is irreducible,
its canonical extension $\trho$ is not.
In what follows,
we will describe how $\trho$ is decomposed.
We will say $\rho$ has the \emph{Connes--Takesaki module}
\index{endomorphism!with Connes--Takesaki module}
when $\trho(Z(\tM))=Z(\tM)$
(\cite[Definition 4.1]{Iz}).

\begin{thm}\label{thm:decomp}
Let $\rho$ be an irreducible endomorphism with finite index
on an infinite factor $M$.
\begin{enumerate}
\item
There exists a finite family of isometries $\{V_i\}_{i=1}^n$
in $\tM$ and $\be\in \End(\tM)_0$ such that
\begin{itemize}
\item
$\dsp \trho(x)=\sum_{i=1}^n V_i\be(x)V_i^*$ for all $x\in\tM$,

\item
$(\trho,\trho)=Z((\trho,\trho))\vee\{V_iV_j^*\}_{i,j}''$,

\item
$(\be,\be)=V_1^*Z((\trho,\trho))V_1$,
\end{itemize}
where the number $n$ is unique
and so is the endomorphism $\be$ as a sector on $\tM$.

\item
The endomorphism $\be$ can be taken
so that $\be$ is commuting with $\th$.
In this case,
$\be|_M$ is irreducible, and
there exists $s\in\R$ such that
$\be=\th_s\wdt{\be_M}$.

\item
If $Z((\trho,\trho))=Z(\tM)$,
then there exists a modular endomorphism
$\si\in \End(M)_{0}$
and $\ps\in\End(M)_{0}$ such that
\begin{itemize}
\item
$\si$ and $\ps$ are irreducible;
\item
$\rho=\si\ps$;
\item
$(\trho,\trho)=Z(\tM)\vee (\tsi,\tsi)$;
\item
$(\tps,\tps)=Z(\tM)$.
\end{itemize}
This decomposition is unique in the following sense:
If $\rho=\si_1\ps_1$ is another such decomposition,
then there exists $\ka\in\Aut(M)_{\rm m}$
such that $\ps_1=\ka\circ\ps$.
In addition to this, if
$\rho$ has the Connes--Takesaki module,
then so does $\ps$ and $\mo(\rho)=\mo(\ps)$.
\end{enumerate}
\end{thm}
\begin{proof}
(1).
Let $N:=\rho(M)$.
Since the inclusion $Z(\tN)\subs \tN'\cap \tM$ has finite index,
the relative commutant $\tN'\cap\tM$ is finite and of type I
by Proposition \ref{prop:typeI}.
The ergodicity of $\th$ on $\tN'\cap \tM$ or $Z(\tN'\cap\tM)$
implies that the algebra is of type I$_n$ for some $n\in\N$.
Let $p_1,\dots,p_n$ be
a partition of unity in $\tN'\cap \tM$
which are orthogonal abelian projections
with central support 1.
Then they are mutually equivalent in $\tN'\cap\tM$,
and we can take partial isometries $v_{ij}\in\tN'\cap\tM$
such that $v_{ij}v_{ij}^*=p_i=v_{ii}$ and $v_{ij}^*=v_{ji}$.
Since $Z(\tM)\subs Z(\tN'\cap\tM)$,
the central supports of $p_i$'s in $\tM$ are also equal to 1.
Moreover, $p_i\tM p_i$ contains $\tN p_i$,
and $p_i$ is properly infinite.
Thus each $p_i$ is equivalent to $1$ in $\tM$.
We take an isometry $V_1\in\tM$ with $V_1V_1^*=p_1$,
and set $V_i=v_{i1}V_1$ for $i\neq1$.
Then we obtain a decomposition like (1)
by putting $\be(x):=V_1^*\trho(x)V_1$.

Let $\trho(\cdot)=\sum_{j=1}^m W_i \be'(\cdot)W_i^*$
be another such decomposition.
Note that $W_jW_k^*\in \tN'\cap \tM$.
Then by our assumption,
$\tN'\cap\tM=Z(\tN'\cap\tM)\vee\{W_jW_k^*\}_{j,k}$.
In particular,
$(\tN'\cap\tM)_{W_1W_1^*}=Z(\tN'\cap\tM)_{W_1W_1^*}$,
and $W_1W_1^*$ is an abelian projection in $\tN'\cap \tM$.
Since $W_1W_1^*$ is equivalent to $W_jW_j^*$
in $\tN'\cap\tM$ for all $j$,
we have $m=n$, and 
the central support of $W_1W_1^*$ in $\tN'\cap\tM$
is equal to 1.
Hence there exists $v\in \tN'\cap \tM$
such that $V_1V_1^*=v^*v$ and $W_1W_1^*=vv^*$.
Then the unitary $u:=W_1^*vV_1$ satisfies
\[
u\be(x)=W_1^*vV_1\cdot V_1^*\trho(x)V_1
=
W_1^*v\trho(x)V_1
=
W_1^*\trho(x)W_1W_1^*vV_1
=\be'(x)u
\]
for $x\in\tM$.

(2).
We show there exists a Borel map $\R\ni t\mapsto w_t\in \tN'\cap \tM$
such that $\th_t(V_1V_1^*)=w_t^*w_t$ and $V_1V_1^*=w_tw_t^*$.
Set $p:=V_1V_1^*$.
Let
\[
S_1:=\{(w, \theta_t(p))\mid ww^*=p,w^*w=\theta_t(p),t\in \R\},
\quad
S_2=\{\theta_t(p)\mid t\in \R\}.
\]
The set $S_2$ is an $\mF_\si$-set
of the unit ball of $\tM$
because it is written as a union of continuous
images of closed intervals.
Similarly, we can show $S_1$ is an $\mF_\si$-set
of the unit ball of $\tM\times \tM$.
Hence they are standard Borel spaces.
Let $f$ be the projection from $S_1$ onto $S_2$ defined by 
$f(w,\theta_t(p))=\theta_t(p)$.
The map $f$ is indeed surjective
because the abelian projection
$\th_t(p)$ is equivalent to $p$ in $\tN'\cap\tM$.
In particular, $S_1\neq\emptyset$.
Obviously, $f$ is continuous.
Thanks to the measurable cross section theorem,
we obtain a Borel map
$g\colon S_2\ra S_1$ such that $f\circ g=\id_{S_2}$.
Hence $g(\theta_t(p))=(w_t,\theta_t(p))$ for some $w_t$. 
Then the composed map
\[
t\mapsto \theta_t(p)\stackrel{g}{\mapsto}
(w_t,\theta_t(p))\mapsto w_t
\]
is a desired Borel map.
Then we set a unitary
$u_t:=V_1^*w_t\th_t(V_1)$,
which satisfies
\begin{align*}
u_t^*\be(x)u_t
&=
\th_t(V_1^*)w_t^*\trho(x)w_t\th_t(V_1)
=
\th_t(V_1^*)\trho(x)w_t^*w_t\th_t(V_1)
\\
&=
\th_t(V_1^*)\trho(x)\th_t(V_1)
=
\th_t\be\th_{-t}(x)
\quad
\mbox{for all }x\in\tM.
\end{align*}
This implies the unitary $\mu_{s,t}:=u_s\th_s(u_t)u_{s+t}^*$
belongs to $\be(\tM)'\cap\tM$.
Trivially, we have
$\mu_{r,s}\mu_{r+s,t}=u_r\th_r(\mu_{s,t})u_r^* \mu_{r,s+t}$
and
\[
\mu_{r,s}=V_1^*w_r\th_r(w_s)w_{r+s}^*V_1.
\]
Since $\be(\tM)'\cap \tM=V_1^*Z(\tN'\cap\tM)V_1$,
we have $w_r\th_r(w_s)w_{r+s}^*\in Z(\tN'\cap\tM)V_1V_1^*$.
Hence there uniquely exists a unitary $z_{r,s}\in Z(\tN'\cap\tM)$
such that
$w_r\th_r(w_s)w_{r+s}^*=z_{r,s}V_1V_1^*$.
Then $z$ is a 2-cocycle for $\th$.
Indeed, we have
\begin{align*}
z_{r,s}z_{r+s,t}V_1V_1^*
&=
z_{r,s}V_1V_1^*\cdot
z_{r+s,t}V_1V_1^*
\\
&=
w_r\th_r(w_s)w_{r+s}^*
\cdot
w_{r+s}\th_{r+s}(w_t)w_{r+s+t}^*
\\
&=
w_r\th_r(w_s)\th_{r+s}(V_1V_1^*)
\th_{r+s}(w_t)w_{r+s+t}^*
\\
&=
w_r\th_r(w_s)\th_{r+s}(w_t)w_{r+s+t}^*,
\end{align*}
and
\begin{align*}
\th_r(z_{s,t})z_{r,s+t}V_1V_1^*
&=
\th_r(z_{s,t})w_rw_r^*\cdot
z_{r,s+t}V_1V_1^*
\\
&=
w_r\th_r(z_{s,t})w_r^*\cdot
z_{r,s+t}V_1V_1^*
\\
&=
w_r\th_r(z_{s,t})w_r^*w_rw_r^*\cdot
z_{r,s+t}V_1V_1^*
\\
&=
w_r\th_r(z_{s,t}V_1V_1^*)w_r^*\cdot
z_{r,s+t}V_1V_1^*
\\
&=
w_r\th_r(w_s\th_s(w_t)w_{s+t}^*)w_r^*\cdot
w_r\th_r(w_{s+t})w_{r+s+t}^*
\\
&=
w_r\th_r(w_s\th_s(w_t)w_{s+t}^*V_1V_1^*w_{s+t})w_{r+s+t}^*
\\
&=
w_r\th_r(w_s\th_s(w_t)\th_{s+t}(V_1V_1^*))w_{r+s+t}^*
\\
&=
w_r\th_r(w_s\th_s(w_t\th_{t}(V_1V_1^*)))w_{r+s+t}^*
\\
&=
w_r\th_r(w_s)\th_{r+s}(w_t)w_{r+s+t}^*.
\end{align*}
Since the flow $\{\th,Z(\tN'\cap\tM)\}$ is ergodic,
$z$ is a coboundary from \cite[Proposition A.2]{CT}
(when $Z(\tN'\cap\tM)=\C$,
then it is well-known that $H^2(\R,\T)$ is trivial).
Thus there exists a Borel map $b\col\R\ra Z(\tN'\cap\tM)$
with $b_s^*b_s=1$
and $b_r\th_r(b_s)z_{r,s}b_{r+s}^*=1$.
We let $b_s':=V_1^*b_sV_1$,
which is a unitary in $\be(\tM)'\cap\tM$.
Then we have $\Ad (b_s' u_s)^*\circ\be=\th_s\be\th_{-s}$,
and
\begin{align*}
&b_s'u_s\th_s(b_t'u_t)(b_{s+t}' u_{s+t})^*
\\
&=
V_1^*b_sV_1\cdot
V_1^*w_s\th_s(V_1)
\cdot
\th_s(V_1^*b_tV_1\cdot
V_1^*w_s\th_t(V_1))
\cdot
\th_{s+t}(V_1^*)w_{s+t}^*V_1
\cdot
V_1^*b_{s+t}^*V_1
\\
&=
V_1^*b_s
w_s\th_s(V_1V_1^*b_tw_s\th_t(V_1))
\cdot
\th_{s+t}(V_1^*)w_{s+t}^*b_{s+t}^*V_1
\\
&=
V_1^*b_s
w_s\th_s(b_tV_1V_1^*w_s)
\cdot
\th_{s+t}(V_1V_1^*)w_{s+t}^*b_{s+t}^*V_1
\\
&=
V_1^*b_s
w_s\th_s(b_tw_s)w_{s+t}^*b_{s+t}^*V_1
\\
&=
V_1^*b_s\th_s(b_t)z_{s,t}b_{s+t}^* V_1V_1^*V_1
=1.
\end{align*}
This shows $b_s'u_s$ is a $\th$-cocycle.
By stability of $\{\tM,\th\}$ \cite[Theorem 5.1 (ii)]{CT},
there exists $w\in U(\tM)$ such that $b_s'u_s=w\th_s(w^*)$.
Then
\[
\th_s(w^*\be(x)w)
=
w^* b_s'u_s\th_s(\be(x))u_s^*(b_s')^*w
=
w^*\be(\th_s(x))w.
\]
Thus we may and do assume that $\be$ is commuting with $\th$.
By the equality $\Ad u_t\circ \be=\th_t\be\th_{-t}=\be$,
we see $u_t\in \be(\tM)'\cap \tM$.
Hence there uniquely exists a unitary $\nu_t\in Z(\tN'\cap\tM)$
such that $u_t=V_1^*\nu_t V_1$.
Then we have $\th_t(V_1)=w_t^*\nu_tV_1$.
We show the isomorphism
$Z(\tN'\cap\tM)\ni z\mapsto V_1^*zV_1\in
\be(\tM)'\cap\tM$
is intertwining the flow $\th$.
Indeed, we have
\[
\th_t(V_1^*zV_1)
=
V_1^*\nu_t^*w_t \cdot \th_t(z)\cdot w_t^*\nu_t V_1
=
V_1^*\th_t(z)\nu_t^*w_tw_t^*\nu_tV_1
=V_1^*\th_t(z)V_1.
\]
In particular, $\th$ is ergodic on $\be(\tM)'\cap \tM$.

Let $\ph_{\trho}$ be the canonical extension
of the standard left inverse $\ph_\rho$
as defined in \cite[Lemma 3.5]{MT2}.
Then the conditional expectation $E_{\trho}:=\trho\circ\ph_{\trho}$
satisfies
the Pimsner--Popa inequality $E_{\trho}(x)\geq d(\rho)^{-2}x$
for all $x\in\tM_+$.
Denote by $E_\be$ the map
$\tM\ni x\mapsto\sum_{i=1}^n
\be(\ph_{\trho}(V_ixV_i^*))\in\be(\tM)$,
which is a conditional expectation onto $\be(\tM)$.
Since $\be(\cdot)=V_1^*\trho(\cdot)V_1$,
we have for $x\in\tM_+$,
\begin{align*}
E_\be(x)
&=
\sum_{i=1}^n
\be(\ph_{\trho}(V_ixV_i^*))
=
\sum_{i=1}^nV_1^* E_{\trho}(V_ixV_i^*)V_1
\\
&\geq
d(\rho)^{-2}
\sum_{i=1}^nV_1^* V_ixV_i^*V_1
\\
&=
d(\rho)^{-2}x.
\end{align*}
Hence $E_\be$ has finite probability index.
The map $\th_tE_\be\th_{-t}$
is also a conditional expectation
onto $\be(\tM)$ for each $t\in\R$
since $\be$ and $\th$ are commuting.
To show $\th_tE_\be\th_{-t}=E_\be$,
it suffices to prove they coincide
on $\be(\tM)'\cap\tM$.
Since $\trho$ and $\be$ are commuting with $\th$,
we have
\[
\sum_{i=1}^n V_i\be(x)V_i^*
=
\sum_{i=1}^n \th_t(V_i)\be(x)\th_t(V_i^*).
\]
Hence $c(t)_{ij}:=V_i^*\th_t(V_j)\in \be(\tM)'\cap \tM$.
Note that $(c(t)_{ij})_{ij}\in (\be(\tM)'\cap \tM)\oti M_n(\C)$
is a unitary matrix.
Take any $x\in \be(\tM)'\cap\tM$,
which is commuting with $c(t)_{ij}$
because $\be(\tM)'\cap\tM$ is commutative.
Then we have
\begin{align*}
\th_t(E_\be(\th_{-t}(x)))
&=
\sum_{i=1}^n
\be(\ph_{\trho}(\th_t(V_i)x\th_t(V_i^*)))
\\
&=
\sum_{i,j,k=1}^n
\be(\ph_{\trho}(V_j c(t)_{ji}xc(t)_{ki}^*V_k^*)))
\\
&=
\sum_{i,j,k=1}^n
\be(\ph_{\trho}(V_j x c(t)_{ji}c(t)_{ki}^*V_k^*)))
\\
&=
\sum_{j=1}^n
\be(\ph_{\trho}(V_j x V_j^*)))
\\
&=
E_\be(x).
\end{align*}
Hence $E_\be$ commutes with $\th$.
By Lemma \ref{lem:beta-can},
we obtain $\be=\th_s\wdt{\be|_M}$ for some $s\in\R$.

(3).
Let $c(t)_{ij}$ be as in the proof of (2).
Then $c(t)_{ij}\in \be(\tM)'\cap \tM
=V_1^*Z(\tN'\cap\tM)V_1=Z(\tM)$.
Employing Lemma \ref{lem:modular-can},
we see $\rho_V$ is the canonical extension of
the modular endomorphism $\rho_V|_M=:\si$.
We let $\ps:=\be|_M$.
Since $\rho=\si\ps$ and $\rho$ is irreducible,
so is $\si$.
By (2),
there exists $s\in\R$ such that $\be=\th_s\tps$.
Then we have $\trho=\th_s\tsi\tps=\th_s\wdt{\si\ps}$.
Then by definition of canonical extension,
we have
\[
d(\rho)^{it}[D\vph\circ\ph_\rho:D\vph]_t
=
e^{-ist}d(\si\ps)^{it}[D\vph\circ\ph_{\si\ps}:D\vph]_t.
\]
This implies $\ph_\rho=e^{-s}\ph_{\si\ps}$, and $s=0$
since $\ph_\rho(1)=1=\ph_{\si\ps}(1)$.

If $\rho=\si_1\ps_1$ is another decomposition as stated in (3),
then $(\tps,\tps_1)\neq0$.
Since $\tps$ and $\tps_1$ are both irreducible and co-irreducible
(see Definition \ref{defn:irr-co}),
there exists a unitary $u\in (\tps,\tps_1)$.
Then $u^*\th_t(u)$ is a $\th$-cocycle evaluated in $Z(\tM)$.
Thus $\ka:=\Ad u|_M$ is an extended modular automorphism
satisfying $\ps_1=\ka\circ\ps$.

From the formula $\tps=\be=V_1^*\trho(\cdot)V_1$,
it is clear that  $\rho$ has the Connes--Takesaki module
if and only if  so does $\ps$,
and those modules are equal.
\end{proof}

An endomorphism $\rho$ on an infinite factor $M$ with finite index
is decomposed to the direct sum of irreducibles
as $\oplus_{i=1}^n m_i \rho_i$,
where $m_i=\dim(\rho_i,\rho)$ and $(\rho_i,\rho_j)=0$
if $i\neq j$.
The collection of modular endomorphisms among them
is
called the \emph{modular part} of $\rho$,
and denoted by $\rho_{\rm m}$.

Note that the conjugate endomorphism
of $\tps\in\End(\tM)_0$ equals the canonical extension
of $\ops$ \cite[Proposition 2.5 (3)]{Iz}.
Thus $\tps$ is irreducible if and only
if $\wdt{\ops}$ is co-irreducible
(see Section \ref{subsect:index}).
Recall the following result proved in \cite[Theorem 3.7]{Iz}.

\begin{lem}\label{lem:ps-ops}
Let $\ps$ be an irreducible endomorphism on $M$ with finite index.
Then the following statements are equivalent:
\begin{enumerate}
\item
$\tps$ is co-irreducible, that is, $Z(\tM)=\tps(\tM)'\cap\tM$;
\item
$(\ps\ops)_{\rm m}=\id$ in $\Sect(M)$,
that is,
the modular part of $\ps\ops$ is trivial.
\end{enumerate}
\end{lem}

\begin{lem}\label{lem:rho123}
Let $\rho_1,\rho_2,\rho_3$ be irreducible endomorphisms on $M$
with finite index.
If two of them and $\rho_1\rho_2\rho_3$ are modular,
then so is the other one.
\end{lem}
\begin{proof}
This is obvious from \cite[Proposition 3.4 (1)]{Iz}.
\end{proof}

Then we can deduce the modular part of $\rho\orho$
is equal to a canonical endomorphism as below.

\begin{lem}\label{lem:rho-mod-part}
Let $\rho$ be an irreducible endomorphism satisfying
the conditions stated in Theorem \ref{thm:decomp} (3),
and
$\rho=\si\ps$ the decomposition given there. 
Then $(\ps\ops)_{\rm m}=\id$ in $\Sect(M)$,
and $(\rho\orho)_{\rm m}=\si\osi$.
In particular, $(\rho\orho)_{\rm m}$ is a canonical endomorphism.
\end{lem}
\begin{proof}
By Theorem \ref{thm:decomp},
we can find a modular endomorphism $\si$
and an endomorphism $\ps$ such that $\rho=\si\ps$
and $\tps(\tM)'\cap\tM=Z(\tM)$.
Hence we get $\rho\orho=\si\ps\ops\osi$.
Employing Lemma \ref{lem:ps-ops} and Lemma \ref{lem:rho123},
we obtain $(\ps\ops)_{\rm m}=\id$ and
$(\rho\orho)_{\rm m}=\si(\ps\ops)_{\rm m}\osi=\si\osi$.
\end{proof}

Now we will prove the following main result of this section.

\begin{thm}\label{thm:rho-orho}
Let $M$ be an infinite factor
and $\rho\in\End(M)_{0}$.
Then the following statements are equivalent:
\begin{enumerate}
\item
$\rho\orho$ and $\orho\rho$ are modular;
\item
$\rho\orho$ is modular,
and $\rho$ has the Connes--Takesaki module;

\item
There exist a modular endomorphism $\si$
and an automorphism $\ps$ on $M$
satisfying $\rho=\si\ps$;
\item
$\rho$ has the Connes--Takesaki module
and $\tM=\trho(\tM)\vee(\trho(\tM)'\cap\tM)$.
\end{enumerate}
In the above, $\ps$ is uniquely determined up to an extended modular automorphism.
\end{thm}
\begin{proof}
(1)$\Rightarrow$(2).
Since the canonical extensions of
$\rho\orho$ and $\orho\rho$ are inner,
they trivially act on $Z(\tM)$.
Then it turns out that $\rho$ has the Connes--Takesaki module.

(2)$\Rightarrow$(3).
We first show the statement for $\rho$ being irreducible.
Since $\rho\orho$ is modular from our assumption,
we have isometries $U_r\in\tM$ such that $\sum_{r=1}^s U_rU_r^*=1$
and $\wdt{\rho\orho}(x)=\sum_{r=1}^s U_rxU_r^*$. 

Let $N:=\rho(M)$.
Recall that the inclusion $Z(\tM)\subs\tN'\cap\tM$
is with finite index since it is anti-isomorphic
to $Z(\widetilde{\orho}(\tM))\subs \widetilde{\orho}(\tM)'\cap\tM$.
Using a disintegration $\tN'\cap\tM$ over $Z(\tM)$
and the ergodicity of $\th$ on $Z(\tM)$,
we see there exist (non-unital)
type I$_n$ factors $A_1,\dots,A_m$ in $\tN'\cap\tM$
such that $\tN'\cap \tM=Z(\tM)\vee \bigoplus_{i=1}^m A_i$,
which gives a tensor product decomposition.
Let $\{e_{jk}^{(i)}\}_{j,k=1}^n$ be a matrix unit of $A_i$.
Then $e_{11}^{(i)}$ is a properly infinite projection
with central support 1 in $\tM$,
and we can take an isometry $V_1^{(i)}\in \tM$
such that $e_{11}^{(i)}=V_1^{(i)}(V_1^{(i)})^*$.
We let $V_j^{(i)}:=e_{j1}^{(i)}V_1^{(i)}$
and $\be^{(i)}(x)=(V_1^{(i)})^* \trho(x)V_1^{(i)}$,
which is an irreducible and co-irreducible endomorphism
because $\trho$ has the Connes--Takesaki module
(see Definition \ref{defn:irr-co}).
Then for all $x\in\tM$ we have
\[
\trho(x)=\sum_{i=1}^m \sum_{j=1}^n V_j^{(i)}\be^{(i)}(x)(V_j^{(i)})^*.
\]
We show $\cE(\tM,\be^{(i)}(\tM))\neq\emptyset$ for all $i$.
For each $i$,
let us introduce the map $F^{(i)}\col \tM\ra \be^{(i)}(\tM)$ as follows
\[
F^{(i)}(x):=(V_1^{(i)})^* E_{\trho}(V_1^{(i)}x(V_1^{(i)})^*)V_1^{(i)},
\quad x\in\tM.
\]
Then $F$ is a faithful normal completely positive map satisfying
$F^{(i)}(\be^{(i)}(x))=\be^{(i)}(x)F(1)$, where $F^{(i)}(1)\in Z(\tM)$.
Since $E_\rho$ satisfies the Pimsner--Popa inequality,
so does $E_{\trho}$.
Thus
\[
F^{(i)}(1)
=
(V_1^{(i)})^* E_{\trho}(V_1^{(i)}(V_1^{(i)})^*)V_1^{(i)}
\geq
d(\rho)^{-1}
(V_1^{(i)})^* \cdot V_1^{(i)}(V_1^{(i)})^*\cdot V_1^{(i)}
=d(\rho)^{-1}.
\]
This shows $F^{(i)}(1)$ is invertible,
and the normalized map $F^{(i)}(1)^{-1}F$
gives a faithful normal conditional expectation
from $\tM$ onto $\be^{(i)}(\tM)$.

If we disintegrate $\wdt{\orho}(\tM)'\cap\tM$ over $Z(\tM)$ as before,
then we have a similar decomposition as follows:
\[
\wdt{\orho}(x)=\sum_{k=1}^p
\sum_{\el=1}^q W_\el^{(k)}\de^{(k)}(x)(W_\el^{(k)})^*,
\]
where $W_{\el}^k$ is an isometry in $\tM$, and $\de^{(k)}$ is
an irreducible and co-irreducible endomorphism on $\tM$.

Let $z\in Z(\tM)$ be a non-zero projection.
we claim that $(\id,\be^{(i)}\de^{(k)})z\neq 0$
for all $i,k$.
If not, then for some $i,k$,
$\be^{(i)}((W_\el^{(k)})^*)(V_j^{(i)})^*U_rz=0$ for all $j,\el,r$.
However the centrality of $z$ immediately implies $z=0$,
and this is a contradiction.
Thus we see that $(\id,\be^{(i)}\de^{(k)})$ contains
an isometry by Zorn's lemma.
Employing Theorem \ref{thm:isom-conj},
we see $\de^{(k)}$ coincides with $\ovl{\be^{(i)}}$
as a sector of $\tM$.
In particular, $\be^{(i)}$'s are mutually equivalent,
but this is possible only when $m=1$.
Hence $\trho(\tM)'\cap\tM=Z(\tM)\vee A_1$,
and $Z(\trho(\tM)'\cap\tM)=Z(\tM)$.

Let $\rho=\si\ps$ be a decomposition
as given in Theorem \ref{thm:decomp} (3).
Since $\rho\orho=(\rho\orho)_{\rm m}$,
it turns out that
$\ps\ops=(\ps\ops)_{\rm m}$ from Lemma \ref{lem:rho-mod-part}.
By Lemma \ref{lem:ps-ops},
we see $\ps\ops$ is an inner automorphism,
and $\ps$ is an automorphism.

Next we let $\rho$ be a not necessarily irreducible endomorphism
which satisfies the condition of (2).
Then it turns out that each irreducible
contained in $\rho$
satisfies the condition in (2).
Hence $\rho$ is of the form $\oplus_{i=1}^n\si_i\ps_i$,
where $\si_i$'s are modular endomorphisms
and $\ps_i$'s are automorphisms.
However, $\rho\orho$ contains $\si_1\ps_1\ps_i^{-1}\osi_i$,
which is modular by our assumption.
Thus by \cite[Proposition 3.4 (1)]{Iz}, each $\ps_i$ is equal to $\ps_1$
up to an extended modular automorphism.
Hence $\rho=\si\ps_1$ for some modular endomorphism $\si$.

(3)$\Rightarrow$(4).
Since $\ps$ is an automorphism,
we have $\trho(\tM)=\tsi(\tM)$.
The statement (4) follows from the equality
$\tM=\tsi(\tM)\vee (\tsi(\tM)'\cap\tM)$.
It is trivial that $\mo(\rho)=\mo(\ps)$.

(4)$\Rightarrow$(3).
We first show the statement (3) when $\rho$ is irreducible.
From the condition of (4), we have $Z((\trho,\trho))=Z(\tM)$.
By Theorem \ref{thm:decomp},
we have a decomposition $\rho=\si\ps$,
where $\si$ is modular and $\ps$ satisfies
$\tps(\tM)'\cap \tM=Z(\tM)$.
Take a finite family of isometries $\{V_i\}_{i=1}^n$ in $\tM$
implementing $\tsi$.
Then we have
\[
\trho(\tM)'\cap\tM=\tsi(\tps(\tM)'\cap\tM)\vee\{V_iV_j^*\}_{i,j}''
=
Z(\tM)\vee\{V_iV_j^*\}_{i,j}''.
\]
This implies
$\tM=\trho(\tM)\vee(\trho(\tM)'\cap\tM)
=\trho(M)\vee\{V_iV_j^*\}_{i,j}''$.
Hence for any $x\in \tM$,
$\tsi(x)$ belongs to $\trho(\tM)\vee\{V_iV_j^*\}_{i,j}''$,
that is, $\tsi(x)\in\trho(\tM)=\tsi\tps(\tM)$, and $x\in \tps(\tM)$.
Therefore, $\tps$ is an automorphism.

Next we treat a general $\rho$.
Let $\rho_1$ be an irreducible summand contained in $\rho$.
Then $\rho_1$ has the Connes--Takesaki module
and satisfies
$\tM=\trho_1(\tM)\vee (\trho_1(\tM)'\cap\tM)$.
Thus each irreducible in $\rho$
is a composition of a modular endomorphism
and an automorphism.

Let $\rho=\oplus_i m_i\si_i\ps_i$
be an irreducible decomposition,
where the integer $m_i$ is the multiplicity,
and $\si_i$ and $\ps_i$ are a modular endomorphism
and an automorphism, respectively,
such that $\si_i\ps_i$ and $\si_j\ps_j$
are not equivalent if $i\neq j$.
Assume that
$\ps_i\neq\ps_j\bmod\Aut(M)_{\rm m}$
for some $i\neq j$.
Let $z_i$ be the central projection in $(\rho,\rho)$
corresponding to $m_i\si_i\ps_i$.
Then $z_i\in (\trho,\trho)$.
Since $(\tsi_i\tps_i,\tsi_j\tps_j)=0$ if $i\neq j$
by \cite[Proposition 3.4 (1)]{Iz},
it turns out that $z_i\in Z((\trho,\trho))$.
The assumption of (4)
implies $z_i\in Z(\tM)$.
However, $z_i$ is contained in $M$, and $z_i$
must be equal to 1.
This is a contradiction.
Hence all $\ps_i$'s are equivalent,
and $\rho$ is a composition of a modular endomorphism
and an automorphism.

(3)$\Rightarrow$(1).
Since $\rho\orho=\si\ps\ps^{-1}\osi=\si\osi$
and $\orho\rho=\ps^{-1}\osi\si\ps$ in $\Sect(M)$,
it turns out that $\rho\orho$ and $\orho\rho$ are modular.
\end{proof}


\begin{rem}
We can show this theorem
for a type III$_\la$ ($0<\la\leq1$) factor
and an irreducible endomorphism
by another direct approach. 
When $M$ is of type III$_1$,
we know that $\Ga:=\{t\in\R\col\si_t^\vph\prec\rho\orho\}$
is a subgroup because of Frobenius reciprocity.
In particular, if $0\neq t$ would belong to $\Ga$,
the group $\Z t$ would be also contained in $\Ga$.
However, $\Z t\in s\mapsto [\si_s^\vph]\in\Out(M)$
is injective because Connes' $T$-set $T(M)$ equals $0$.
Hence we would have to obtain $d(\rho\orho)=\infty$,
but this is a contradiction.
Hence $\rho\orho=\id$, and $\rho$ is an automorphism.

When $M$ is of type III$_\la$ with $0<\la<1$,
the canonical endomorphism $\rho\orho$ is decomposed
to the direct sum of
$\id_M, \si_{t_1}^\vph,\dots,\si_{t_n}^\vph$,
where $\vph$ is a periodic state
with period $T:=-2\pi/\log\la$.
By \cite[Theorem 4.1]{Iz-f}, 
$\rho(M)=M^\alpha$ for some free action $\alpha$ of $\Z/n\Z$.
%
%
However, it follows that $M^\al$ is of type III$_{\la^n}$.
Thus we must get $n=1$ since $\rho(M)$ is of type III$_\la$.
This yields $\rho\orho=\id$, and $\rho$ is an automorphism.
\end{rem}

\subsection{Extended modular automorphisms}\label{subsect:extend}
Let $\vph$ be a dominant weight on an infinite factor $M$
and $\theta^0_t$ a trace scaling
one-parameter automorphism group on $M_\varphi$,
and $M=M_\varphi\rtimes_{\theta^0}\mathbb{R}$
the continuous decomposition.
Let $K:=L^2(M_\varphi)$ be the standard Hilbert space, and set
$\wdt{K}:=K\otimes L^2(\R)$.
Define an action $\pi^0$ of $M_\varphi$, and  one-parameter unitary groups
$\lambda^l_t$, $\mu_p$, $\lambda^r_t$ on $\wdt{K}$
by the following formulae:
\[\left(\pi^0(x)\xi\right)(s)=\theta_{-s}^0(x)\xi(s),
\quad
\left(\lambda_t^l\xi\right)(s)=\xi(s-t),
\]
\[
\left(\mu_p\xi\right)(s)=e^{-ips}\xi(s),
\quad
\left(\lambda^r_t\xi\right)(s)=\xi(s+t)
\quad
\mbox{for }x\in M_\vph,\ \xi\in \wdt{K}.
\]
By Takesaki duality,
$M_\varphi\rtimes_{\theta^0}\R$
and
$\left(M_\varphi\rtimes_{\theta^0}\R\right)\rtimes_{\sigma^\varphi}\R$
are isomorphic to
$\pi^0(M_\varphi)\vee \{\lambda^l_t\}_{t\in \R}$
and
$\pi^0(M_\varphi)\vee \{\lambda^l_t\}_{t\in \R}\vee\{\mu_p\}_{p\in \R}$,
respectively, and the second dual action $\hat{\hat{\theta^0}}_t$ is given
by $\theta_t^0\otimes \Ad \lambda^r_t$.
With this identification,
the core $\tM$ is equal to
$\pi^0(M_\varphi)\vee \{\lambda^l_t\}_{t\in \R}\vee\{\mu_p\}_{p\in \R}$,
and the dual flow $\th_t$ is $\theta_t^0\otimes \Ad \lambda^r_t$.
%
%

For a $\th^0$-cocycle $c_t \in U(Z(M_\varphi))$,
we define the unitary $V_c$ by 
\[
(V_c\xi)(s):=c_{-s}^*\xi(s)
\quad
\mbox{for }\xi\in\wdt{K}.
\]
Then clearly, $V_c$
belongs to $Z(M_\vph)\vee L^\infty(\R)$.
Hence $V_c\in\tM$.
Let us define $\si_c^\vph\in\Aut(M)$ by
\[
\pi^0(\si_c^\vph(x))=V_c \pi^0(x) V_c^*
\quad
\mbox{for all }x\in M.
\]
This is the original definition of the extended modular automorphism
(see Section \ref{subsection:modular-endo} and \cite[XII, Theorem 1.10]{Ta}).

\begin{lem}\label{lem:Vc-th}
Let $c,c'\in Z_{\th^0}^1(\R,U(Z(M_\vph)))$.
Then the following hold:
\begin{enumerate}
\item
$\wdt{\sigma^\varphi_c}=\Ad V_c$ on $\tM$;
\item
$V_c^*\theta_t(V_c)=c_t\otimes 1$ for all $t\in\R$;
\item
$V_{c}V_{c'}=V_{cc'}=V_{c'}V_{c}$, $V_c^*=V_{c^*}$;
\item
If $c_t=u\theta_t^0(u^*)$ for some $u\in U(Z(M_\varphi))$,
then $V_{c}=\pi^0(u)(u^*\otimes 1)$.
\end{enumerate}
\end{lem}
\begin{proof}
(1).
We verify that
$\Ad V_c=\id$ on $\pi^0(M_\varphi)$,
and $\Ad V_c(\lambda^0_t)=\pi^0(c_t)\lambda^l_t $
for all $t\in\R$.
For $x\in M$, we have the following.
\begin{align*}
\left(V_c\pi^0(x)V_c^*\xi\right)(s)
&=
c_{-s}^* \left(\pi^0(x)V_c^*\xi\right)(s)
=
(c_{-s})\theta_{-s}(x) \left(V_c^*\xi\right)(s)
\\
&=
c_{-s}\theta_{-s}(x) c_{-s}^*\xi(s)
=
\theta_{-s}^0(x) \xi(s)
\\
&=
\left(\pi^0(x)\xi\right)(s).
\end{align*}
The second equality is verified as follows.
\begin{align*}
\left(V_c\lambda^0_tV_c^*\xi\right)(s)
&=
c_{-s}^* \left(\lambda^0_tV_c^*\xi\right)(s)
=
c_{-s}^* \left(V_c^*\xi\right)(s-t)
\\ 
&=
c_{-s}^*c_{t-s} \xi(s-t)
=
c_{-s}^*c_{-s}\theta_{-s}(c_t) \xi(s-t)
\\
&=
\theta_{-s}(c_t) \xi(s-t)
=
\left(\pi^0(c_t)\lambda^l_t\right)(s). 
\end{align*}

Since $\mu$ is an
implementing unitary for $\sigma^\varphi$, and $\varphi$ is
$\sigma^\varphi_c$-invariant, we have to show $\Ad V_c(\mu_p)=\mu_p$.
This is trivial
since $V_c$ is belonging to $B(K)\oti L^\infty(\R)$.

(2).
We obtain $V_c^*\theta_t(V_c)=c_t$ as follows:
\begin{align*}
(V_c^*\th_t(V_c)\xi)(s)
&=
c_{-s}\left(\th_{t}(V_c)\xi\right)(s)
=
c_{-s}\th_{t}^0(c_{-(s+t)}^*)\xi(s)
\\ 
&=
c_{-s}\th_{t}^0(\th_{-t}(c_{-s})c_{-t}^*)\xi(s)
=
\th_{t}^0(c_{-t}^*)\xi(s)
=
c_t\xi(s).
\end{align*}

(3). This is trivial.

(4). By the definition of $V_c$, 
\begin{align*}
\left(V_c\xi\right)(s)
&=c_{-s}^*\xi(s,t)
=\theta_{-s}^0(u)u^*\xi(s)
\\
&=\th_{-s}^0(u)
\left((u^*\otimes 1)\xi\right)(s)
=
\left(\pi^0(u)(u\otimes 1)\xi\right)(s).
\end{align*}

\end{proof}

Note that
$V_c$'s are generating a commutative von Neumann algebra
in $\tM$.

\section{Kac algebras}

We briefly summarize the terminology and the notation
about Kac algebras which are used in \cite{MT1}.
Our standard references are \cite{BaSk,ES,MT1,Wor}.

\subsection{Compact or discrete Kac algebras}

Let $\bG$ be a \emph{compact quantum group}.
\index{quantum group!compact--}
Namely,
$\bG$ is a pair of a unital C$^*$-algebra $C(\bG)$
and a coproduct $\de$ that is a unital faithful $*$-homomorphism
from $C(\bG)$ into $C(\bG)\oti C(\bG)$
such that $(\de\oti\id)\circ\de=(\id\oti\de)\circ\de$.
We moreover assume the cancellation property,
that is,
the subspaces $\de(C(\bG))(C(\bG)\oti\C)$
and $\de(C(\bG))(\C\oti C(\bG))$ are norm dense
in $C(\bG)\oti C(\bG)$.
Thanks to Woronowicz's result,
there exists a state $h$,
which is called the \emph{Haar state}
\index{Haar state}
such that
\[
(\id\oti h)(\de(x))=h(x)1=(h\oti\id)(\de(x))
\quad
\mbox{for all }x\in C(\bG).
\]
We will assume that $h$ is faithful and tracial in this paper.
Then $\bG$ is said to be of \emph{Kac type}
\index{quantum group!of Kac type}
or a \emph{compact Kac algebra}.
\index{Kac algebra!compact--}

Let $L^2(\bG)$ be the GNS Hilbert space of $\CG$ associated with $h$
and $\Om_h$ the GNS cyclic vector.
The weak closure of $\CG$ is denoted by $\lG$.
Then the multiplicative unitary $V$ is defined by
\[
V^*(x\Om_h\oti y\Om_h)=\de(y)(x\Om_h\oti\Om_h)
\quad
\mbox{for }x,y\in\CG.
\]
Thanks to the double commutant theorem,
$V$ belongs to $\lG\oti B(\ltG)$.
The coproduct $\de$ extends to $\lG$
by setting $\de(x)=V^*(1\oti x)V$ for $x\in\lG$.
Moreover, $V$ satisfies the following pentagon equation:
\[
V_{12}V_{13}V_{23}=V_{23}V_{12}.
\]
From this, the set
$\lhG:=\ovl{\spa}^{\rm w}\{(\om\oti\id)(V)\mid \om\in \lG_*\}$
is a von Neumann algebra with a coproduct $\De$
defined by
$\De(x):=V(x\oti 1)V^*$ for $x\in\lhG$.
If we restrict the canonical trace $\Tr$ of $B(\ltG)$ on $\lhG$,
we obtain the semifinite weight $\vph$,
which satisfies the following bi-invariance:
\[
(\id\oti\vph)(\De(x))=\vph(x)=(\vph\oti\id)(\De(x))
\quad\mbox{for all }x\in \lhG_+.
\]
Then we call $\bhG:=(\lhG,\De,\vph)$ the dual discrete Kac algebra
or the dual discrete quantum group of Kac type.
\index{quantum group!discrete--}
Note that any discrete Kac algebra
\index{Kac algebra!discrete--}
arises like this
due to the duality.

Next we review the representation theory of $\bG$.
A \emph{unitary representation}
\index{unitary representation!of a compact Kac algebra}
of $\bG$ on a Hilbert space $H$ means
a unitary $v\in B(H)\oti\lG$ such that $(\id\oti\de)(v)=v_{12}v_{13}$.
Let $v$ and $w$ be unitary representations on Hilbert spaces $H$ and $K$,
respectively.
We associate with them the intertwiner subspace $(v,w)$ of $B(H,K)$,
that is, $S\in (v,w)$ if and only if $(S\oti1)v=w(S\oti1)$.
If $(v,v)=\C$,
then $v$ is said to be irreducible.
Then Peter--Weyl theory still holds as in the case of a compact group.
Let $\IG$ be a complete set of unitary equivalence classes
of irreducible representations of $\bG$.
For each $\pi\in\IG$,
we choose the corresponding representation $v(\pi)$ on $H_\pi$.
We fix an orthonormal base $\{\vep_{\pi_i}\}_{i\in I_\pi}$ of $H_\pi$.
The $(i,j)$-matrix element of $v(\pi)$ is denoted by $v(\pi)_{ij}$.
The \emph{conjugate unitary representation}
\index{unitary representation!conjugate--}
of $\pi$
is defined by
$v(\opi)_{i,j}=v(\pi)_{i,j}^*$.

For $\pi,\rho\in\IG$, we define the \emph{tensor product representation}
\index{unitary representation!tensor product--}
$\pi\rho$ by the unitary $v(\pi)_{13}v(\rho)_{23}$.
Then if $\rho$ equals $\opi$, that is, $\rho$ is conjugate to $\opi$,
then $\pi\opi$ contains the trivial representation $\btr$
with multiplicity 1.
Indeed, the following isometry
$T_{\pi,\opi}$ is intertwining $\btr$ and $\pi\opi$:
\[
T_{\pi,\opi}
=\sum_{i\in I_\pi}\frac{1}{\sqrt{d(\pi)}}\vep_{\pi_i}\oti\vep_{\opi_i},
\]
where $d(\pi)$ denotes the dimension of $H_\pi$,
and the coupling of $\vep_{\pi_i}$ and $\vep_{\opi_j}$ equals $\de_{i,j}$.

Let $\pi,\rho,\si\in\IG$
and $S\col H_\si\ra H_\pi\oti H_\rho$ be an intertwiner.
If we want to specify the representation spaces,
we use the notation $S_{\pi,\rho}^\si$.
The matrix element of $S_{\pi,\rho}^\si$
with respect to the vectors $\vep_{\pi_i}\oti\vep_{\rho_j}$ and $\vep_{\si_k}$
is denoted by $S_{\pi_i,\rho_j}^{\si_k}$.
Note that
$(\si,\pi\rho)$ is a Hilbert space with the inner product $(S,T):=T^*S$.
We will fix an orthonormal base from now on and denote the family
by $\ONB(\si,\pi\rho)$.

\subsection{Generalized central quantum subgroups}
\label{subsect:GCQS}




Most of the results obtained in this and the next subsections
are probably well-known for experts.
We, however, present their proofs for reader's convenience.
Let us begin with the notion of a quantum subgroup.

\begin{defn}
Let $\bG=(\CG,\de_\bG)$ and $\bK=(\CK,\de_\bK)$
be compact Kac algebras.
Then $\bK$ is called a \emph{quantum subgroup}
\index{quantum subgroup!of a compact Kac algebra}
of $\bG$
when there exists a surjective
$*$-homomorphism $r\col \CG\ra\CK$
such that $(r\oti r)\circ\de_\bG=\de_\bK\circ r$.
The map $r$ is called a \emph{restriction map}.
\index{restriction map}
\end{defn}

Readers should note that the restriction map is not uniquely
determined from $\bK$ and $\bG$ in general.
However, we often abbreviate it for simplicity,
and the situation is denoted by $\bK\subs\bG$.

For a quantum subgroup $\bK$ as above,
we define the restriction of a representation $v_\pi$, $\pi\in\IG$
as follows:
\[
v_{\pi|_\bK}:=(\id\oti r)(v_\pi).
\]
Then the surjectivity of $r$ implies that
the matrix entries of $v(\pi|_\bK)$, $\pi\in\IG$,
are densely spanning $\CK$.

The quantum subgroup
$\bK$ acts on $\CG$ from the left and right as follows:
\[
\ga^\el:=(r\oti \id)\circ\de_\bG,
\quad
\ga^r:=(\id\oti r)\circ\de_\bG.
\]
The fixed point algebras of $\ga^\el$ and $\ga^r$
are denoted by $C(\bK\backslash\bG)$ and $C(\bG/\bK)$,
respectively.
Then the normality of quantum subgroup is introduced
as follows.

\begin{defn}
Let $\bK$ be a quantum subgroup of $\bG$.
We will say that $\bK$ is \emph{normal}
\index{quantum subgroup!normal--}
if $C(\bG/\bK)=C(\bK\backslash\bG)$.
\end{defn}

The following result may be well-known to experts.

\begin{lem}\label{lem:K-normal}
A quantum subgroup $\bK$ of $\bG$
is normal
if and only if
$C(\bG/\bK)$ is a Kac subalgebra.
\end{lem}
\begin{proof}
The ``only if'' part is trivial.
We show the ``if'' part.
Take $\pi\in \IG$ such that $C(\bG/\bK)_\pi\neq0$.
Then $\pi|_\bK$ contains the trivial representation.
Thus we may and do assume $I_\pi$ has the decomposition
$I_\pi=I_\pi^\btr\sqcup J_\pi$
such that the vectors $\{\vep_{\pi_i}\}_{i\in I_\pi^\btr}$
span the trivial representation space of $\bK$.
Let $i\in\IG$ and $j\in I_\pi^\btr$.
Then we have
$\de(v_{\pi_{ij}})=\sum_{k}v_{\pi_{ik}}\oti v_{\pi_{kj}}$.
Since
$\de(C(\bG/\bK))\subs C(\bG/\bK)\oti C(\bG/\bK)$,
we have
$v_{\pi_{ik}}\in C(\bG/\bK)$ for all $k\in I_\pi$.
In the end,
$I_\pi$ coincides with $I_\pi^\btr$,
that is, $\pi|_\bK$ is trivial.
This implies the normality of $\bK$.
\end{proof}

We consider a special kind of a normal quantum subgroup.

\begin{defn}
Let $\bK\subs \bG$ be a normal quantum subgroup.
If the restriction $\pi|_\bK$ decomposes
to only one one-dimensional representation of $\bK$
for any $\pi\in\IG$,
then we will say
that $\bK$ is a \emph{generalized central quantum subgroup}.
\index{quantum subgroup!generalized central--}
\end{defn}

Indeed,
if $\bG$ is a compact group,
then such normal subgroup must be contained in $Z(\bG)$.
Let $\bK$ be a generalized central quantum subgroup of $\bG$
with restriction map $r_\bK\col\CG\ra\CK$.
By definition, there exists a unitary $g_\pi\in\CK$
such that
\[
(\id\oti r_\bK)(v_\pi)=1\oti g_\pi.
\]
The $g_\pi$ is a one-dimensional representation of $\bK$
which satisfies $g_\opi=g_\pi^*$.
Since all irreducible representations come from restrictions
of $\pi$'s,
it turns out that $\bK$ is co-commutative.
Thus the dual of $\bK$ is a discrete group.

\subsection{Normal subcategory}

\begin{defn}\label{defn:normal-subcat}
Let $\bG$ be a compact Kac algebra.
\begin{enumerate}
\item
Let $\La$ be a subset of $\IG$.
Suppose that $\La$ is closed under conjugation and product,
that is,
if $\rho_1,\rho_2\in\La$,
then $\orho_1$ and each irreducible in $\rho_1\rho_2$
still belongs to $\La$.
In this case, we will say
that $\La$ is a \emph{subcategory}
\index{subcategory!of $\IG$}of $\IG$;

\item
A subset $\La\subs \IG$ is said to be \emph{normal}
\index{subcategory!normal--}
if $\btr\in\La$ and
each irreducible in $\pi\rho\opi$ is contained in $\La$
for all $\pi\in\IG$ and $\rho\in\La$.
\end{enumerate}
\end{defn}

Let $\bK\subs\bG$ be a generalized central quantum subgroup,
and $\La_\bK\subs \IG$ the collection of $\rho\in\IG$
such that $\rho|_\bK$ is a trivial representation.

\begin{thm}\label{thm:corr-gc-nc}
The map $\bK\mapsto \La_\bK$
is a bijective correspondence between
the set of generalized central quantum subgroups of $\bG$,
and the set of normal subcategories of $\IG$.
\end{thm}
\begin{proof}
Let $\bK\subs\bG$ be a generalized central quantum subgroup
with the restriction map $r_\bK\col\CG\ra\CK$.
First observe the map given above is well-defined.
Indeed, let $\pi\in\IG$ and $\rho\in\La_\bK$.
Then we have
\[
(\id_\pi\oti\id_\rho\oti\id_\opi\oti r_\bK)
(v_\pi v_\rho v_\opi)
=g_\pi\cdot 1\cdot g_\opi=1.
\]
which shows $\La_\bK$ is normal.
By similar computation, it is easy to see that
$\La_\bK$ is a subcategory of $\IG$.

Next we will construct the inverse map.
Let $\La$ be a normal subcategory of $\IG$.
Then
the linear span of $C(\bG)_\rho$, $\rho\in\La$
is a $*$-algebra.
We denote it by $C(\bH_\La)$, that is,
\[
C(\bH_\La)=\ovl{\spa}^{\|\cdot\|}
\{v(\rho)_{ij}\mid \rho\in\La, i,j\in I_\rho\}.
\]
Then clearly $C(\bH_\La)$ is a compact Kac subalgebra.
We show there exists a quantum subgroup $\bK_\La$
of $\bG$ such that $C(\bH_\La)=C(\bG/\bK_\La)$.
Thanks to \cite[Theorem 3.17]{T2},
it suffices to check that $L^\infty(\bH_\La)$
has the expectation property and the coaction symmetry.
Since $\lG$ is finite,
the expectation property automatically holds
(this is still true for a compact quantum group
because of the global invariance of the modular group $\si^h$).
The coaction symmetry means
for any $\pi\in\IG$, $i,j\in I_\pi$,
$\rho\in\La$ and $p,q\in I_\rho$,
the element $\sum_{k\in I_\pi}v_{\pi_{ki}}^*v_{\rho_{pq}}v_{\pi_{kj}}$
belongs to $C(\bH_\La)$.
Indeed, this is trivial
because the normality of $\La$ implies
the matrix entries of $v_\opi v_\rho v_\pi$
are contained in $C(\bH_\La)$.

By Lemma \ref{lem:K-normal},
$\bK_\La$ is normal, and $\rho|_{\bK_\La}$ is trivial for all $\rho\in\La$.
Let $r_{\bK_\La}\col \CG\ra C(\bK_\La)$ be the restriction map.
The normality of $\La$ implies that for $\pi\in\IG$ and $\rho\in\La$,
\[
(\id_\pi\oti\id_\rho\oti\id_\opi\oti r_{\bK_\La})(v_\pi v_\rho v_\opi)=1
.\]
It turns out from this equality
that there exists a unitary $g_\pi\in C(\bK_\La)$
such that $(\id\oti r_{\bK_\La})(v_\pi)=1_\pi\oti g_\pi$,
which means $\bK_\La$ is a generalized central quantum subgroup of $\bG$. 
It immediately follows that the map $\bK\mapsto\La_\bK$
is a bijection.
\end{proof}

\begin{ex}\label{ex:conn-simple-normal}
If $\bG$ is a connected simple Lie group,
then any subcategory of $\IG$ is normal.
Indeed, let $\La\subs \IG$ be a subcategory.
Then it corresponds to the quotient
by a normal closed subgroup $\bK_\La$ of $\bG$
such that
\[
C(\bG/\bK_\La)
=\ovl{\spa}^{\|\cdot\|}\{v_{\pi_{i,j}}\mid\pi\in\La,i,j\in I_\pi\},
\]
that is, $\bK_\La$ is nothing but the intersection
of $\ker(\pi)$ for all $\pi\in\La$.
Since the Lie group $\bG$ is connected and simple,
any normal subgroup is central.
This implies that $\La$ is normal.

Next let us consider the $q$-deformation $\bG_q$
of such $\bG$ with $q\neq1$.
Then the representation category is unchanged,
and any subcategory of $\bG_q$ is still normal.
\end{ex}

\subsection{2-cocycles of a discrete Kac algebra}\label{sec:2-cocycle}
In Section \ref{sec:modular},
we will study an action $\al$ such that
all $\al_\pi$'s are modular.
Such actions are related with
2-cocycles of $\bhG$.
In this subsection,
we summarize some results on 2-cocycles.
Our standard references are \cite{IK,Wass-cptII}.

\begin{defn}\label{df:2-cocycle}
Let $M$ be an abelian von Neumann algebra.
\begin{enumerate}
\item
A unitary $\om\in M\otimes \lhG\otimes \lhG$ is called an
$M$-\emph{valued  2-cocycle}
\index{2-cocycle!$M$-valued--}
if
\[
\om_{123}(\id_M \otimes \Delta\otimes \id)(\om)=
\om_{134}(\id_M\otimes\id\otimes \Delta)(\om),
\]
and $\om_{\cdot,\btr}=\om_{\btr,\cdot}=1$.
We will denote by $Z^2(\bhG,M)$ the set of all $M$-valued 2-cocycles.

\item
Let $\om^1$, $\om^2 \in Z^2(\bhG,M)$.
We will say that $\om^1$ and $\om^2$ are \emph{equivalent}
\index{2-cocycle!equivalent--}
if $\om^2=u_{12}u_{13}\om^1(\id_M\otimes \Delta)(u^*)$
for some unitary $u\in M\otimes \lhG$,
and then we write $\om^1\sim \om^2$.
By $H^2(\bhG,M)$, we denote
the set of all equivalence classes.

\item
We will say $\om\in Z^2(\bhG,M)$ is \emph{normalized}
\index{2-cocycle!normalized--}
if $\om(1\oti\delta(e_\btr))=1\oti\delta(e_\btr)$.

\item
We will say $\om\in Z^2(\bhG,\mathbb{C})$ is \emph{co-commutative}
\index{2-cocycle!co-commutative--}
if $\Delta^\om:=\Ad \om\circ\Delta$ is co-commutative,
that is, $\Delta^\om=F\circ \Delta^\om$,
where $F$ denotes the flip automorphism
on $\lhG\oti\lhG$.
\end{enumerate}
\end{defn}

In (3), we note that $(x\oti1)\De(e_\btr)=(1\oti \ka(x))\De(e_\btr)$
for $x\in\lhG$,
where $\ka$ denotes the antipode of $\lhG$.
Thus our definition is the same as \cite[Definition 13.1]{IK}.

When
$M=\mathbb{C}$ in Definition \ref{df:2-cocycle},
we will call $\om$ a 2-cocycle shortly.
The set of 2-cocycles is denoted by $Z^2(\bhG)$.
Let $Z^2(\bhG)_c$ and $Z^2(\bhG)_n$ be
the sets of co-commutative 2-cocycles
and normalized 2-cocycles, respectively.
We set $Z^2(\bhG)_{c,n}:=Z^2(\bhG)_c\cap Z^2(\bhG)_n$.

\begin{rem}
Let $\om\in Z^2(\bhG)$.
Then $\Delta^\om=\Ad \om\circ\Delta$ gives
a new coproduct on $\lhG$.
Moreover if $\om$ is normalized,
then we can show the weight $\vph$ is still invariant for $\De^\om$.
Thus $\bhG^\om:=(\lhG, \Delta^\om, \varphi)$ becomes a Kac algebra
(see Section \ref{subsect:cocycle-Kac} for the proof).
Note that if $\om\in Z^2(\bhG)_n$, then $\om^*\in Z^2(\bhG^\om)_n$.
\end{rem}

\begin{defn}
The \emph{diagonal}
\index{diagonal of 2-cocycle}
$v(\om)$ of $\om\in Z^2(\bhG,M)$
is the element of $M\oti\lhG$
such that
\[
(v(\om)\oti1)(1\oti\De(e_\btr))=\om(1\oti\De(e_\btr)).
\]
By definition,
we have $v(\om)=(\id\oti m\circ(\id\oti \ka))(\om))$,
where $m$ denotes the linear map $\lhG\oti\lhG\ni a\oti b\to ab\in \lhG$.
\end{defn}

\begin{lem}
Let $\om\in Z^2(\bhG,M)$
and $u\in M\oti\lhG$ a unitary.
Let $\om^u:=(u\oti u)\om \De(u^*)$.
Then $v(\om^u)=uv(\om)\ka(u)$.
\end{lem}

The next result has been proved in the co-commutative case in
\cite[Lemma 13]{Wass-cptII},
and in the finite dimensional case in \cite[Lemma 13.2]{IK},
but we present a proof for readers' convenience.

\begin{lem}
Let $M$ be a commutative von Neumann algebra.
Let $\om$ be an $M$-valued 2-cocycle of $\bhG$.
Then the following hold:
\begin{enumerate}
\item
The diagonal $v(\om)$ of $\om$ is a unitary satisfying $\ka(v(\om))=v(\om)$;
\item
Denote $v(\om)$ by $v$.
Then one has $\wdt{\om}=(\ka(v)^*\oti\ka(v)^*)\om\De(\ka(v))$,
where $\wdt{\om}:=F((\ka\oti\ka)(\om^*))$.
\end{enumerate}
\end{lem}
\begin{proof}
(1), (2).
Since $M$ is commutative,
we may and do assume that $M=\C$ by considering the disintegration.
Let us denote $v(\om)$ by $v$.
We shall show $v$ is a unitary.

Let $V$ be an $\om$-representation on some Hilbert space $H$.
Then $\om_{23}(\id\oti\De)(V)=V_{13}V_{12}$.
Applying the linear map
$(\id \oti F)\circ (\id\oti\ka\oti\ka)$,
we have
\[
(\id\oti\De)((\id\oti\ka)(V))\wdt{\om}_{23}^*
=(\id\oti\ka)(V)_{12}(\id\oti\ka)(V)_{13},
\]
where
$\wdt{\om}$ denotes $F((\ka\oti\ka)(\om^*))$.
By definition of $v$ and $V$,
we have $1\oti \ka(v)= V (\id\oti\ka)(V)$.
Then by these equalities, we obtain
\[
V_{13}V_{12}(\id\oti\De)(V^*(1\oti \ka(v)))\wdt{\om}_{23}^*
=1\oti \ka(v)\oti\ka(v).
\]
Taking the norm of the both sides,
we have $\|\ka(v)\|=1=\|v\|$.
Together with $(\id\oti\tr_\pi)(v_\pi^*v_\pi)=1$,
which is proved in \cite[Lemma 5.6]{MT1},
it turns out that $v$ is a unitary satisfying (1) and (2).

Thus we have shown $\wdt{\om}=(\ka(v^*)\oti\ka(v^*))\om\De(\ka(v))$,
and $\wdt{\om}$ is also a 2-cocycle.
The previous lemma implies
$v(\wdt{\om})=\ka(v^*)vv^*=\ka(v^*)$.
By definition of $v(\wdt{\om})$,
we have
$(v(\wdt{\om})\oti1)\De(e_\btr)=\wdt{\om}\De(e_{\btr})$.
Applying $F\circ(\ka\oti\ka)$ to the both sides,
we have
$v(\wdt{\om})=v^*$.
Thus we obtain $\ka(v)=v$.
\end{proof}

\begin{lem}\label{lem:normalize}
Let $M$ be a commutative von Neumann algebra.
Then any $M$-valued 2-cocycle of a discrete Kac algebra
is normalizable.  
\end{lem}
\begin{proof}
Let $\om$ and $v$ be as above.
Since $v$ is fixed by the antipode $\ka$,
there exists $\nu\in\lhG$ such that $v^*=\nu^2$ and $\ka(\nu)=\nu$.
Then $\om^\nu:=(\nu\oti\nu)\om\De(\nu^*)$
is a normalized 2-cocycle.
\end{proof}

Let $U_0(\bhG):=\{u\in U(A)\mid u_{\mathbf{1}}=1,\ \ka(u)=u^*\}$.
If $\om_1, \om_2\in Z^2(\bhG)_n$ are equivalent,
any unitary perturbing $\om_1$ to $\om_2$
must sit in $U_0(\bhG)$.

Next we recall the following result \cite[Theorem 2]{Wass-cptII}.

\begin{thm}[Wassermann]\label{thm:erg-coc}
Let $\bG$ be a compact group.
Then there exists one to one correspondence
between $H^2(\bhG)$ and
the set of conjugacy classes of ergodic $\bG$-actions
with full multiplicity.
\end{thm}

The correspondence is obtained as follows.
Let $\om\in Z^2(\bhG)$,
and we consider a unitary
$V_\om\in B(L^2(\bG))\oti \lhG$ defined by
$V_\om:=\om V$.
Then $V_\om$ is an $\om$-representation,
that is,
\[
(V_\om)_{12}(V_\om)_{13}=\om_{23}(\id\oti\De)(V_\om).
\]
Let $A_\om:=\{(\id\oti\ph)(V_\om)\mid \ph\in \lhG_*\}''$.
We introduce an ergodic action $\al_g$ of $\bG$ on $A_\om$
defined by $\Ad \rho(g)$, where $\rho$ denotes
the right regular representation.
Then we have
\[
(\al_g\oti\id)(V_\om)=V_\om(1\oti \la(g)).
\]

Conversely, let $\al$ be an ergodic action of $\bG$
on a von Neumann algebra $A$ which has full multiplicity.
Then
there exists a unitary $V\in A\oti\lhG$ such that
$(\al_g\oti\id)(V)=V(1\oti \la(g))$.
By the ergodicity, there exists $\om\in Z^2(\bhG)$
such that
$V_{12}V_{13}=\om_{23}(\id\oti\De)(V)$.

\subsection{Skew symmetric bicharacters}

By \cite[Lemma 26]{Wass-cptII}
or Lemma \ref{lem:skew-symm},
$\beta_\om:=\om^*F(\om)=\om_{12}^*\om_{21}$
is a \emph{skew symmetric bicharacter}
\index{skew symmetric bicharacter}
of $\bhG$ if $\om\in Z^2(\bhG)_c$.
Namely,
$\be_\om$ satisfies the following identities
putting $\be=\be_\om$:
\begin{equation}\label{eq:beom1}
(\Delta\otimes \id)(\beta)=\be_{23}\be_{13},
\quad(\id\otimes\Delta)(\beta)=\be_{12}\be_{13},
\end{equation}
\begin{equation}\label{eq:beom2}
F(\beta)=\beta^*,
\quad
(\id\otimes \ka)(\beta)=\be^*=(\ka\otimes \id)(\beta).
\end{equation}
We set the following space
\[
(\De^\op,\De):=\{x\in \lhG\oti\lhG\mid x\De(y)_{21}=\De(y)x
\ \mbox{for all }y\in\lhG\}.
\]
Note the following useful fact,
whose proof is a straightforward computation. 
\begin{lem}\label{lem:remark}
Let $\om\in Z^2(\bhG)$.
Then $\om\in Z^2(\bhG)_c$ if and only if
$\beta_\om\in (\De^\op,\De)$.
\end{lem}

We say that a skew symmetric bicharacter
$\be$ is \emph{normal}
\index{skew symmetric bicharacter!normal--}
if $\be\De(e_\btr)=\De(e_\btr)$.
Let $X^2(\bhG)_n$ be the set of normal skew symmetric bicharacters
on $\bhG$.

When $\bhG$ is a discrete abelian group $\Ga$,
the image of the map $\om\mapsto\be_\om$
is precisely equal to the set
of all skew symmetric bicharacters which take the value 1
on the diagonal set $\{(g,g)\in\Ga\times\Ga\mid g\in\Ga\}$
as shown in \cite[Proposition 3.2]{OPT}
(see also \cite[Theorem 7.1]{Kl}).
Readers should notice that
in the statement of \cite[Proposition 3.2]{OPT},
the condition on the value is missing,
but that is implicitly used
there to show that the group $\mathscr{C}$
is indeed a subgroup of $\mathscr{G}$.
Thus we have to impose the normality of $\be$
to prove the bijective correspondence
in general Kac algebra case.

Let $\om\in Z^2(\bhG)_c$.
We show the normality of $\be_\om$.
Take a unitary $u\in \lhG$ such that $\om':=(u\oti u)\om\De(u^*)$
is a normalized cocycle.
Then
\begin{align*}
\be_{\om'}
&=
\De(u)\om^*(u^*\oti u^*)\cdot (u\oti u)\om_{21}\De(u^*)_{21}
\\
&=
\De(u)\om^* \cdot\De^\om(u^*)_{21}\cdot\om_{21}
=
\De(u)\om^* \De^\om(u^*)\cdot \om_{21}
=
\be_\om.
\end{align*}
Thus we may and do assume that $\om$ is normalized.
Then we have
\[
\be_\om\De(e_\btr)=\om^*\om_{21}\De(e_\btr)
=\om^*\De(e_\btr)=\De(e_\btr).
\]
Hence the map
$\be\col H^2(\bhG)_c\ni[\om]\mapsto \be_\om\in X^2(\bhG)_n$
is well-defined.
In the following theorem,
the surjectivity
has been already proved in \cite{NeTu}
by using the Doplicher--Roberts duality.
We will present a proof in our setting for readers' convenience.

\begin{thm}\label{thm:bichara}
The map $\be\col H^2(\bhG)_c\ra X^2(\bhG)_n\cap(\De^\op,\De)$
is bijective.
\end{thm}
\begin{proof}
We show the injectivity of the map.
Let $\beta_{\om_1}=\beta_{\om_2}$.
Then $F(\om_1\om_2^*)=\om_1\om_2^*$.
Thus $\om_1\om_2^*$ is a symmetric 2-cocycle
of the co-commutative Kac algebra $(\lhG,\De^{\om_2})$.
By \cite[Lemma 21]{Wass-cptII},
$\om_1\om_2^*=(u\otimes u)\Delta^{\om_2}(u^*)$
for some $u\in U_0(\bhG)$,
and hence $\om_1=(u\otimes u)\om_2\Delta(u^*)$.

The surjectivity is checked as follows.
Let $\Rep_f(\bG)$ be the representation category of $\bG$,
that is, its objects and morphisms consist
of finite dimensional unitary representations
and intertwines, respectively.
In the following $X,Y,Z$ denote objects of $\Rep_f(\bG)$.
The flip symmetry of $\Rep_f(\bG)$ is written
by $\Si$, which is a collection
of a map $\Si_{X,Y}\col X\oti Y \ra Y\oti X$
given by $\Si_{X,Y}(\xi\oti \eta)=\eta\oti\xi$.
For each $X\in \Rep_f(\bG)$,
we have a $*$-homomorphism $\pi_X\col \lhG\ra B(X)$,
the differential representation.
Now we let $\be\in X^2(\bhG)_n\cap (\De^\op,\De)$,
and
$R_{X,Y}:=\Si_{X,Y}(\pi_X\oti\pi_Y)(\be^*)=(\pi_Y\oti\pi_X)(\be)\Si_{X,Y}$.
Then for all $a\in\lhG$,
we have
\begin{align*}
R_{X,Y}\pi_{X\oti Y}(a)
&=
\Si_{X,Y}(\pi_X\oti\pi_Y)(\be^*\De(a))
=
\Si_{X,Y}(\pi_X\oti\pi_Y)(\De(a)_{21}\be^*)
\\
&=
\pi_{Y\oti X}(a)R_{X,Y}.
\end{align*}
Thus $R_{X,Y}\in \Hom(X\oti Y,Y\oti X)$.
In fact, $R$ gives a permutation symmetry on $\Rep_f(\bG)$.
Indeed, the naturality of $R_{X,Y}$ follows
from the identity $S\pi_X(a)=\pi_Y(a)S$ for any $S\in \Hom(X,Y)$
and $a\in\lhG$.
The braiding relation
$R_{X,Y\oti Z}=(1_Y\oti R_{X,Z})(R_{X,Z}\oti1_Z)$
is observed as
\begin{align*}
R_{X,Y\oti Z}
&=
\Si_{X,Y\oti Z} (\pi_X\oti\pi_{Y\oti Z})(\be^*)
=
\Si_{X,Y\oti Z} (\pi_X\oti\pi_{Y}\oti\pi_{Z})((\id\oti\De)(\be^*))
\\
&=
\Si_{X,Y\oti Z} (\pi_X\oti\pi_{Y}\oti\pi_{Z})(\be_{13}^*\be_{12}^*)
\\
&=
\Si_{X,Y\oti Z} (\pi_X\oti\pi_{Y}\oti\pi_{Z})(\be_{13}^*)
(\Si_{Y,X}\oti1_Z)
(R_{X,Y}\oti1_Z)
\\
&=
\Si_{X,Y\oti Z}(\Si_{Y,X}\oti1_Z)
(\pi_Y\oti\pi_{X}\oti\pi_{Z})(\be_{23}^*)
(R_{X,Y}\oti1_Z)
\\
&=
(1_Y\oti \Si_{X,Z})(\pi_Y\oti\pi_{X}\oti\pi_{Z})(\be_{23}^*)
(R_{X,Y}\oti1_Z)
\\
&
=
(1_Y\oti R_{X,Z})(R_{X,Z}\oti1_Z).
\end{align*}
Similarly, we obtain $R_{X\oti Y,Z}=(R_{X,Z}\oti1_Y)(1_X\oti R_{X,Z})$.
Since $F(\be)=\be^*$,
we have $R_{Y,X}R_{X,Y}=1$.
Thus $(\Rep_f(\bG),R)$ is a symmetric C$^*$-tensor category.

Next we show that the category is even
in the sense of \cite[Definition B.8]{Mu},
that is, it has the trivial twist $\Th(X)=1_X$ for all $X$.
To prove that, it suffices to prove $\Th(X)=1_X$
for each irreducible module $X$.
In this case,
the twist $\Th(X)$ is given by
$\Th(X)=(\Tr_X\oti\id_X)(R_{X,X})$, where $\Tr_X$ denotes
the non-normalized trace on $B(X)$ (see \cite[Definition A.40]{Mu}).
Setting $T_{\ovl{X},X}:=\sum_i\ovl{\xi_i}\oti\xi_i$
as before,
we have
\begin{align*}
\Th(X)
&=
(T_{\ovl{X},X}^*\oti1_X)(1_{\ovl{X}}\oti R_{X,X})
(T_{\ovl{X},X}\oti1_X)
\\
&=
(T_{\ovl{X},X}^*\oti1_X)
(1_{\ovl{X}}\oti \Si_{X,X})
(\pi_{\ovl{X}}\oti \pi_{X}\oti \pi_{X})(\be_{23}^*)
(T_{\ovl{X},X}\oti1_X)
\\
&=
(T_{\ovl{X},X}^*\oti1_X)
(1_{\ovl{X}}\oti \Si_{X,X})
(\pi_{\ovl{X}}\oti \pi_{X}\oti \pi_{X})(\be_{13})
(T_{\ovl{X},X}\oti1_X)
\quad
\mbox{by }(\ref{eq:beom1})
\\
&=
(T_{\ovl{X},X}^*\oti1_X)
(\pi_{\ovl{X}}\oti \pi_{X}\oti \pi_{X})(\be_{12})
(1_{\ovl{X}}\oti \Si_{X,X})
(T_{\ovl{X},X}\oti1_X)
\\
&=
(T_{\ovl{X},X}^*\oti1_X)
(1_{\ovl{X}}\oti \Si_{X,X})
(T_{\ovl{X},X}\oti1_X)
\quad
\mbox{by normality}
\\
&=
1_X.
\end{align*}

Thus by the Doplicher--Roberts duality \cite{DR}
\index{Doplicher--Roberts duality}
(see also \cite[Theorem B.11]{Mu}),
there exists a $*$-preserving symmetric fiber functor
$E\col (\Rep_f(\bG),R)\ra (\Hilb_f,\Si)$,
where $\Hilb_f$ denotes the (strict) C$^*$-tensor category
consisting of finite dimensional Hilbert spaces.
We let $d_{X,Y}\col E(X)\oti E(Y)\ra E(X\oti Y)$
be the associated unitary.

Since the $\Rep_f(\bG)$ is amenable,
the functor $E$ preserves the dimension of each Hilbert space
(see \cite[Proposition A.4]{NeTu}).
Moreover, the semisimplicity of $\Rep_f(\bG)$
implies $E$ is equivalent to the forgetful functor
$F\col \Rep_f(\bG)\ra \Hilb_f$
as a C$^*$-category (not as a C$^*$-tensor category).
Take an associated unitary $u_X\col X\ra E(X)$ for each $X$.
Then we obtain
\[
\Si_{E(X),E(Y)}=d_{Y,X}^*E(R_{X,Y})d_{X,Y}
=d_{Y,X}^*u_{Y\oti X} R_{X,Y} u_{X\oti Y}^*d_{X,Y}.
\]
Since
$\Si_{E(X),E(Y)}=\Si_{u_X X,u_Y Y}
=(u_Y\oti u_X)\Si_{X,Y}(u_X^*\oti u_Y^*)$,
we have
\begin{equation}
\label{eq:RSi}
R_{X,Y}
=
u_{Y\oti X}^*d_{Y,X}(u_Y\oti u_X)
\Si_{X,Y}
(u_X^*\oti u_Y^*)
d_{X,Y}^* u_{X\oti Y}.
\end{equation}
We show that there exists $\om\in\lhG\oti\lhG$
such that
$(\pi_X\oti\pi_Y)(\om)
=(u_X^*\oti u_Y^*)d_{X,Y}^* u_{X\oti Y}$.
For any $\bG$-modules $X,X',Y, Y'$
and any intertwiners $S\in \Hom(X,X')$, $T\in\Hom(Y,Y')$,
we only have to check
\[
(S\oti T)(u_X^*\oti u_Y^*)d_{X,Y}^* u_{X\oti Y}
=
(u_{X'}^*\oti u_{Y'}^*)d_{X',Y'}^* u_{X'\oti Y'}
(S\oti T).
\]
Indeed, we have
\begin{align*}
(S\oti T)(u_X^*\oti u_Y^*)d_{X,Y}^* u_{X\oti Y}
&=
(u_{X'}^*\oti u_{Y'}^*)
(E(S)\oti E(T))d_{X,Y}^* u_{X\oti Y}
\\
&=
(u_{X'}^*\oti u_{Y'}^*)
d_{X',Y'}^*(E(S\oti T)) u_{X\oti Y}
\\
&=
(u_{X'}^*\oti u_{Y'}^*)d_{X',Y'}^* u_{X'\oti Y'}
(S\oti T).
\end{align*}
By definition, $\om$ is a unitary.
The 2-cocycle relation of $\om$ is shown as follows:
\begin{align*}
&(\pi_X\oti\pi_Y\oti\pi_Z)
(\om_{12}(\De\oti\id)(\om))
\\
&=
((u_X^*\oti u_Y^*)d_{X,Y}^* u_{X\oti Y}\oti 1_Z)
\cdot
(u_{X\oti Y}^*\oti u_Z^*)d_{X\oti Y,Z}^* u_{X\oti Y\oti Z}
\\
&=
(u_X^*\oti u_Y^*\oti u_Z^*)
d_{X,Y}^*d_{X\oti Y,Z}^*u_{X\oti Y\oti Z}
\\
&=
(u_X^*\oti u_Y^*\oti u_Z^*)
d_{Y,Z}^*d_{X,Y\oti Z}^*u_{X\oti Y\oti Z}
\\
&=
(\pi_X\oti\pi_Y\oti\pi_Z)
(\om_{23}(\id\oti\De)(\om)).
\end{align*}
Thus by (\ref{eq:RSi}),
$\be=\om^*\om_{21}=\be_\om$,
and $\om$ is co-commutative from Lemma \ref{lem:remark}.
\end{proof}

\subsection{A Kac algebra associated with $\be_\om$}
By the bicharacter formulae (\ref{eq:beom1}) and (\ref{eq:beom2}),
the first tensor component of $\be$ spans a Hopf $*$-subalgebra
of $\lhG$, which we denote by $\meR_\be$,
that is,
\[
\meR_\be:=\ovl{\spa}^{\rm w}\{(\id\oti\ph)(\be)\mid \ph\in \lhG_*\}.
\]
The same argument as \cite[Lemma 34]{Wass-cptII}
implies
that $\be\in \meR_\be\oti \meR_\be$ and $\ka(\meR_\be)=\meR_\be$.

Thanks to \cite[Lemma 2.6, 3.14, Theorem 3.18]{T2},
$\meR_\be$ is the group algebra of a quantum subgroup
$\bK_\be$ of $\bG$,
which is of Kac type,
that is, $\meR_\be\cong L^\infty(\bhK_\be)$ as Kac algebras.

We let $p_\be$ be the projection in $\lhG$
which
corresponds to the trivial representation
of $\bK_\be$.

\begin{lem}
The projection $p_\be$ satisfies the following:
\begin{enumerate}
\item
$\be(p_\be\oti1)=p_\be\oti1=(p_\be\oti1)\be$;
\item
$\be(1\oti p_\be)=1\oti p_\be=(1\oti p_\be)\be$;
\item
$p_\be \lhG=\{x\in \lhG\mid \be(x\oti1)=x\oti1\}$.
\end{enumerate}
\end{lem}
\begin{proof}
Since $p_\be$ satisfies $\De(x)(p_\be\oti1)=p_\be\oti x$
and $\De(x)(1\oti p_\be)=x\oti p_\be$ for $x\in \meR_\be$,
we obtain the equalities of (1) and (2)
using (\ref{eq:beom1}).
On (3), we let $p$ be the projection such that
$p\lhG$ is equal to the right hand side.
It is trivial that $p_\be\leq p$.
By linearity,
we have $xp=\vep(x)p$ for every $x\in \meR_\be$.
Since $\vep(p_\be)=1$, we have $p_\be p=p$, and $p\leq p_\be$.
Thus $p=p_\be$.
\end{proof}

In the case that $p_\be=e_\btr$,
$\om$ is said to be \emph{totally skew}
\index{2-cocycle!totally skew}
(cf. \cite{LK} for abelian groups).
In this paper,
following \cite[p.220]{Wass-coact},
we will say that $\be_\om$ is \emph{non-degenerate}
\index{skew symmetric bicharacter!non-degenerate--}
in that case.
Recall the following result due to A. Wassermann \cite[Theorem 12]{Wass-cptII}.

\begin{thm}[A. Wassermann]\label{thm:non-deg-factor}
The following statements are equivalent:
\begin{enumerate}
\item
$\be_\om$ is non-degenerate;
\item
$\meR_{\be_\om}=\lhG$;
\item
The 2-cocycle $\om$
corresponds to an ergodic action of $\bG$
on a factor.
\end{enumerate}
\end{thm}

Now we will understand what is $\bK_\be$.
Take $\om\in Z^2(\bhG)_c$ with $\be=\om^*\om_{21}$
from Theorem \ref{thm:bichara}.
In fact, we may and do assume that $\om\in\meR_\be\oti\meR_\be$
because $\be$ is a skew symmetric bicharacter on $\meR_\be$.
Let us shed a light on the compact group
$\bG_\om$, the dual cocycle twisting of $\bG$.
Then $\om^*$ is a 2-cocycle of $\wdh{\bG}_\om$,
and that corresponds to an ergodic action of $\bG_\om$
with full multiplicity by Theorem \ref{thm:erg-coc}.
Then the action is induced from an ergodic action
of a closed subgroup $\bG_\om^0$ on a finite factor
\cite[Theorem 7]{Wass-cptI}.
Then again by the correspondence,
we obtain a 2-cocycle $\om^0$ of the dual of $\bG_\om^0$.
Since the cohomology class of $\om^*$ and $\om^0$ are equal,
there exists a unitary $u\in \lhG$
such that
$\om^0=(u\oti u)\om^*\De^\om(u^*)$,
that is,
\[
\be^0:=\om_{21}^0(\om^0)^*=(u\oti u)\be_\om^*(u^*\oti u^*),
\]
where $\be$ is the one introduced in \cite[p.1513]{Wass-cptII}.
Thus $u \meR_\om u^*$ is the weak closure of
the linear span of $(\id\oti\ph)(\be^0)$, $\ph\in\lhG_*$.
Since
$\be^0:=\om_{21}^0(\om^0)^*$ is non-degenerate
from Theorem \ref{thm:non-deg-factor},
$u\meR_\om u^*=L^\infty(\widehat{\bG}_\om^0)$,
the dual of $\bG_\om^0$.

%


\begin{ex}
Let us consider a finite dimensional Kac algebra $\bG$
and a 2-cocycle $\om\in Z^2(\bhG)_{c}$
such that $\be_\om$ is non-degenerate.
Then there exists an ergodic action
$\al$ of $\bG$ on $B(\C^n)$, where $n^2=\dim C(\bG)$.
This equality is a useful obstruction
so that $\wdh{\Ga}$ has a non-trivial 2-cocycle.
Readers are referred to a further investigation
of dual 2-cocycles of finite Kac algebras
due to Izumi--Kosaki \cite{IKf,IK,IzKo-coh}.
\end{ex}

\begin{ex}
When $\bG=SU(2)$ or $SU_{-1}(2)$,
\index{$SU(2)$}
\index{$SU_{-1}(2)$}
any 2-cocycle of $\bhG$ is trivial.
Thus the dual cocycle twisting of them
are nothing but $\bG$.
\end{ex}

\begin{ex}
We consider when $\bG=SO(3)$.
\index{$SO(3)$}
It is known that the cohomology set
$H^2(\bhG)$
consists of two elements \cite[Theorem 6]{Wass-cptIII}.
The non-trivial cocycle $\om$
comes from the Krein four group
$D_2:=\Z/2\Z\times \Z/2\Z$.
\index{Krein four group $D_2$}
Let $\bG_\om=SO(3)_\om$ be the twisting of $SO(3)$ by $\om$.
Then the non-normality of $D_2$ in $SO(3)$ implies
that $SO(3)_\om$ is not a compact group
(see also Remark \ref{rem:Net-Tus}).
In fact, it is isomorphic to $SO_{-1}(3)$
\index{$SO_{-1}(3)$}
in the sense of \cite[Definition 3.1]{BaBi}.
Readers should note that the definition of
$SO_{-1}(3)$ there differs from that of \cite{KS},
where $SO_{-1}(3)$ is nothing but $SO(3)$.
Then \cite[Theorem 3.1]{BaBi} implies
that $SO(3)_\om$ is isomorphic to $\mathscr{Q}_4$,
the quantum permutation group on 4 points
\cite{Wan-sym}.
\index{quantum permutation group on 4 points $\mathscr{Q}_4$}
\end{ex}

\subsection{Actions and Cocycle actions}
\label{subsect:endo-action}
Let $M$ be a von Neumann algebra.
Let $\al\col M\ra M\oti\lhG$ be
a unital faithful normal $*$-homomorphism
and $u\in M\oti\lhG\oti\lhG$ be a unitary.
The pair $(\al,u)$
is called a \emph{cocycle action}
\index{action!cocycle--}
when it satisfies
\[
(\al\oti\id)\circ\al=(\id\oti \De)\circ\al,
\quad
(u\oti1)(\id\oti\De\oti\id)(u)=\al(u)(\id\oti\id\oti\De)(u),
\]
where $\al(u)$ means $(\al\oti\id\oti\id)(u)$.
We will often use such a convention to abbreviate tensor notation.
The unitary $u$ is also called a \emph{2-cocycle}.
\index{2-cocycle!of a cocycle action}
When $u=1$, $\al$ is called an \emph{action}.
\index{action!of a discrete Kac algebra}
If $v\in M\oti\lhG$ is a unitary,
then the perturbed cocycle action
\index{action!perturbed cocycle--}
$(\al^v,u^v)$ is defined by
\[
\al^v:=\Ad v\circ\al,
\quad
u^v:=(v\oti1)\al(v)u(\id\oti\De)(v^*).
\]
If $u=1=u^v$, $v$ is called an $\al$-\emph{cocycle}.
\index{cocycle of an action}

By the analogue of Peter--Weyl theory,
we have the following decomposition:
\[
\lhG=\bigoplus_{\pi\in\IG}B(H_\pi).
\]
Thus $\al$ consists of the family of $\al_\pi$
which are the maps from $M$ into $M\oti B(H_\pi)$.

We will recall the notion of the freeness for cocycle actions.

\begin{defn}
Let $(\al,u)$ be a cocycle action of $\bhG$
on a von Neumann algebra $M$.
Then we will say that $(\al,u)$ is \emph{free}
\index{action!free cocycle--}
if for every $\pi\in\IG\setminus\{\btr\}$,
there exists no non-zero $a\in M\oti B(H_\pi)$
such that $a(x\oti1_\pi)=\al_\pi(x)a$ for all $x\in M$.
\end{defn}

If $M$ is properly infinite,
then $M\oti B(H_\pi)$ is isomorphic to $M$.
So, an action of $\bhG$ means a certain system
of endomorphisms.

Let us clarify this point.
Let $M$ be a properly infinite von Neumann algebra
and $\al\col M\ra M\oti\lhG$ an action.
We associate a family of endomorphisms $\{\be_\pi\}_{\pi\in\IG}$
on $M$ with $\al$ as follows.
For each $\pi\in\IG$, take a Hilbert space $K_\pi\subs M$
whose dimension equals $d(\pi)$.
Let $\{V_{\pi_i}\}_{i\in I_\pi}$ be an orthogonal base of $K_\pi$.
Then we set
\[
\be_\pi(x):=\sum_{i,j\in I_\pi}V_{\pi_i}\al_{\pi_{ij}}(x)V_{\pi_j}^*
\quad\mbox{for }x\in M.
\]
Using $\al_\pi\circ\al_\rho=(\id\oti_\pi\De_\rho)\circ\al$,
we can compute the composition $\be_\pi\be_\rho$ as follows:
\begin{align*}
\be_\pi(\be_\rho(x))
&=
\sum_{i,j}
\be_\pi(V_{\rho_i})\be_\pi(\al_{\rho_{ij}}(x))\be_\pi(V_{\rho_j})^*
\\
&=
\sum_{i,j,r,s}
\be_\pi(V_{\rho_i})V_{\pi_r}\al_{\pi_{r,s}}(\al_{\rho_{ij}}(x))
V_{\pi_s}^*\be_\pi(V_{\rho_j})^*
\\
&=
\sum_{i,j,r,s,a,b,\si,S}
\be_\pi(V_{\rho_i})V_{\pi_r}
S_{\pi_r,\rho_i}^{\si_a}
\al_{\si_{ab}}(x)
\ovl{S_{\pi_s,\rho_j}^{\si_b}}
V_{\pi_s}^*\be_\pi(V_{\rho_j})^*
\\
&=
\sum_{i,j,r,s,a,b,\si,S}
\be_\pi(V_{\rho_i})V_{\pi_r}
S_{\pi_r,\rho_i}^{\si_a}
V_{\si_a}^*
\be_{\si}(x)
V_{\si_b}
\ovl{S_{\pi_s,\rho_j}^{\si_b}}
V_{\pi_s}^*\be_\pi(V_{\rho_j})^*,
\end{align*}
where the summation is taken for
$i,j\in I_\pi$, $r,s\in I_\rho$, $a,b\in I_\si$,
$\si\in\IG$ and $S\in\ONB(\si,\pi\rho)$.
We let
$v(S):=
\sum_{i,r,a}\be_\pi(V_{\rho_i})V_{\pi_r}S_{\pi_r,\rho_i}^{\si_a}V_{\si_a}^*$.
Then it is easy to see that $\{v(S)\}_{S\in\ONB(\si,\pi\rho)}$
is an orthonormal base of a Hilbert space in $M$
with support 1,
and we obtain
\begin{equation}
\label{eq:be-sect}
\be_\pi(\be_\rho(x))
=\sum_{\si\in\IG}\sum_{S\in\ONB(\si,\pi\rho)}
v(S)\be_\si(x)v(S)^*.  
\end{equation}
Thus we have the following result.
\begin{lem}
In $\Sect(M)$, one has
\[
[\be_\pi][\be_\rho]=\sum_{\sigma\prec\pi\rho}N_{\pi,\rho}^\si[\be_\si],
\]
where $N_{\pi,\rho}^\si$ denotes $\dim(\si,\pi\rho)$.
\end{lem}

\begin{lem}
For each $\pi\in\IG$,
$\be_\opi$ is the conjugate endomorphism of $\be_\pi$.
\end{lem}
\begin{proof}
We let $v:=v(T_{\pi,\opi})$ and $w:=v(T_{\opi,\pi})$.
Then by direct computation,
we see $v\in (\id,\be_\pi\be_\opi)$ and $w\in(\id,\be_\opi\be_\pi)$,
and they are satisfying
$v^*\be_\pi(w)=1/d(\pi)=w^*\be_\opi(v)$.
Then it turns out that $\be_\opi$ is the conjugate endomorphism
of $\be_\pi$ from \cite[Theorem 5.3]{Is}.
\end{proof}

Now we introduce the following key-notion ``modular part'' in our work.
Recall that $\al_\pi$ is said to be \emph{modular}
if the canonical extension $\tal_\pi$ is implemented by a unitary
contained in $\widetilde{M}\oti B(H_\pi)$.
(See \cite[Appendix]{MT3} for the definition of $\tal$.)

\begin{defn}\label{defn:modular-part}
Let $\bG$ be a compact Kac algebra
and $(\al,u)$ a cocycle action of $\bhG$ on a factor $M$.
Then the set $\La(\al):=\{\pi\in\IG\mid \al_\pi\mbox{ is modular}\}$
is called the \emph{modular part}
\index{modular part!of a cocycle action}
of $\al$.
If $\La(\al)=\IG$, $\al$ is said to be \emph{modular}.
When $\La(\al)=\{\btr\}$,
we will say $\al$ has \emph{trivial modular part}.
\index{modular part!trivial--}
\end{defn}

\begin{thm}
Let $\al\col M\ra M\oti\lhG$ be an action on an infinite factor $M$.
Then the following statements hold:
\begin{enumerate}
\item
Let $\be\col\IG\ra\End(M)_0$
be the map associated with $\al$ as before.
Then for $\pi\in\IG$,
$\al_\pi$ is modular if and only if $\be_\pi$ is modular;

\item
The action
$\al$ has the normal modular part if and only if $\al_\pi\al_\opi$ is
modular
for all $\pi\in\IG$.
In this case, $\al$ has the Connes--Takesaki module;

\item
If $\al$ has the normal modular part,
then $\be_\pi$ does not contain a modular endomorphism for all
$\pi\nin\La(\al)$.
\end{enumerate}
\end{thm}
\begin{proof}
(1) has been shown in \cite[Lemma A.4]{MT3}.
We will prove (2).
When the modular part $\La(\al)$ is normal,
each irreducible $\rho\prec\pi\opi$ is contained in $\La(\al)$
because $\btr\in\La(\al)$.
Thus $\be_\rho$ is modular, and $\be_\pi\be_\opi$ is modular
from (\ref{eq:be-sect}).

Conversely, suppose that $\be_\pi\be_\opi$ is modular for each $\pi$.
The previous result shows $\be_\opi=\ovl{\be_\pi}$.
By Theorem \ref{thm:rho-orho}, $\be_\pi$ has the Connes--Takesaki module.
Then $\be_\pi\be_\rho\be_\opi$ is modular for all $\rho\in\La(\al)$.
Thus each $\be_\si$ for $\si\prec\pi\rho\opi$ is modular,
and $\La(\al)$ is normal.

(3).
Let $\pi\nin\La(\al)$.
Assume $\be_\pi$ would contain a modular endomorphism $\si$.
Since $\be_\pi$ is not modular,
it also contains an endomorphism $\de$ that is not modular.
Then $\de\osi$ is contained in $\be_\pi\be_{\opi}$,
and that is modular.
The Frobenius reciprocity implies $\de$ is modular,
and this is a contradiction.
\end{proof}

As a corollary, we have the following result.

\begin{cor}\label{cor:Vps}
Let $\bG$ be a compact Kac algebra
and $(\al,u)$ a cocycle action of $\bhG$ on a factor $M$.
If $\al$ has the normal modular part,
then $\al$ has the Connes--Takesaki module,
and there exist a unitary $V_\pi\in\tM\oti B(H_\pi)$
and $\psi_\pi\in\Aut(M)$ for each $\pi\in\IG$
such that
\[
\widetilde{\alpha}_\pi=\Ad V_\pi\circ (\widetilde{\psi}_\pi\otimes1),
\quad
V_\pi^*(\th_t\oti\id)(V_\pi)\in Z(\tM)\oti B(H_\pi)
\mbox{ for all }t\in\R.
\]
Moreover, $\ps_\pi$ is uniquely determined up to $\Aut(M)_{\rm m}$.
\end{cor}
\begin{proof}
When $M$ is infinite, the 2-cocycle $u$ is a coboundary
from \cite[Lemma 3.2]{MT3}.
Thus we may and do assume that $\al$ is an action.
Then the statement follows from
Theorem \ref{thm:rho-orho}
and the previous theorem
(see the proof of \cite[Lemma A.4]{MT3}).
Next we consider the case that $M$ is finite.
We set the cocycle action $\be:=\id\oti \al$
on the infinite factor $N:=B(\el_2)\oti M$.
Then $\tN=B(\el_2)\oti \tM$ and $\tde=\id\oti\tal$,
and $\de$ has the normal modular part.
Thus there exist a unitary $V_\pi\in\tN\oti B(H_\pi)$,
and $\psi_\pi\in\Aut(N)$ such that
\[
\tde_\pi=\Ad V_\pi\circ (\widetilde{\psi}_\pi\otimes1),
\quad
V_\pi^*(\th_t\oti\id)(V_\pi)\in Z(\tM)\oti B(H_\pi)
\mbox{ for all }t\in\R.
\]
Since $\ps_\pi(B(\el_2))$ is a type I subfactor in $N$,
there exists a unitary $v\in N$ such that
$\Ad v\circ\ps_\pi=\id$ on $B(\el_2)$.
Hence we may and do assume that $\ps_\pi=\id$ on $B(\el_2)$.
Then $\ps_\pi$ is regarded as an automorphism on $M$.
From the equality $x\oti1_\pi=\tde_\pi(x)=V_\pi(x\oti1)V_\pi$
for all $x\in B(\el_2)$,
it turns out that $V_\pi\in \tM\oti B(H_\pi)$.
\end{proof}

Therefore, when the modular part is normal,
the canonical extension of a cocycle action
consists of two parts,
that is, a unitary implementation
and the canonical extensions of automorphisms.
We will prove the map $\ps\col \IG\ra\Aut(M)$
is possessing the property that
$\ps_\pi=\ps_\rho \bmod \Aut(M)_{\rm m}$
if $\pi=\rho$ in some ``quotient group'' $\Ga$ of $\IG$
by $\La(\al)$ in the next section,
and
then we will classify such maps,
which we will call modular $\Ga$-kernels.

The normality of a modular part does not hold in general,
but so does for $\bG$ being a connected simple compact Lie group
as follows.

\begin{thm}\label{thm:connected}
If $\bG$ is a connected simple compact Lie group
or its $q$-deformation with $q=-1$,
then the modular part of any cocycle action of $\bhG$
on a factor $M$ is trivial or normal.
\end{thm}
\begin{proof}
Let $\bG$ be such a compact Kac algebra,
and
$\al\col M\ra M\oti\lhG$ an action on a factor $M$,
with the non-trivial modular part $\La(\al)$.
From the observation in Example \ref{ex:conn-simple-normal},
it turns out that $\La(\al)$ is normal.
\end{proof}

\subsection{Dual actions of finite groups}
In this subsection,
we discuss when a dual action of a finite group
has the Connes--Takesaki module, and when that is modular.
Before that,
we will recall the notion of the minimality of a group action.

\begin{defn}
Let $G$ be a locally compact group
and $\al$ an action on a factor $M$.
Then $\al$ is said to be \emph{minimal}
\index{action!minimal--}
if the following conditions hold:
\begin{itemize}
\item
$(M^\al)'\cap M=\C$;

\item
The group homomorphism
$\al\col G\to\Aut(M)$ is injective.
\end{itemize}
\end{defn}

Let $\Ga$ be a finite group
and $\al$ a minimal action of $\Ga$ on a factor $N$.
We denote by $\La$ the modular part of $\al$,
that is, $\La$ consists of $h\in\Ga$ such that
$\tal_h=\Ad u_h$ for some unitary $u_h\in\tN$.
Assume that the characteristic invariant of $\al$ is trivial.
Namely, we have
\[
\tal_g(u_h)=u_{ghg^{-1}},
\quad
u_hu_k=u_{hk},
\quad g\in\Ga,\ h,k\in\La.
\]
Note that
the normal subgroup $\La$ must be abelian.
Indeed, $u_hu_ku_h^*u_k^*$ implements $\al_{hkh^{-1}k^{-1}}$,
but the dual action $\th$ fixes the unitary, and that is contained in $M$.
Thus $hkh^{-1}k^{-1}=e$ because of the minimality.

Let $M:=N\rti_\al\Ga$.
The core $\tM$ is naturally identified with $\tN\rti_\tal\Ga$.
By $\la^\tal(g)$, we denote the implementing unitary in $\tM$.
Then the following holds (see \cite{Kw-Tak,Se-flow}):
\[
\tN'\cap(\tN\rti_\tal\Ga)=\spa\{Z(\tN)W_h\mid h\in\La\},
\]
where $W_h:=u_h^*\la^\tal(h)$.
Note that $W$ is a representation of $\La$, that is, $W_hW_k=W_{hk}$.
Then an element $x=\sum_{h\in\La}z_h W_h$, $z_h\in Z(\tN)$
is contained in $Z(\tM)$
if and only if
\begin{equation}\label{eq:al-z}
\tal_g(z_h)=z_{ghg^{-1}}
\quad
\mbox{for all }g\in\Ga,h\in\La.
\end{equation}
Thus
the following element $x(h)$ is central:
\[
x(h):=\sum_{g\in \Ga}W_{ghg^{-1}}.
\]

Now we consider the dual action $\hal\col M\ra M\oti L^\infty(\hGa)$.
Then the canonical extension of $\be:=\hal$ on $\tM$
is identified
with the dual action of $\tal$ on $\tM=\tN\rti_\tal\Ga$.
Then $\be$ fixes $\tN$, and
\[
\tbe_\pi(W_h)=W_h\oti \pi(h),
\quad \pi\in\IG,\ h\in\La.
\]

When $\be_\pi$ has the Connes--Takesaki module,
$\tbe_\pi(x(h))$ belongs to $Z(\tM)\oti \C1_\pi$.
This means
\[
\sum_{g\in\Ga}\pi(ghg^{-1})_{ij}W_{ghg^{-1}}\in Z(\tM)
\quad
\mbox{for all }i,j\in I_\pi.
\]
By (\ref{eq:al-z}), we have
$\pi(h)=\pi(ghg^{-1})$ for all $g\in \Ga$,
and $\pi(h)$ is scalar for all $h\in\La$.

Conversely, we suppose that $\pi$ maps $\La$ into $\C$.
Let $x=\sum_{h\in\La}z_hW_h\in Z(\tM)$.
Then we have $\be_\pi(x)=\sum_{h\in\La}\pi(h)z_hW_h$.
It is immediately checked that the coefficients $\pi(h)z_h$
satisfies (\ref{eq:al-z}).
Hence we have proved the following result.

\begin{prop}
\label{prop:CT-cent}
Let $\pi\in\Irr(\Ga)$.
Then $\be_\pi$ has the Connes--Takesaki module
if and only if $\pi(\La)$ consists of scalars.
In particular,
the action $\be$ has the Connes--Takesaki module
exactly when $\La$ is contained in the center $Z(\Ga)$.
\end{prop}

Next we will determine when $\be_\rho$ is modular.
Suppose that a unitary $U_\rho\in \tM\oti B(H_\rho)$ satisfies
$\be_\rho=\Ad U_\rho$.
Since $\be$ fixes $\tN$, $U$ sits in $(\tN'\cap\tM)\oti B(H_\rho)$.
Thus there exists $z_h\in Z(\tN)\oti B(H_\rho)$ such that
\[
U_\rho=\sum_{h\in\La}z_h(W_h\oti1).
\]
Since $\be_\rho(\la^\tal(g))=\la^\tal(g)\oti\rho(g)$,
we have $U_\rho=(1\oti\rho(g))\cdot(\Ad \la^\tal(g)\oti\id)(U_\rho)$.
Using $\Ad \la^\tal(g)(W_h)=W_{ghg^{-1}}$,
we obtain
\begin{equation}\label{eq:z-rho-al}
z_{ghg^{-1}}=(1\oti\rho(g))(\tal_g\oti\id)(z_h).
\end{equation}
We equip a $\Ga$-module structure on $Z(\tN)\oti \el^\infty(\La)$
defined by $\tal_g\oti \Ad \la^\el(g)\la^r(g)$,
where $\la^\el(g)$ and $\la^r(g)$ denote
the left and right regular representation of $\Ga$, respectively.
Then (\ref{eq:z-rho-al}) means that
$Z(\tN)\oti \el^\infty(\La)$
contains $H_\orho$ as a $\Ga$-module if $z_h\neq0$.
We set $K_\al:=\ker(\mo\al)$.
Then the quotient group $\Ga/K_\al$ faithfully acts on $Z(\tN)$,
and the $\Ga$-module
$Z(\tN)$ is nothing but $\el^\infty(\Ga/K_\al)$ up to multiplicity
(see \cite[Lemma A.1]{Iz}).
Applying the above to $\be_\orho$ that is also modular,
we get the following characterization.

\begin{prop}\label{prop:be-modular}
Let $\rho\in\Irr(\Ga)$.
Then $\be_\rho$ is modular if and only if
the $\Ga$-module $\el^\infty(\Ga/K_\al\times\La)$
contains the irreducible representation $\rho$.
\end{prop}

When $\Ga$ is abelian, $\el^\infty(\La)$ is a trivial $\Ga$-module.
Thus $\be_\rho$ is modular if and only if $\rho|_{K_\al}$ is trivial,
that is, the modular part of $\be$ is the dual group of $\Ga/K_\al$.

\subsection{Actions of $\wdh{S_d}$}\label{sec:dualSd}
We continue to use the same notation in the previous subsection.
Let us consider the case that
$\Ga$ is the symmetric group of degree $d$ which we denote by $S_d$.
If $d\geq5$,
then a normal subgroup of $\Ga$ must be one of $\{e\}$, $A_d$ and $S_d$,
where $A_d$ denotes the alternative subgroup.
Since the modular part $\La$ of a minimal action $\al$ is abelian,
$\La$ exactly equals to $\{e\}$,
that is, $\al$ is centrally free.
The definition of the centrally freeness is the following:

\begin{defn}
Let $\Ga$ be a discrete group and $\al$ an action on a factor $M$.
Then we will say that $\al$ is \emph{centrally free}
\index{action!centrally free--}
if $\al$ is point-wise centrally non-trivial,
that is, $\al_t^\om$ is non-trivial on $M_\om$ whenever $t\neq e$.
\end{defn}

Note that the central freeness implies that
$\al^\om$ is point-wise outer on $M_\om$
(see \cite[p.38]{Oc}).

Then we have $Z(\tM)=Z(\tN)^\tal$.
This implies $\be=\hal$ has the trivial Connes--Takesaki module.

Next we study the modular part $\IH$ of $\hal$.
The corresponding quotient group $\bH$
must be one of $\Ga$, $\Z_2=\Ga/A_d$ and $\{e\}$.
Suppose that $H=\Z_2$.
Then $\IH=\{\btr,\sgn\}$,
where $\sgn$ is the sign representation.

Then there exists a unitary $U\in \tM$
which implements $\hal_{\sgn}$.
Since $\hal_{\sgn}$ fixes $\tN$,
$U$ must sit in $\tN'\cap\tM=Z(\tN)$.
By the equality $\tbe_{\sgn}(\la^\tal(g))=\sgn(g)\la^\tal(g)$,
we have
\[
\tal_g(U)=\sgn(g)U
\quad
\mbox{for all }g\in\Ga.
\]

We show the fact that $K:=\ker(\mo(\al))=A_d$.
Note that the subgroup $K$ is normal,
and $K=\{e\}, A_d$ or $\Ga$.
Indeed, if $g$ is a transposition,
then $\tal_g(U_\rho)=-U\neq U$.
This shows $K\neq\Ga$.
If $K=\{e\}$, then $\Ga$ faithfully acts on the flow space $Z(\tN)$.
Thanks to \cite[Theorem 5.9]{Iz},
$\hal$ would be modular.
Thus $K=A_d$.

\begin{prop}\label{prop:S5}
Let $\al$ be a minimal action of $S_d$ on an injective factor $N$.
If $d\geq5$,
then the following statements hold:
\begin{enumerate}
\item
$\al$ is centrally free;

\item
The dual action $\hal$ of $\wdh{S_d}$
has the trivial Connes--Takesaki module;

\item
If $\ker(\mo(\al))=\{e\}$, then $\hal$ is modular;
\item
If $\ker(\mo(\al))=A_d$, then the modular part of $\hal$ is $\{\btr,\sgn\}$;
\item
If $\ker(\mo(\al))=S_d$, then $\hal$ is centrally free.
\end{enumerate}
\end{prop}

If $d=3,4$, then $S_d$ has a non-trivial abelian normal subgroup.
Since that is not central,
we can construct a minimal dual action $\hal$ which does not have
the Connes--Takesaki module (see Proposition \ref{prop:CT-cent}).
Let us discuss it in a more general setting.

Let $\Gamma_0=\Lambda_0 \rtimes K$ be a finite group such that
$\Lambda_0$ is
abelian. 
We define an action of $K$ on $\Lambda=\widehat{\Lambda_0}$ by 
$\langle kpk^{-1},n\rangle=\langle p,k^{-1}nk\rangle$ for $p\in \Lambda $.

In what follows, we use letters $p,q$ to denote elements in $\Lambda$,
$k,l$ those in $K$, and $m,n$ those in $\Lambda_0$, respectively.

Let $(X,\mathcal{F}_t^X)$ be an ergodic flow, and 
$c(t,x)$ a minimal cocycle to $\Gamma_0$, i.e.,
$c(t, \mathcal{F}_s^Xx)c(s,x)=c(t+s,x)$ holds. 
Let $Z:=X\rtimes_c \Gamma_0$ and 
$\mathcal{F}^Z$ be the skew product flow.
Namely, $X\rtimes_c \Gamma_0=X\times
\Gamma_0$ as a set, and a flow $\mathcal{F}^Z_t$ is given by 
\[
 \mathcal{F}_t^Z(x,g)=(\mathcal{F}_t^Xx, c(t,x)g). 
\]
The minimality of
$c(t,x)$ implies the ergodicity of $(Z,\mathcal{F}_t^Z)$. Let
$\gamma_g$ be an action of $\Gamma_0$ on $Z$ by $\gamma_g(x,h)=(x,hg^{-1})$.

Let us decompose $c(t,x)$ as  
$c(t,x)=c^K(t,x)c^{\La_0}(t,x)$, $c^K(t,x)\in K$,
and $c^{\Lambda_0}(t,x)\in \La_0$.
We have
\begin{align*}
&c^K(t+s,x)c^{\Lambda_0}(t+s,x)
\\
&=c(t+s,x)
= c(t,\mathcal{F}_s^X x)c(s,x) \\
&=
 c^K(t,\mathcal{F}_s^X x) c^{\Lambda_0}(t,\mathcal{F}_s^X x)c^K(s,x)c^{\Lambda_0}(s,x)\\
&=
c^K(t,\mathcal{F}_s^X x)c^K(s,x)
\cdot
c^K(s,x)^{-1} 
c^{\Lambda_0}(t,\mathcal{F}_s^X x)c^K(s,x)c^{\Lambda_0}(s,x).
\end{align*}

Thus $c^K(t,x)$ is a cocycle to $K$, 
and $$c^K(s,x)^{-1}c^{\Lambda_0}(t,\mathcal{F}_s^X x)c^K(s,x)c^{\Lambda_0}(s,x)
=c^{\Lambda_0}(t+s,x)$$ holds.

Let $M$ be an injective factor
with the flow of weights
$(Z,\mathcal{F}_t^Z)$, and $\alpha$ an action of $\Gamma_0$ on $M$
with $\mathrm{mod}(\alpha_g)=\gamma_g$. Note
$\alpha$ is outer, or more precisely, centrally free
since $\alpha$ has the non-trivial Connes--Takesaki module.

Let 
$\mathcal{P}:=M\rtimes_\alpha \Lambda_0$,
and $v_n$ be the
implementing unitary.
We can canonically extend $\alpha_k$, $k\in K$, 
to an action on $\mathcal{P}$ by 
$\alpha_k(v_{k^{-1}nk})=v_n$.
We can easily to see
$\alpha_k\hat{\alpha}_{k^{-1}pk}\alpha_{k^{-1}}=\hat{\alpha}_p$. Thus 
we have an action $\beta$ of $\Gamma=\Lambda\rtimes K$ on ${P}$ by  
$\beta_{pk}=\hat{\alpha}_p\alpha_k$.
It is easy to see the outerness of $\beta$.
Since $Z(\tP)=Z(\tM)^{\Lambda_0}$, the flow of weights of $\mP$ is given by 
skew product flow $(Y,\mathcal{F}^Y_t):=X\rtimes_{c^K}K$ associated with 
$c^K(t,x)$. 

Let $u_p\in L^\infty(Z)=Z(\tM)$ be a unitary given by $u_p(x,kn)={\langle p,n\rangle}$.  
Obviously,
$u_p$ is a unitary representation of $\La=\widehat{\La_0}$. 
Moreover,
we have $\tilde{\beta}_k(u_{k^{-1}pk})=u_p$.

We have $\tilde{\beta_p}=\Ad u_p$.
To see this, we only have to verify 
\[
 \tilde{\beta}_p(x)=x=\Ad u_p(x),\,\, x\in \tM
\]
and 
\[
 \tilde{\beta}_p(v_n)=\overline{\langle p,n \rangle}v_n=u_pv_nu_p^*.
\]
Actually, they are just straightforward calculations.
Note that we have used the canonical
identification 
\[
\widetilde{M\rtimes_\alpha \Lambda_0}=\tM\rtimes_{\tilde{\alpha}}\Lambda_0.
\]
Thus $\beta_p$ is an extended modular automorphism.
Let us compute the 1-cocycle associated with $\beta_p$.
\begin{align*}
\theta_{t}(u_p)(x,km)
&=u_p\left(\mathcal{F}_{-t}^Z(x,km)\right)\\
&=u_p(F_{-t}^Xx, c(-t,x)km) \\
&=u_p(F_{-t}^Xx, c^K(-t,x)kk^{-1}c^{\La_0}(-t,x)km) \\
&={\langle p, k^{-1}c^{\La_0}(-t,x)km\rangle} \\
&=\langle p, k^{-1}c^{\La_0}(-t,x)k\rangle \langle p,m\rangle  \\
&={\langle p, k^{-1}c^{\La_0}(-t,x)k\rangle}u_p(x,km)
\end{align*}

Define a cocycle $d(t,(x,k))\colon Y\rightarrow \Lambda_0$ by 
$d(t,(x,k))=k^{-1}c^{\Lambda_0}(t,x)k$.
For $p\in \Lambda$, define $\hat{d}_p(t)\in Z^1(\mathbb{R},
U(L^\infty(Y)))$ by 
\[
 \hat{d}_p(t)(y)=\langle p, d(-t,y)\rangle.
\]
%
Thus we have
\[
\mathrm{mod}(\beta_k)
=\gamma_k,
\quad
\tilde{\beta}_p=\Ad u_p,
\quad
\theta_t(u_p)=\hat{d}_p(t)u_p,
\]
\[
\tilde{\beta}_k(u_{k^{-1}pk})=u_p,
\quad
u_pu_q=u_{pq}.
\]
This implies that the modular part of $\be$ equals $\Lambda$.
Summarizing the above argument,
we obtain the following result.

\begin{prop}
There exists an outer action $\beta$
of $\Gamma=\widehat{\La_0}\rtimes K$ on an injective factor $P$
such that
$\ker(\mo)=\widehat{\La_0}$, which is equal to the modular part of $\be$,
and
the modular invariant is given by $\hat{d}_p(t)$
and, moreover, the characteristic invariant is trivial.
\end{prop}

In particular,
if $K$ acts on $\Lambda_0$ non-trivially,
then the dual action $\hat{\beta}$ does not have the Connes--Takesaki module
by Proposition \ref{prop:CT-cent}.

If we put ${M}$ is an injective factor of type III$_0$,
and
$\alpha$ an action of $\Gamma=S_3=\mathbb{Z}_3\rtimes\mathbb{Z}_2$,
or $S_4=(\mathbb{Z}_2\times\mathbb{Z}_2)\rtimes S_3$,
which
faithfully acts on the flow of weights
of $M$ in the above construction,
then by the above procedure, we construct an outer action $\beta$ of $\Gamma$
whose dual action $\hat{\beta}$ of $\hat{\Gamma}$
does not have a Connes--Takesaki module.

\section{Classification of modular kernels}

As is remarked after Corollary \ref{cor:Vps},
any action of compact Kac algebra with normal modular part
consists of compositions of modular maps and automorphisms.
In this section, we study such system of automorphisms called
a modular kernel,
and discuss a classification of modular kernels in this section.

\subsection{Modular kernels}
\label{subsect:modularkernelclass}
Recall the set $\Aut(M)_{\rm m}$
that is the collection of all
extended modular automorphisms on a factor $M$.
In this section,
$\Ga$ denotes a discrete group.

\begin{defn}
Let $M$ be a factor.
Let $\al\col\Ga\ra \Aut(M)$ be a map with $\al_e=\id$.
We will say that
\begin{enumerate}
\item
$\al$ is called a \emph{modular $\Ga$-kernel}
\index{modular kernel}
if $\al_p\al_q\al_{pq}^{-1}\in \Aut(M)_{\rm m}$
for all $p,q\in \Ga$;

\item
$\al$ is said to be \emph{modularly free}
\index{modularly free}
if $\al_p\nin \Aut(M)_{\rm m}$ for all $p\neq e$.
\end{enumerate}
\end{defn}

We introduce two natural perturbations of a modular $\Ga$-kernel
as follows.

\begin{defn}
Let $\al,\be$
be modular $\Ga$-kernels on a factor $M$.
\begin{enumerate}
\item
We will say they are \emph{strongly conjugate}
\index{conjugate!strongly--}
if there exists $\th\in\oInt(M)$
such that $\th\al_p\th^{-1}=\be_p$
for all $p\in\Ga$.

\item
We will say they
are \emph{strongly outer conjugate}
\index{conjugate!strongly outer--}
if there exists $\th\in\oInt(M)$ and a unitary function
$v\col \Ga\ra U(M)$
such that
$\Ad v_p\circ\th\al_p\th^{-1}=\be_p$
for all $p\in\Ga$.
\end{enumerate}
\end{defn}

Let $\al\col\Ga\ra \Aut(M)$ be a modular
$\Ga$-kernel.
Set $\si_{p,q}^\al:=\al_p\al_q\al_{pq}^{-1}\in \Aut(M)_{\rm m}$.
Then we have
\[
\al_p\al_q=\si_{p,q}^\al\al_{pq}\quad \mbox{on }M.
\]
The canonical extension $\wdt{\si_{p,q}^\al}$
is implemented by some unitary
$w_{p,q}^\al\in U(\tM)$,
but a choice of $w_{p,q}^\al$ has ambiguity
in the expression.
When we want to specify $\al$ with $w_{p,q}^\al$,
we use the triple $(\al,\si^\al,w^\al)$.
From the following equality:
\[
\tal_p\tal_q=\Ad w_{p,q}^\al\tal_{pq}\quad \mbox{on }\tM,
\]
we obtain the cohomological invariants
$d_1^\al$ and $d_2^\al$ as follows:
\begin{equation}\label{eq:d1d2}
w_{p,q}^\al w_{pq,r}^\al=d_1^\al(p,q,r)\tal_p(w_{q,r}^\al)w_{p,qr}^\al,
\quad
\th_s(w_{p,q}^\al)=d_2^\al(s;p,q)w_{p,q}^\al.
\end{equation}
The both $d_1^\al(p,q,r)$ and $d_2^\al(s;p,q)$
are unitaries in $Z(\tM)$,
and they are depending on the choice of $w^\al$.
This ambiguity is resolved in a suitable frame work of a cohomology group.
Now we will state the main theorem of this section.

\begin{thm}\label{thm:mod-ker}
Let $\Ga$ be a discrete amenable group
and $M$ a McDuff factor.
Let $(\al,\si^\al,w^\al)$
and $(\be,\si^\be,w^\be)$
be modular $\Ga$-kernels on $M$.
Assume that they satisfy the following conditions:
\begin{itemize}
\item
$\al_p=\be_p\bmod\oInt(M)$ for all $p\in\Ga$;

\item
$\al$ and $\be$ are modularly free;

\item
$d_1^\al(p,q,r)=d_1^\be(p,q,r)$ for all $p,q,r\in\Ga$;

\item
$d_2^\al(s;p,q)=d_2^\be(s;p,q)$
for all $p,q\in\Ga$, $s\in\R$.
\end{itemize}
Then there exist $\th\in\oInt(M)$ and a map
$v\col \Ga\ra U(M)$
such that
\[
\Ad v_p\circ\th\al_p\th^{-1}=\be_p,
\quad
w_{p,q}^\al=v_p\cdot \th\be_p\th^{-1}(v_q)\cdot
\th(w_{p,q}^\be)v_{pq}^*,
\]
for all $p,q\in\Ga$.
In particular,
$\al$ and $\be$ are strongly outer conjugate.
\end{thm}

Let $\al$ and $\be$ as in the theorem above.
Then we have two maps $w^\al,w^\be\col \Ga\times\Ga\ra U(\tM)$
with
\[
w_{p,q}^\al w_{pq,r}^\al=d_1(p,q,r)\tal_p(w_{q,r}^\al)w_{p,qr}^\al,
\quad
\th_s(w_{p,q}^\al)=d_2(s;p,q)w_{p,q}^\al
\]
and
\[
w_{p,q}^\be w_{pq,r}^\be=d_1(p,q,r)\tbe_p(w_{q,r}^\be)w_{p,qr}^\be,
\quad
\th_s(w_{p,q}^\be)=d_2(s;p,q)w_{p,q}^\be,
\]
where $d_1:=d_1^\al=d_1^\be$ and $d_2:=d_2^\al=d_2^\be$.
Since $\al_p\be_p^{-1}\in\oInt(M)$,
there exists a sequence $(u_p^\nu)_\nu$ in $U(M)$
such that
\[
\al_p=\lim_{\nu\to\infty}\Ad u_p^\nu \be_p
\quad\mbox{for all }p\in\Ga,
\]
where the convergence is taken in the $u$-topology of $\Aut(M)$.
Hence when we consider the unitary $U_p:=\pi_\om((u_p^\nu)_\nu)$
in $M^\om$,
we obtain
$\al_p^\om=\Ad U_p\circ\be_p^\om$ on $M$.
Moreover, the automorphism
$\ga_p:=\Ad U_p\circ\be_p^\om$ on $M^\om$
acts on $M_\om$ globally invariantly.
Let us introduce the unitary
$\chi_{p,q}^{\al,\be}:=w_{p,q}^\al(w_{p,q}^\be)^*$
which connects $\si^\be$ to $\si^\al$ as
$\si_{p,q}^\al=\Ad \chi_{p,q}^{\al,\be}\circ \si_{p,q}^\be$.

\begin{lem}
The unitary $\chi_{p,q}^{\al,\be}$ belongs to $M$.
\end{lem}
\begin{proof}
Since $d_2^\al=d_2^\be$, $\chi_{p,q}^{\al,\be}$
is fixed by the $\R$-action $\th$.
Thus $\chi_{p,q}^{\al,\be}\in M$.
\end{proof}

\begin{lem}
Set
$V_{p,q}:=U_p\be_p^\om(U_q)(\si_{p,q}^\be)^\om(U_{pq}^*)
(\chi_{p,q}^{\al,\be})^*$
for each $p,q\in\Ga$.
Then the following hold:
\begin{enumerate}
\item
$V_{p,q}$ is an element in $M_\om$;
\item
$(\ga,V)$ is a cocycle action on $M_\om$.
\end{enumerate}
\end{lem}
\begin{proof}
(1).
Let $\vph\in M_*$.
Then the first statement is verified as follows:
\begin{align*}
\al_p(\al_q(\vph))
&\sim
\Ad u_p^\nu(\be_p(\al_q(\vph)))
\\
&\sim
\Ad u_p^\nu\be_p(u_q^\nu)(\be_p\be_q(\vph))
\\
&\sim
\Ad u_p^\nu\be_p(u_q^\nu)(\si_{p,q}^\be(\be_{pq}(\vph)))
\\
&\sim
\Ad u_p^\nu\be_p(u_q^\nu)\si_{p,q}^\be((u_{pq}^\nu)^*)(\si_{p,q}^\be\al_{pq}(\vph))
\\
&\sim
\Ad
u_p^\nu\be_p(u_q^\nu)\si_{p,q}^\be((u_{pq}^\nu)^*)
(\chi_{p,q}^{\al,\be})^*(\al_{p}(\al_q(\vph))),
\end{align*}
where
for sequences $\chi_\nu^1,\chi_\nu^2$ of normal functionals
on a von Neumann algebra $N$,
the notation $\chi_\nu^1\sim\chi_\nu^2$
means $\dsp\lim_{\nu\to\infty}\|\chi_\nu^1-\chi_\nu^2\|_{N_*}=0$.

(2).
We first show
$\si_{p,q}^\be(U_{pq}^*)(\chi_{p,q}^{\al,\be})^*U_{pq}\in M$.
Indeed, using $\tsi_{p,q}^\be=\Ad w_{p,q}^\be$ on $\tM$,
we have
\[
\si_{p,q}^\be((u_{pq}^\nu)^*)(\chi_{p,q}^{\al,\be})^*
u_{pq}^\nu
=
w_{p,q}^\be (u_{pq}^\nu)^* w_{p,q}^\al u_{pq}^\nu \in M,
\]
which converges to $w_{p,q}^\be
\tbe_{pq}\tal_{pq}^{-1}((w_{p,q}^\al)^*)\in \tM$ in the strong*
topology.
However, the condition that
$\be_{pq}\al_{pq}^{-1}$ has the trivial Connes--Takesaki
module
and $d_2^\al=d_2^\be$
implies $w_{p,q}^\be
\tbe_{pq}\tal_{pq}^{-1}((w_{p,q}^\al)^*)\in M$.
Hence
\begin{equation}\label{eq:siU}
\si_{p,q}^\be(U_{pq}^*)(\chi_{p,q}^{\al,\be})^*U_{pq}
=w_{p,q}^\be \tbe_{pq}\tal_{pq}^{-1}((w_{p,q}^\al)^*)\in M.
\end{equation}

Next we prove $\ga_p\ga_q=\Ad V_{p,q}\circ\ga_{pq}$ on $M_\om$.
Since $\si_{p,q}^\be(U_{pq}^*)(\chi_{p,q}^{\al,\be})^*U_{pq}\in M$
and $\be_{pq}^\om=\be_p^\om\be_q^\om$ on $M_\om$,
we have
\[
\Ad V_{p,q}\circ\ga_{pq}
=
\Ad U_p\be_p^\om(U_q)\circ \be_{p}^\om\be_q^\om
=
\ga_p\ga_q.
\]

Then $V$ satisfies the cocycle identity,
\begin{equation}\label{eq:V}
V_{p,q}V_{pq,r}=\ga_p(V_{q,r})V_{p,qr}.
\end{equation}
Indeed, the left hand side is equal to
\[
U_p\be_p^\om(U_q) (\si_{p,q}^{\be})^\om(U_{pq}^*)(\chi_{p,q}^{\al,\be})^*
\cdot
U_{pq}\be_{pq}^\om(U_r)(\si_{pq,r}^\be)^\om(U_{pqr}^*)(\chi_{pq,r}^{\al,\be})^*,
\]
and the right is equal to
\[
U_p\be_p^\om(U_q\be_q^\om(U_r)\si_{q,r}^\be(U_{qr}^*)
(\chi_{q,r}^{\al,\be})^*)U_p^*
\cdot
U_{p}\be_{p}^\om(U_{qr})(\si_{p,qr}^\be)^\om(U_{pqr}^*)
(\chi_{p,qr}^{\al,\be})^*.
\]
Thus (\ref{eq:V}) holds if and only if
\begin{align*}
&(\si_{p,q}^{\be})^\om(U_{pq}^*)(\chi_{p,q}^{\al,\be})^*
U_{pq}
\cdot
\be_{pq}^\om(U_r)(\si_{pq,r}^\be)^\om(U_{pqr}^*)
(\chi_{pq,r}^{\al,\be})^*
\\
&=
\be_p^\om(\be_q^\om(U_r)\cdot\si_{q,r}^\be(U_{qr}^*)
(\chi_{q,r}^{\al,\be})^*U_{qr})
\cdot
(\si_{p,qr}^\be)^\om(U_{pqr}^*)(\chi_{p,qr}^{\al,\be})^*
\\
\Leftrightarrow
\ &
\be_p^\om(\be_q^\om(U_r^*))\cdot
(\si_{p,q}^{\be})^\om(U_{pq}^*)
(\chi_{p,q}^{\al,\be})^*
U_{pq}
\cdot
\be_{pq}^\om(U_r)
\cdot
(\si_{pq,r}^\be)^\om(U_{pqr}^*)
(\chi_{pq,r}^{\al,\be})^*
U_{pqr}
\\
&=
\be_p^\om(\si_{q,r}^\be(U_{qr}^*)
(\chi_{q,r}^{\al,\be})^*U_{qr})
\cdot
(\si_{p,qr}^\be)^\om(U_{pqr}^*)
(\chi_{p,qr}^{\al,\be})^*U_{pqr}
\\
\Leftrightarrow
\ &
\be_p^\om(\be_q^\om(U_r^*))\cdot
w_{p,q}^\be \be_{pq}\al_{pq}^{-1}((w_{p,q}^\al)^*)
\cdot
\be_{pq}^\om(U_r)
\cdot
w_{pq,r}^\be \be_{pqr}\al_{pqr}^{-1}((w_{pq,r}^\al)^*)
\\
&=
\be_p^\om(
w_{q,r}^\be \be_{qr}\al_{qr}^{-1}((w_{q,r}^\al)^*))
\cdot
w_{p,qr}^\be \be_{pqr}\al_{pqr}^{-1}((w_{p,qr}^\al)^*)
\quad \mbox{by }(\ref{eq:siU})
\\
\Leftrightarrow
\ &
w_{p,q}^\be
\be_{pq}^\om(U_r^*\al_{pq}^{-1}((w_{p,q}^\al)^*)U_r)
\cdot
w_{pq,r}^\be \be_{pqr}\al_{pqr}^{-1}((w_{pq,r}^\al)^*)
\\
&=
\be_p^\om(
w_{q,r}^\be \be_{qr}\al_{qr}^{-1}((w_{q,r}^\al)^*))
\cdot
w_{p,qr}^\be \be_{pqr}\al_{pqr}^{-1}((w_{p,qr}^\al)^*)
\\
\Leftrightarrow
\ &
w_{p,q}^\be
\be_{pq}(\be_{r}\al_r^{-1}\al_{pq}^{-1}((w_{p,q}^\al)^*))
\cdot
w_{pq,r}^\be \be_{pqr}\al_{pqr}^{-1}((w_{pq,r}^\al)^*)
\\
&=
\be_p^\om(
w_{q,r}^\be \be_{qr}\al_{qr}^{-1}((w_{q,r}^\al)^*))
\cdot
w_{p,qr}^\be \be_{pqr}\al_{pqr}^{-1}((w_{p,qr}^\al)^*)
\\
\Leftrightarrow
\ &
w_{p,q}^\be w_{pq,r}^\be 
\be_{pqr}(\al_r^{-1}\al_{pq}^{-1}((w_{p,q}^\al)^*)
\al_{pqr}^{-1}((w_{pq,r}^\al)^*))
\\
&=
\be_p^\om(w_{q,r}^\be)
w_{p,qr}^\be
\be_{pqr}(\al_{qr}^{-1}((w_{q,r}^\al)^*)\al_{pqr}^{-1}((w_{p,qr}^\al)^*))
\\
\Leftrightarrow
\ &
d_1(p,q,r)
\be_{pqr}(\al_r^{-1}\al_{pq}^{-1}((w_{p,q}^\al)^*)
\al_{pqr}^{-1}((w_{pq,r}^\al)^*))
\\
&=
\be_{pqr}(\al_{qr}^{-1}((w_{q,r}^\al)^*)\al_{pqr}^{-1}((w_{p,qr}^\al)^*))
\\
\Leftrightarrow
\ &
d_1(p,q,r)
\be_{pqr}(\al_{pqr}^{-1}((w_{pq,r}^\al)^*(w_{p,q}^\al)^*w_{pq,r}^\al)
\al_{pqr}^{-1}((w_{pq,r}^\al)^*))
\\
&=
\be_{pqr}(\al_{qr}^{-1}\al_p^{-1}(\al_p(w_{q,r}^\al)^*)
\al_{pqr}^{-1}((w_{p,qr}^\al)^*))
\\
\Leftrightarrow
\ &
d_1(p,q,r)
\be_{pqr}(\al_{pqr}^{-1}((w_{pq,r}^\al)^*(w_{p,q}^\al)^*)
\\
&=
\be_{pqr}(\al_{pqr}^{-1}((w_{p,qr}^\al)^*\al_p(w_{q,r}^\al)^*w_{p,qr}^\al)
\al_{pqr}^{-1}((w_{p,qr}^\al)^*))
\\
\Leftrightarrow
\ &
d_1(p,q,r)
\be_{pqr}(\al_{pqr}^{-1}((w_{pq,r}^\al)^*(w_{p,q}^\al)^*)
\\
&=
\be_{pqr}(\al_{pqr}^{-1}((w_{p,qr}^\al)^*\al_p(w_{q,r}^\al)^*)
\\
\Leftrightarrow
\ &
d_1(p,q,r)
\be_{pqr}(\al_{pqr}^{-1}(d_1(p,q,r)^*))=1.
\end{align*}
This is true because $\be_{pqr}\al_{pqr}^{-1}\in\oInt(M)$.
\end{proof}

\begin{rem}
Recall that $M^\om\subs\tM^\om$
(see Corollary \ref{cor:core-ultra}).
Since $\tilde{\alpha}_p=\lim_\nu \Ad u_p^\nu \circ\tbe_p$
holds in $\Aut(\tM)$,
we can regard $U_p=(u_\pi^\nu)\in \tM^\omega$.
Denote $\wdt{\ga}_p=\Ad U_p\circ\tbe_p$.
Let $V'_{p,q}:=U_p\tbe_p(U_q)w^\beta_{p,q}U_{pq}^*$. Since
 $\mathrm{mod}(\alpha_p\beta_p^{-1})=\id $, 
$[U_p,x]=0$ for $x\in Z(\tM)$.
If we note this fact, the following equality:
\[
V'_{p,q}V_{pq,r}'=d_1(p,q,r)\wdt{\ga}_p(V_{q,r}')V_{p,qr}'
\]
can be shown
in the same way as in the argument before Lemma \ref{lem:nearvanish}.
Since
$w_{p,q}^\alpha w_{pq,r}^\alpha
=d_1(p,q,r)\alpha_p(w_{q,r}^\alpha)w_{p,qr}$, 
we can show the 2-cocycle identify for $V_{p,q}$ in a similar way as in
 \cite[Lemma 4.4]{M}.
%
\end{rem}

\begin{lem}\label{lem:perturb}
Let $\al$ and $\be$ be modular $\Ga$-kernels
as in Theorem \ref{thm:mod-ker}.
Then there exists a sequence of unitaries
$(u_p^\nu)_\nu$ in $M$ for each $p\in\Ga$
such that
\begin{itemize}
\item
$\dsp\al_p=\lim_{\nu\to\infty}\Ad u_p^\nu \be_p$ in $\Aut(M)$.

\item
$\dsp\lim_{\nu\to\om}u_p^\nu \be_p(u_q^\nu)\si_{p,q}^\be((u_{pq}^\nu)^*)
(\chi_{p,q}^{\al,\be})^*=1$.
\end{itemize}
\end{lem}
\begin{proof}
Since $\Ga$ is amenable and $M$ is McDuff,
the cocycle action $(\ga,V)$ on $M_\om$ can be perturbed
to an action as proved in \cite[Lemma 4.3]{MT1}.
Thus we have a unitary $a_p\in M_\om$
with $a_p\ga_p(a_q)V_{p,q}a_{pq}^*=1$.
Since $\si_{p,q}^{\al,\be}$ is centrally trivial,
we have
\[
a_pU_p\be_p^\om(a_q U_q)\si_{p,q}^{\be}(U_{pq}^*a_{pq}^*)
(\chi_{p,q}^{\al,\be})^*
=
a_p\ga_p(a_q)V_{p,q}a_{pq}^*
=1.
\]
Thus a unitary representing sequence of
$a_pU_p$ has the desired properties.
\end{proof}

Recall the following result \cite[Lemma 4.9]{M}:
\begin{lem}\label{lem:nearvanish}
Let $\Ga$ be a discrete amenable group.
Assume $F\Subset \Ga $,
$\Psi\Subset M_*$, $\Phi\Subset (M_*)_+$ and
$\varepsilon>0$ are given.  
Let $S$ be an $(F,\varepsilon)$-invariant finite set. 
Let $\{\gamma_g\}_{g\in \Ga} \subset \Aut(M)$ be a set 
such that $\gamma$ induces a free action on $M_\omega$.
Assume that a family of unitaries $\{u_g\}_{g\in \Ga}\subset U(M)$ satisfies 
\[
\|[\psi, u_k]\|<(3|S|)^{-1}\varepsilon,\,\, k\in S,\,\psi\in \Psi,\]
\[
\|\varphi\cdot(u_g\gamma_g(u_h)u_{g h}^*-1)\|<5\varepsilon,
\quad
\|(u_g\gamma_g(u_h)u_{g h}^*-1)\cdot \varphi\|<5\varepsilon,
\quad
\varphi\in \Phi,\, g\in F, \,h\in S.
\]
Then there exists
$w\in U(M)$ such that
\[\|[w,\psi]\|<\varepsilon,\quad \ps\in \Psi,
\]
\[
\|(u_g\gamma_g(w)w^*-1)\cdot \varphi \|<7\sqrt[4]{\varepsilon},\quad
\|\varphi\cdot(u_g\gamma_g(w)w^*-1) \|<7\sqrt[4]{\varepsilon},
\ g\in F,\
\varphi\in \Phi.
\]
\end{lem}

Together with the two lemmas above,
we can make use of the Bratteli--Elliott--Evans--Kishimoto
intertwining argument as \cite[Theorem 7.1]{MT1},
and then Theorem \ref{thm:mod-ker} will be proved.

Let $\be\col \Ga\ra\Aut(M)$ be a modular $\Ga$-kernel
with $\si_{p,q}^\be\in \Aut(M)_{\rm m}$ as before.
Take a unitary $w_{p,q}^\be\in \tM$ implementing
the canonical extension of $\si_{p,q}^\be$.
When we consider a unitary perturbation of
the pair $(\be,w^\be)$ by $u\in U(M)$,
we always take the following $w^{\Ad u\circ\be}$:
\[
w_{p,q}^{\Ad u\circ\be}:=u_p\be_p(u_q)w_{p,q}^\be u_{pq}^*.
\]
Then the $d_1$ and $d_2$-invariants of
$(\Ad u\circ\be,w_{p,q}^{\Ad u\circ\be})$ are equal to those of
$(\be,w^\be)$.
Indeed, we have
\begin{align*}
w_{p,q}^{\Ad u\circ\be}w_{pq,r}^{\Ad u\circ\be}
&=
u_p\be_p(u_q)w_{p,q}^\be u_{pq}^*
\cdot
u_{pq}\be_{pq}(u_r)w_{pq,r}^\be u_{pqr}^*
\\
&=
u_p\be_p(u_q)w_{p,q}^\be\be_{pq}(u_r)w_{pq,r}^\be u_{pqr}^*
\\
&=
u_p\be_p(u_q)\be_{p}(\be_q(u_r))w_{p,q}^\be w_{pq,r}^\be u_{pqr}^*
\\
&=
u_p\be_p(u_q\be_q(u_r)w_{q,r}^\be u_{qr}^*\cdot u_{qr}(w_{q,r}^\be)^*)
w_{p,q}^\be w_{pq,r}^\be u_{pqr}^*
\\
&=
u_p\be_p(w_{q,r}^{\Ad u\circ\be}\cdot u_{qr}(w_{q,r}^\be)^*)
w_{p,q}^\be w_{pq,r}^\be u_{pqr}^*
\\
&=
\Ad u_p\circ\be_p(w_{q,r}^{\Ad u\circ\be})u_p
\be_p(u_{qr})\cdot
\be_p(w_{q,r}^\be)^*
w_{p,q}^\be w_{pq,r}^\be u_{pqr}^*
\\
&=
\Ad u_p\circ\be_p(w_{q,r}^{\Ad u\circ\be})u_p
\be_p(u_{qr})
\cdot
d_1^\be(p,q,r)w_{p,qr}^\be u_{pqr}^*
\\
&=
d_1^\be(p,q,r)\Ad u_p\circ\be_p(w_{q,r}^{\Ad u\circ \be})
w_{p,qr}^{\Ad u \circ\be},
\\
\th_s(w_{p,q}^{\Ad u\circ\be})
&=
u_p\be_p(u_q)\th_s(w_{p,q}^\be)u_{pq}^*
=
d_2(s;p,q)w_{p,q}^{\Ad u\circ\be}.
\end{align*}

\vspace{10pt}
\noindent{\it Proof of Theorem \ref{thm:mod-ker}.}
Let $n\in\N$. Set $\vep_n:=4^{-n}$,
$(\ga^{(0)},\si^{(0)},w^0):=(\al,\si^{\al},w^\al)$,
$(\ga^{(-1)}, \si^{(-1)},w^{-1}):=(\be,\si^{\be},w^{\be})$
and $\chi^1:=w^{0}(w^{-1})^*=\chi^{\al,\be}$.

Take a finite subset $\Ps_n\subs M_*$ such that
$\emptyset=:\Ps_0\subs \Ps_1\subs
\Ps_2\subs\cdots$ and $\bigcup_{n\geq1}\Ps_n$ is total in $M_*$.
We may and do assume that $\Ps_n^*=\Ps_n$.
Fix a finite subset $F_n \subs \Ga$    
and an $(F_n, \varepsilon_n)$-invariant finite subset
$S_n\subset \Ga$ such that
$F_0=\emptyset$, $e\in F_1$, $F_n\subset F_{n+1}$, 
$\bigcup_{n\geq1} F_n=\Ga$,
$F_n\subset S_n$
and $F_nS_n\subset F_{n+1}$.
Let $\vph_0\in M_*$ be a faithful state.

For each integer $n\geq1$,
we will inductively construct the following ones:
\begin{enumerate}
\renewcommand{\labelenumi}{($n$.\arabic{enumi})}
\item
$\La_n,\Ps_n',\Ph_n'\subs M_*$, finite subsets,
and $\La_n$ is increasing,
\item
$\cu^n\col \Ga\ra U(M)$,
\item
$\ga^{(n)}\col\Ga\ra \Aut(M)$,
\item
$\si^{(n)}\col \Ga^2\ra \Aut(M)_{\rm m}$,
\item
$w^n\col \Ga^2\ra U(\tM)$,
\item
$\chi^n\col\Ga^2\ra U(M)$,
\item
$\mu^n\in U(M)$,
\item
$\de_n>0$,
\end{enumerate}
which satisfy the following recursive formulae:
\begin{enumerate}
\addtocounter{enumi}{+8}
\renewcommand{\labelenumi}{($n$.\arabic{enumi})}
\item
$\dsp
\Ps_n':=\Ps_{n}\cup \Ad (\mu_{n-1}\mu_{n-3}\cdots)(\Ps_{n})
\cup\bigcup_{p\in F_n}\{\cu_p^{n-1}\cdot\vph_0,\vph_0\cdot \cu_p^{n-1}\}$.
\item
$\dsp\Ph_n:=\bigcup_{p\in F_{n+1}}
(\ga_p^{(n-2)})^{-1}(\Ps_n'\cup \La_{n})
\cup
(\ga_p^{(n-1)})^{-1}(\Ps_n'\cup \La_{n})$.

\item
$\ga_{p}^{(n)}:=\Ad \cu_p^n \ga_p^{(n-2)}$,
\item
$\si_{p,q}^{(n)}:=\Ad \cu_p^n\ga_p^{(n-2)}(\cu_q^n)\circ\si_{p,q}^{(n-2)}\circ\Ad (\cu_{pq}^n)^*$,
\item
$w_{p,q}^n:=\cu_p^n\ga_p^{(n)}(\cu_q^n)w_{p,q}^{n-2}(\cu_{pq}^n)^*$,
\item
$\chi_{p,q}^n:= w_{p,q}^{n-1}(w_{p,q}^{n-2})^*$,
\end{enumerate}
and the conditions:
\begin{enumerate}
\addtocounter{enumi}{+14}
\renewcommand{\labelenumi}{($n$.\arabic{enumi})}
\item
If $u\in U(M)$ satisfies
$\|[u,\phi]\|<\de_n$
for all $\ph\in \La_n$,
then
\[
\|(\si_{p,q}^{(n-1)}(u)-u)\cdot\ps\|<\vep_n,\ 
\|\ps\cdot(\si_{p,q}^{(n-1)}(u)-u)\|<\vep_n
\]
for all $(p,q)\in F_n\times S_n, \ps\in\Ps_n'$;

\item
$\dsp\|\ga_p^{(n)}(\ph)-\ga_p^{(n-1)}(\ph)\|
<\frac{\vep_n\de_{n}}{6|S_n|},
\ p\in F_{n+1},
\dsp\ph\in\Ph_n$;\\

\item
$\|(\cu_p^n \ga_p^{(n-2)}(\cu_q^n)
\si_{p,q}^{(n-2)}(\cu_{pq}^n)^*
(\chi_{p,q}^n)^*-1)\cdot\ps\|<\vep_n,\\
\|\ps\cdot(\cu_p^n \ga_p^{(n-2)}(\cu_q^n)
\si_{p,q}^{(n-2)}(\cu_{pq}^n)^*(\chi_{p,q}^n)^*-1)\|<\vep_n,
(p,q)\in F_n\times S_n, \ps\in \Ps_n'$;

\item
$\|(\chi_{p,q}^n-1)\ps\|<\vep_{n-1}$,\\
$\|\ps(\chi_{p,q}^n-1)\|<\vep_{n-1}$,
$(p,q)\in F_{n-1}\times S_{n-1}$, $\ps\in \Ps_{n-1}'$;

\item
$\|(\cu_p^n\ga_p^{(n-2)}(\mu^n)(\mu^n)^*-1)\cdot\ps\|
<7\vep_{n-1}^{1/4}$;
\\
$\|\ps\cdot(\cu_p^n\ga_p^{(n-2)}(\mu^n)(\mu^n)^*-1)\|
<7\vep_{n-1}^{1/4}$,
$\ps\in\Ps_{n-1}'$, $p\in F_{n-1}$;

\item
$\|[\mu^n,\ps]\|<\vep_{n-1}$, $\ps\in \Ps_{n-1}'$.
\end{enumerate}
Hence the only $\de_n$, $\La_n$, $\cu^n$ and $\mu^n$ are indeterminates,
and we will find them in each step
by using central triviality of $\si^{(n-1)}$,
Lemma \ref{lem:perturb} and Lemma \ref{lem:nearvanish}.
The initial data are
\[
(\ga^{(-1)},\si^{(-1)},w^{-1})=(\be,\si^\be,w^\be),
\ (\ga^{(0)},\si^{(0)},w^{0})=(\al,\si^\al,w^\al),
\]
\[
\de_0=1,\ \Ph_0=\Ps_0=\La_0=\emptyset.
\]
\noindent
{\bf Step 1.}
Since $\si_{p,q}^{(0)}$ is an extended modular automorphism,
it is centrally trivial \cite[Theorem 1]{KST}.
Hence we can take $\de_1$ and $\La_1$ as in (1.15).
By Lemma \ref{lem:perturb},
we can take $\cu_p^1\in U(M)$ for $p\in\Ga$
such that
\begin{itemize}
\item
$\dsp\|\Ad \cu_p^1\circ\ga_p^{(-1)}(\ph)-\ga_p^{(0)}(\ph)\|
<\frac{\vep_1\de_1}{6|S_1|}$ for $p\in F_2$,
$\ph\in \Ph_1$,

\item
$\|(\cu_p^1 \ga_p^{(-1)}(\cu_q^1)
\si_{p,q}^{(-1)}(\cu_{pq}^1)^*
(\chi_{p,q}^1)^*-1)\cdot\ps\|<\vep_1,\\
\|\ps\cdot((\cu_p^1 \ga_p^{(-1)}(\cu_q^1))
\si_{p,q}^{(-1)}(\cu_{pq}^1)^*(\chi_{p,q}^1)^*-1)\|<\vep_1$,
$(p,q)\in F_1\times S_1, \ps\in \Ps_1'$.
\end{itemize}
We set $\ga^{(1)}$, $\si^{(1)}$, $w^1$
as (1.11), (1.12) and (1.13).
Then $\ga^{(1)}$ is a modular $\Ga$-kernel on $M$
with
$\ga_p^{(1)}\ga_q^{(1)}=\si_{p,q}^{(1)}\ga_{pq}^{(1)}$
and $\wdt{\si_{p,q}^{(1)}}=\Ad w_{p,q}^1$.
Note that $\chi^1$ is just equal to $\chi^{\al,\be}$.
The conditions (1.16) and (1.17) hold,
and the others (1.18), (1.19) and (1.20) are empty,
but we set $\mu^1:=1$.

\noindent
{\bf Step 2.}
Using central triviality of $\si^1$,
we take $\de_2>0$ and $\La_2\subs M_*$ as in (2.15) and $\La_1\subs \La_2$.
Set $\chi^2$ as (2.14).
We apply Lemma \ref{lem:perturb} to $\ga^{(0)}$ and $\ga^{(1)}$,
and obtain $\cu^2$ such that
\begin{itemize}
\item
$\dsp\|\Ad \cu_p^2\circ\ga_p^{(0)}(\ps)-\ga_p^{(1)}(\ps)\|
<\frac{\vep_2\de_2}{6|S_2|}$
for $p\in F_3$, $\dsp \ps\in\Ph_2$,

\item
$\|((\cu_p^2 \ga_p^{(0)}(\cu_q^2))
\si_{p,q}^{(0)}(\cu_{pq}^2)^*
(\chi_{p,q}^2)^*-1)\cdot\ps\|<\vep_2,\\
\|\ps\cdot((\cu_p^2 \ga_p^{(0)}(\cu_q^2))
\si_{p,q}^{(0)}(\cu_{pq}^2)^*(\chi_{p,q}^2)^*-1)\|<\vep_2$,
$(p,q)\in F_2\times S_2, \ps\in \Ps_2'$.
\end{itemize}
Let $\ga^{(2)}$, $\si^{(2)}$ and $w^n$ as in (2.11), (2.12)
and (2.13).
Then (2.16) and (2.17) hold.
By definition of $\chi^2$,
we have
$\chi_{p,q}^2=
\cu_p^1\ga_p^{(-1)}(\cu_q^1)
\si_{p,q}^{(-1)}((\cu_{pq}^1)^*)(\chi_{p,q}^1)^*$.
Then (1.17) implies (2.18).
By (1.16) and (2.16),
we have
\[
\|\Ad \cu_{p}^2\circ\ga_{p}^{(0)}(\ps)-\ga_{p}^{(0)}(\ps)\|
<\frac{\vep_1\de_1}{3|S_1|},\ p\in F_2,
\ps\in
\bigcup_{r\in F_2}(\ga_r^{(0)})^{-1}(\Ps_1\cup\La_1),
\]
and in particular,
\[
\|[\cu_{p}^2,\ph]\|
<\frac{\vep_1\de_1}{3|S_1|}<\de_1,
\ p\in F_2,
\ph\in \Ps_1\cup \La_1.
\]
From the inequality above and the assumption $F_1S_1\subs F_2$,
it turns out the family $\{\cu_{pq}^2\}_{p\in F_1, q\in S_1}$
satisfies the assumption of (1.15),
and we obtain
\[
\|(\si_{p,q}^{(0)}(\cu_{pq}^2)-\cu_{pq}^2)\cdot\ps\|<\vep_1,
\ 
\|\ps\cdot(\si_{p,q}^{(0)}(\cu_{pq}^2)-\cu_{pq}^2)\|<\vep_1,
\ (p,q)\in F_1\times S_1, \ps\in\Ps_1'.
\]
Thus for $(p,q)\in F_1\times S_1$ and $\ph\in \Ps_1'$, we have
\begin{align*}
&\|
\ph
\cdot
(\cu_p^2 \ga_p^{(0)}(\cu_q^2)
(\cu_{pq}^2)^*-1)
\|
=\|
\ph
\cdot
\cu_p^2 \ga_p^{(0)}(\cu_q^2)
-
\ph\cdot\cu_{pq}^2
\|
\\
&\leq
\|
\ph
\cdot
\cu_p^2 \ga_p^{(0)}(\cu_q^2)
-
\ph\cdot\si_{p,q}^{(0)}(\cu_{pq}^2)
\|
+
\|
\ph\cdot
(\si_{p,q}^{(0)}(\cu_{pq}^2)
-
\cu_{pq}^2)
\|
\\
&<
\|
\ph
\cdot
\cu_p^2 \ga_p^{(0)}(\cu_q^2)
\si_{p,q}^{(0)}(\cu_{pq}^2)^*
-
\ph
\|
+
\vep_1
\\
&\leq
\|
\ph
\cdot
\cu_p^2 \ga_p^{(0)}(\cu_q^2)
\si_{p,q}^{(0)}(\cu_{pq}^2)^*
(\chi_{p,q}^2)^*
-
\ph
\|
+
\|\ph\cdot (\chi_{p,q}^2)^*-\ph\|
+
\vep_1
\\
&<
\vep_2+2\vep_1
<3\vep_1,
\end{align*}
and
\begin{align*}
&\|
(\cu_p^2 \ga_p^{(0)}(\cu_q^2)
(\cu_{pq}^2)^*-1)
\cdot\ph
\|
\\
&\leq
\|
(\cu_p^2 \ga_p^{(0)}(\cu_q^2)
\cdot
((\cu_{pq}^2)^*-\si_{p,q}^{(0)}(\cu_{pq}^2)^*)
\cdot\ph
\|
+
\|
\cu_p^2 \ga_p^{(0)}(\cu_q^2)\si_{p,q}^{(0)}(\cu_{pq}^2)^*
\cdot\ph
-\ph
\|
\\
&\leq
\|
((\cu_{pq}^2)^*-\si_{p,q}^{(0)}(\cu_{pq}^2)^*)
\cdot\ph
\|
+
\|
\cu_p^2 \ga_p^{(0)}(\cu_q^2)\si_{p,q}^{(0)}(\cu_{pq}^2)^*
\cdot
(1-(\chi_{p,q}^2)^*)
\cdot\ph
\|
\\
&+
\|
\cu_p^2 \ga_p^{(0)}(\cu_q^2)\si_{p,q}^{(0)}(\cu_{pq}^2)^*
(\chi_{p,q}^2)^*
\cdot\ph
-\ph
\|
\\
&<
\vep_1+\|(1-(\chi_{p,q}^2)^*)\cdot\ph\|+\vep_2
<2\vep_1+\vep_2<3\vep_1.
\end{align*}
Then by Lemma \ref{lem:nearvanish},
there exists $\mu^2\in U(M)$ with
(2.19) and (2.20).

\noindent{\bf Step {\boldmath$n$}.}
We suppose that the construction up to $(n-1)$ have been done.
Using central triviality of $\si^{(n-1)}$,
we take $\de_n>0$ and $\La_n\subs M_*$ as in ($n$.15)
and $\La_{n-1}\subs \La_n$.
Set $\chi^n$ as ($n$.14).
We apply Lemma \ref{lem:perturb} to $\ga^{(n-2)}$ and $\ga^{(n-1)}$,
and obtain $\cu^n$ such that
\begin{itemize}
\item
$\dsp\|\Ad \cu_p^n\circ\ga_p^{(n-2)}(\ps)-\ga_p^{(n-1)}(\ps)\|
<\frac{\vep_n\de_n}{6|S_n|}$
for $p\in F_{n+1}$, $\dsp \ps\in\Ph_n$,

\item
$\|((\cu_p^n \ga_p^{(n-2)}(\cu_q^n))
\si_{p,q}^{(n-2)}(\cu_{pq}^n)^*
(\chi_{p,q}^n)^*-1)\cdot\ps\|<\vep_n,
\\
\|\ps\cdot((\cu_p^n \ga_p^{(n-2)}(\cu_q^n))
\si_{p,q}^{(n-2)}(\cu_{pq}^n)^*(\chi_{p,q}^n)^*-1)\|<\vep_n$,
\\
$(p,q)\in F_n\times S_n, \ps\in \Ps_n'$.
\end{itemize}
Let $\ga^{(n)}$, $\si^{(n)}$ and $w^n$ as in ($n$.11), ($n$.12)
and ($n$.13).
Then ($n$.16) and ($n$.17) hold.
By definition of $\chi^n$,
we have
$\chi_{p,q}^n=
\cu_p^{n-1}\ga_p^{(n-3)}(\cu_q^1)
\si_{p,q}^{(n-3)}((\cu_{pq}^{n-3})^*)(\chi_{p,q}^{n-1})^*$.
Then ($(n-1)$.17) implies ($n$.18).
By ($(n-1)$.16) and ($n$.16),
we have
\[
\|\Ad \cu_{p}^n\circ\ga_{p}^{(n-2)}(\ps)-\ga_{p}^{(n-2)}(\ps)\|
<\frac{\vep_{n-1}\de_{n-1}}{3|S_{n-1}|}
\]
for
\[
p\in F_n,
\quad
\ps\in
\bigcup_{r\in F_{n}}(\ga_r^{(n-2)})^{-1}(\Ps_{n-1}'\cup\La_{n-1}),
\]
and in particular,
\[
\|[\cu_{p}^n,\ps]\|
<\frac{\vep_{n-1}\de_{n-1}}{3|S_{n-1}|}<\de_{n-1},
p\in F_n,
\ps\in \Ps_{n-1}'\cup \La_{n-1}.
\]
From the inequality above and the assumption $F_{n-1}S_{n-1}\subs F_n$,
it turns out the family $\{\cu_{pq}^2\}_{p\in F_{n-1}, q\in S_{n-1}}$
satisfies the assumption of ($(n-1)$.15),
and we obtain
\[
\|(\si_{p,q}^{(n-2)}(\cu_{pq}^n)-\cu_{pq}^n)\cdot\ps\|<\vep_{n-1},
\ 
\|\ps\cdot(\si_{p,q}^{(n-2)}(\cu_{pq}^n)-\cu_{pq}^n)\|<\vep_{n-1}
\]
for all $(p,q)\in F_{n-1}\times S_{n-1}, \ps\in\Ps_{n-1}'$.
Thus, for $(p,q)\in F_{n-1}\times S_{n-1}$ and $\ps\in \Ps_{n-1}'$,
we have
\begin{align*}
&\|
\ps
\cdot
(\cu_p^n \ga_p^{(n-2)}(\cu_q^n)
(\cu_{pq}^n)^*-1)
\|
=\|
\ps
\cdot
\cu_p^n \ga_p^{(n-2)}(\cu_q^n)
-
\ps\cdot\cu_{pq}^n
\|
\\
&\leq
\|
\ps
\cdot
\cu_p^n \ga_p^{(n-2)}(\cu_q^n)
-
\ps\cdot\si_{p,q}^{(n-2)}(\cu_{pq}^n)
\|
+
\|
\ps\cdot
(\si_{p,q}^{(n-2)}(\cu_{pq}^n)
-
\cu_{pq}^n)
\|
\\
&<
\|
\ps
\cdot
\cu_p^n \ga_p^{(n-2)}(\cu_q^n)
\si_{p,q}^{(n-2)}(\cu_{pq}^n)^*
-
\ps
\|
+
\vep_{n-1}
\\
&\leq
\|
\ps
\cdot
\cu_p^n \ga_p^{(n-2)}(\cu_q^n)
\si_{p,q}^{(n-2)}(\cu_{pq}^n)^*
(\chi_{p,q}^n)^*
-
\ps
\|
+
\|\ps\cdot (\chi_{p,q}^n)^*-\ps\|
+
\vep_{n-1}
\\
&<
\vep_n+2\vep_{n-1}
<3\vep_{n-1},
\end{align*}
and
\begin{align*}
&\|
(\cu_p^n \ga_p^{(n-2)}(\cu_q^n)
(\cu_{pq}^n)^*-1)
\cdot\ps
\|
\\
&\leq
\|
(\cu_p^n \ga_p^{(n-2)}(\cu_q^n)
\cdot
((\cu_{pq}^n)^*-\si_{p,q}^{(n-2)}(\cu_{pq}^n)^*)
\cdot\ps
\|
\\
&\quad
+
\|
\cu_p^n \ga_p^{(n-2)}(\cu_q^n)\si_{p,q}^{(n-2)}(\cu_{pq}^n)^*
\cdot\ps
-\ps
\|
\\
&\leq
\|
((\cu_{pq}^n)^*-\si_{p,q}^{(n-2)}(\cu_{pq}^n)^*)
\cdot\ps
\|
+
\|
\cu_p^n \ga_p^{(n-2)}(\cu_q^n)\si_{p,q}^{(n-2)}(\cu_{pq}^n)^*
\cdot
(1-(\chi_{p,q}^n)^*)
\cdot\ps
\|
\\
&+
\|
\cu_p^n \ga_p^{(n-2)}(\cu_q^n)\si_{p,q}^{(n-2)}(\cu_{pq}^n)^*
(\chi_{p,q}^n)^*
\cdot\ps
-\ps
\|
\\
&<
\vep_{n-1}+\|(1-(\chi_{p,q}^n)^*)\cdot\ps\|+\vep_n
<2\vep_{n-1}+\vep_n<3\vep_{n-1}.
\end{align*}
Then by Lemma \ref{lem:nearvanish},
there exists $\mu^n\in U(M)$ satisfying
($n$.19) and ($n$.20),
and the induction is completed.

We set
$\bu_p^n:=\cu_p^n\ga_p^{(n-2)}(\mu_n)\mu_n^*$
and
$u_p^n:=\bu_p^n\mu_n u_p^{n-2}\mu_n^*$ with $\mu_1=u_p^{-1}=u_p^0=1$.
Then
\[
\ga_p^{(1)}
=
\Ad u_p^1
\circ\Ad \mu_1\circ \be_p \circ\Ad\mu_1^*,
\]
\[
\ga_p^{(2)}
=
\Ad u_p^2
\circ\Ad\mu_2\circ \al_p\circ \Ad\mu_2^*.
\]
We put $\bar{\mu}_{2n}:=\mu_{2n}\mu_{2n-2}\cdots\mu_0$
and $\bar{\mu}_{2n+1}:=\mu_{2n+1}\mu_{2n-1}\cdots\mu_{-1}$.
By induction, we obtain
\begin{align*}
\ga_p^{(2n+1)}
&=
\Ad\cu_p^{2n+1}\circ\ga_p^{(2n-1)}
=
\Ad\bu_p^{2n+1} \circ \Ad\mu_{2n+1}
\circ \ga_p^{(2n-1)}\circ\Ad\mu_{2n+1}^*
\\
&=
\Ad\bu_p^{2n+1} \mu_{2n+1}u_p^{2n-1}\mu_{2n+1}^*
\circ
\Ad\bar{\mu}_{2n+1}
\circ
\be_p\circ
\Ad\bar{\mu}_{2n+1}^*
\\
&=
\Ad u_p^{2n+1}
\circ
\Ad\bar{\mu}_{2n+1}
\circ
\be_p\circ
\Ad\bar{\mu}_{2n+1}^*.
\end{align*}
and
\begin{align*}
\ga_p^{(2n)}
&=
\Ad\cu_p^{2n}\circ\ga_p^{(2n-2)}
=
\Ad\bu_p^{2n} \circ \Ad\mu_{2n}\circ \ga_p^{(2n-2)}\circ\Ad\mu_{2n}^*
\\
&=
\Ad\bu_p^{2n} \mu_{2n}u_p^{2n-2}\mu_{2n}^*
\circ
\Ad\bar{\mu}_{2n}
\circ
\al_p\circ\Ad\bar{\mu}_{2n}^*
\\
&=
\Ad u_p^{2n}
\circ
\Ad\bar{\mu}_{2n}
\circ
\al_p\circ\Ad\bar{\mu}_{2n}^*.
\end{align*}

The condition ($n$.20) implies
the convergences
$\dsp\bar{\th}_1:=\lim_{n\to\infty}\Ad \bar{\mu}_{2n+1}$
and
$\dsp\bar{\th}_0:=\lim_{n\to\infty}\Ad \bar{\mu}_{2n}$
in $\Aut(M)$ with respect to the $u$-topology.
Moreover, 
$\{u_p^{2n}\}_n$ and $\{u_p^{2n+1}\}_n$ 
are Cauchy sequences,
and the limits $\dsp \hat{u}_p^1:=\lim_{n\to\infty}u_p^{2n+1}$
and $\dsp \hat{u}_p^0:=\lim_{n\to\infty}u_p^{2n}$ exist.
Then the condition ($n$.16) yields
\[
\Ad \hat{u}_p^0\circ \bar{\th}_0\al_p\bar{\th}_0^{-1}
=
\Ad \hat{u}_p^1\circ \bar{\th}_1\be_p\bar{\th}_1^{-1}.
\]
Setting
$v_p:=\bar{\th}_0^{-1}\left((\hat{u}_p^0)^*\hat{u}_p^1\right)$
and
$\th:=\bar{\th}_0^{-1}\bar{\th}_1$,
we obtain
\[
\al_p=\Ad v_p\circ \th\circ\be_p\circ\th^{-1}.
\]
About $w^{(2n)},w^{(2n-1)}$, we can prove the following formulae
by induction:
\begin{align}
w_{p,q}^{(2n)}
&=
u_p^{(2n)}\cdot \Ad \bar{\mu}_{2n}
\circ\al_p\circ\Ad \bar{\mu}_{2n}^*
(u_q^{(2n)})
\cdot
\Ad \bar{\mu}_{2n}^*
(w_{p,q}^\al)
\cdot
(u_{pq}^{(2n)})^*,
\label{al:w2n}
\\
w_{p,q}^{(2n+1)}
&=
u_p^{(2n+1)}\cdot \Ad \bar{\mu}_{2n+1}
\circ\be_p\circ\Ad \bar{\mu}_{2n+1}^*
(u_q^{(2n+1)})
\cdot
\Ad \bar{\mu}_{2n+1}^*
(w_{p,q}^\al)
\cdot
(u_{pq}^{(2n+1)})^*
.\label{al:w2n1}
\end{align}
The condition ($n$.18) implies
that $\chi_{p,q}^{n}=w_{p,q}^{n-1}(w_{p,q}^{n-2})^*$ converges to 1
in the strong $*$-topology.
Hence by (\ref{al:w2n}) and (\ref{al:w2n1}),
we obtain
\[
\hat{u}_p^0\cdot
\bar{\th}_0\al_p\bar{\th}_0^{-1}(\hat{u}_q^0)
\cdot
\bar{\th}_0(w_{p,q}^\al)
\cdot
(\hat{u}_{pq}^0)^*
=
\hat{u}_p^1\cdot
\bar{\th}_1\be_p\bar{\th}_1^{-1}(\hat{u}_q^1)
\cdot
\bar{\th}_1(w_{p,q}^\be)
\cdot
(\hat{u}_{pq}^1)^*,
\]
which is rewritten as
\[
w_{p,q}^\al=
v_p\cdot \th\be_p\th^{-1}(v_q)\cdot\th(w_{p,q}^\be)\cdot
v_{pq}^*.
\]
\hfill
$\Box$

Let $d_1^\al$ and $d_2^\al$ be as before.
We define a function $c^\al$ of $\wdt{\Ga}:=\Ga\times\R$
as follows:
\[
c^\al((p,s),(q,t),(r,u)):=d_1^\al(p,q,r)\tal_p(d_2^\al(s;q,r)^*).
\]
Let us write $\meC$ for $Z(\widetilde{M})$ in what follows.
Then $c^\al$ belongs to $Z^3(\wdt{\Ga},U(\meC_M))$,
the set of 3-cocycles,
\index{3-cocycle}
where $(p,s)\in\wdt{\Ga}$ acts on $\meC_M$ as $\tal_p\th_s$.

\begin{lem}
The cohomology class of $c$ is an invariant
of the modular kernel.
\end{lem}
\begin{proof}
Let $\al$ and $w$ be as before.
If we replace $w_{p,q}^\al$ with $w_{p,q}^\al z_{p,q}$,
where $z_{p,q}\in U(\meC_M)$,
then the $d_1^\al$, $d_2^\al$ and $c^\al$ change as
\[
(d_1^\al)'(p,q,r)=d_1^\al(p,q,r)z_{p,q}z_{pq,r}z_{p,qr}^*\tal_p(z_{q,r}^*),
\ 
(d_2^\al)'(s;q,r)=d_2^\al(s;q,r)z_{q,r}^*\th_s(z_{q,r}),
\]
and
\[
(c^\al)'((p,s),(q,t),(r,u))
=
c^\al((p,s),(q,t),(r,u))
z_{p,q}z_{pq,r}\tal_p(\th_s(z_{q,r}^*))z_{p,qr}^*.
\]
Thus $[c^\al]=[(c^\al)']$ in $H^3(\wdt{\Ga},U(\meC_M))$.

Next we assume that
a modular kernel $\be$ is strongly cocycle conjugate to $\al$.
Then there exists a function $v\col\Ga\ra U(M)$ and $\th\in\oInt(M)$
such that
$\be_p=\Ad v_p\circ\th\al_p\th^{-1}$.
Since the unitary $w_{p,q}^\be:=v_p\th\al_p\th^{-1}(v_q)\th(w_{p,q}^\al)v_{pq}^*$
implements $\be_p\be_q\be_{pq}^{-1}$,
we have
\[
d_1^\be(p,q,r)=d_1^\al(p,q,r),
\quad
d_2^\be(s;q,r)=d_2^\al(s;q,r).
\]
Thus $c^\be=c^\al$.
Therefore, $c^\al$ is a strong cocycle conjugacy invariant.
\end{proof}

\begin{defn}
In the notation above,
the cohomology class $\Obm(\al):=[c^\al]$ in $H^3(\wdt{\Ga},U(\meC_M))$
is called a \emph{modular obstruction}.
\index{modular obstruction}
Denote the pair $(\mo(\al),\Obm(\al))$ by $\Inv(\al)$.
\end{defn}

By definition, $\Inv(\al)$ is an element
in
$\Hom(\Ga,\Aut_\th(\meC_M))\times H^3(\wdt{\Ga},U(\meC_M))$.
Now we will strengthen the statement of Theorem \ref{thm:mod-ker}.

\begin{thm}\label{thm:main-kernel-cohom}
The map $\Inv$ gives a complete invariant.
Namely,
if two free modular $\Ga$-kernels $\al$ and $\be$
satisfy
$\Inv(\al)=\Inv(\be)$,
then they are strongly cocycle conjugate.
\end{thm}
\begin{proof}
Let $\al$, $\be$ be as above.
Take $w_{p,q}^\al$ and $w_{p,q}^\be$ as before.
We set $d_i^\al$ and $d_i^\be$ associated with those unitaries.
For $g=(p,s)\in\wdt{\Ga}$, we set $\ps_g:=\tal_p\th_s$.
Then we get a function $z\col \wdt{\Ga}\times\wdt{\Ga}\ra U(\meC_M)$
such that
\[
c^\be(g,h,k)=c^\al(g,h,k)z_{g,h}z_{gh,k}z_{g,hk}^*\ps_g(z_{h,k}^*).
\]

We claim that $z$ satisfies $z_{(e,s),(q,0)}=1$
for all $(q,s)\in\wdt{\Ga}$.

Let $g=(p,s)$, $h=(q,t)$, $k=(r,u)$. If we put $s=t=u=0$, then we have 
\begin{equation}
\label{eq:dbe1}
d^\beta_1(p,q,r)
=d^\alpha_1(p,q,r)\tal_p(z_{q,r}^*)z_{p,qr}^*z_{p,qr}z_{p,q}.
\end{equation}
If we put $p=e$, $t=u=0$, then we have
\begin{equation}
\label{eq:dbe2}
d_2^\beta(s;q,r)
=d^\alpha_2(s;q,r)\theta_s(z_{q,r})z_{s,qr}z_{s,q}^*z_{q,r}^*,
\end{equation}
where we have simply written
$(p,0)$ and $(e,s)$ as $p$ and $s$, respectively.

Here,
$c^\alpha(g,h,k)$
belongs to $Z^3(\wdt{\Ga},U(\meC_M))$
if and only if the following relation
is satisfied \cite[Lemma 2.5]{KtT-outerI};
\[\theta_s(d^\alpha_1(p,q,r))d^\alpha_1(p,q,r)^*
=\alpha_p(d^\alpha_2(s;q,r)^*)
d^\alpha_2(s;p,qr)^*d_2^\alpha(s;pq,r)d^\alpha_2(s;p,q).
\]
We only have to check the 3-cocycle identity for its proof.
The same relation holds for $c^\beta$.
If we substitute (\ref{eq:dbe1}) and (\ref{eq:dbe2})
to the above relation,
then we get $z_{s,qr}=z_{s,q}$.
Thus we have $z_{s,r}=z_{(e,s),(r,0)}=1$.

Then replacing $w_{p,q}^\al$ with $w_{p,q}^\al z_{(p,0),(q,0)}$,
we may and do assume that
$d^\al=d^\be$ (see the proof of the previous lemma).
It follows from Theorem \ref{thm:mod-ker}
that $\al$ and $\be$ are strongly cocycle conjugate.
\end{proof}

\subsection{Models of modular kernels}
We will construct a model of a modular kernel
which attains given invariants $d_1$ and $d_2$.

Let $\Gamma$ be a discrete amenable group,
and suppose that we have a $\Gamma$-modular kernel $\al$
on an injective factor $M$.
Let $\tal $ be the canonical extension,
and fix $w_{p,q}\in U(\tM)$ with $\tal_p\tal_q=\Ad w_{p,q}\circ\tal_{pq}$
as before.
Then we obtain 
$d_1(p,q,r)$, $d_2(t;q,r)\in U(\mathscr{C}_M)$ such that
\[
w_{p,q}w_{pq,r}=d_{1}(p,q,r)\tal_p(w_{q,r})w_{p,qr},
\quad
\theta_t(w_{q,r})=d_2(t;q,r)w_{q,r}.
\]
We set $\wdt{\Ga}:=\Gamma\times \mathbb{R}$,
$\tal_{p,s}:=\tal_p\theta_s$,
$w_{(p,s),(q,t)}:=w_{p,q}$
and
\begin{equation}
\label{eq:cdd}
c((p,s),(q,t),(r,u)):=d_1(p,q,r)\tal_p(d_2(s;q,r)^*).  
\end{equation}
Then $(\tal, w)$ is a $\wdt{\Ga}$-kernel on $\tM$
with a 3-cocycle $c$.
The $c$ has the following important property:
\[
c((p,s),(q,t),(r,u))=c((p,s),(q,0),(r,0)).
\]
Thus we can recover $d_1(p,q,r)$ and $d_2(s;q,r)$ by 
\[
d_1(p,q,r)=c((p,0),(q,0),(r,0)),
\quad
d_2(s;q,r)=c((e,s),(q,0),(r,0))^*.
\]

Conversely for any given action
$\gamma$ of $\Ga$ on $\mathscr{C}_M$, $d_1(p,q,r)$ and $d_2(s;q,r)$,
we will construct a model $\Gamma$-modular kernel with $d_1$ and $d_2$
as below.
Let $c$ be the corresponding 3-cocycle defined in (\ref{eq:cdd}).
By the argument above,
we only have to realize a $\wdt{\Ga}$-kernel on $\tM$
with $c$.

\begin{lem}\label{lem:c-kernel}
Let $c$ be a 3-cocycle defined by (\ref{eq:cdd}).
Then there exists a $\wdt{\Ga}$-kernel $\ga$ on $\tM$
and a function $u\col\wdt{\Ga}\times\wdt{\Ga}\ra U(\tM)$
such that
\begin{itemize}
\item
$\ga_g\ga_h=\Ad u_{g,h}\circ\ga_{hk}$ for all $g,h\in\wdt{\Ga}$;
\item
$\Tr\circ\ga_{(p,s)}=e^{-s}\Tr$ for all $(p,s)\in\wdt{\Ga}$;
\item
$u_{(p,s),(q,t)}=u_{(p,0),(q,0)}$ for all $(p,s),(q,t)\in\wdt{\Ga}$;
\item
$u_{(e,s),(q,0)}=u_{(q,0),(e,s)}=1$ for all $q\in\Ga$, $s\in\R$;
\item
$\ga_{(e,s)}=\th_s$ for $s\in\R$;
\item
$c(g,h,k)=u(g,h)u(gh,k)u(g,hk)^*\ga_g(u(h,k)^*)$
for all $g,h,k\in\wdt{\Ga}$.
\end{itemize}
\end{lem}

If this lemma holds,
$\al_p:=\ga_{(p,0)}|_M$ gives a model as explained below.
Letting $g=(p,0)$, $h=(q,0)$ and $k=(r,0)$
in the final equality in the lemma,
we have
\[
d_1(p,q,r)=w_{p,q}w_{pq,r}w_{p,qr}^*\ga_p(w_{q,r}^*),
\quad
w_{p,q}:=u((p,0),(q,0)).
\]
To compute $d_2(s;q,r)$,
we put $g=(e,s)$, $h=(q,0)$ and $k=(r,0)$
in that equality, and we obtain
\[
\th_s(w_{q,r})=d_2(s;q,r)w_{q,r}.
\]
Thus $\si_{q,r}:=\Ad w_{q,r}|_M$
is an extended modular automorphism (see Lemma \ref{lem:modular-can}).
Finally we show $\tal_p=\ga_{(p,0)}$.
From Lemma \ref{lem:beta-can}, there exists $s_p\in\R$ such that
$\ga_{(p,0)}=\tal_p\th_s$.
However, $\ga$ and $\tal$ preserves the trace $\Tr$,
we have $s=0$.
Therefore, the modular $\Ga$-kernel
$\al$ gives the invariant $(d_1,d_2)$.

Before we start to prove the previous lemma,
we explain how such a kernel can be constructed.
Let $\alpha$ be an action of a locally compact group $G$ on a von
Neumann algebra $P$.
Let $N:=\{g\in G\mid \alpha_g\in \Int(P)\}$,
and fix $u_n\in U(P)$ with $\alpha_n=\Ad u_n$.
Then $\lambda(g,n)$, $\mu(m,n)\in U(Z(P))$ are obtained by 
\[
\alpha_g(u_{g^{-1}ng})=\lambda(g,n)u_n,\,\,\, u_mu_n=\mu(m,n)u_{mn}.
\]

Fix a section $Q:=G/N\ni p\mapsto \tilde{p}\in G$. 
Let $\mathfrak{n}(p,q):=\tilde{p}\tilde{q}\widetilde{pq}^{-1}\in N$.
Then we get the $Q$-kernel $\gamma_p:=\alpha_{\tilde{p}}$.
We compute a 3-cocycle $c(p,q,r)$ associated with $\gamma$.
Let $v(p,q):=u_{\mathfrak{n}(p,q)}$.
Then $\gamma_p\gamma_q=\Ad v(p,q)\gamma_{pq}$,
and 
$c(p,q,r)$ is obtained as follows; 
\begin{equation}\label{eq:cv}
c(p,q,r):=v(p,q)v(pq,r)v(p,qr)^*\gamma_{p}(v(q,r))^*.
\end{equation}
Thus we have
\begin{align*}
 c(p,q,r)
&=u_{\mathfrak{n}(p,q)}u_{\mathfrak{n}(pq,r)}
u_{\mathfrak{n}(p,qr)}^*\alpha_{\tilde{p}}(u_{\mathfrak{n}(q,r)}^*) \\
&=
\lambda(\tilde{p},\tilde{p}\mathfrak{n}(q,r)\tilde{p}^{-1})^*
\mu(\mathfrak{n}(p,q),\mathfrak{n}(pq,r))u_{\mathfrak{n}(p,q)\mathfrak{n}(pq,r)}
u_{\mathfrak{n}(p,qr)}^*u_{\tilde{p}\mathfrak{n}(q,r)\tilde{p}^{-1}}^* \\
&=
{\lambda(\tilde{p},\tilde{p}\mathfrak{n}(q,r)\tilde{p}^{-1})}^*\mu(\mathfrak{n}(p,q),
\mathfrak{n}(pq,r))
{\mu(\tilde{p}\mathfrak{n}(q,r)\tilde{p}^{-1},\mathfrak{n}(p,qr))}^*.
\end{align*}

It is easy to see that the cohomology class of $c(p,q,r)$ does not depend
on the choice of $u_n$ (and hence $\lambda(g,n)$ and $\mu(m,n)$).
If we first have a pair $(\la,\mu)$, then by defining the 3-cocycle
$c$ as above,
we obtain the following Huebschmann--Jones--Ratcliffe map (HJR-map):
\index{Huebschmann--Jones--Ratcliffe map}
\[
\delta:\Lambda(G,N,U(Z(P)))\rightarrow H^3(Q,U(Z(P))),
\]
which maps $[\lambda,\mu]$ to $[c(p,q,r)]$.
See \cite{J-act} for a detailed account.

Thus for a given 3-cocycle $c(p,q,r)$,
to realize a $Q$-kernel as (\ref{eq:cv}),
what we have to do is to find the following objects:
\begin{enumerate}
\item
a group $G$ and its normal subgroup $N$ with $G/N=Q$;

\item
a characteristic invariant $[\lambda,\mu]$ such that
$\delta([\lambda,\mu])=[c] $;
\item
an action of $G$
with a given characteristic invariant.
\end{enumerate}


\begin{proof}[Proof of Lemma \ref{lem:c-kernel}]
Following the argument presented in \cite[Lemma 2.14]{KtT-outerI},
we first solve those three problems.
We set $B:=U(\mathscr{C}_M\otimes L^\infty(\wdt{\Ga}))=U(L^\infty(\wdt{\Ga},\mathscr{C}_M))$.
Consider the following exact sequence
\[
1\longrightarrow U(\mathscr{C}_M)\longrightarrow B
\stackrel{\pi}{\longrightarrow} C\longrightarrow 1,
\]
where $C$ is a quotient group. 
We define an action of
$\wdt{\Ga}$ on $\mathscr{C}_M\otimes L^\infty(\wdt{\Ga})$
by 
$\ga_g:=\ps_g\otimes \Ad \rho_g$,
where $\rho_g$ is the right regular representation.
Then $\wdt{\Ga}$ naturally acts on $C$.

Define $u(g,h)\in B$
by $u(g,h)(x):=\ps_x^{-1}(c(x,g,h)^*)$ for $x,g,h\in\wdt{\Ga}$.
Then
\begin{equation}
\label{eq:uuc}
u(g,h)u(gh,k)=c(g,h,k)\gamma_{g}(u(h,k))u(g,hk).  
\end{equation}
Indeed, by the
3-cocycle relation 
\[
c(x,g,h)c(xg,h,k)^*c(x,gh,k)c(x,g,hk)^*\ps_x(c(g,h,k))=1,
\]
we have 
\begin{align*}
\lefteqn{ u(g,h)(x)u(gh,k)(x)u(g,hk)^*(x)\gamma_g(u(h,k))^*(x) } \\
&=
\ps_x^{-1}(c(x,g,h)^*)
\ps_x^{-1}(c(x,gh,k)^*)
\ps_x^{-1}(c(x,g,hk))
\ps_g(\ps_{xg}^{-1}(c(xg,h,k)))\\
&=
\ps_x^{-1}\left(c(x,g,h)^*
c(x,gh,k)^*
c(x,g,hk)
c(xg,h,k)\right)\\
&=c(g,h,k).
\end{align*}

Thus
$\pi(u(g,h))$ is a $C$-valued 2-cocycle for $\wdt{\Ga}$.
The equality $c(x,(p,s),(q,t))=c(x,(p,0),(q,0))$ implies
$u((p,t),(q,s))=u((p,0),(q,0))$.
Hence we have
$\mathfrak{n}(p,q):=\pi((p,0),(q,0))$
is an element of $Z^2(\Gamma,C)$.
Next, we set $g=(e,t)$, $h=(q,0)$ and $k=(r,0)$ in (\ref{eq:uuc}). 
Using $u(g,h)=u((e,t),(q,0))=1$,
we have
\[
\theta_t(u((q,0),(r,0)))=d_2(t;q,r)u((q,0),(r,0)).
\]
Thus
$\theta_t(\mathfrak{n}(q,r))=\mathfrak{n}(q,r)$ in $C$.
Let $N:=\{\ps_{(p,0)}(\mathfrak{n}(q,r))\mid p,q,r\in \Gamma\}$.
Then $N\subs C$ is a countable abelian group fixed by $\th$.
Hence the cocycle twisted product
$N\rtimes_{\mathfrak{n}}\wdt{\Ga}$ is canonically identified with 
$(N\rtimes_{\mathfrak{n}} \Gamma) \times \mathbb{R}$,
where the product of two elements is described as
\[
(m,g)(n,h)=(m\ga_g(n)\mathfrak{n}(g,h),gh)
\quad
\mbox{for }m,n\in N,\ g,h\in\wdt{\Ga}.
\]
Let $H(c):=N\rtimes_{\mathfrak{n}} \Gamma$, which is a discrete amenable
group, and denote the canonical section
$\wdt{\Gamma}\rightarrow H(c)\times\R$
and the quotient map by $\mathfrak{s}_{c}$
and
$q_c$, respectively.
Thus we have the following exact sequence;
\[
1\longrightarrow N
\longrightarrow H(c)\times \mathbb{R}
\stackrel{q_c}{\longrightarrow}
\wdt{\Ga}\longrightarrow 1.
\]
Then every $\wdt{\Ga}$-module
is naturally regarded as an $H(c)\times \mathbb{R}$-module
through $q_c$.
%
%
%

Let $E:=\pi^{-1}(N)\subs B$.
We have the following $H(c)$-equivariant exact sequence. 
\[
1 \longrightarrow  U(\mathscr{C}_M) \longrightarrow E
\stackrel{\pi_E}{\longrightarrow }
N\longrightarrow 1,
\]
where $\pi_E=\pi|_E$.
Let $\mathfrak{s}_N$ be a section of $\pi_E$.
Then we get 
the following characteristic cocycle
$[\lambda,\mu]$ belonging to $\Lambda(H(c),N,U(\mathscr{C}_M))$:
for $g\in H(c),\ m,n\in N$,
\[
\gamma_g(\mathfrak{s}_N(g^{-1}ng))=\lambda(g,n)\mathfrak{s}_N(n),
\quad
\mathfrak{s}_N(m)\mathfrak{s}_N(n)
=\mu(m,n)\mathfrak{s}_N(mn).
\]
We will verify the equality $\delta([\lambda,\mu])=[c]$.
Consider the following exact sequence 
\[
1 \longrightarrow  U(\mathscr{C}_M) \longrightarrow E
\stackrel{\pi_E}{\longrightarrow }
H(c)\times \mathbb{R}
\stackrel{q_c}{\longrightarrow} \wdt{\Ga}\longrightarrow 1.
\]

By the definition of $H(c)$,
we have
$\mathfrak{s}_c(g)\mathfrak{s}_c(h)
=
\mathfrak{n}(g,h)\mathfrak{s}_c(gh)$,
$g,h\in \wdt{\Ga}$.
Set $f(g,h):=\mathfrak{s}_N(\mathfrak{n}(g,h))\in E$.
By $\pi_E(u(g,h))=\mathfrak{n}(g,h)=\pi_E(f(g,h))$, 
$f(g,h)=z(g,h)u(g,h)$ for some $z(g,h)\in U(\mathscr{C}_M)$. 
On one hand, we have
\begin{align*}
\lefteqn{ f(g,h)f(gh,k)f(g,hk)^*\gamma_g(f(h,k))^*}\\
&=z(g,h)z(gh,k)\gamma_g(z(h,k)^*)z(g,hk)^*
u(g,h)u(gh,k)
u(g,hk)^*\gamma_g(u(h,k)^*) \\
&=z(g,h)z(gh,k)\gamma_g(z(h,k)^*)z(g,hk)^*c(g,h,k).
\end{align*}

On the other hand, we have
\begin{align*}
\lefteqn{ f(g,h)f(gh,k)f(g,hk)^*\gamma_g(f(h,k))^* } \\
&=
\mathfrak{s}_N(\mathfrak{n}(g,h))
\mathfrak{s}_N(\mathfrak{n}(gh,k))
\gamma_g(\mathfrak{s}_N(f(h,k))^*)
\mathfrak{s}_N(f(g,hk)^*) \\
&=\lambda(\mathfrak{s}_c(g),\mathfrak{s}_c(g)f(h,k)\mathfrak{s}_c(g)^{-1})^*
\mu(\mathfrak{n}(g,h),\mathfrak{n}(gh,k))
\mu(\mathfrak{s}_c(g)\mathfrak{n}(h,k)\mathfrak{s}_c(g)^{-1},\mathfrak{n}(g,hk))^*.
\end{align*}
This shows that $\delta([\lambda,\mu])=[c]$.

It is known that for a group homomorphism
$\ps\col H(c)\times\R\to \Aut_\theta(\mathscr{C}_M)$
with $\gamma_n=\id$, $n\in N$, 
and
a characteristic element
$[\lambda,\mu]\in \Lambda(H(c)\times \mathbb{R}, N, U(\mathscr{C}_M))$,
there exists an action $\beta$ of $H(c)\times \mathbb{R}$ on $\tM$
with $\mathrm{Inv}(\beta)=(N,\gamma_g,[\lambda,\mu])$.
For readers' convenience, we briefly recall the
construction presented in \cite{M}, which does not involve groupoid theory.

Set $\lambda'(g,n):=\lambda((g,0),n)$ and 
$c_t(n)=\lambda((e,t),n)$, $g\in H(c)$, $n\in N$, $t\in\mathbb{R}$. 
Let $\varphi$ be a dominant weight
on $M$
and $\alpha^{(0)}$ a free action of $H(c)$ on $R_0$,
the injective factor of type II$_1$.
We use the notation in Section \ref{subsect:extend}, and 
we denote the embedding of $M_\varphi$ into $M$ by $\pi^0$.
We identify $\mathscr{C}_M$ with $Z(M_\varphi)$.
Set $\alpha^{(1)}_n:=\sigma^\varphi_{c(n)^*}\otimes \alpha^{(0)}_n$. 
Then $(\alpha^{(1)}_n,\pi^0(\mu(m,n))\otimes 1)$ is a cocycle action of
$N$ on $M\otimes R_0$.
Since each $\alpha_n^{(1)}$, $n\neq 1$
is not modular,
the flow of the weights of
$(M\otimes R_0)\rtimes_{\alpha^{(1)},\pi^0(\mu)\oti1} N$
coincides with that of $M$
(see Section \ref{subsect:cocycle-cross}
for the notion of cocycle crossed product).
Thus those factors are isomorphic since they are injective.

Thanks to \cite[Corollary 1.3]{ST},
we can lift $\ps_g$ to an action of $\wdt{\Ga}$ on $M$,
which is denoted by $\ga_g\in\Aut(M)$.
Then we extend an action
$\beta_g:=\gamma_{q_c(g)}\otimes \alpha^{(0)}_g$ of $H(c)$
on $M\otimes R_0$ to
$(M\otimes R_0)\rtimes_{\alpha^{(1)},\mu\oti1} N$
by
\[
\beta_g(w_{g^{-1}ng})=(\pi^0(\lambda'(g,n))\otimes 1)w_n,
\]
where $w_n$ is the implementing unitary
(for simplicity, we write $\pi^0(\lambda(g,n))$ for $\pi^0(\lambda(g,n))\otimes 1$).
For if we put $G=N$, $\alpha_n=\alpha^{(1)}$, $\theta=\beta_g$, and
$u_n^\theta =\pi^0(\lambda'(g,n))$ in Theorem \ref{thm:ext},
then they satisfy the required condition,
where $G$ acts on $N$ by conjugation.

Let $\beta$ be an action of $H(c)$ constructed as above. We will see
$\beta$ realizes the given invariant. 

\begin{cla}
For $n\in N$,
$\tbe_n= \Ad \left((V_{c(n)}\otimes 1)w_n\right)$
on $(\tM\oti R_0)\rtimes_{\tal,\pi^0(\mu)}N$.
\end{cla}
\begin{proof}[Proof of Claim]
For $x\in \tM\otimes R_0$, we have
\[
\Ad \left((V_{c(n)}\otimes 1)w_n\right) x
=
\Ad (V_{c(n)}\otimes 1)\circ
(\sigma_{(c(n))^*}^\varphi\otimes\alpha_n^{(0)} )(x)
=(\id\otimes \alpha_n^{(0)})(x)=\beta_n(x),
\]
since $q_c(n)=e$ for $n\in N$.
For $w_m$, we have
\begin{align*}
\Ad \left((V_{c(n)}\otimes 1)w_n\right) w_m
&=
\pi^0(\mu(m,n)\mu(n,n^{-1}mn)^*)\Ad \left((V_{c(n)}\otimes 1)\right)w_{nmn^{-1}}
\\
&=\pi^0(\lambda(m,nmn^{-1}))w_{nmn^{-1}}
\left(\wdt{\sigma_{c(nmn^{-1})}^\vph}(V_{c(n)})\otimes 1\right)
(V_{c(n)}^*\oti1)
\\
&= \pi^0(\lambda(m,nmn^{-1}))w_{nmn^{-1}}
\\
&= \beta_n(w_m).
\end{align*}
Thus $\tbe_n=\Ad (V_{c(n)}\otimes 1)w_n$ holds.
\end{proof}

Next,
we show that $g\in N$ and only if
$\tbe_g$ is inner on
$\tM\rtimes_{\alpha^{(1)},\mu} N$.
Assume $g\in H(c)\backslash N$ and 
$a\in \tM\rtimes_{\alpha^{(1)},\mu} N$
satisfies
$ax=\tbe_g(x)a$ for all $x\in \tM\rtimes_{\alpha^{(1)},\mu} N$.
Let $a=\sum_{n\in N}a_nw_n$ be the formal expansion of $a$.
Then for $x\in \tM\oti R_0$, we have
\[
a_n(\sigma^\varphi_{c(n)^*}\otimes \alpha^{0}_n)(x)
=
(\wdt{\ga}_{q_c(g)}\otimes \alpha_g^0)(x)a_n
=
\tbe_g(x)a_n.
\]
Thus we have
\[
a_n x
=
(\wdt{\ga}_g\sigma^\varphi_{c(n)}\otimes \alpha_{gn^{-1}}^0)(x)a_n
\quad\mbox{for all }x\in \tM\oti R_0.
\]
By assumption, $\alpha_{gn^{-1}}^0$ is outer for all $n$.
Hence $a_n=0$, and so is $a=0$.
This means that the modular part of $\be$ equals $N$.

Using the identity
$c_t(m)c_t(n)=c_t(mn)\mu(m,n)^*\theta_t(\mu(m,n))$
and Lemma \ref{lem:Vc-th} (3),
we have the following for $m,n\in N$:
\begin{align*}
(V_{c(m)}\otimes 1)w_m (V_{c(n)}\otimes 1)w_n
&=
(V_{c(m)}V_{c(n)}\otimes 1)w_m w_n \\ 
&=\pi^0(\mu(m,n))(V_{c(m)c(n)}\otimes 1)w_{mn}\\
&=\pi^0(\mu(m,n))(V_{c(mn)\partial(\mu(m,n)^*)}\otimes 1)w_{mn}\\
&=\mu(m,n)(V_{c(mn)}\otimes 1)w_{mn}.
\end{align*}

By 
$\gamma_g(c_t(g^{-1}ng))=c_t(n)\lambda(g,n)^*\theta_t(\lambda(g,n))$ and
Lemma \ref{lem:Vc-th} (3), we get the following for $g\in H(c)$ and $n\in N$,
\begin{align*}
\tbe_g\left((V_{c(g^{-1}ng)}\otimes 1)w_{g^{-1}ng}\right)
&=
\pi^0(\lambda(g,n))\left(V_{\gamma_g(c(g^{-1}ng))}\otimes 1\right)w_{n}
\\
&=\lambda(g,n)\left(V_{c(n)}\otimes 1\right)w_n.
\end{align*}

By Lemma \ref{lem:Vc-th} (3), we obtain
\[
\theta_t\left((V_{c(n)}\otimes 1)w_n\right)=c_t(n)
(V_{c(n)^*}\otimes 1)w_n.
\]

Therefore,
$\beta$ is an action
with the invariant $(N,\ps,[\lambda,\mu])$.
In particular,
$\beta_{\mathfrak{s}_c(p)}$,
$p\in\Ga$ is a desired $\Gamma$-modular kernel.
\end{proof}

From the previous lemma,
we obtain the following.

\begin{thm}\label{thm:model-kernel}
For any pair $(d_1,d_2)$ as before,
there exists a modular $\Ga$-kernel which attains $(d_1,d_2)$.
\end{thm}

\section{Classification of actions with non-trivial modular parts}
Let $\al \col M\ra M\oti\lhG$ be an action of $\bhG$ on a factor $M$.
Recall the modular part $\La(\al)$ that is defined
in Definition \ref{defn:modular-part} as follows:
\[
\La(\al):=\{\rho\in\IG\mid \al_\rho\mbox{ is modular}\}.
\]

It is trivial that $\La(\al)$ generates a fusion algebra,
that is,
if $\pi,\rho$ belongs to $\La(\al)$,
then so do $\opi$
and every irreducible summand of the tensor product
representation $\pi\rho$.
This means $\CG_\pi$ for $\pi\in\La(\al)$
generates a function algebra of a compact Kac algebra $\bH_\al$,
that is,
\[
C(\bH_\al)=\ovl{\spa}^{\|\cdot\|}\{\CG_\pi\mid\pi\in\La(\al)\}.
\]
Then all irreducible representations of $\bH_\al$
are belonging to $\La(\al)$.
Thus $\bH_\al$ is regarded as a quotient of $\bG$.

We will study classification on a case-by-case basis.
In this section, we study the case that $\Irr(\bH_\al)\neq\{\btr\}$,
and in the next, $\Irr(\bH_\al)=\{\btr\}$.
The latter (i.e. the centrally free case)
has been partially proved in \cite{MT2,MT3},
that is,
the authors proved the cocycle conjugacy of
invariantless (i.e. approximately inner and centrally free) actions
on injective factors.
However, we can improve those proofs to be
adapted to actions with non-trivial Connes--Takesaki module.

To study the first case,
we have to set up the following
assumptions on actions at this time:
\begin{ass}\label{ass:modular}
In this section, $\La(\al)$ is normal.
\end{ass}

From Corollary \ref{cor:Vps},
the assumption implies $\al$ has the Connes--Takesaki module.
Also recall that Theorem \ref{thm:connected}
states all connected simple Lie groups
or their $q$-deformations
with $q=-1$ satisfy the assumption above.
In what follows, we simply denote $\bH_\al$ by $\bH$.

%
%
%

\subsection{Discrete group $\Ga$}
We want to describe the ``kernel'' of the quotient map
$\bG\ra\bH$ or equivalently, the ``quotient'' $\IG$ by $\IH$.
For the sake of this,
we introduce a relation $\sim$ on $\IG$,
which is defined as $\pi_1\sim\pi_2$
if each irreducible in $\pi_1\ovl{\pi_2}$
belongs to the modular part $\IH$
of the action $\al$.

\begin{lem}\label{lem:sim}
The relation $\sim$ satisfies the following conditions:
\begin{enumerate}
\item
$\sim$ is an equivalence relation.

\item
$\rho\sim\btr$ if and only if $\rho\in\IH$.

\item
Each irreducible contained in
$\pi_1\pi_2$ for all $\pi_1,\pi_2\in\IG$
is equivalent to each other.
\end{enumerate}
\end{lem}
\begin{proof}
(1).
Let $\pi,\rho,\si\in\IG$.
By our assumption, $\pi\opi$ decomposes into
the irreducibles in $\IH$.
This means $\pi\sim\pi$.

Suppose that $\pi\sim \rho$, that is,
each irreducible contained in $\pi\orho$
is in $\IH$.
Since the set $\IH$ is closed under conjugation,
the decomposition of $\rho\opi=\ovl{\pi\orho}$
is contained in $\IH$, that is, $\rho\sim\pi$.

When $\pi\sim\rho$ and $\rho\sim\si$,
the morphism $\al_{\pi\orho\rho\osi}=\al_{\pi\orho}\al_{\rho\osi}$
is implemented by unitary,
and so is the morphism $\al_{\orho\rho}$ by our assumption.
Thus $\al_{\pi\osi}$ is implemented by unitary.
Thanks to \cite[Proposition 3.4 (1)]{Iz},
we see $\al_\xi$ is implemented by unitary
for all $\xi\prec\pi\osi$,
and $\xi\in\IH$.
Hence $\pi\sim\si$.

(2).
It is trivial.

(3).
Let $\xi,\eta\prec\pi_1\pi_2$.
Then the morphism $\al_{\xi\ovl{\eta}}$
is contained in
$\al_{\pi_1\pi_2\opi_2\opi_1}
=
\al_{\pi_1}\al_{\pi_2\opi_2}\al_{\opi_1}$
that is implemented by unitary.
Again by \cite[Proposition 3.4 (1)]{Iz},
we see $\xi\sim\eta$ as in the proof of (2).
\end{proof}

Let us introduce
the set of tensor product representations
denoted by $\Rep(\bG)_0$,
that is,
$\Rep(\bG)_0
:=\{\pi_1\pi_2\cdots\pi_n\mid \pi_i\in\IG, 1\leq i\leq n,n\in\N\}$.
Then the equivalence relation $\sim$ naturally extends to
$\Rep(\bG)_0$.
Note that if $\xi,\eta\prec \pi_1\pi_2\cdots\pi_n$
are irreducibles, then $\xi\sim\eta$.

Let us denote by $\Ga$ and $[\pi]$
the quotient set $\IG/\sim=\Rep(\bG)_0/\sim$
and the equivalence class of $\pi\in\Rep(\bG)_0$,
respectively.
When $\IG$ is a discrete group,
then $\Ga$ is nothing but the quotient group $\IG/\IH$.

\begin{lem}
The set $\Ga$ possesses a discrete group structure as follows:
\begin{enumerate}
\item $[\pi_1][\pi_2]:=[\pi_1\pi_2]$.

\item $[\btr]$ is the unit.

\item $[\opi]$ is the invertible element of $[\pi]$.
\end{enumerate}
\end{lem}
\begin{proof}
(1). We show well-definedness.
Let $\pi_1\sim\pi_1'$ and $\pi_2\sim\pi_2'$.
Then $\pi_1\pi_2\sim\pi_1'\pi_2'$
because $\al_{\pi_1\pi_2\ovl{\pi_2'}\ovl{\pi_1'}}$
is implemented by unitary as the proof of the previous lemma.
(2), (3). They are trivial.
\end{proof}

\begin{lem}
If $\bhG$ is amenable,
so is $\Ga$.
\end{lem}
\begin{proof}
Let $f\col \el^\infty(\Ga)\ra \lhG$ be
the embedding defined by
$f(\de_{[\pi]})=\sum_{\pi'\sim\pi}1_{\pi'}$
for each $\pi\in\IG$.
Let $m$ be an invariant mean on $\lhG$.
We show $m\circ f$ is an invariant mean.

Let $\pi_1\in\IG$.
We claim that
$(\tr_{\opi_1}\oti\id)(\De(f(\de_{[\pi]})))=f(\de_{[\pi_1][\pi]})$,
where $\tr_{\opi_1}$ denotes the normalized trace on $B(H_{\opi_1})$.
Let $\rho\in\IG$.
Then we get
\begin{equation}\label{eq:trf}
(\tr_{\opi_1}\oti\id)(\De(f(\de_{[\pi]})))1_\rho
=
\left(
\sum_{\pi'\sim\pi}
\frac{\dim(\pi',\opi_1\rho)d(\pi')}{d(\pi_1)d(\rho)}
\right)1_\rho.
\end{equation}
Suppose that $\rho$ is contained in the
conjugacy class $[\pi_1\pi]$.
Then we can take an irreducible $\pi'$
such that $\pi'\prec\opi_1\rho$ and $\pi'\sim\pi$.
By Lemma \ref{lem:sim} (3),
each irreducible contained in $\opi_1\rho$
is equivalent to $\pi$.
Hence the summation in (\ref{eq:trf})
is taken for all irreducibles in $\opi_1\rho$,
and that is equal to $1_\rho$.
If $\rho$ is not contained in the conjugacy class
$[\pi_1\pi]$, then it turns out
that the both sides of (\ref{eq:trf}) are equal to 0.
Thus we have proved the claim.
This claim immediately yields the invariance of $m\circ f$.
\end{proof}

Now recall the discussion we had in Section \ref{subsect:GCQS}.
There exists a generalized central quantum subgroup
$\bK$ corresponding to $\IH$
as shown in Theorem \ref{thm:corr-gc-nc}.
For the restriction map $r_\bK\col\CG\ra\CK$,
we have a unitary $g_\pi\in \CK$ such that
\[
(\id\oti r_\bK)(v_\pi)=1\oti g_\pi
\quad
\mbox{for }\pi\in\IG.
\]
Then these $g_\pi$'s are generating the dual of $\bK$,
which we denote by $\wdh{\bK}$.

\begin{prop}
The map $g\col\IG\ra \wdh{\bK}$ factors through
$\Ga$, and this factor map is a group isomorphism.  
\end{prop}
\begin{proof}
Let $\pi,\rho\in\IG$ with $g_\pi=g_\rho$.
Then we have
$(\id_\pi\oti\id_\rho\oti r_\bK)(v_\pi v_\orho)=1$.
This means that each irreducible of $\pi\orho$
is contained in $\IH$,
which is equivalent to $[\pi]=[\rho]$ in $\Ga$.
Thus $g$ factors through $\Ga$
and the factor map $g\col\Ga\ra \wdh{\bK}$
is bijective.
Since
\[
(\id_\pi\oti\id_\rho\oti r_\bK)(v_\pi v_\rho)
=
g_\pi g_\rho,
\]
any irreducible $\si$ of $\pi\rho$ satisfies
$g_\si=g_\pi g_\rho$.
Hence $g$ is a group isomorphism.
\end{proof}

Thus we have
\[
C(\bG/\wdh{\Ga})=\CH.
\]

\begin{ex}
Let $\bG=SU(N)$ with $N\geq2$.
\index{$SU(N)$}
Then the center $Z(SU(N))$ coincides with
the finite group $\{\om^k 1_N\mid k=0,1,\dots,N-1\}$
where $\om:=e^{2\pi i/N}$.
If $N=2$, then a non-trivial subgroup
is $Z(SU(2))$ itself.
Hence if an action
$\al\col M\ra M\oti L^\infty(\widehat{SU(2)})$
has non-trivial modular part $\IH$,
then $\IH=\{\pi_n\mid n=0,1,\dots\}$, that is,
$\bH=SO(3)$ and $\Ga=\Z/2\Z$.
\end{ex}

Let us take a unitary $V$ and a map $\ps$ as in Corollary \ref{cor:Vps}.
Then we claim that
$\pi_1\sim\pi_2$
if and only if $\ps_{\pi_1}=\ps_{\pi_2} \bmod \Aut(M)_{\rm m}$.
Indeed, the condition $\pi_1\sim\pi_2$ implies
$\tal_{\pi_1\opi_2}=\Ad V_{\pi_1}\tps_{\pi_1}(V_{\opi_2})
\circ(\tps_{\pi_1}\tps_{\opi_2}\oti1)$
is implemented by unitary,
and so is $\tps_{\pi_1}\tps_{\opi_2}$
by \cite[Proposition 3.4 (1)]{Iz}.
Thus $\ps_{\pi_1}\ps_{\opi_2}\in\Aut(M)_{\rm m}$.
Putting $\pi_2=\pi_1$,
we get $\ps_{\pi_1}\ps_{\opi_1}=\id\bmod \Aut(M)_{\rm m}$.
Hence the claim has been proved.
This fact enables us to take
a map $\ps\col\IG\ra \Aut(M)$ that is constant
in each conjugacy class.
Thus we can regard $\ps$ as a map from $\Ga$ into $\Aut(M)$.
Of course, we may and do assume that $\ps_{[\btr]}=\id$.

\begin{lem}
The map $\ps$ is a modularly free modular $\Ga$-kernel.
\end{lem}
\begin{proof}
Let $\pi_1,\pi_2\in\IG$.
By definition of $\ps$,
it is trivial that
$\tal_{\pi_1\pi_2}\tps_{[\pi_1][\pi_2]}^{-1}$
is implemented by unitary,
and
$\ps_{[\pi_1]}\ps_{[\pi_2]}=\ps_{[\pi_1][\pi_2]}\bmod\Aut(M)_{\rm m}$.
Thus $\ps$ is a modular $\Ga$-kernel.
The modular freeness is trivial.
\end{proof}

Summarizing our argument, we have obtained the following.

\begin{prop}
Let $\bG$ be a compact Kac algebra.
If an action $\alpha$ of $\bhG$ on a factor $M$
is satisfying Assumption \ref{ass:modular},
then there exist a map $V\col\IG\ra U(\tM)$
and a modularly free modular $\Ga$-kernel
$\psi\col\Ga\ra\Aut(M)$
such that for all $\pi\in\IG$, $t\in\R$,
\[
\widetilde{\alpha}_\pi=\Ad V_\pi\circ
(\widetilde{\psi}_{[\pi]}\otimes1),
\quad
V_\pi^*(\th_t\oti\id)(V_\pi)\in Z(\tM)\oti B(H_\pi).
\]
\end{prop}

\subsection{Invariants of actions}
Take $V$ and $\ps$ as in the previous proposition.
Let $\si_{[\pi],[\rho]}:=\ps_{[\pi]}\ps_{[\rho]}\ps_{[\pi][\rho]}^{-1}$
for each $\pi,\rho\in\IG$,
which is an extended modular automorphism on $M$.
We take a unitary $w_{[\pi],[\rho]}\in U(\tM)$
with $\tsi_{[\pi],[\rho]}=\Ad w_{[\pi],[\rho]}$.
Recall the following invariants $d_1,d_2$
introduced in (\ref{eq:d1d2}):
for $p,q,r\in\Ga, s\in\R$,
\[
w_{p,q}w_{pq,r}=d_1(p,q,r)\tps_p(w_{q,r}^\al)w_{p,qr},
\quad
\th_s(w_{p,q})=d_2(s;p,q)w_{p,q}.
\]

\begin{lem}
\label{lem:ad}
For $\pi,\rho\in \IG$,
we set
$a_{\pi,\rho}:=
V_{\pi\rho}^*V_\pi (\tps_{[\pi]}\oti\id_\rho)(V_\rho)w_{[\pi],[\rho]}$.
Then the following hold:
\begin{enumerate}
\item
$a_{\pi,\rho}$ is a unitary
contained in $Z(\tM)\oti B(H_\pi)\oti B(H_\rho)$;

\item
$
a_{\pi,\rho}^*a_{\pi\rho,\si}^*
=
d_1(\pi,\rho,\si)^*
\tps_{[\pi]}(a_{\rho,\si}^*)a_{\pi,\rho\si}^*
$;

\item
$\th_t(a_{\pi,\rho})
=
d_2(t;[\pi],[\si])
c_{\pi\rho}(t)^*
a_{\pi,\rho}
c_{\pi}(t)\tps_{[\pi]}(c_\rho(t))$.
\end{enumerate}
\end{lem}
\begin{proof}
(1). This follows from $\al_\pi\al_\rho=\al_{\pi\rho}$.

(2).
By definition, we have
\[
V_{\pi\rho}\tps_{[\pi\rho]}(V_\si)w_{[\pi\rho],[\si]}=V_{\pi\rho\si}a_{\pi\rho,\si}.
\]
The left hand side is computed as follows:
\begin{align*}
&V_\pi\tps_\pi(V_\rho)w_{[\pi],[\rho]}a_{\pi,\rho}^*
\cdot
\tps_{[\pi\rho]}(V_\si)w_{[\pi\rho],[\si]}
\\
&=
V_\pi\tps_\pi(V_\rho)a_{\pi,\rho}^*
\cdot
\tps_{[\pi]}\tps_{[\rho]}(V_\si)w_{[\pi],[\rho]}w_{[\pi\rho],[\si]}
\\
&=
V_\pi\tps_\pi(V_\rho)
\cdot
\tps_{[\pi]}\tps_{[\rho]}(V_\si)
a_{\pi,\rho}^*w_{[\pi],[\rho]}w_{[\pi\rho],[\si]}
\\
&=
V_\pi\tps_\pi(V_\rho\tps_{[\rho]}(V_\si))
\cdot
a_{\pi,\rho}^*w_{[\pi],[\rho]}w_{[\pi\rho],[\si]}.
\end{align*}
The right hand side is equal to
\[
V_\pi\tps_{[\pi]}(V_{\rho\si})w_{[\pi],[\rho][\si]}a_{\pi,\rho\si}^*a_{\pi\rho,\si}
=
V_\pi\tps_{[\pi]}(V_\rho\tps_{[\rho]}(V_\si)w_{\rho,\si}a_{\rho,\si}^*)
w_{[\pi],[\rho][\si]}a_{\pi,\rho\si}^*a_{\pi\rho,\si}
.
\]
Comparing these, we have
\[
a_{\pi,\rho}^*w_{\pi,\rho}w_{\pi\rho,\si}
=
\tps_{[\pi]}(w_{\rho,\si}a_{\rho,\si}^*)w_{\pi,\rho\si}a_{\pi,\rho\si}^*a_{\pi\rho,\si},
\]
that is,
\[
a_{\pi,\rho}^*a_{\pi\rho,\si}^*
=
d_1(\pi,\rho,\si)^*
\tps_{[\pi]}(a_{\rho,\si}^*)a_{\pi,\rho\si}^*.
\]

(3).
This is computed as follows:
\begin{align*}
\th_t(a_{\pi,\rho})
&=
\th_t(V_{\pi\rho}^* V_\pi\tps_{[\pi]}(V_\rho)w_{[\pi],[\si]})
\\
&=
c_{\pi\rho}(t)^*V_{\pi\rho}^* \cdot
V_\pi c_{\pi}(t)\cdot
\tps_{[\pi]}(V_\rho c_\rho(t))
\cdot
d_2(t;[\pi],[\si])w_{[\pi],[\si]}
\\
&=
c_{\pi\rho}(t)^*V_{\pi\rho}^* \cdot
V_\pi \cdot
\tps_{[\pi]}(V_\rho )w_{[\pi],[\si]}
\cdot
c_{\pi}(t)\tps_{[\pi]}(c_\rho(t))
d_2(t;[\pi],[\si])
\\
&=
c_{\pi\rho}(t)^*
a_{\pi,\rho}
c_{\pi}(t)\tps_{[\pi]}(c_\rho(t))
d_2(t;[\pi],[\si]).
\end{align*}
\end{proof}

\begin{rem}
The element $a\in Z(\tM)\oti \lhG\oti\lhG$ generalizes
the characteristic invariant studied in \cite{ST}.
Indeed, let $\al\col\La\ra\Aut(M)$ be an action
of a discrete group $\La$ on a factor $M$.
Let $\La_0$ be the modular part of $\La$.
%
For $n\in\La_0$, take $v_n\in U(\tM)$ such that $\tal_n=\Ad v_n$.
Let $\Ga:=\Lambda/\Lambda_0$.
Fix a section $p\in \Ga\rightarrow s(p)\in \La$. 
Define $n(p,q):=s(p)s(q)s(pq)^{-1}$.
Set $\psi_p:=\alpha_{s(p)}$,
$w_{p,q}:=v_{n(p,q)}$.
Hence $\tps_p\tps_q=\Ad w_{p,q}\circ\tps_{pq}$.
Then $\alpha_{ns(p)}=\Ad v_n\circ\psi_p$ for $n\in\La_0$ and $p\in\Ga$.
For $g=ms(p)$, $h=ns(q)$,  
$\alpha_{gh}=\Ad v_{ms(p)ns(p)^{-1}n(p,q)}\circ\psi_{pq}$ holds. Thus 
$a_{g,h}$ is given by 
\begin{align*}
 a_{g,h}&=v_{ms(p)ns(p)^{-1}n(p,q)}^* v_m\psi_p(v_n)w_{p,q} \\
&= v_{ms(p)ns(p)^{-1}n(p,q)}^* v_m\alpha_{s(p)}(v_n)v_{n(p,q)} \\
&= \lambda(s(p), s(p)ns(p)^{-1})v_{ms(p)ns(p)^{-1}n(p,q)}^*
v_mv_{s(p)ns(p)^{-1}}v_{n(p,q)} \\
&=\lambda(s(p), s(p)ns(p)^{-1})\mu(m, s(p)ns(p)^{-1}) \mu(ms(p)ns(p)^{-1},n(p,q)).
\end{align*}
If we put $p=q=e$, then we can recover $\mu(m,n)$ from $a_{g,h}$, thus
 so does $\lambda(g,h)$.
\end{rem}

We will state our main result in this subsection as follows.

\begin{thm}\label{thm:al-be-a-c}
Let $\al$ and $\be$ be actions on an injective factor $M$.
Assume the following conditions:
\begin{itemize}
\item
$\al$ and $\be$ have common normal modular part;
\item
$\mo(\al)=\mo(\be)$;
\item
$a^\al=a^\be$;
\item
$c_\pi^\al(t)=c_\pi^\be(t)$ for all $\pi\in\IG$ and $t\in\R$.
\end{itemize}
Then $\al$ and $\be$ are strongly cocycle conjugate.
\end{thm}
\begin{proof}
Take $V^\al,V^\be,\ps^\al$ and $\ps^\be$ as follows:
\[
\tal_\pi=\Ad V_\pi^\al\circ(\tps_\pi^\al\oti1),
\quad
\tbe_\pi=\Ad V_\pi^\be\circ(\tps_\pi^\be\oti1),
\]
\[
a_{\pi,\rho}^\al
=
(V_{\pi\rho}^\al)^* V_\pi^\al\tps_{[\pi]}^\al(V_\rho^\al)w_{[\pi],[\rho]}^\al,
\quad
a_{\pi,\rho}^\be
=
(V_{\pi\rho}^\be)^* V_\pi^\be\tps_{[\pi]}^\be(V_\rho^\be)w_{[\pi],[\rho]}^\be,
\]
and
\[
c_\pi^\al(t)=(V_\pi^\al)^* \th_t(V_\pi^\al),
\quad
c_\pi^\be(t)=(V_\pi^\be)^* \th_t(V_\pi^\be).
\]
Then by the previous lemma,
we have $d_1^\al=d_1^\be$ and $d_2^\al=d_2^\be$.
By Theorem \ref{thm:mod-ker},
there exist a map $v\col \IG\ra U(M)$
and an automorphism $\bar{\th}\in\oInt(M)$ such that
\[
\ps_{[\pi]}^\al=\Ad v_{[\pi]}\circ
\bar{\th}\ps_{[\pi]}^\be\bar{\th}^{-1},
\quad
w_{[\pi],[\rho]}^\al=
v_{[\pi]}\cdot
\bar{\th}\ps_{[\pi]}^\be\bar{\th}^{-1}(v_{[\rho]})
\cdot
\bar{\th}(w_{[\pi],[\rho]}^\be)
v_{[\pi\rho]}^*.
\]

The action $\ga_\pi:=(\bar{\th}\oti\id)\circ\be_\pi\circ\bar{\th}^{-1}$
satisfies
\[
\wdt{\ga}_\pi
=\Ad \bar{\th}(V_\pi^\be)\circ
(\bar{\th}\ps_{[\pi]}^\be\bar{\th}^{-1}\oti1)
=
\Ad \bar{\th}(V_\pi^\be)v_{[\pi]}^*
\circ
(\ps_{[\pi]}^\al\oti1).
\]
Let $V_\pi^\ga:=\bar{\th}(V_\pi^\be)v_{[\pi]}^*$,
$\ps^\ga:=\ps^\al$ and
$w_{[\pi],[\rho]}^\ga:=w_{[\pi],[\rho]}^\al$.
Then $a^\ga$ is computed as
\begin{align*}
a_{\pi,\rho}^\ga
&=
(V_{\pi\rho}^\ga)^*
\cdot
V_\pi^\ga
\cdot
\tps_{[\pi]}^\ga(V_\rho^\ga)
\cdot
w_{[\pi],[\rho]}^\ga
\\
&=
(\bar{\th}(V_{\pi\rho}^\be)v_{[\pi][\rho]}^*)^*
\cdot
\bar{\th}(V_\pi^\be)v_{[\pi]}^*
\cdot
\tps_{[\pi]}^\al(\bar{\th}(V_\rho^\be)v_{[\rho]}^*)
\cdot
w_{[\pi],[\rho]}^\ga
\\
&=
v_{[\pi][\rho]}\bar{\th}(V_{\pi\rho}^\be)^*
\cdot
\bar{\th}(V_\pi^\be)
\cdot
v_{[\pi]}^*\tps_{[\pi]}^\al(\bar{\th}(V_\rho^\be)v_{[\rho]}^*)
v_{[\pi]}
\cdot
v_{[\pi]}^*w_{[\pi],[\rho]}^\ga
\\
&=
v_{[\pi][\rho]}\bar{\th}\left((V_{\pi\rho}^\be)^*V_\pi^\be\right)
\cdot
\bar{\th}\ps_{[\pi]}^\be
\bar{\th}^{-1}(\bar{\th}(V_\rho^\be)v_{[\rho]}^*)
\cdot
v_{[\pi]}^*w_{[\pi],[\rho]}^\ga
\\
&=
v_{[\pi][\rho]}
\cdot
\bar{\th}\left(
(V_{\pi\rho}^\be)^*V_\pi^\be
\ps_{[\pi]}^\be(V_\rho^\be)
\cdot
\ps_{[\pi]}^\be\bar{\th}^{-1}(v_{[\rho]}^*)
\right)
\cdot
v_{[\pi]}^*w_{[\pi],[\rho]}^\ga
\\
&=
v_{[\pi][\rho]}
\cdot
\bar{\th}\left(
a_{\pi,\rho}^\be (w_{[\pi],[\rho]}^\be)^*
\cdot
\ps_{[\pi]}^\be\bar{\th}^{-1}(v_{[\rho]}^*)
\right)
\cdot
v_{[\pi]}^*w_{[\pi],[\rho]}^\ga
\\
&=
\bar{\th}(a_{\pi,\rho}^\be)
\cdot
v_{[\pi][\rho]}
\cdot
\bar{\th}
(w_{[\pi],[\rho]}^\be)^*
\cdot
\bar{\th}\ps_{[\pi]}^\be\bar{\th}^{-1}(v_{[\rho]}^*)
\cdot
v_{[\pi]}^*w_{[\pi],[\rho]}^\ga
\\
&=
a_{\pi,\rho}^\be
\cdot
(w_{[\pi],[\rho]}^\al)^*
\cdot
w_{[\pi],[\rho]}^\ga
\\
&=
a_{\pi,\rho}^\al.
\end{align*}
We also have $c^\ga(t)=c^\be(t)$.
We denote by $\ps$ the modular $\Ga$-kernel $\ps^\al=\ps^\ga$.
Set $v_\pi:=V_\pi^\al (V_\pi^\ga)^*$.
It is trivial that $\al=\Ad v\circ \ga$ by definition,
and $v\in M\oti \lhG$ because of $c^\al=c^\ga$.
We check $v$ is a $\ga$-cocycle as follows:
\begin{align*}
v_\pi\ga_\pi(v_\rho)
&=
V_\pi^\al(V_\pi^\ga)^*
\cdot
\tga_\pi(V_\rho^\al(V_\rho^\ga)^*)
\\
&=
V_\pi^\al
\cdot
(V_\pi^\ga)^*
\tga_\pi(V_\rho^\al(V_\rho^\ga)^*)
V_\pi^\ga
\cdot
(V_\pi^\ga)^*
\\
&=
V_\pi^\al
\cdot
\tps_\pi(V_\rho^\al(V_\rho^\ga)^*)
\cdot
(V_\pi^\ga)^*
\\
&=
V_\pi^\al
\tps_{[\pi]}(V_\rho^\al)
\cdot
(V_\pi^\ga\tps_{[\pi]}(V_\rho^\ga))^*
\\
&=
V_{\pi\rho}^\al a_{\pi,\rho}^\al
(w_{[\pi],[\rho]}^\al)^*
\cdot
(V_{\pi\rho}^\ga a_{\pi,\rho}^\ga
(w_{[\pi],[\rho]}^\ga)^*)^*
\\
&=
V_{\pi\rho}^\al a_{\pi,\rho}^\al
(w_{[\pi],[\rho]}^\al)^*
\cdot
(V_{\pi\rho}^\ga a_{[\pi],[\rho]}^\al
(w_{[\pi],[\rho]}^\al)^*)^*
\\
&=
v_{\pi\rho}.
\end{align*}
Therefore $\al$ and $\be$ are strongly cocycle conjugate.
\end{proof}

Let $\ps\col \Ga\ra \Aut(M)$ be a modular kernel as before.
Then we denote by $\mathcal{A}(\ps)$
the set of unitary $a\in Z(\tM)\oti\lhG\oti\lhG$
such that
$a_{\btr,\cdot}=1=a_{\cdot,\btr}$
and
it gives a function $d_1\col \Ga^3\ra U(Z(\tM))$ as
\[
a_{\pi,\rho}^*a_{\pi\rho,\si}^*
=d_1([\pi],[\rho],[\si])^*
\tps_{[\pi]}(a_{\rho,\si}^*)a_{\pi,\rho\si}^*.
\]

Let $Z^1(F^M,\lhG)$ be the set of 1-cocycles,
that is, the collection of a unitary
$c(t)\in \meC_M\oti\lhG$
such that $c(t)_\btr=1$ and $c(s+t)=c(s)(\th_s\oti\id)(c(t))$.

We introduce
the subset $\mX(\ps,F^M)$
of the direct product $\mathcal{A}(\ps)\times Z^1(F^M,\lhG)$.
An element
$(a,c)$ is contained in $\mX(\ps,F^M)$
if and only if there exists a function $d_2\col \R\times\Ga^2\ra U(Z(\tM))$
such that
\[
\th_t(a_{\pi,\rho})
=
d_2(t;[\pi],[\rho])
c_{\pi\rho}(t)^*a_{\pi,\rho}\tps_{[\pi]}(c_\rho(t))c_\pi(t).
\]

Next we introduce the equivalence relation $\sim$
in $\mX(\ps,F^M)$
as
$(a^1,c^1)\sim (a^2,c^2)$ if and only if
there exist
$c^0\in Z^1(F^M,\el^\infty(\Ga))$,
$\la\in U(Z(\tM)\oti\el^\infty(\Ga^2))$
and
$\mu\in U(Z(\tM)\oti \lhG)$
such that
\[
a_{\pi,\rho}^2
=
\la_{[\pi],[\rho]}
\mu_{\pi\rho}a_{\pi,\rho}^1 \tps_{[\pi]}(\mu_\rho^*)\mu_\pi^*,
\quad
c_\pi^2(t)=c_{[\pi]}^0(t)\mu_\pi c_\pi^1(t)\th_t(\mu_\pi^*).
\]
The quotient space is denoted by
$X(\ps, F^M)$.
Then it is not so hard to show the following lemma.

\begin{lem}
The following hold:
\begin{enumerate}
\item
For an action $\al\col M\ra M\oti\lhG$ with $\mo(\al_\pi)=\mo(\ps_{[\pi]})$,
the element $[(a^\al,c^\al)]\in X(\ps,F^M)$
is well-defined.
\item
If two actions $\al,\be\col M\ra M\oti\lhG$
with $\mo(\al_\pi)=\mo(\be_\pi)=\mo(\ps_{[\pi]})$are
strongly cocycle conjugate,
then
$[(c^\al,a^\al)]=[(c^\be,a^\be)]$.
\end{enumerate}
\end{lem}

We will state our main result in this section.

\begin{thm}
\label{thm:non-triv-modular-main}
Let $\al,\be\col M\ra M\oti\lhG$
be actions on a McDuff factor $M$.
Suppose the following conditions:
\begin{itemize}
\item
$\al$ and $\be$ have
common normal modular part $\IH$;
\item
$\mo(\al_\pi)=\mo(\be_\pi)$ for all $\pi\in\IG$;
\item
$[(a^\al,c^\al)]=[(a^\be,c^\be)]$
in $X(\ps,F^M)$.
\end{itemize}
Then $\al$ and $\be$ are strongly cocycle conjugate.
\end{thm}
\begin{proof}
We decompose $\al$ and $\be$
as $\al_\pi=\Ad V_\pi^\al\circ(\ps_\pi^\al(\cdot)\oti1)$
and $\be_\pi=\Ad V_\pi^\be\circ(\ps_\pi^\be(\cdot)\oti1)$
so that
they have $a_{\pi,\rho}^\al$ and $a_{\pi,\rho}^\be$
as invariants by taking $w^\al$ and $w^\be$ as before.
Take $c^0$, $\la$ and $\mu$ as above.
Then we have
\[
a_{\pi,\rho}^\al
=
\la_{[\pi],[\rho]}
\mu_{\pi\rho}a_{\pi,\rho}^\be \tps_{[\pi]}(\mu_\rho^*)\mu_\pi^*,
\quad
c_\pi^\al(t)=c_{[\pi]}^0(t)\mu_\pi c_\pi^\be(t)\th_t(\mu_\pi^*).
\]
By stability of $\{\tM,\th\}$,
we can take a unitary $u_{[\pi]}\in \tM$
such that $c^0(t)_{[\pi]}=u_{[\pi]}\th_t(u_{[\pi]}^*)$.

Putting
\[
V_\pi':=V_\pi^2 \mu_\pi (u_{[\pi]}^*\oti1),
\quad
\ps_{[\pi]}':=\Ad u_{[\pi]}\circ\ps_{[\pi]}^\al,
\]
\[
w_{[\pi],[\rho]}'
:=\la_{[\pi],[\rho]}^*
u_{[\pi]}\ps_{[\pi]}^\al(u_{[\rho]})w_{[\pi],[\rho]}^\al u_{[\pi\rho]}^*,
\]
\[
a_{\pi,\rho}'
:=V_{\pi\rho}'^* V_\pi'\ps_{[\pi]}^\al(V_\rho')w_{[\pi],[\rho]}'
,
\quad
c'(t):=(V')^*(\th_t\oti\id)(V'),
\]
we have
\[
\al_\pi=\Ad V_\pi'\circ(\ps^\al\oti1),
\quad
a'=a^\be,
\quad
c'(t)_\pi=c^\be(t)_\pi.
\]
Then by Theorem \ref{thm:al-be-a-c},
we are done.
\end{proof}

\subsection{Model action 1}

We shall construct a model action
with a given invariant.

\begin{thm}\label{thm:model1}
Let $\bH$ be a quotient of $\bG$
such that $\pi\opi\in \Rep(\bH)_0$
for any $\pi\in\IG$,
and $\Ga:=\IG/\IH$ the associated discrete group.
Let $M$ be an injective infinite factor
and $\ps\col \Ga\ra \Aut(F^M)$ a group homomorphism.
Then for any $(\al,c)\in \mX(F^M,\ps)$,
there exists an action $\al$ of $\bhG$ on $M$
such that
$(a^\al,c^\al)=(a,c)$.
\end{thm}
\begin{proof}
Let $d_1,d_2$ be the associated maps with $(a,c)$.
By Theorem \ref{thm:model-kernel},
we can take a free modular $\Ga$-kernel
$(\ps,w)$ on $M$ realizing $\ps$, $d_1$ and $d_2$
(by identifying $\tps|_{Z(\tM)}$ with given $\ps$).
Then we have
\begin{align}
&a_{\pi,\rho}^*w_{[\pi],[\rho]}a_{\pi\rho,\si}^*w_{[\pi\rho],[\si]}
=\tps_{[\pi]}(a_{\rho,\si}^*w_{[\rho],[\si]})a_{\pi,\rho\si}^*w_{[\pi],[\rho\si]},
\label{eq:aw}\\
&
\th_t(a_{\pi,\rho}^*w_{[\pi],[\rho]})
=
c_\pi(t)^*\tps_{[\pi]}(c_\rho(t)^*)a_{\pi,\rho}^*w_{[\pi],[\rho]}c_{\pi\rho}(t).
\label{eq:thaw}
\end{align}

Thanks to \cite[Theorem 3.3]{Iz},
there exists a unitary $V\in \tM\oti\lhG$
such that $c(t)=V^*\th_t(V)$,
and the map $\Ad V|_M$ is modular on $M$.
Then by (\ref{eq:thaw})
and the equality
$\tps_{[\pi]}(c_\rho(t))c_\pi(t)
=\tps_{[\pi]}(V_\rho^*)c_\pi(t)\tps_{[\pi]}(\th_t(V_\rho))$,
we have
\[
\th_t(V_\pi\tps_{[\pi]}(V_\rho)a_{\pi,\rho}^*w_{[\pi],[\rho]}V_{\pi\rho}^*)
=
V_\pi\tps_{[\pi]}(V_\rho)a_{\pi,\rho}^*w_{[\pi],[\rho]}V_{\pi\rho}^*.
\]
Thus the element
$u_{\pi,\rho}:=V_\pi\tps_{[\pi]}(V_\rho)a_{\pi,\rho}^*w_{[\pi],[\rho]}V_{\pi\rho}^*$
is contained in $M\oti B(H_\pi)\oti B(H_\rho)$.

Let $\al^0:=\Ad V\circ \ps$ on $M$.
We show that $(\al^0,u)$ is a cocycle action.
For $x\in M$,
we have
\begin{align*}
\al_\pi^0(\al_\rho^0(x))
&=
V_\pi \tps_{[\pi]}(V_\rho)
\tps_{[\pi]}(\tps_{[\rho]}(x))\tps_{[\pi]}(V_\rho^*)V_\pi^*
\\
&=
V_\pi \tps_{[\pi]}(V_\rho)a_{\pi,\rho}^*
\tps_{[\pi]}(\tps_{[\rho]}(x))a_{\pi,\rho}\tps_{[\pi]}(V_\rho^*)V_\pi^*
\\
&=
u_{\pi,\rho}\al_{\pi\rho}^0(x)u_{\pi,\rho}^*.
\end{align*}
On the 2-cocycle relation, we have
\begin{align*}
u_{\pi,\rho}u_{\pi\rho,\si}
&=
V_\pi\tps_{[\pi]}(V_\rho)a_{\pi,\rho}^*w_{[\pi],[\rho]}V_{\pi\rho}^*
\cdot
V_{\pi\rho}\tps_{[\pi\rho]}(V_\si)a_{\pi\rho,\si}^*w_{[\pi\rho],[\si]}V_{\pi\rho\si}^*
\\
&=
V_\pi\tps_{[\pi]}(V_\rho)\tps_{[\pi]}(\tps_{[\rho]}(V_\si))
\cdot
a_{\pi,\rho}^*w_{[\pi],[\rho]}
a_{\pi\rho,\si}^*w_{[\pi\rho],[\si]}V_{\pi\rho\si}^*
\\
&=
V_\pi\tps_{[\pi]}(V_\rho)\tps_{[\pi]}(\tps_{[\rho]}(V_\si))
\cdot
\tps_{[\pi]}(a_{\rho,\si}^*w_{[\rho],[\si]})
a_{\pi,\rho\si}^*w_{[\pi],[\rho\si]}
V_{\pi\rho\si}^*
\quad(\mbox{by }(\ref{eq:aw}))
\\
&=
V_\pi\tps_{[\pi]}(u_{\rho,\si})V_\pi^*
\cdot
V_\pi\tps_{[\pi]}(V_{\rho\si})a_{\pi,\rho\si}^*w_{[\pi],[\rho\si]}
V_{\pi\rho\si}^*
\\
&=
\al_\pi^0(u_{\rho,\si})u_{\pi,\rho\si}.
\end{align*}
Thus $(\al^0,u)$ is a cocycle action on $M$.
When $M$ is properly infinite,
the 2-cocycle is a coboundary.
Hence we may and do assume that $V\in \tM\oti\lhG$
satisfies $\al:=\Ad V\circ \ps|_M$ is an action,
and
$1=V_\pi\tps_{[\pi]}(V_\rho)a_{\pi,\rho}^*w_{[\pi],[\rho]}V_{\pi\rho}^*$.
It is trivial that $\tal=\Ad V\circ \tps$.
Therefore, the invariant of $\al$
is equal to $(a,c)$.
\end{proof}

Next we study how the freeness of $\al$ is characterized
by using the characteristic invariant.
Let $\al$ be an action of $\bhG$ on a factor $M$.
Suppose that the modular part $\IH$ of $\al$ is normal.
Take $V$ and $\ps$ as before and let $(a,c)$ be the associated
characteristic cocycle.
Let us denote by $K_c$ the minimal subgroup of $(c(t)_\rho)_{\rho\in\IH}$.
Note that each $\rho\in\IH$ induces a unitary representation
$\pi\col K_c\ra B(H_\rho)$ by $\rho(c(x,t))=c(x,t)_\rho$
for all $(x,t)\in X\times\R$.

\begin{prop}\label{prop:al-free}
In the setting above,
consider the following statements:
\begin{enumerate}
\item $\al$ is free;

\item Each representation
$\rho\in\IH\setminus\{\btr\}$ of $K_c$ does not contain $\btr$;

\item Each representation
$\rho\in\IH\setminus\{\btr\}$ of $K_c$ is irreducible.
\end{enumerate}
Then (1) and (2) are equivalent,
and (1) always implies (3).
If $d(\rho)\geq2$ for all $\rho\in\IH\setm\{\btr\}$,
then (1), (2) and (3) are equivalent.
\end{prop}
\begin{proof}
(1)$\Rightarrow$(2).
Suppose that $\al$ is free.
Let $\rho\in\IH$
and take
a projection $p\in B(H_\rho)$
such that $c(t)_\rho(1\oti p)=(1\oti p)$
for all $t\in\R$.
We will show that $p=0$.
By our assumption, $v:=V_\rho(1\oti p)$
is fixed by $\th\oti\id$, and $v\in M\oti B(H_\rho)$.
Then clearly we have
$v(x\oti1)=\al_\rho(x)v$ for all $x\in M$.
Since $\al$ is free, we have $p=0$.

(2)$\Rightarrow$(1).
Suppose that for some $\pi\in\IG\setminus\{\btr\}$,
we have a non-zero $b\in M\oti B(H_\pi)$ satisfying
$b(x\oti1)=\al_\pi(x)b$ for all $x\in M$,
which is equivalent to say $V_\pi^*b(x\oti1)=(\tps_{[\pi]}(x)\oti1)V_\pi^*b$.
Thus $(\id,\tps_{[\pi]})\neq0$.
However, this implies $\ps_{[\pi]}$ is modular by \cite[Proposition 3.4]{Iz},
and $\al_\pi$ is modular.
By definition of $\bH$, $\pi$ belongs to $\IH$, and $\ps_{[\pi]}=\id$.
Hence $V_\pi^*b\in (M'\cap \tM)\oti B(H_\pi)$,
where the Connes--Takesaki relative commutant theorem
yields $M'\cap\tM=\meC_M$
(see \cite[Theorem II.5.1]{CT} and \cite[Lemma 1.1]{KtST}).

We put $b':=V_\pi^* b$.
Then we have $c(t)_\pi(\th_t\oti\id_\pi)(b')=b'$ for all $t\in\R$.
Thanks to \cite[Theorem 3.14]{Z} (or \cite[Theorem 3.5 (4)]{Iz}),
this means the representation $\pi$ contains the trivial representation $\btr$,
but this is a contradiction.
Hence $\al$ is free.

(1)$\Rightarrow$(3).
Suppose that $\al$ is free.
Let $\rho\in\IH$ and take
a non-zero projection $p\in B(H_\rho)$
such that $(1\oti p)c(t)_\rho=c(t)_\rho(1\oti p)$
for all $t\in\R$.
Since $c(t)_\rho=V_\rho^*(\th_t\oti\id)(V_\rho)$,
it turns out that $q:=V(1\oti p)V^*$ is fixed by $\th\oti\id$.
Hence $q\in M\oti B(H_\rho)$.
It is clear that $q$ is commuting with $\al_\rho(M)$,
and $q$ is a scalar by \cite[Lemma 2.8]{MT1}.
However, it shows $p=0$ or $1$, which means that
$\rho$ is irreducible.

(3)$\Rightarrow$(2).
If a representation $\rho\in\IH$ of $K_c$ contains a trivial representation,
then we have $d(\rho)=1$ and $c(t)_\rho=1$ for all $t\in\R$
because of the irreducibility of $\rho$ on $K_c$.
Thus $\rho$ equals $\btr$ by our assumption.
\end{proof}

\subsection{Model action 2}
When we make a further assumption,
a model action is constructed in a relatively direct way.
On a 2-cocycle deformation of a discrete Kac algebra,
readers are referred to Section \ref{sec:2-cocycle}.

Let $\IH\subs \IG$ be a normal subcategory
and $\Ga:=\IG/\IH$ the quotient group as before.
Let us treat the invariants
$\ps\col\Ga\ra\Aut_\th(\meC_M)$, $a$ and $c$ such that
\begin{itemize}
\item
$\ps_{[\pi]}=\id_{\meC_M}$, $\pi\in\IG$;

\item
$a\in\C\oti\lhG\oti\lhG$;

\item
$c_\pi(t)$ is a non-trivial irreducible cocycle for all $\pi\in\IG$;

\item
$a_{\pi\rho,\si}a_{\pi,\rho}=a_{\pi,\rho\si}a_{\rho,\si}$, $\pi,\rho,\si\in\IG$;
\item
$a_{\pi,\rho}=c_{\pi\rho}(t)^*a_{\pi,\rho}c_\rho(t)c_{\pi}(t)$,
$\pi,\rho,\si\in\IG$.
\end{itemize}
In particular, $d_1$ and $d_2$ are trivial in this case.
Put $\om:=a^*$, which is a 2-cocycle of $\bhG$.
The last equality implies
that the $\th$-cocycle $c$ is evaluated in $\bG_\om$,
the dual cocycle twisting of $\bG$ by the dual cocycle $\om$.
From the third, it turns out that $\bG_\om$ is the minimal
subgroup of $c$.
Thus the twisted coproduct $\De^\om$ is co-commutative.

Let $\ga$ be a free action of $\bhG$
on an injective factor $M^0$ of type II$_1$.
Since
$(\gamma\otimes \id)\circ\gamma
=\Ad a\circ(\id \otimes \De^\om)\circ \gamma$,
we may regard $\gamma$ as a cocycle action of
$(\bhG_\om, \De^\om)$ with the 2-cocycle $a$.
If we collect the ``$\IH$-part'' of $\ga$,
then $\ga$ naturally defines a cocycle action of $\bhH_\om$,
which we also denote by $\ga$.

Let $Q:=M^0\rtimes_{\ga,a}\bhG_\om$,
and $P:=M^0\rtimes_{\ga,a}\bhH_\om$
be the cocycle crossed products,
which are injective factors of type II$_1$.
Then we have a natural inclusion $P\subs Q$.
Let $\lambda_\pi$ be the implementing unitary in $Q\oti B(H_\pi)$.
Then
\begin{equation}\label{eq:la-a}
\lambda_{12}\lambda_{13}
=a(\id\oti\De^\om)(\la)
=(\id\oti\De)(\la)a.
\end{equation}

Employing Theorem \ref{thm:corr-gc-nc},
we see that there exists a closed subgroup
$Z\subset Z(\bG_\om)$ such that
with $\bG_\om/Z=\mathbb{H}_\om$.
Then the restriction of $\pi\in\IG$ on $Z$
is a character.
Thus we have a map from $\IG$ into $\wdh{Z}$.
It is not hard to see that the map induces a group isomorphism
between $\Ga$ and $\wdh{Z}$.

Let $\{X_M,F^M\}$ be the point realization of the flow $\{\meC_M,\th\}$.
For the $\theta$-cocycle
$c(t)\in \meC_M\otimes \lhG$,
we set
$c(x,t):=c(t)_{x}$ for $x\in X_M$ and $t\in\R$.
Then $c$ is a groupoid homomorphism from the transformation groupoid
$X_M\rti\R$ into $\bG_\om$.
Define $\theta^0_{x,t}\in \Aut(Q)$
by $\theta^0_{x,t}(y)=y$ for $y\in M^0$,
and
\[
(\theta^0_{x,t}\oti\id)(\lambda_\pi)=\lambda_{\pi}(1\oti c_\pi(x,t))
\quad
\mbox{for }
(x,t)\in X_M\rti\R, \pi\in\IG.
\]
Set
$\theta_{x,t}^1:=\theta_{x,t}\otimes \theta^0_{x,t}$
on $\tM(x)\otimes Q$,
and we obtain a one-parameter automorphism group
$\theta_t^1\in \Aut(\tM\otimes Q)$
defined by
\[
\th_t^1(b)
=\int_{X_M}^\oplus
\th_{x,t}^1(b(F_{-t}^M x))\,d\mu(x)
\quad\mbox{for }y\in \tM\oti Q.
\]
Note that $\tM\oti P$ is globally invariant under $\th^1$.

Let $\ta$ be the trace on $\tM$
such that $\ta\circ\th_s=e^{-s}\ta$ for $s\in\R$,
and $\ta_1:=\ta\oti\tr$,
where $\tr$ is the normalized trace on $Q$.
Then we have $\ta_1\circ\th_s=e^{-s}\ta_1$.
Thus the covariant systems
$\{\tM\oti Q,\th^1\}$ and $\{\tM\oti P,\th^1\}$
are the cores of their crossed products.
In fact, those crossed products are isomorphic to $M$
since the flow spaces are equal to $\{X_M,F_M\}$.

Let us consider the dual action $\wdh{\ga}$ of $\bG_\om$ on $Q$.
Since $\wdh{\ga}$ is minimal,
so is the restriction $\wdh{\ga}|_Z$.
By direct computation,
we see $\th^1$ and $\id_{\tM}\oti \wdh{\ga}_Z$ are commuting.

\begin{lem}
The action $\id_{\tM}\oti \wdh{\ga}|_Z$ on $(\tM\oti Q)^{\th^1}$
is minimal.
\end{lem}
\begin{proof}
First observe that the fixed point algebra is exactly equal
to $(\tM\oti P)^{\th^1}$,
which contains $M \oti M^0$.
Since $M'\cap \tM=\meC_M$
by \cite[Theorem II.5.1]{CT},
we have
\[
\big{(}(\tM\oti P)^{\th^1}\big{)}'
\cap
(\tM\oti Q)^{\th^1}
\subs
(M \oti M^0)'
\cap
(\tM\oti Q)^{\th^1}
\subs
(\meC_M\oti \C)^{\th^1}
=\C.
\]

Next we have to show that $\id_{\tM}\oti\wdh{\ga}|_Z$ is faithful.
Take a one-parameter unitary $v(p)$, $p\in\R$ in $\tM$
such that $\th_s(v(p))=e^{-isp}v(p)$ as usual.
It is trivial that $\th_s^1(v(p)\oti1)=e^{-isp}(v(p)\oti1)$,
and $\th^1$ is a dual action.
Then we naturally have
\[
(\tM\oti P)^{\th^1}\vee\{v(\R)\oti1\}''
\cong
(\tM\oti P).
\]
Since $\id_{\tM}\oti\wdh{\ga}|_Z$ fixes $v(p)$
and
the action is faithful
on the right hand side,
the action on $(\tM\oti P)^{\th^1}$ is faithful.
\end{proof}

Since $(\tM\oti Q)^{\th^1}\cong M$ is an infinite factor,
the minimal action $\ga|_Z$ is dual.
Hence there exists a unitary representation
$U$ of $\wdh{Z}$ in $(\tM\oti Q)^{\th^1}$
such that $\wdh{\gamma}_z(U_\chi)=\chi(z)U_{\chi}$
for $z\in Z$, $\chi\in\wdh{Z}$.
Then $\psi_\chi:=\Ad U_\chi|_{\tM\oti P}$
gives an action of $\wdh{Z}$ on $\tM\oti P$
because $P$ is the fixed point algebra by the action $\wdh{\ga}|_Z$.
Note that $\ps$ is commuting with $\th^1$.
Since $Z(\tM\oti P)=Z(\tM)\oti \C$,
we have
\begin{equation}\label{eq:ps-id}
\ps_\chi=\id
\quad
\mbox{on }Z(\tM\oti P)
\mbox{ for }\chi\in \wdh{Z}.
\end{equation}

Set
$V_\pi:=(1\oti \lambda_{\pi})(U_{[\pi]}^*\otimes 1_\pi)$
which belongs to $\tM\oti P\otimes B(H_\pi)$.
Then we have
\begin{equation}\label{eq:th-V-c}
(\th_t^1\oti\id)(V_{23})
=
V_{23} c(t)_{13}.
\end{equation}
Let $m_\pi(x):=\Ad (1\oti \lambda_{\pi})(x\otimes 1_\pi)$
for $\pi\in\IG$
and $x\in \tM\oti P$.
Then $m$ is an action of $\bhG$ on $\tM\oti P$
because of (\ref{eq:la-a}),
and it is commuting with $\th^1$.
Trivially, the following equality holds:
\begin{equation}\label{eq:mVps}
m_\pi(x)=V_\pi (\psi_{[\pi]}(x)\otimes 1)V_\pi^*
\quad
\mbox{for }\pi\in\IG, x\in \tM\oti P.
\end{equation}

We set $M_1:=(\tM\otimes P)\rti_{\th^1}\R$ that is isomorphic to $M$.
Then the actions $m$ and $\ps$ naturally extend to the actions on $M_1$,
which are also denoted by $m$ and $\ps$, respectively.
Then (\ref{eq:mVps}) holds on $M_1$.
Thanks to Takesaki duality,
we obtain an isomorphism of covariant systems
between $\{M_1\rti_{\si^{\wdh{\ta_1}}}\R,\wdh{\si^{\wdh{\ta_1}}}\}$
and $\{B(L^2(\R))\oti \tM\oti P,\Ad\rho(\cdot)\oti \th^1\}$,
where
$\rho$ denotes the regular representation of $\R$.
Then by simple calculation,
it turns out that
the canonical extension $\wdt{m}$ and $\wdt{\ps_\chi}$
are transformed to $\id\oti m$ and $\id\oti \ps_\chi$,
respectively.
From (\ref{eq:ps-id}) and (\ref{eq:th-V-c}),
we see $m$ has the given invariants $\ps$, $a$ and $c$.

\subsection{Modular actions}\label{sec:modular}
Consider the invariant set $\mX$
when the modular part is full,
that is, $\IH=\IG$.
Then $\mX$ consists of all elements
$(\mu^*,c)$
such that $\mu$ is a $Z(\tM)$-valued 2-cocycle of $\bhH$
and $c(t)\in Z(\tM)\oti\lhG$
is
a $Z(\tM)$-valued $\th$-cocycle,
and they satisfy
\[
\th_t(\mu^*)=(\id\oti\De)(c(t)^*)\mu^*c(t)_{12}c(t)_{13}.
\]

We recall the notion of the irreducibility of $c$
which is introduced in \cite[p.17]{Iz}.
Namely, $c_\rho$ is \emph{irreducible}
\index{2-cocycle!irreducible--}
when the minimal subgroup of $c_\rho$
is irreducibly acting on $H_\rho$.
Note that if an irreducible cocycle $c$ is a coboundary,
then the dimension of $H_\rho$ must be equal to $1$
for any $\rho\in\IH$.

\begin{lem}\label{lem:bichara}
Let $(\mu^*,c)\in\mX$ such that $c_\rho$ is irreducible
and not a coboundary for each $\rho\in\IH$.
Then the bicharacter $\beta_{\mu}=\mu^* F(\mu)$
is an element in
$\mathbb{C}\otimes \lhG\otimes \lhG$,
and $\Ad \beta_\mu (1\oti \Delta^\opp(x))=1\oti\De(x)$ holds
for all $x\in\lhG$.
\end{lem}
\begin{proof}
We will present two different proofs.
The first proof comes from \cite{Iz,Z}.
From the following equalities:
\[
\mu_{123}(\id\oti\De\oti\id)(\mu)
=
\mu_{134}(\id\oti\id\oti\De)(\mu),
\quad
\th_t(\mu^*)=(\id\oti\De)(c(t)^*)\mu^*c(t)_{12}c(t)_{13},
\]
we may and do assume that
$c(t)$ is a minimal cocycle \cite[Corollary 3.8(i)]{Z} with
a minimal subgroup $K_c\subs U(\lhH)$.
Then the projection of $\lhH$ onto $B(H_\rho)$
induces an irreducible representation of $K_c$.
Then we have
\begin{align*}
\th_t(\be_\mu)
&=
\th_t(\mu^*)\th_t(\mu_{132}^*)
\\
&=
(\id\oti\De)(c(t)^*)\mu^*c(t)_{12}c(t)_{13}
\cdot
c(t)_{12}^*c(t)_{13}^*\mu_{132}(\id\oti\De^\opp)(c(t))
\\
&=
(\id\oti\De)(c(t)^*)\be_\mu(\id\oti\De^\opp)(c(t)).
\end{align*}
This is rewritten as
\[
\be_\mu(\om)_{\pi,\rho}
c(\om,t)_{\rho\pi}
=
c(\om,t)_{\pi\rho}
\be_\mu(F_{-t}^M\om)_{\pi,\rho}
\quad\mbox{a.e. }\om.
\]
Thanks to \cite[Theorem 3.14]{Z},
the function $X_M\ni\om\to\be_\mu(\om)\in \lhG\oti\lhG$
is constant almost everywhere,
and $\be_\mu\in \C\oti(\De^\opp,\De)$.

For the other proof,
let us take an action $\al$ of $\bhH$ on $M$
such that $(a^\al,c^\al)=(\mu^*,c)$
as Theorem \ref{thm:model1}.
Then there exists a unitary $V\in\tM\oti\lhH$
such that $\al=\Ad V|_M$ and
\[
\mu^*=(\id\oti\De)(V^*)V_{12}V_{13},
\quad
c(t)=V^*\th_t(V).
\]

The irreducibility of $c$ implies that
$\al_\rho$ is irreducible by \cite[Corollary 3.6]{Iz}.
Then the action $\al$ is free.
Indeed, we let $a\in M\oti B(H_\rho)$ satisfy
$a(x\oti1)=\al_\rho(x)a$ for every $x\in M$,
which means $\id_M$ is contained in the corresponding
endomorphism (see Section \ref{subsect:endo-action}).
Thanks to \cite[Proposition 3.4]{Iz},
$\al_\rho$ is modular,
that is, $c_\rho$ is a coboundary.
This is a contradiction.
Hence it follows that
$(\alpha_\pi, \alpha_\rho)=(\pi,\rho)$
for any (reducible)
representations $\pi,\rho$.

Set $\cE:=V_{12}V_{13}V_{12}^*V_{13}^*$,
which is considered as a permutation symmetry
of $\Rep(\bH)$ as \cite[p.28]{Iz}.
Since
\[
\theta_t(\cE)
=V_{12}V_{13}c(t)_{12}c(t)_{13}
c(t)_{12}^{*}c(t)_{13}^{*}V_{12}^*V_{13}^*
=\cE,
\]
it turns out that $\cE\in M\otimes A\otimes A$.
Moreover,
\[
\cE
(\alpha\otimes \id)(\alpha(x))_{132}
=(\alpha\otimes\id)(\alpha(x)) \cE.
\]
Then the freeness of $\al$ implies that
$\cE\in \C\oti(\De^\opp, \Delta)$.

Next we have
$\cE=V_{12}V_{13}V_{12}^*V_{13}^*
=(\id\otimes \Delta)(V)\be_\mu
(\id\otimes \De^\opp)(V^*)$.
Hence $\be_{\mu}=(\id \otimes
\Delta )(U^*)\cE(\id\otimes \De^\opp)(U)
=\cE$.
Therefore, $\be_\mu\in \C\oti(\De^\opp,\De)$.
\end{proof}

Let $M$ be a type III factor, 
and $\alpha$ a free and modular action of $\bhG$ on $M$.
Thus we have $\bG=\bH$ by the notation in the previous section.
This kind of actions are already treated by Yamanouchi in
\cite[Section 6]{Y},
and we will further extend his argument.
Let $(\mu^*,c)\in\mX$ be the invariant
corresponding to a choice of a unitary $V\in \tM\oti\lhH$
such that $\al=\Ad V|_M$,
$\mu=V_{13}^*V_{12}^*(\id\oti\De)(V)$
and $c(t)=V^*\th_t(V)$.
Then they satisfy
\[
\mu_{123}(\id\oti\De\oti\id)(\mu)
=
\mu_{134}(\id\oti\id\oti\De)(\mu),
\ 
\th_t(\mu^*)=(\id\oti\De)(c(t)^*)\mu^*c(t)_{12}c(t)_{13}.
\]

\begin{thm}\label{thm:H-minimal}
Let $(\mu^*,c)\in\mX$
such that $c_\rho$ is irreducible and not a coboundary
for each $\rho\in\IH\setm\{\btr\}$.
Then the following statements hold:
\begin{enumerate}
\item
$\mu$ is equivalent to a scalar valued co-commutative 2-cocycle $\mu_0$;
\item
The minimal subgroup $K_c$ is the dual cocycle twisting of $\bH$ by $\mu_0$.
\end{enumerate}
\end{thm}
\begin{proof}
(1).
We represent $Z(\tM)$ as $L^\infty(X, dm)$ for some measure space $(X,dm)$. 
Let $\mu=\int^\oplus_X\mu(x)\,dm(x)$ be the central decomposition.
By Lemma \ref{lem:remark} and Lemma \ref{lem:bichara},
$\mu(x)\in Z^2(\bhH)_c$ for almost every $x$.
Employing Theorem \ref{thm:bichara},
we see that $\mu(x)$ is equivalent
to some $\mu_0\in Z^2(\bhH)_c$ for almost every $x$.
Hence there exists a measurable function
$X\ni x\to z(x)\in U(\lhG)$
with
$\mu_0=(z(x)\otimes z(x))\mu(x)\Delta(z(x)^*)$.
Let $z=\int^\oplus_X z(x)\,dm(x)$.
By replacing $V$ with $Vz$, we are done.

(2).
Since the action $\al$ is free,
each projection $\lhH$ onto $B(H_\pi)$ gives an irreducible
representation of $K_c$, and all the irreducibles arise like this.
Thus the group algebra of $K_c$ is nothing but $\lhH$
with the coproduct $\De^{K_c}(k)=k\oti k$ for $k\in K_c$.
This is rewritten as $(\id\oti\De^{K_c})(c(t))=c(t)_{12}c(t)_{13}$,
which is equal to $(\id\oti\De^{\mu_0})(c(t))$.
Hence $\De^{K_c}=\De^{\mu_0}$.
\end{proof}

\begin{rem}
This result is a special case
studied in \cite{IzKo-coh},
and $\bH_\mu$ is a minimal compact group which 
corresponds to $c(t)$ (see \cite{Iz}).
We will explain a modular action studied by Izumi
in \cite[\S 5.2]{Iz}.
We are firstly given a system $\mD$ of modular endomorphisms
on a type III factor $M$.
Then $\mD$ naturally induces a $\th$-cocycle $c$
with compact minimal subgroup $K_c$,
and we obtain a free modular action of $\wdh{K_c}$ on $M$
whose 2-cocycle $\mu$ is trivial.
\end{rem}

In the following,
it turns out that
the freeness of a modular action forces
both $M$ and $\bH$ to have some properties.

\begin{cor}
Let $\al$ be a free modular action of a discrete Kac algebra
$\bhH$ on a factor $M$.
Then the following statements hold:
\begin{enumerate}
\item
$M$ is of type III;
\item
If $M$ is not of type III$_0$,
then $\bH$ is a compact abelian group;
\item
If $M$ is of type III$_0$,
then $\bH$ is a dual cocycle twisting of a compact group.
In particular, the fusion rule algebra of $\bH$ is commutative.
\end{enumerate}
\end{cor}
\begin{proof}
(1).
Suppose that $M$ is semifinite.
We may assume that $M$ is infinite by considering $B(\el^2)\oti M$.
By freeness of $\al$,
each map $\al_\pi$ corresponds to an irreducible modular endomorphism.
However, such endomorphism must be an inner automorphism
by \cite[Lemma 4.10, Theorem 4.12]{MT2}.
This contradicts the freeness of $\al$.

(2).
As explained in (1),
each corresponding endomorphism must be an outer modular automorphism.
In particular, $\lhH$ is commutative.
Thus we must have $K_c=\bH$, and $\bH$ must be a compact abelian group.

(3).
This is an immediate consequence of the previous theorem.
\end{proof}

For (3) in the above,
the converse statement holds.
\begin{prop}
Let $\bH$ be a compact Kac algebra
which is a dual cocycle twisting of a compact group.
Then $\bhH$ admits a free modular action
on any type III$_0$ factor.
\end{prop}
\begin{proof}
Let $\mu\in \lhH\oti\lhH$ be a 2-cocycle
such that $K:=\bH_\mu$ is a compact group,
that is,
the coproduct of $\wdh{K}$ is given by
$\Ad \mu\circ\De$.
Note that each $\rho\in\IH$
gives an irreducible representation of $K$.

Let $M$ be a type III$_0$ factor and $\th$ the dual flow
on $\tM$.
We denote by $X_M$ the flow space of $M$ associated with $\th$.
Thanks to \cite[Proposition A.5]{Iz},
there exists a minimal cocycle $c\col X\times\R\ra K$.
Regarding $c$ as a $\th\oti\id$-cocycle $c(s)$
in $\meC_M\oti\lhH=\meC_M\oti L^\infty(\wdh{K})$,
we have $(\id\oti\De^\mu)(c(s))=c(s)_{12}c(s)_{13}$
since $c(x,s)\in K$ for all $x\in X_M$.
Then it is clear that $(\mu^*,c)$ is
the characteristic cocycle.
By Theorem \ref{thm:model1},
we obtain an action $\al$ of $\bhH$ on $M$
which attains $(\mu^*,c)$ as its invariant.
Employing Proposition \ref{prop:al-free},
we see $\al$ is free.
\end{proof}

\subsection{Examples}
Let us consider a free modular action $\al$ on a type III$_0$ factor $M$
when $\bH=SU(2)$ or $SU_{-1}(2)$.
\index{$SU(2)$}
\index{$SU_{-1}(2)$}

\begin{prop}
If $\al$ is a free modular action
of the dual of $SU(2)$ on a type III$_0$ factor,
then we can take a unitary $V\in\tM\oti\lhH$
such that $\tal=\Ad V$ and $(\id\oti \De)(V)=V_{12}V_{13}$.
\end{prop}
\begin{proof}
It is shown in \cite[Theorem 2]{Wass-cptIII}
that every ergodic action of $SU(2)$ of full multiplicity
is isomorphic to the regular action on $L^\infty(SU(2))$.
Then every 2-cocycle is a coboundary from \cite[Theorem 2]{Wass-cptII}.
\end{proof}

Hence the invariant of $\al$ consists of
the $\th$-cocycle $c\in Z(\tM)\oti\lhH$.
Putting $\mu_0=1$ in Theorem \ref{thm:H-minimal},
we may assume that the target space
of $c$ is $SU(2)$ that is in fact the minimal group of $c$.
By Theorem \ref{thm:al-be-a-c},
two free modular actions are strongly cocycle conjugate
if and only if
those $\th$-cocycles are equal up to equivalence.

Next we study the case of $\bH=SU_{-1}(2)$.
By Theorem \ref{thm:H-minimal},
we know that the minimal group $K_c$
is a dual cocycle twisting of $SU_{-1}(2)$.
However, every 2-cocycle of the dual of $SU_{-1}(2)$
is a coboundary
(cf. \cite[Corollary 5.10]{BRV}).
This is a contradiction.
We can also derive
the contradiction in a somewhat direct way as follows.

Observe that
the fusion rule of $K_c$ coincides with that of $SU_{-1}(2)$.
Let $\{\pi_\nu\}_{\nu}$, $\nu\in (1/2)\Z_{+}$
be the irreducible representations of $K_c$
such that $\pi_{1/2}\pi_{1/2}=\btr+\pi_1$ and so on.
Note that $d(\pi_{1/2})=2$ because of the twisting.
Hence $K_c$ is regarded as a closed subgroup of $U(\C^2)$.
The action of $K_c$ on $\C^2\bigwedge\C^2$
is given by taking the determinant.
However, the decomposition rule $\pi_{1/2}\pi_{1/2}=\btr+\pi_1$
forces $K_c$ to be a closed subgroup of $SU(2)$.
The fusion rule
also implies that
the restriction of every irreducible representation of $SU(2)$
is irreducible.
Thus $K_c=SU(2)$ because each irreducible representation
in the $SU(2)$-algebra $C(SU(2)/K_c)$ must be trivial on $K_c$.
However, we know that the dual of $SU(2)$ has no non-trivial 2-cocycle,
and this is a contradiction.
Hence we have the following result.

\begin{prop}
The Kac algebra $SU_{-1}(2)$ is not a dual cocycle twisting of a compact group.
In particular,
the dual of $SU_{-1}(2)$ does not have a free modular action
on a type III$_0$ factor.
\end{prop}


\begin{rem}
\label{rem:Net-Tus}
As is remarked in \cite[Remark 4.2]{IKf},
if $\bG$ is a connected compact group,
then any co-commutative cocycle $\om$ of $\bhG$
comes from a closed central subgroup,
and the dual cocycle twisting $\bG_\om$ coincides with $\bG$.
This relates with the following result due to Neshveyev--Tuset.

\begin{thm}[Neshveyev--Tuset]
Let $G$ be a connected compact group.
Then
\[
H_G^2(\widehat{G})=H^2(\widehat{Z(G)}),
\]
where $H_G^2(\bhG)$
denotes the cohomology group of $G$-invariant cocycles,
and $Z(G)$ the center of $G$.
\end{thm}
\end{rem}

\subsection{Crossed products}\label{sec:cross}

We will investigate the flow of weights of
the crossed product $M\rtimes_\alpha \bhH$
of a free modular action $\al$ on $M$.
We may and do assume that an associated 2-cocycle $\mu$ is scalar-valued.
Let $\psi$ be a faithful normal semifinite weight on $M$,
and extend $\psi$
on $M\rtimes_\alpha \bhH$ by $\psi\circ E$. We can identify 
$(M\rtimes_\alpha \bhH)\rtimes_{\sigma^\psi}\mathbb{R}$ with 
$(M\rtimes_{\sigma^\psi}\mathbb{R})\rtimes_{\tal}\bhH$.
We fix $V_\pi$ in the previous section
so that $\mu$ is in $Z^2(\bhH)_{c,n}$.
Let $\lambda_{\pi}$ be the implementing unitary in $M\rtimes_\alpha\bhH$, 
and put $W_{\pi}:=V_\pi^*\lambda_\pi$.
A simple computation shows 
\[
B:=\tM'\cap (\tM\rtimes_{\tal}\bhH)
=
\ovl{\spa}^{\rm w}\left\{zW_{\pi_{i,j}}\mid z\in Z(\tM),
\ \pi\in\IH,\ i,j\in I_\pi\right\}.
\]
In particular, $W$ is an element of $B\oti\lhG$.
The dual action $\theta_t$ on $B$ is described as
$\theta_t(zW_{\pi_{i,j}})
=\sum_{k}\theta_t(z)c(t)_{\pi_{k,i}}^*W_{\pi_{k,j}}$.
The relative commutant $B$
is endowed with the dual action of $\tal$ by $\bH$.
Then for each $j$ and a non-zero $z\in Z(\tM)$,
$\{zW_{\pi_{j,i}}\}_{i\in I_\pi}$ span the copy of the irreducible
$\bH$-module $H_\pi$.
Hence the $\pi$-spectral subspace $B_\pi$ is
$\si$-weakly spanned by them.

The unitary $W\in (\tM\rti_\tal\bhH)\oti\lhH$
is a $\mu$-representation.
Indeed, we have
\begin{align*}
W_{12}W_{13}
&=V_{12}^*\lambda_{12}V_{13}^*\lambda_{13}
=V_{12}^*(\tal\otimes \id)(V)\lambda_{12}\lambda_{13}\\
&=V_{13}^*V_{12}^*\lambda_{12}\lambda_{13}
=\mu(\id\otimes \Delta)(W).
\end{align*}
Thus $\Ad W^*$ gives an inner action on $B$,
and $Z(B)=B^{\Ad W^*}$.
Since $\mu$ is scalar-valued,
\[
L_\mu:=\ovl{\spa}^{\rm w}\{W_{\pi_{i,j}}\mid\pi\in\IH, i,j\in I_\pi\}
\]
is a subalgebra of $B$,
and we have $B\cong Z(\tM)\otimes L_\mu$.
Note that $L_\mu$ is nothing but the \emph{link algebra}
\index{link algebra}
of $\lH$ and $L^\infty(K_c)$
that was introduced in \cite{BRV}.
The link algebra naturally admits two commuting ergodic actions
of $K_c$ and $\bH$ as follows.
For $k\in K_c$, the map $\ga_k$ is defined by
\[
(\gamma_k\otimes \id_\pi)(W_\pi)=(1\oti\pi(k)^* )W_\pi,
\]
where $\pi\col K_c\ra U(H_\pi)$ is an irreducible representation
defined by $\pi(c(t,\om)):=c(t,\om)_\pi$ for $t\in\R$ and $\om\in X_M$.
The action of $\bH$ is introduced by the restriction of the dual
action $\hal$, that is,
\[
(\hal\oti\id)(W)=W_{13}\la_{23}.
\]
Indeed, $\ga$ gives us an action as seen below:
\begin{align*}
(\gamma_k\otimes \id_{\pi\otimes \rho})(W_{\pi,12}W_{\rho,13})
&=
(\gamma_k\otimes \id_{\pi\otimes \rho})
(\mu_{\pi,\rho} (\id\oti{}_\pi\De_\rho)(W))
\\
&=
(\gamma_k\otimes \id_{\pi\otimes \rho})
((\id\oti{}_\pi\De_\rho^\mu)(W)\mu_{\pi,\rho})
\\
&=(1\oti\pi(k)^*\otimes \rho(k)^*)
(\id\oti{}_\pi\De_\rho^\mu)(W)\mu_{\pi,\rho} 
\\
&=(1\oti \pi(k)^*\otimes \rho(k)^*)\mu_{\pi,\rho}
(\id\oti{}_\pi\De_\rho)(W)
\\
&=
(\gamma_k\otimes \id_{\pi\otimes \rho})(W_{\pi,12})
(\gamma_k\otimes \id_{\pi\otimes \rho})(W_{\rho,13}),
\end{align*}
\begin{align*}
(\ga_k\oti\id_\pi)(W_\pi^*)
&=
(\ga_k\oti\id_\pi)((\id\oti\ka)(W_\opi))
\quad(\mbox{since $\mu$ is normalized})
\\
&=
(\id\oti\ka)((1\oti \opi(k)^*)W_\opi)
\\
&=
W_\pi^*(1\oti \pi(k)).
\end{align*}

We next compute the flow of weights of $M\rtimes_\alpha \bhH$.

\begin{thm}\label{thm:center}
$Z(\tM\rtimes_{\tal}\bhH)$ equals $Z(B)$,
which is described as follows: 
\[
Z(\tM\rtimes_{\tal}\bhH)
=
\ovl{\spa}^{\rm w}\left\{
z (W(1\oti p_\mu))_{\pi_{i,j}}
\mid z\in Z(\tM),
\
\pi\in\IH,
\
i,j\in I_\pi
\right\}.
\]
\end{thm}
\begin{proof}
Note that $Z(\tM)\subs Z(\tM\rti_\tal\bhH)$
because the action $\tal$ is inner.
It is trivial that $\tM\rti_\tal\bhH$
is generated by $\tM$ and $W_{\pi_{i,j}}$'s.
Thus $Z(\tM\rti_\tal\bhH)$ is the fixed point
subalgebra of $B$ by the inner action $\Ad W^*$.
We show the inner action $\Ad W^*$ on $B$ is acting on $B_\pi$
for each $\pi\in\IH$.
Since we have
\begin{align*}
 W_{13}^*W_{12}^*W_{13}W_{12}
&=(\id\otimes \De)(W^*)\mu^*F(\mu)
(\id\otimes \Delta^*))(W)\\
&=(\id \otimes \De^\opp)(W^*)(\id\otimes \De^\opp)(W)\beta_\mu
=\beta_\mu,
\end{align*}
$W_{13}W_{12}=W_{12}W_{13}\beta_\mu$ holds.
Then for $z\in Z(\tM)$,
we have
\[
W_{13}^*W_{12}(z\oti1\oti1)W_{13}
=
W_{13}^*W_{12}W_{13}(z\oti1\oti1)
=
W_{12}(z\oti\be_\mu^*).
\]
This implies
\begin{equation}\label{eq:WzW}
\Ad W_\rho^*(zW_{\pi_{i,j}}\oti1_\rho)=\sum_{k\in I_\pi}zW_{\pi_{i,k}}\oti (\be_\mu^*)_{\pi_{k,j},\rho},  
\end{equation}
and $W^*(B_\pi\oti\C)W\subs B_\pi\oti \lhH$.
Thus the fixed point algebra $B^{\Ad W^*}$
is generated by subspaces $B_\pi^{\Ad W^*}$ for $\pi\in\IG$.
Take any $x\in B_\pi^{\Ad W^*}$.
Then there exists $z_{\pi_{i,j}}\in Z(\tM)$
such that $x=\sum_{i,j}z_{\pi_{j,i}}W_{\pi_{i,j}}$.
By (\ref{eq:WzW}),
we must have
$z_{\pi_{k,i}}\oti1_\rho=\sum_j z_{\pi_{j,i}}\oti(\be_\mu^*)_{\pi_{k,j},\rho}$.
Hence $z:=\sum_{i,j} z_{\pi_{i,j}}\oti e_{\pi_{i,j}}$
is a member of $Z(\tM)\oti p_\mu\lhH$.
\end{proof}

\begin{cor}
The flow space of $M\rti_\al\bhH$ is identified with
$X_M\times K_c/G_{\mu,c}$,
where $G_{\mu,c}$ is a closed subgroup of $K_c$.
The flow is given by $(x,\ovl{k})\cdot t=(x\cdot t,\ovl{c(t,x)k})$,
where $\ovl{k}$ denotes the equivalence class of $k\in K_c$.
\end{cor}
\begin{proof}
From the previous theorem,
$Z(\tM\rti_\tal\bhH)$ coincides
with $B^{\bK_\mu}$,
the fixed point algebra by the restriction of $\hal$.
Thus $Z(\tM\rti_\tal\bhH)\cong Z(\tM)\oti L_\mu^{\bK_\mu}$.
Since $K_c$ acts on the abelian von Neumann algebra
$L_\mu^{\bK_\mu}$ ergodically,
there exists a closed subgroup $G_{\mu,c}$ in $K_c$
such that $L_\mu^{\bK_\mu}\cong L^\infty(K_c/G_{\mu,c})$.
Hence the flow space of $M\rti_al\bhH$
is identified with $X_M\times K_c/G_{\mu,c}$.
On the flow,
we have $(\th_t\oti\id)(W(1\oti p_\mu))=c(t)^*W(1\oti p_\mu)$.
Since $(\ga_k\oti\id)(W)=(1\oti k^*)W$,
we see the flow is given by the skew product.
\end{proof}

Since we have $Z(L_\mu)\cong L^\infty(K_c/G_{\mu,c})$,
the ergodic action $\ga$ on $L_\mu$ is
induced from an ergodic action of $G_{\mu,c}$ on a finite factor.
Thus there exists a unitary $\nu\in L^\infty(\wdh{K_c})=\lhH$
such that
$\om:=(\nu\oti\nu)\mu^*\De^\mu(\nu^*)$
is contained in $L^\infty(\wdh{G_{\mu,c}})\oti L^\infty(\wdh{G_{\mu,c}})$,
and the corresponding bicharacter $\be:=\om_{21}\om^*$ is non-degenerate
by \cite[Theorem 12]{Wass-cptII},
where $\be$ is the one introduced in \cite[p.1513]{Wass-cptII}.
Again thanks to \cite[Theorem 12]{Wass-cptII},
the von Neumann algebra generated by $(\id\oti\ph)(\be)$,
$\ph\in L^\infty(\wdh{G_{\mu,c}})_*$
coincides with $L^\infty(\wdh{G_{\mu,c}})$.

Let $\be_\mu:=\mu^*F(\mu)$.
Then
the computation $(\id\oti\ph)(\be)=\nu(\id\oti\nu^*\ph\nu)(F(\be_\mu))\nu^*$
implies $L^\infty(\wdh{G_{\mu,c}})=\nu L^\infty(\wdh{\bK_\mu})\nu^*$.
Since $\om$ is a 2-cocycle of $\wdh{G_{\mu,c}}$,
we see that
$(\nu^*\oti\nu^*)\mu\De(\nu)
\in L^\infty(\wdh{\bK_\mu})\oti L^\infty(\wdh{\bK_\mu})$.
Therefore, we obtain the following result.
\begin{cor}
The compact group $G_{\mu,c}$ introduced in the previous corollary
is a dual cocycle twisting of $\wdh{\bK_\mu}$.
\end{cor}
\
We have shown $\wdh{K_c}=\bhH^\mu$.
Since
$(\alpha\otimes \id)\circ\alpha
=(\id\otimes \Delta)\circ \alpha
=\Ad (1\otimes \mu^*)(\id\otimes \Delta^\mu)\circ \alpha $, 
$\alpha$ is a cocycle action of $\wdh{K_c}$ with a 2-cocycle $\mu^*$.
Since $M$ is of type III, we can perturb $\alpha$
by a unitary $v\in M\otimes A$
so that $\Ad v\alpha$ is a genuine action.
Let $P:=M\rtimes_{\Ad v\alpha}\wdh{K_c}$,
which is isomorphic to the cocycle crossed product $M\rtimes_{\alpha,\mu^*}\wdh{K_c}$. 
In fact, $P$ is nothing but a skew product of $\alpha$ by a
cocycle $c_t$ in the sense of \cite[Definition 5.6]{Iz}.
In particular, the inclusion
$M\subset M\rtimes_{\alpha,\mu^*} \wdh{K_c}$
isomorphic to $P^{K_c}\subset (P\otimes L_\mu)^{K_c}$
by \cite[Theorem 5.11]{Iz}.
The isomorphism is checked by direct computation as follows.
Let $\lambda^{\alpha,\mu^*}$
be the implementing unitary in $P$.
Then
\begin{align*}
\int_{K_c}
(\hal_k\otimes \gamma_k)
(\lambda^{\alpha,\mu^*}_{\pi_{ij}}\otimes W_{\rho_{rs}})\,dk
&=
\sum_{m,n}\bigg{(}
\int_{K_c}\pi(k)_{mj}\pi(k^{-1})_{rn}\,dk\bigg{)}
(\lambda^{\alpha,\mu^*}_{\pi_{im}}\oti W_{\pi_{ns}})
\\
&=
\delta_{j,r}
\sum_m
\lambda^{\alpha,\mu^*}_{\pi_{in}}\otimes W_{\pi_{ns}}
\\
&=
\de_{j,r}
(\lambda^{\alpha,\mu^*}_{13}W_{23})_{\pi_{i\el}}.
\end{align*}
Thus an isomorphism between $M\rtimes_{\al,\mu^*} \wdh{K_c}$
and $(P\otimes L_\mu)^{K_c}$
is given by sending $\lambda^{\al,\mu^*}$
to $\lambda^{\alpha,\mu^*}_{13}W_{23}$.
Since the action $\hal$ on $L_\mu$ is commuting with $\ga$,
the isomorphism intertwines $\bH$-actions.

\subsection{Hamachi--Kosaki decomposition}
We will close this section by studying the Hamachi--Kosaki decomposition
\index{Hamachi--Kosaki decomposition}
of $M\subset M\rtimes_\alpha \bhH$.
Let $\mZ:=Z(\tM \rtimes_{\tal}\bhH)$.
By Theorem \ref{thm:center},
we have $\mZ=Z(\tM'\cap(\tM\rtimes_{\tal}\bhH))$.
Set the following:
\[
\widetilde{\mA}
:= \mZ'\cap(\tM \rti_{\tal}\bhH)=\tM\rtimes_{\tal}\bhH,
\quad
\widetilde{\mB}:= \tM \vee \mathcal{Z},
\quad
\mA:=\widetilde{\mA}^\th,
\quad
\mB:=\widetilde{\mB}^\th.
\]
Then the two step inclusion $M\subset\mathcal{B}\subset \mathcal{A}=M\rti_\al\bhH$ is the Hamachi--Kosaki decomposition.
The Galois correspondence proved by Izumi--Longo--Popa
in \cite[Theorem 4.4]{ILP}
asserts that $\mB$ corresponds to a left coideal of $\lH$.
\index{coideal}
Indeed, the left coideal is $L^\infty(\bH/\bK_\mu)$ as computed below.

By Theorem \ref{thm:center}, we have
\[
\wdt{\mB}
=\ovl{\spa}^{\rm w}
\{(\id\oti\vph_\pi)(a_\pi W_\pi(1\oti p_\mu))\mid a\in \tM\oti\lhH,
\pi\in\IG\}.
\]
Note that $\tM$ and $L_\mu$ are independent with respect to
the averaging expectation of $\hal$.
Using $(\th_t\oti\id)(W)=c(t)^*W$ and a formal decomposition,
we obtain
\[
\mB
=\ovl{\spa}^{\rm w}
\{(\id\oti\vph_\pi)(a_\pi\la_\pi(1\oti p_\mu))\mid a\in M\oti\lhH,
\pi\in\IG\}.
\]
Since $p_\mu$ corresponds to the trivial representation of $\bK_\mu$,
$\mB$ is the fixed point algebra by the action of $\bK_\mu$ obtained from
the restriction of $\hal$.
Thus $M\subs (M\rti_\al\bhH)^{\bK_\mu}\subs M\rti_\al\bhH$
is the Hamachi--Kosaki decomposition.


\section{Classification of centrally free actions}
In this section,
we will study centrally free actions
on a McDuff factor.
The strategy to classify them
is same as that of \cite{MT1},
that is, we first construct
a Rohlin tower for a centrally free action,
and next we prove a 2-cohomology vanishing theorem.
However, the proof of \cite{MT1} is based on the existence
of a trace, and we should improve the argument
to be adapted to a general McDuff factor.
Readers are referred to \cite[Section 3]{MT1}
for usage of ultraproduct von Neumann algebras
and actions on them.

\subsection{Settings}
\begin{defn}
Let $\bhG$ be a discrete Kac algebra
and $M$ a von Neumann algebra.
\begin{itemize}
\item
A map $\th\col\IG\ra \Aut(M)$
is said to be a \emph{homomorphism}
when
$\th_\pi\th_\rho=\th_\si$
for all $\pi,\rho,\si\in\IG$ such that $\si\prec\pi\rho$.

\item
Let $(\al,u)$
be a cocycle action of $\bhG$ on $M$.
We will say that $(\al,u)$
is \emph{approximately inner modulo automorphism}
\index{approximately inner}
when
there exists a homomorphism $\th\col\IG\ra\Aut(M)$
such that the map
$\al_\pi\th_\pi^{-1}$ is approximately inner
for all $\pi\in\IG$.
\end{itemize}
\end{defn}

A typical example is the following.

\begin{ex}
\label{ex:modulo-auto}
If $M$ is injective and a cocycle action
$(\al,u)$ has the Connes--Takesaki module,
then it is an approximately inner modulo automorphism.
Indeed, let $s\col\Aut_\th(Z(\tM))\ra \Aut(M)$
be a homomorphic section of the Connes--Takesaki module map
constructed in \cite{ST} (also see \cite[Theorem 4]{Ham}).
If we set $\th_\pi:=s(\mo(\al_\pi))$,
then $\mo(\al_\pi\th_\pi^{-1})=\id$.
Thus $\al_\pi\th_\pi^{-1}$ is approximately inner
from \cite[Theorem A.6 (1)]{MT3}.
\end{ex}

\begin{ass}\label{ass:app-cent}
In this section, we always assume the following:
\begin{itemize}
\item
$M$ is a von Neumann algebra such that $M_\om$ is of type II$_1$;
\item
$(\al,u)$ is approximately inner
modulo automorphism;
\item
$(\al,u)$ is centrally free.
\index{action!centrally free cocycle--}
\end{itemize}
\end{ass}

In the above,
the central freeness of an action $\al$ on $M$
means that
$\al_\pi$ is properly centrally non-trivial
for each $\pi\in\IG\setm\{\btr\}$,
that is, there exists no nonzero element $a\in M\oti B(H_\pi)$
such that $\al_\pi^\om(x)a=(x\oti1)a$ for all $x\in M_\om$
(see \cite[Definition 8.1]{MT1}).

By definition of the approximate innerness \cite[Definition 3.5]{MT1},
we can take and fix
a unitary $U=(U^\nu)^\om\in M^\om\oti\lhG$
and a homomorphism $\th\col\IG\ra\Aut(M)$
such that
\[
\lim_{\nu\to\infty}(\chi\oti\tr_\pi)\circ\Ad (U_\pi^\nu)^*
=
\chi\circ\th_\pi\circ\Ph_\pi^\al
\quad\mbox{for all }\chi\in M_*
\
\pi\in\IG,
\]
where $\Ph_\pi^\al$ denotes the standard left inverse of $\al$,
that is,
\[
\Ph_\pi^\al(x)
=(1\oti T_{\opi,\pi}^*)(\al_{\opi}\oti\id)(x)(1\oti T_{\opi,\pi})
\quad
\mbox{for }x\in M\oti B(H_\pi).
\]

\subsection{Action $\ga$}
We introduce the cocycle action $(\ga,w)$ on $M^\om$
defined by $\ga:=\Ad U^* \circ\al^\om$
and $w:=U_{12}^*\al^\om(U^*)u(\id\oti\De)(U)$.
Then $\ga_\pi\circ (\th_\pi^\om)^{-1}$
preserves $M_\om$ by \cite[Lemma 3.7]{MT1},
and so does $\ga_\pi$ for each $\pi\in\IG$.
\begin{lem}
The restriction of the cocycle action $(\ga,w)$ on $M_\om$
gives a cocycle action,
that is, $w$ is evaluated in $M_\om$.
\end{lem}
\begin{proof}
Let $\chi\in M_*$ and $\pi,\rho\in\IG$.
Let $w^\nu$ be a unitary representing sequence of $w$.
We will prove
$(w^\nu)^*(\chi\oti\tr_\pi\oti\tr_\rho)w^\nu$
converges to $(\chi\oti\tr_\pi\oti\tr_\rho)$
in the norm topology of $(M\oti B(H_\pi)\oti B(H_\rho))_*$.
Then the matrix entries of $(w^\nu)_\nu$ is centralizing
by \cite[Lemma 3.6]{MT1}.
We set $\Ph_\pi^{\al\th^{-1}}:=\th_\pi\circ\Ph_\pi^\al$,
which is a left inverse of $\al_\pi\th_\pi^{-1}$.
Then we get
\begin{align*}
&(w^\nu)^*(\chi\oti\tr_\pi\oti\tr_\rho)w^\nu
\\
=&
(\id\oti\De)((U^\nu)^*)\al(U^\nu)u^*\cdot U_{12}^\nu
(\chi\oti\tr_\pi\oti\tr_\rho)
(U_{12}^\nu)^*
\cdot\al((U^\nu)^*)u(\id\oti\De)(U^\nu)
\\
\sim&
(\id\oti\De)((U^\nu)^*)u^*\al(U^\nu)\cdot
(\chi\circ\Ph_\pi^{\al\th^{-1}}\oti\tr_\rho)
\cdot\al((U^\nu)^*)u(\id\oti\De)(U^\nu)
\\
=&
(\id\oti\De)((U^\nu)^*)u^*\cdot
\left((\chi\circ\th_\pi\oti\tr_\rho)\circ \Ad (U_\rho^\nu)^*
\circ (\Ph_\pi^\al\oti\id_\rho)
\right)
\cdot u(\id\oti\De)(U^\nu)
\\
\sim&
(\id\oti\De)((U^\nu)^*)u^*\cdot
\left(\chi\circ\th_\pi\circ \Ph_\rho^{\al\th^{-1}}
\circ (\Ph_\pi^\al\oti\id_\rho)
\right)
\cdot u(\id\oti\De)(U^\nu)
\\
=&
(\id\oti\De)((U^\nu)^*)u^*\cdot
\left(\chi\circ\th_\pi\th_\rho\circ \Ph_\rho^{\al}
\circ (\Ph_\pi^\al\oti\id_\rho)
\right)
\cdot u(\id\oti\De)(U^\nu).
\end{align*}
Recall the following composition rule (\cite[Lemma 2.5]{MT1}):
\[
\Ph_\rho^\al\big{(}
(\Ph_\pi^\al\oti\id)(u_{\pi,\rho}xu_{\pi,\rho}^*)\big{)}
=
\sum_{\si\prec\pi\oti \rho}
\sum_{T\in \ONB(\si,\pi\rho)}
\frac{d(\si)}{d(\pi) d(\rho)}
\Ph_\si^\al((1\oti T^*)x(1\oti T)).
\]
Using this and $\th_\pi\th_\rho=\th_\si$ if $\si\prec\pi\rho$,
we have
\begin{align*}
&(w^\nu)^*(\chi\oti\tr_\pi\oti\tr_\rho)w^\nu
\\
\sim&
\sum_{\si\prec\pi\oti\rho}
\sum_{T\in \ONB(\si,\pi\oti\rho)}
\frac{d(\si)}{d(\pi) d(\rho)}
(1\oti T)(U_\si^\nu)^*\cdot
\left(\chi\circ \th_\si\circ\Ph_\si^\al
\right)\cdot U_\si^\nu(1\oti T^*)
\\
=&
\sum_{\si\prec\pi\oti\rho}
\sum_{T\in \ONB(\si,\pi\oti\rho)}
\frac{d(\si)}{d(\pi) d(\rho)}
(1\oti T)(U_\si^\nu)^*\cdot
\left(\chi\circ\Ph_\si^{\al\th^{-1}}
\right)\cdot U_\si^\nu(1\oti T^*)
\\
\sim&
\sum_{\si\prec\pi\oti\rho}
\sum_{T\in \ONB(\si,\pi\oti\rho)}
\frac{d(\si)}{d(\pi) d(\rho)}
(1\oti T)\cdot
\left(\chi\oti\tr_\si
\right)\cdot(1\oti T^*)
\\
=&
\chi\oti\tr_\pi\oti\tr_\rho.
\end{align*}
Thanks to \cite[Lemma 3.6]{MT1},
we see the 2-cocycle $w$ is evaluated in $M_\om$.
\end{proof}

\begin{thm}\label{thm:app-cocycle}
Let $(\al,u)$ be a cocycle action of $\bhG$
on a von Neumann algebra $M$
such that Assumption \ref{ass:app-cent} is fulfilled.
Then there exists a unitary $U\in M^\om\oti\lhG$
such that
\begin{itemize}
\item
$\Ad U^\nu$ approximates $\al\th^{-1}$,
that is, $\dsp\lim_{\nu\to\infty}(\chi\oti\tr_\pi)\circ\Ad
(U^\nu)^*=\chi\circ\Ph_\pi^{\al\th^{-1}}$
for all $\chi\in M_*$ and $\pi\in\IG$;
\item
$U_{12}^*\al^\om(U^*)u(\id\oti\De)(U)=1$.
\end{itemize}
\end{thm}
\begin{proof}
Let $(\ga,w)$ be as in the previous lemma.
Thanks to the 2-cohomology vanishing theorem \cite[Lemma 4.3]{MT1},
the 2-cocycle $w$ is a coboundary.
Thus we may and do assume that
$U_{12}^*\al^\om(U^*)u(\id\oti\De)(U)=1$.
\end{proof}

We always assume that $U$ is taken as above in what follows.
Then $\ga$ is an action.

\begin{lem}
The equality $\Ph_\pi^\ga=\Ph_\pi^{\al^\om}\circ\Ad U_\pi$ holds
for all $\pi\in\IG$.
\end{lem}
\begin{proof}
Let $x\in M^\om\oti B(H_\pi)$.
Then
\begin{align*}
\Ph_\pi^\ga(x)
&=
(1\oti T_{\opi,\pi}^*)
\ga_\opi(x)
(1\oti T_{\opi,\pi})
\\
&=
(1\oti T_{\opi,\pi}^*)
U_\opi^*\al_\opi^\om(x)U_\opi
(1\oti T_{\opi,\pi})
\\
&=
(1\oti T_{\opi,\pi}^*)
U_\opi^*\al_\opi^\om(U_\pi^*)\al_\opi^\om(U_\pi xU_\pi^*)
\al_\opi^\om(U_\pi)U_\opi
(1\oti T_{\opi,\pi})
\\
&=
(1\oti T_{\opi,\pi}^*)
(\id\oti\De)(U^*)u^*
\al_\opi^\om(U_\pi xU_\pi^*)
u(\id\oti\De)(U)
(1\oti T_{\opi,\pi})
\\
&=
(1\oti T_{\opi,\pi}^*)
u^*
\al_\opi^\om(U_\pi xU_\pi^*)
u
(1\oti T_{\opi,\pi})
=\Ph_\pi^{\al^\om}(U_\pi x U_\pi^*).
\end{align*}
\end{proof}

\begin{lem}\label{lem:ta-ga}
The equality $(\ta^\om\oti\id)(\ga_\pi(x))=\th_\pi(\ta^\om(x))\oti1$
holds for all $\pi\in\IG, x\in M^\om$.
\end{lem}
\begin{proof}
Let $\pi\in\IG$, $y\in B(H_\pi)$ and $\chi\in M_*$.
It suffices to show that
$(\chi\circ\ta^\om\oti\tr_\pi y)(\ga(x))=\chi(\th_\pi(\ta^\om(x)))\tr_\pi(y)$.
Using $(\chi\oti\tr_\pi)\circ\Ad (U^\nu)^*\to\chi\circ\Ph_\pi^{\al\th^{-1}}$
in $(M\oti B(H_\pi))_*$ as $\nu\to\infty$,
we have
\begin{align*}
(\chi\circ\ta^\om\oti\tr_\pi y)(\ga_\pi(x))
&=
\lim_{\nu\to\om}
(\chi\oti\tr_\pi y)((U^\nu)^*\al_\pi(x^\nu)U^\nu)
\\
&=
\lim_{\nu\to\om}
(\chi\oti\tr_\pi)((U^\nu)^*\cdot U^\nu(1\oti y)(U^\nu)^*\al_\pi(x^\nu)U^\nu)
\\
&=
\lim_{\nu\to\om}
\chi\circ\th_\pi\circ\Ph_\pi^\al(U^\nu(1\oti y)(U^\nu)^*\al_\pi(x^\nu))
\\
&=
\lim_{\nu\to\om}
\chi\circ\th_\pi(\Ph_\pi^\al(U^\nu(1\oti y)(U^\nu)^*)x^\nu)
\\
&=
\chi(\ta^\om(\th_\pi^\om\circ\Ph_\pi^{\al^\om}(U(1\oti y)U^*)\th_\pi^\om(x)))
\\
&=
\chi(\ta^\om(\th_\pi^\om(\Ph_\pi^{\ga}(1\oti y))\th_\pi^\om(x)))
\\
&=
\chi(\ta^\om(\tr_\pi(y)\th_\pi^\om(x)))
=\chi(\th_\pi(\ta^\om(x)))\tr_\pi(y),
\end{align*}
where we have used the fact that $\ga$ is an action.
\end{proof}

\begin{lem}\label{lem:g-ps-vph}
Let $x\in M_\om \oti\lhG\oti\lhG$, $y\in M^\om\oti\lhG\oti\lhG$
and $g\in M\oti\lhG$ be a unitary.
Let $\ps\in (M^\om)_*$ be a state.
Then the following inequalities hold
\begin{enumerate}
\item
For $\pi,\rho\in\IG$,
one has
\[
|(g(\ps\oti\vph_\pi)g^*\oti\vph_\rho)(xy)|
\leq d(\pi)|x|_{\ps\oti\vph_\pi\oti\vph_\rho}\|y\|.
\]
\item
For central projections $F$ and $K$ in $\lhG$
with finite support, one has
\[
|(g(\ps\oti\vph_F)g^*\oti\vph_K)(xy)|
\leq |F|_\vph|x|_{\ps\oti\vph_F\oti\vph_K}\|y\|.
\]
\end{enumerate}
\end{lem}
\begin{proof}
(1).
By using the matrix representation of $g_{12}^*xyg_{12}$,
we have
\begin{align*}
(g(\ps\oti\vph_\pi)g^*\oti\vph_\rho)(xy)
&=
\sum_{i,j,k,\el,m,n,p}
d(\pi)d(\rho)
\ps\left(
(g^*)_{i,j}
x_{(j,k),(\el,m)}
y_{(\el,m),(p,k)}
g_{p,i}
\right)
\\
&=
\sum_{i,j,k,\el,m,n,p}
d(\pi)d(\rho)
\ps\left(
x_{(j,k),(\el,m)}
(g^*)_{i,j}
y_{(\el,m),(p,k)}
g_{p,i}
\right).
\end{align*}
The matrix element
$h_{(\el,m),(j,k)}
:=\sum_{i,p}
(g^*)_{i,j}
y_{(\el,m),(p,k)}
g_{p,i}
$
defines the element $h\in M^\om\oti\lhG\oti\lhG$
by
\[
h_{132}=\sum_{j,\el}
(\id\oti\Tr_\pi\oti\id)
((g^*\oti1)(1\oti e_{\pi_{j,\el}}\oti 1)y(g\oti1))
\oti e_{\pi_{\el,j}}.
\]
Then we have
$(g(\ps\oti\vph_\pi)g^*\oti\vph)(xy)
=
(\ps\oti\vph_\pi\oti\vph_\rho)(xh)
$.
Since $(\id\oti\tr_\pi\oti\id_\rho)=d(\pi)^{-1}(\id\oti\Tr_\pi\oti\id_\rho)$
is a contraction
and $\sum_{j,\el}e_{\pi_{j,\el}}\oti e_{\pi_{\el,j}}$ is unitary,
we have
$\|h\|=\|h_{132}\|\leq d(\pi)\|(g^*\oti1)y(g\oti1)\|=d(\pi)\|y\|$.
Thus
$|(g(\ps\oti\vph_\pi)g^*\oti\vph_\rho)(xy)|
\leq
|x|_{\ps\oti\vph_\pi\oti\vph_\rho}\|h\|
\leq
d(\pi)|x|_{\ps\oti\vph_\pi\oti\vph_\rho}\|y\|$.

(2). The second statement is an immediate consequence of (1).
\end{proof}

\subsection{Rohlin tower}
Let $(\al,u)$ be a cocycle action of $\bhG$ on a McDuff factor $M$
as in the previous subsection.
Fix a unitary $v\in M^\om\oti\lhG$ which possesses
the properties of $U$
stated in Theorem \ref{thm:app-cocycle}.
In this subsection,
we will reprove the Rohlin type theorem \cite[Theorem 5.9]{MT1}
for the action $\ga:=\Ad v^*\circ\al$
in a more general fashion.
We keep the following notation throughout this section:
\begin{itemize}
\item
$\ph\in M_*$, a faithful state,
$\ps:=\ph\circ\ta^\om \in (M^\om)_*$;
\item
$\mG\subs U(M\oti\lhG)$, a finite subset such that $1\in \mathcal{G}$;
\item
$S\subs M^\om$, a countable set;
\item
$F$, a central projection with finite support in $\lhG$
such that $F=\ovl{F}$ and $Fe_\btr=e_\btr$.
Let $\mF:=\{\pi\in\IG\mid F 1_\pi\neq0\}$;
\item
$\de$, a positive number with $\de^{1/4}|F|_\vph<1$;
\item
$K$, an $(F,\de)$-invariant central projection
with finite support in $\lhG$ satisfying $Ke_\btr=e_\btr$.
Let $\mK:=\{\pi\in\IG\mid K 1_\pi\neq0\}$.
\end{itemize}

In the above, the $(F,\de)$-invariance of $K$ means
the following inequality holds
(See \cite[Definition 2.1]{MT1}):
\index{$(F,\de)$-invariance}
\[
|(F\oti1)\De(K)-F\oti K|_{\vph\oti\vph}<\de|F|_\vph|K|_\vph.
\]
This implies the following:
\begin{equation}
\label{eq:FDeK}
|(F\oti K)\De(K^\perp)|_\vph<\de|F|_\vph|K|_\vph
\end{equation}

We start to construct a Rohlin tower for the action $\ga$.
The construction presented here is almost same as \cite{MT1},
that is, so-called the \emph{build and destroy} method
\cite[Lemma 6.4, p.48]{Oc}.
However, our construction at this time requires
more technical estimates of inequalities.
Let us recall the \emph{diagonal operator}
\index{diagonal of a 2-cocycle}
$a$ of a cocycle action $(\al,u)$
introduced in \cite[Definition 5.5]{MT1}.
That is the unique operator such that
\[
(a\oti1)(1\oti \De(e_{\btr}))=u(1\oti\De(e_\btr)).
\]
By the norm estimate $\|a_\pi\|\leq d(\pi)$,
the diagonal operator $a$ is affiliated with $M\oti\lhG$,
a priori.
We have the following equality
(see \cite[Lemma 5.6 (2)]{MT1}):
\begin{equation}
\label{eq:aa}
(\id\oti\vph_\rho)(a^*a)=d(\rho)^2
\quad
\mbox{for }\rho\in\IG.
\end{equation}

Let us introduce the set $\meJ$ that is the collection
of a projection $E\in M_\om\oti\lhG$ such that
\begin{enumerate}[(E.1)]
\item
$E=E(1\oti K)$;
\label{item:E-K}

\item
In the decomposition of
\[
E=\sum_{\rho\in\mK}\sum_{i,j\in I_\rho}
d(\rho)^{-1}f_{\rho_{i,j}}\oti e_{\rho_{i,j}},
\]
the family $\{f_{\rho_{i,j}}\}_{i,j\in I_\rho}$ is a system of
matrix units and they are orthogonal with respect to $\rho\in\mK$,
that is, the following holds:
\[
f_{\pi_{i,j}}f_{\rho_{k,\el}}=\de_{\pi,\rho}\de_{j,k}f_{\pi_{i,\el}};
\]
\label{item:ga-E}

\item
In the decomposition of
\[
a^*vEv^*a=\sum_{\rho\in\mK}\sum_{i,j\in I_\rho}
d(\rho)^{-1}f_{\rho_{i,j}}^\al\oti e_{\rho_{i,j}},
\]
the family $\{f_{\rho_{i,j}}^\al\}_{i,j\in I_\rho}$ is a system of
matrix units and they are orthogonal with respect to $\rho\in\mK$;
\label{item:ga-avE}

\item
The initial projection of the partial isometry $a^*vE$ is equal to $E$;
\label{item:EvaavE}

\item
$(\id\oti\vph_\rho)(E)=(\id\oti\vph_\rho)(a^*vEv^*a)
\in S'\cap M_\om$ for all $\rho\in\mK$;
\label{item:vph-E}

\item
$(\ta^\om\oti\id)(E)=(\ta^\om\oti\id)(vEv^*)\in\C K$.
\label{item:ga-taE}
\end{enumerate}

These conditions derive useful properties of $E$.
(E.\ref{item:ga-E}) and (E.\ref{item:ga-avE})
implies the sliced element of $E$,
$(\id\oti\vph_\rho)(E)=(\id\oti\vph_\rho)(a^*vEv^*a)$
is a projection.
If we integrate $a^*vE$ over $K$, we get
\[
\mu_E:=(\id\oti\vph)(a^*vE),
\]
which is a partial isometry satisfying
$\mu_{E}^*\mu_{E}=(\id\oti\vph)(E)=\mu_{E}\mu_{E}^*$.
By (E.\ref{item:ga-E}), we have
\begin{equation}\label{eq:Emu}
(\mu_E\oti1)E=a^*vE.
\end{equation}
Indeed, we have
\begin{align*}
\mu_E E_{\si_{p,q}}
&=
\sum_{\rho\in\mK}\sum_{i\in I_\rho}
d(\rho)d(\si)^{-1}
(a^*vE)_{\rho_{i,i}}f_{\si_{p,q}}
\\
&=
\sum_{\rho\in\mK}\sum_{i,j\in I_\rho}
d(\si)^{-1}
(a^*v)_{\rho_{i,j}}f_{\rho_{j,i}}f_{\si_{p,q}}
\\
&=
\sum_{j\in I_\si}
d(\si)^{-1}
(a^*v)_{\si_{p,j}}f_{\si_{j,q}}
\\
&=
(a^*vE)_{\si_{p,q}}.
\end{align*}

The following equality follows
from (E.\ref{item:ga-taE}):
\begin{equation}
\label{eq:taEK}
(\ta_\om\oti\id)(E)=|K|_\vph^{-1}|E|_{\ps\oti\vph}K.
\end{equation}

For $g\in\mG$,
we introduce the functions $a,b^g,c^g$ on $\meJ$ as follows:
\begin{align*}
a_E
&:=
|F|_\vph^{-1}|\ga_F(E)
-(\id\oti_{F}\De_K)(E)|_{\ps\oti\vph\oti\vph},
\\
b_E^g
&:=|E|_{g(\ps\oti\vph)g^*},
\\
c_E^g
&:=
(\ps\oti\vph\oti\vph)
(g_{12}^*v_{12}
\ga(v)(\ga_F(E)-(\id\oti{}_F\De_K)(E))(\mu_E^*\oti F\oti K)g_{12}).
\end{align*}
In fact, $b_E^g$ does not depend on $g$, that is, $b_E^g=|E|_{\ps\oti\vph}=:b_E$.
Indeed,
\begin{align*}
b_E^g&=
(\ps\oti\vph)(g^*Eg)
=
(\ph\oti\vph)((\ta^\om\oti\id)(g^*Eg))
\\
&=
(\ph\oti\vph)(g^*(\ta_\om\oti\id)(E)g)
\\
&=
(\ph\oti\vph)(g^*(1\oti K)g)|K|_\vph^{-1}b_E
=
b_E
\quad
\mbox{by }
(\ref{eq:taEK}).
\end{align*}

Our task is to prove the following lemma.

\begin{lem}\label{lem:abc}
Suppose that an element $E\in \meJ$ satisfies
$b_E<1-\de^{1/4}$.
Then there exists $E'\in\meJ$ satisfying the following inequalities:
\begin{enumerate}
\item
$a_{E'}-a_E\leq3\de^{1/4}(b_{E'}-b_E)$;
\item
$0<(\de^{1/4}/2)|E'-E|_{\ps\oti\vph}\leq b_{E'}-b_E$;
\item
$|c_{E'}^g|-|c_E^g|\leq3\de^{1/4}|F|_\vph(b_{E'}-b_E)$ for all $g\in\mG$.
\end{enumerate}
\end{lem}

To prove this result,
we need the following lemma.
We let $\mL$ be the subset of $\mK$
such that
$\pi\in \mL$ if and only if
$\pi\prec \rho\si$
for some $\rho\in\mF$ and $\si\nin \mK$.
We claim that the size of $\mL$ is small relative to $\mK$.

\begin{lem}\label{lem:L}
The inequality $|1_\mL|_\vph<\de^{1/2}|F|_\vph|K|_\vph$ holds.
\end{lem}
\begin{proof}
Let us introduce the following set:
\[
\mT:=\{\pi\in\mK
\mid
|(F\oti1_\pi)\De(K)-F\oti 1_\pi|_{\vph\oti\vph}
<\de^{1/2}|F|_\vph|1_\pi|_\vph\}.
\]
Set $K_0:=1_{\mT}$ that is a central projection
supported by $K$.
Then
\begin{align*}
|1_{\mK\setminus\mT}|_\vph
&=
|K-K_0|_\vph
=
\sum_{\pi\in\mK\setminus\mT}
|1_\pi|_\vph
\\
&\leq
\sum_{\pi\in\mK\setminus\mT}
\de^{-1/2}
|F|_\vph^{-1}
|(F\oti1_\pi)\De(K)-F\oti 1_\pi|_{\vph\oti\vph}
\\
&\leq
\sum_{\pi\in\mK}
\de^{-1/2}
|F|_\vph^{-1}
|(F\oti1_\pi)\De(K)-F\oti 1_\pi|_{\vph\oti\vph}
\\
&=
\de^{-1/2}
|F|_\vph^{-1}
|(F\oti K)\De(K)-F\oti K|_{\vph\oti\vph}
\\
&<
\de^{1/2}
|K|_\vph.
\end{align*}
We show $\mL\subset \mK\setminus \mT$.
Now let $\pi\in\mL$.
If $\pi$ would be contained in $\mT$,
then we see that $\De(F)(1\oti 1_{\opi})\leq K\oti1_\opi$
by \cite[Lemma 3.5]{T}.
We give a proof of this fact for readers' convenience.
First, we prove
\begin{equation}\label{eq:FK}
(\vph\oti\vph)(\De(F)(K^\perp\oti1_\opi))
<\de^{1/2}|F|_\vph|1_\opi|_\vph.
\end{equation}
Indeed, by the Frobenius reciprocity, we have
$N_{\si,\opi}^\rho=N_{\rho,\pi}^\si$, and
\begin{align}
(\vph\oti\vph)(\De(F)(K^\perp\oti1_\opi))
&=
\sum_{\rho\in\mF,\si\nin\mK}
(\vph\oti\vph)(\De(1_\rho)(1_\si\oti1_\opi))
\notag\\
&=
\sum_{\rho\in\mF,\si\nin\mK}
N_{\si,\opi}^\rho
d(\si)d(\opi)d(\rho)
\notag\\
&=
\sum_{\rho\in\mF,\si\nin\mK}
N_{\rho,\pi}^\si
d(\si)d(\pi)d(\rho)
\label{eq:Nsirhopi}\\
&=
(\vph\oti\vph)(\De(K^\perp)(F\oti1_\pi))
\notag\\
&=
|(F\oti1_\pi)-(F\oti1_\pi)\De(K)|_{\vph\oti\vph}
\notag\\
&<\de^{1/2}|F|_\vph|1_\pi|_\vph
\quad\mbox{since } \pi\in\mT.
\notag
\end{align}
Second, we show $\De(F)(K^\perp\oti1_\opi)=0$.
By (\ref{eq:FK})
and applying the inequality
$d(\rho)d(\si)\geq N_{\osi,\rho}^\opi d(\opi)=N_{\si,\opi}^\rho d(\pi)$
to (\ref{eq:Nsirhopi}),
we obtain
\[
\de^{1/2}|F|_\vph|1_\pi|_\vph
>
(\vph\oti\vph)(\De(F)(K^\perp\oti1_\opi))
\geq
\sum_{\rho\in\mF,\si\nin\mK}
(N_{\rho,\pi}^\si)^2d(\pi)^2.
\]
This yields
$\sum_{\rho\in\mF,\si\nin\mK}(N_{\rho,\pi}^\si)^2
<\de^{1/2}|F|_\vph<1$,
and $N_{\rho,\pi}^\si=0$ for all $\rho\in\mF,\si\nin\mK$.
Thus $\De(F)(K^\perp\oti1_\opi)=0$.
This implies that $\mF\pi\subs\mK$.
However, since $\pi\in\mL$, we see $\pi\in\mF\si$
for some $\si\nin\mK$.
Thus $\si\in \ovl{\mF}\pi=\mF\pi$,
and this is a contradiction.
Hence $\pi\nin \mT$
and $|1_\mL|_\vph\leq |1_{\mK\setminus\mT}|_\vph<\de^{1/2}|K|_\vph$.
\end{proof}

\begin{proof}[Proof of Lemma \ref{lem:abc}]
By enlarging $S$ if necessary,
we may assume that the entries of the matrix elements of $E$,
$v$ and $u$ are all in $S$, and that $S$ is $\al^\om$-invariant.
We set the finite set
$\mS:=\mF\cdot\mK$.
Thanks to \cite[Lemma 5.3]{MT1},
we can take a non-zero projection $e$ from $S'\cap M_\om$
such that
\begin{equation}\label{eq:ga-e}
(e\oti1_\rho)\ga_\rho(e)=0
\mbox{ for all }
\rho\in
\ovl{\mS}\cdot\mS\setm\{\btr\}.
\end{equation}
Since $e$ commutes with $v$,
$(e\oti1_\rho)\al_\rho^\om(e)=0$ also holds.
By \cite[Lemma 5.8]{MT1},
we have the equality of projections,
\[
(\id\oti\vph)(\ga_\rho(e))=(\id\oti\vph)(a_\rho^*\al_\rho^\om(e)a_\rho)
\in S'\cap M_\om.
\]
Let $N$ be a von Neumann subalgebra in $M^\om$ that is generated
by $M$ and the matrix entries of $\{\al_\rho^\om(e)\}_{\rho\in\IG}$.
Applying the Fast Reindexation Trick (\cite[Lemma 3.10]{MT1})
for $N$ and $S$,
we obtain a map
$\Ps\in\Mor(\wdt{N}, M^\om\oti \lhG)$
as in \cite[Lemma 3.10]{MT1}.
Set $f:=\Ps(e)$ and then $f\in S'\cap M_\om$.
Since
$(f\oti1_\rho)\al_\rho^\om(f)
=(\Ps\oti\id)((e\oti1_\rho)\al_\rho^\om(e))
=0$ for $\rho\in\ovl{\mS}\cdot\mS\setm\{\btr\}$,
the equality $(f\oti1_\rho)\ga_\rho(f)=0$ holds.
Then by \cite[Lemma 5.8]{MT1}, we have
\[
(\id\oti\vph)(\ga_\rho(f))
=
(\id\oti\vph)(a_\rho^* \al_\rho^\om(f)a_\rho)
\\
=
\Ps\big{(}
(\id\oti\vph)(a_\rho^* \al_\rho^\om(e)a_\rho)\big{)}.
\]
Set a projection $f':=(\id\oti\vph)(\ga_K(f))$
that is contained in $S'\cap M_\om$.
Then $f'$ has the following $\ta^\om$-splitting property
for $x\in S\oti \lhG$:
\begin{align*}
&(\ta^\om\oti\id)
\big{(}x(f'\oti1)\big{)}
\\
&=
\sum_{\rho\in\mK}
(\ta^\om\oti\id)
\big{(}x\big{(}\Ps\big{(}
(\id\oti\vph)(a_\rho^* \al_\rho^\om(e)a_\rho)\big{)}\oti1\big{)}\big{)}
\\
&=
\sum_{\rho\in\mK}
(\ta^\om\oti\id)(x)
\cdot
(\ta^\om\oti\id)
\big{(}
\Ps\big{(}(\id\oti\vph)(a_\rho^* \al_\rho^\om(e)a_\rho)\big{)}\oti1
\big{)}
\\
&=
(\ta^\om\oti\id)(x)
\cdot
\big{(}
\ta_\om(f')\oti1\big{)}.
\end{align*}
By Lemma \ref{lem:ta-ga},
we have
$\ta_\om((\id\oti\vph)(\ga_\rho(f)))=\ta_\om(f)d(\rho)^2$
for $\rho\in\mK$.
This implies
\begin{equation}\label{eq:f}
|f'|_\ps=|f'|_{\ta_\om}=\ta_\om(f)|K|_\vph.
\end{equation}

Now we set a projection $E'\in M_\om\oti \lhG$ defined by
\[
E':=E(f'^\per\oti1)+\ga_K(f).
\]
Then all the conditions from (E.1) to (E.6)
are checked as the proof of \cite[Lemma 5.11]{MT1}.
Thus $\meJ$ contains $E'$.
By the $\ta^\om$-splitting property,
we obtain $b_E'=b_E \ps(f'^\perp)+\ps(f')$.
By our assumption of Lemma \ref{lem:abc},
we obtain
\begin{equation}\label{eq:fb}
\de^{1/4}|f'|_\ps< b_{E'}-b_E.
\end{equation}

Now we prove that $E'$ satisfies the statements (1), (2) and (3)
of this lemma.
The first two are nothing
but \cite[Lemma 5.11 (1),(2)]{MT1}.
Thus we will show the remaining (3) as follows.

Recall the projections $e$, $f$
and the condition (\ref{eq:ga-e}).
The support projection of $\mu_{\ga_K(f)}$ is nothing but
$f'=(\id\oti\vph)(\ga_K(f))$, which commutes with $\mu_E$,
and
\[
\mu_{E'}=\mu_E f'^\perp+\mu_{\ga_K(f)}.
\]
Since $E'(f'^\perp\oti1)=E(f'^\perp\oti1)$
and $f'^\perp\mu_{\ga_K(f)}=0$,
we have
\begin{align}
&
(\ga_F(E')-(\id\oti{}_F\De_K)(E'))(\mu_{E'}^*\oti F\oti K)
\notag\\
&=
(\ga_F(E')-(\id\oti{}_F\De_K)(E'))(\mu_{E'}^*f'^\perp\oti F\oti K)
\notag\\
&\quad
+
(\ga_F(E')-(\id\oti{}_F\De_K)(E'))(\mu_{\ga_K(f)}\oti F\oti K)
\notag\\
&=
(\ga_F(E')-(\id\oti{}_F\De_K)(E))(\mu_{E'}^*f'^\perp\oti F\oti K)
\notag\\
&\quad
+
(\ga_F(E')-(\id\oti{}_F\De_K)(\ga_K(f)))(\mu_{\ga_K(f)}^*\oti F\oti K)
\notag\\
&=
\ga_F(E)(\ga_F(f'^\perp)\oti K-f'^\perp\oti F\oti K)
(f'^\perp \mu_E^*\oti F\oti K)
\notag\\
&\quad+
(\ga_F(E)-(\id\oti{}_F\De_K)(E))
(f'^\perp \mu_E^*\oti F\oti K)
\notag\\
&\quad
+
\ga_F(\ga_K(f))(f'^\perp\mu_E^*\oti F\oti K)
\notag\\
&\quad+
\ga_F(E(f'^\perp\oti1))(\mu_{\ga_K(f)}^*\oti F\oti K)
\notag\\
&\quad+
(\ga_F(\ga_K(f))-(\id\oti{}_F\De_K)(\ga_K(f)))(\mu_{\ga_K(f)}^*\oti
F\oti K).
\label{eq:ga-E-f}
\end{align}

To express these terms in a more simple form, we need the following
claim.
\begin{claim}\label{claim:ga-f-F}
One has the following:
\begin{enumerate}
\item 
$\ga_F(f')(f'^\perp\oti F)
=
(\id\oti\id\oti\vph)((\id\oti{}_F\De_K)(\ga_{K^\perp}(f)))$;
\label{eq:ga-f-f}

\item
$(\ga_F(\ga_K(f))-(\id\oti{}_F\De_K)(\ga_K(f)))
(\mu_{\ga_K(f)}^*\oti
F\oti K)=0$;

\item
$(\ta^\om\oti\id)(\al_F^\om(f')(f'^\perp\oti F))
=\ta_\om(f)\oti(\id\oti\vph)({}_F\De_K(K^\perp))$,
which is contained in $\C\oti Z(\lhG)$.
\label{item:ta-alf}
\end{enumerate}
\end{claim}
\begin{proof}[Proof of Claim \ref{claim:ga-f-F}]
(1). Since $\{\ga_{\rho}(f)\}_{\rho\in\mS}$
is a base of a Rohlin tower along with $\mS$,
and $\mS$ contains $\mF\cdot\mK$,
we have
\begin{align*}
(\id\oti{}_F\De_{K})(\ga(f))(f'^\perp\oti F\oti K)
&=(\id\oti{}_F\De_{K})(\ga(f)(f'^\perp\oti1))
\\
&=(\id\oti{}_F\De_{K})(\ga_{K^\perp}(f)).
\end{align*}

(2).
This is because $f'\mu_{\ga_K(f)}=\mu_{\ga_K(f)}$
and $(\id\oti{}_F\De_K)(\ga_{K^\perp}(f)(f'\oti1))=0$.

(3).
By the proof of (1) and the equalities
$\al^\om=\Ad v\circ\ga$ and $[v,f'\oti1]=0$,
we have
\[
(\id\oti{}_F\De_{K})(\al^\om(f))(f'^\perp\oti F\oti K)
=
(\id\oti{}_F\De_{K})(\al^\om_{K^\perp}(f)).
\]
This implies, with $f'=(\id\oti\vph)(a^*\al_K^\om(f)a)$,
the following:
\begin{align*}
&\al_F^\om(f')(f'^\perp\oti F)
\\
&=
(\id\oti\id\oti\vph)
\left(
\al_F(a_K^*)\al_F^\om(\al_K^\om(f))\al_F(a_K)
\right)(f'^\perp\oti F)
\\
&=
(\id\oti\id\oti\vph)
\left(
\al_F(a_K^*)u_{F,K}(\id\oti{}_F\De_K)(\al^\om(f))u_{F,K}^*\al_F(a_K)
\right)(f'^\perp\oti F)
\\
&=
(\id\oti\id\oti\vph)
\left(
\al_F(a_K^*)u_{F,K}(\id\oti{}_F\De_K)(\al_{K^\perp}^\om(f))u_{F,K}^*
\al_F(a_K)
\right).
\end{align*}
Thus we obtain
\begin{align*}
&(\ta^\om\oti\id)(\al_F^\om(f')(f'^\perp\oti F))
\\
&=
\ta_\om(f)
(\id\oti\id\oti\vph)
\left(
\al_F(a_K^*)u_{F,K}(1\oti{}_F\De_K(K^\perp))u_{F,K}^*\al_F(a_K)
\right).
\end{align*}

Using $\vph(\cdot)=\sum_{\rho}d(\rho)^2\om_\rho(\cdot\oti1_\orho)$
and $(a\oti1)T_{\rho,\orho}=uT_{\rho,\orho}$,
where $\om_\rho(x):=T_{\rho,\orho}^* x T_{\rho,\orho}$
for $x\in B(H_\rho)\oti B(H_\orho)$,
we have
\begin{align*}
&(\ta^\om\oti\id)(\al_F^\om(f')(f'^\perp\oti F))
\\
&=
\ta_\om(f)
\sum_{\rho\in\mK}
d(\rho)^2
(\id\oti\id\oti\om_\rho)
\left(
\al_F(a_\rho^*)u_{F,\rho}
(1\oti{}_F\De_\rho(K^\perp))
u_{F,\rho}^*
\al_F(a_\rho)
\right)
\\
&=
\ta_\om(f)
\sum_{\rho\in\mK}
d(\rho)^2
(\id\oti\id\oti\om_\rho)
\left(
\al_F(u_{\rho,\rho}^*)u_{F,\rho}
(1\oti{}_F\De_\rho(K^\perp))
u_{F,\rho}^*
\al_F(u_{\rho,\rho})
\right),
\end{align*}
and from the equalities below
\begin{align*}
&u(\id\oti\De\oti\id)(u)=\al(u)(\id\oti\id\oti\De)(u),
\\
&
(\id\oti\id_F\oti{}_\rho\De_\orho)(u^*)(1\oti F\oti T_{\rho,\orho})
=(1\oti F\oti T_{\rho,\orho}),
\\
&
[(\id\oti{}_F\De_\rho\oti\id_\orho)(u),1\oti{}_F\De_\rho(K^\perp)\oti1_\orho]
=0,
\end{align*}
we get
\begin{align*}
&(\ta^\om\oti\id)(\al_F^\om(f')(f'^\perp\oti F))
\\
&=
\ta_\om(f)
\sum_{\rho\in\mK}
d(\rho)^2
(\id\oti\id\oti\om_\rho)
\left(
1\oti{}_F\De_\rho(K^\perp)
\right)
\\
&=
\ta_\om(f)\oti(\id\oti\vph)({}_F\De_K(K^\perp)).
\end{align*}
\end{proof}

Hence (\ref{eq:ga-E-f}) is simply rewritten as follows
\begin{align*}
&(\ga_F(E')-(\id\oti{}_F\De_K)(E'))(\mu_{E'}^*\oti F\oti K)
\\
&=
\ga_F(E)(\ga_F(f'^\perp)-f'^\perp\oti F\oti K)
(f^\perp\mu_E^*\oti F\oti K)
\\
&\quad+
(\ga_F(E)-(\id\oti{}_F\De_K)(E))
(f'^\perp \mu_E^*\oti F\oti K)
\\
&\quad
+
\ga_F(\ga_K(f))(f'^\perp\mu_E^*\oti F\oti K)
\\
&\quad+
\ga_F(E(f'^\perp\oti1))(\mu_{\ga_K(f)}^*\oti F\oti K).
\end{align*}
If we put
\begin{align*}
A&:=
g_F^*v_F\ga_F(vE)
(\ga_F(f'^\perp)\oti K- f'^\perp\oti F\oti K)
(f'^\perp\mu_E^*\oti F\oti K)g_F,
\\
B&:=
g_F^*v_F\ga_F(v)(\ga_F(E)-(\id\oti{}_F\De_K)(E))
(f'^\perp \mu_E^*\oti F\oti K)
g_F,
\\
C&:=
g_F^*v_F\ga_F(v)
\ga_F(\ga_K(f))(f'^\perp\mu_E^*\oti F\oti K)
g_F,
\\
D&:=
g_F^*v_F\ga_F(vE(f'^\perp\oti1))(\mu_{\ga_K(f)}^*\oti F\oti K)g_F,
\end{align*}
then
\begin{equation}\label{eq:c-E}
c_{E'}^g
=(\ps\oti\vph\oti\vph)(A+B+C+D).
\end{equation}

Using $[E,f'\oti1]=0=[v,f'\oti1]$,
we have
\begin{align*}
A
&=g_F^*v_F
(\ga_F(f'^\perp)\oti K-f'^\perp\oti F\oti K)
(f'^\perp\oti F\oti K)\ga_F(vE)(\mu_E^*\oti F\oti K)g_F
\\
&=
g_F^*
(\al_F^\om(f'^\perp)\oti K-f'^\perp\oti F\oti K)
(f'^\perp\oti F\oti K)v_F\ga_F(vE)(\mu_E^*\oti F\oti K)g_F
\\
&=
g_F^*
(\al_F^\om(f')\oti K-f'\oti F\oti K)
(f'^\perp\oti F\oti K)v_F\ga_F(vE)(\mu_E^*\oti F\oti K)g_F
\\
&=
-g_F^*
(\al_F^\om(f')\oti K)
(f'^\perp\oti F\oti K)v_F\ga_F(vE)(\mu_E^*\oti F\oti K)g_F.
\end{align*}
Put $h_F:=(\id\oti\vph)({}_F\De_K(K^\perp))$
that is central in $\lhG$.
The $\ta^\om$-splitting property of $f$
and Claim \ref{claim:ga-f-F} (\ref{item:ta-alf}) yield the following:
\begin{align*}
&-(\ta^\om\oti\id\oti\id)(A)
\\
&=
\ta_\om(f)
g_F^*
(1\oti h_F\oti K)
\cdot(\ta^\om\oti\id\oti\id)(v_F\ga_F(vE)( \mu_E^*\oti F\oti K))g_F
\\
&=
\ta_\om(f)
(\ta^\om\oti\id\oti\id)\left(
(1\oti h_F^{1/2}\oti K)
g_F^*
v_F\ga_F(vE)\cdot\ga_F(E)( \mu_E^*\oti h_F^{1/2}\oti K)g_F
\right).
\end{align*}
Then we can estimate the first term of (\ref{eq:c-E})
by using the Cauchy--Schwarz inequality as follows:
\begin{align*}
&|(\ps\oti\vph\oti\vph)(A)|
\notag\\
&\leq
\ta_\om(f)
(\ps\oti\vph\oti\vph)
\left(
(1\oti h_F\oti K)
g_F^*v_F\ga_F(vEv^*)v_F^*g_F
\right)^{1/2}
\notag
\\
&\quad
\cdot
(\ps\oti\vph\oti\vph)
\left(g_F^*(\mu_E\oti h_F\oti K)\ga_F(E)(\mu_E^*\oti F\oti K)g_F
\right)^{1/2}
\notag\\
&\leq
\ta_\om(f)
(\ps\oti\vph\oti\vph)
\left(
(1\oti h_F\oti K)
g_F^*v_F\ga_F(vEv^*)v_F^*g_F
\right)^{1/2}
\\
&\quad
\cdot
(\ps\oti\vph)
\left(g_F^*(\mu_E\oti h_F)\ga_F((\id\oti\vph)(E))
(\mu_E^*\oti F)g_F
\right)^{1/2}
\notag
\\
&=
\ta_\om(f)
(\ps\oti\vph\oti\vph)
\left(
(1\oti h_F\oti K)
g_F^*\al_F^\om(vEv^*)g_F
\right)^{1/2}
\\
&\quad
\cdot
(\ps\oti\vph)
\left(g_F^*(\mu_E\oti h_F)\ga_F((\id\oti\vph)(E))
(\mu_E^*\oti F)g_F
\right)^{1/2}
\end{align*}
Using (E.6) and (\ref{eq:taEK}),
we have
$(\ta^\om\oti\id)(\al_F^\om(vEv^*))=|K|_\vph^{-1}b_E (F\oti K)$.
Note that $(\id\oti\vph)(E)$ and $\mu_E\mu_E^*$ are projections.
Then
\begin{align*}
&|(\ps\oti\vph\oti\vph)(A)|
\\
&\leq
\ta_\om(f)
\|K\|_\vph^{-1}b_E^{1/2}
(\ps\oti\vph\oti\vph)(1\oti h_F\oti K)^{1/2}
\cdot
(\ps\oti\vph)
\left(g_F^*(\mu_E\mu_E^*\oti h_F)g_F
\right)^{1/2}
\\
&\leq
\ta_\om(f)
\|K\|_\vph^{-1}b_E^{1/2}
\vph(h_F)^{1/2}\|K\|_\vph
\cdot
\vph(h_F)^{1/2}
\\
&=
\ta_\om(f)b_E^{1/2}\vph(h_F).
\end{align*}
Using (\ref{eq:FDeK})
and (\ref{eq:f}),
we have $\vph(h_F)<\de|F|_\vph|K|_\vph$
and
\begin{equation}
\label{eq:A0}
|(\ps\oti\vph\oti\vph)(A)|
\leq
\ta_\om(f)\de|F|_\vph|K|_\vph
=\de|F|_\vph\ta_\om(f').
\end{equation}

In the second term of (\ref{eq:c-E}),
by the $\ta^\om$-splitting property of $f'^\perp$,
we have
\begin{equation}\label{eq:B-c-E}
(\ps\oti\vph\oti\vph)(B)
=
\ta_\om(f'^\perp)c_E^g.
\end{equation}

In the third term of (\ref{eq:c-E}),
we have
\begin{align*}
&(\ps\oti\vph\oti\vph)(C)
\\
&=
(\ps\oti\vph\oti\vph)
(g_F^*v_F\ga_F(v)\ga_F(\ga_K(f))(\mu_E^*f'^\perp\oti F\oti K)g_F)
\\
&=
(\ps\oti\vph\oti\vph)(g_F^*\al_F^\om(\al_K^\om(f))v_F\ga_F(v)
(\mu_E^*f'^\perp\oti F\oti K)g_F)
\\
&=
(\ps\oti\vph\oti\vph)
(g_F^*\al_F^\om(\al_K^\om(f))
(f'^\perp\oti F\oti K)
\cdot
v_F\ga_F(v)(\mu_E^*\oti F\oti K)g_F).
\end{align*}
By the $\ta^\om$-splitting property
of $\al_F^\om(\al_K^\om(f))(f'^\perp\oti F\oti K)$,
we have
\begin{align}
&(\ps\oti\vph\oti\vph)(C)
\notag\\
&=
(\phi\oti\vph\oti\vph)
\big{(}
g_F^*(\ta^\om\oti\id\oti\id)(\al_F^\om(\al_K^\om(f))(f'^\perp\oti F\oti K))
\notag\\
&\qquad\cdot
(\ta^\om\oti\id\oti\id)(v_F\ga_F(v)(\mu_E^*\oti F\oti K)g_F)\big{)}.
\label{eq:Cmu}
\end{align}

Since
\begin{align*}
&(\ta^\om\oti\id\oti\id)(\al_F^\om(\al_K^\om(f))(f'^\perp\oti F\oti K))
\\
&=
(\ta^\om\oti\id\oti\id)(u_{F,K}(\id\oti{}_F\De_K)(\al^\om(f))u_{F,K}^*
(f'^\perp\oti F\oti K))
\\
&=
(\ta^\om\oti\id\oti\id)
(u_{F,K}(\id\oti{}_F\De_K)(\al_{K^\perp}^\om(f))u_{F,K}^*)
\\
&=
\ta_\om(f)u_{F,K}(1\oti{}_F\De_K(K^\perp))u_{F,K}^*,
\end{align*}
we have
\begin{align*}
(\ref{eq:Cmu})
&=
\ta_\om(f)
(\phi\oti\vph\oti\vph)
\big{(}
g_F^*u_{F,K}(1\oti{}_F\De_K(K^\perp))u_{F,K}^*
\notag\\
&\hspace{50pt}\cdot
(\ta^\om\oti\id\oti\id)(v_F\ga_F(v)(\mu_E^*\oti F\oti K)g_F)\big{)}
\\
&=
\ta_\om(f)
(\ps\oti\vph\oti\vph)
\big{(}
g_F^*u_{F,K}(1\oti{}_F\De_K(K^\perp))
\notag\\
&\hspace{50pt}\cdot
u_{F,K}^*v_F\ga_F(v)(\mu_E^*\oti F\oti K)g_F
\big{)}
\\
&=
\ta_\om(f)
(\ps\oti\vph\oti\vph)
\big{(}
g_F^*u_{F,K}(1\oti{}_F\De_K(K^\perp))
\notag\\
&\hspace{50pt}\cdot
(\id\oti{}_F\De_K)(v)(\mu_E^*\oti F\oti K)g_F\big{)}
\\
&=
\ta_\om(f)
(\ps\oti\vph\oti\vph)
\big{(}
g_F^*u_{F,K}\cdot
(\id\oti{}_F\De_K)(v_{K^\perp})(\mu_E^*\oti F\oti K)g_F\big{)}.
\end{align*}
Thus
\begin{align*}
&|(\ps\oti\vph\oti\vph)(C)|
\\
&\leq
\ta_\om(f)
(\ps\oti\vph\oti\vph)
\big{(}g_F^*u_{F,K}u_{F,K}^*g_F\big{)}^{1/2}
\\
&\hspace{30pt}
\cdot
(\ps\oti\vph\oti\vph)
\big{(}
g_F^*(\mu_E^*\oti F\oti K)
(\id\oti{}_F\De_K)(v_{K^\perp}^*v_{K^\perp})
(\mu_E^*\oti F\oti K)g_F\big{)}^{1/2}
\\
&=
\ta_\om(f)
(\ps\oti\vph\oti\vph)(1\oti F\oti K)^{1/2}
\cdot
(\ps\oti\vph\oti\vph)
(
g_F^*
(\mu_E^*\mu_E\oti{}_F\De_K(K^\perp))g_F)^{1/2}
\\
&\leq
\ta_\om(f)
\|F\|_\vph\|K\|_\vph
\cdot
(\ps\oti\vph\oti\vph)
(
g_F^*
(1\oti{}_F\De_K(K^\perp))g_F)^{1/2}
\\
&=
\ta_\om(f)
\|F\|_\vph\|K\|_\vph
(\ps\oti\vph)(g_F^* (1\oti h_F) g_F)^{1/2}
\\
&=
\ta_\om(f)
\|F\|_\vph\|K\|_\vph \vph(h_F)^{1/2}
\\
&<
\ta_\om(f)
\de^{1/2}|F|_\vph|K|_\vph,
\end{align*}
and
\begin{equation}
\label{eq:psCF}
|(\ps\oti\vph\oti\vph)(C)|
\leq
\de^{1/2}|F|_\vph\ta_\om(f').
\end{equation}

Next we estimate the last term of (\ref{eq:c-E}).
We set
\begin{align*}
D_1&:=
\ga_F(E(f'^\perp\oti1)v^*)(f'\oti F\oti K)v_F^*g_F,
\\
D_2&:=
\ga_F(E(f'^\perp\oti1))
(\mu_{\ga_K(f)}^*\oti F\oti K)g_F.
\end{align*}
By Claim \ref{claim:ga-f-F} (1),
we see that
$\ga_F(f'^\perp)$ and $f'\oti F$
are commuting.
Hence we have $D=D_1^*D_2$,
and by applying the Cauchy--Schwarz inequality,
we get
\begin{equation}\label{eq:D12}
|(\ps\oti\vph\oti\vph)(D)|
\leq
(\ps\oti\vph\oti\vph)(D_1^*D_1)^{1/2}
(\ps\oti\vph\oti\vph)(D_2^*D_2)^{1/2}.
\end{equation}
Then we have
\begin{align}
&(\ps\oti\vph\oti\vph)(D_1^*D_1)
\notag\\
&=
(\ps\oti\vph\oti\vph)(
g_F^*v_F\ga_F(vE(f'^\perp\oti1)v^*)(f'\oti F\oti K)v_F^*g_F)
\notag\\
&=
(\ps\oti\vph\oti\vph)(
g_F^*v_F\ga_F(vEv^*)v_F^*
(\al_F^\om(f'^\perp)(f'\oti F)\oti K)g_F)
\notag\\
&=
(\ps\oti\vph\oti\vph)(
g_F^*\al_F^{\om}(vEv^*)
(\al_F^\om(f'^\perp)(f'\oti F)\oti K)g_F)
\notag\\
&=
(\ph\oti\vph\oti\vph)
\big{(}
g_F^*
(\ta^\om\oti\id\oti\id)(\al_F^{\om}(vEv^*))
\notag\\
&\hspace{90pt}\cdot
(\ta^\om\oti\id\oti\id)(\al_F^\om(f'^\perp)
(f'\oti F\oti K)g_F)
\big{)}
\quad \mbox{by $\ta^\om$-splitting}
\notag
\\
&=
b_{E}|K|_\vph^{-1}
(\ph\oti\vph\oti\vph)
\left(
g_F^*
(\ta^\om\oti\id\oti\id)(\al_F^\om(f'^\perp)(f'\oti F\oti K)
g_F)
\right)
\notag\\
&=
b_{E}
(\ph\oti\vph)
\left(
g_F^*
(\ta^\om\oti\id)(\al_F^\om(f'^\perp)(f'\oti F))
g_F
\right).
\label{eq:D1-g-a}
\end{align}

To continue our estimate further,
we need the following claim.
\begin{claim}\label{claim:alf}
One has the following:
%
\[
(\ta^\om\oti\id)(\al_F^\om(f'^\perp)(f'\oti F))
=
\ta_\om(f)\oti h_F.
\]
\end{claim}
\begin{proof}[Proof of Claim \ref{claim:alf}]
Note that
\[
\al_F^\om(f'^\perp)(f'\oti F)
=
f'\oti F+\al_F^\om(f')(f'^\perp\oti F)-\al_F^\om(f').
\]
By Claim \ref{claim:ga-f-F} (\ref{item:ta-alf}),
we have
\begin{align*}
&(\ta^\om\oti\id)(\al_F^\om(f'^\perp)(f'\oti F))
\\
&=
\ta_\om(f')\oti F
+(\ta^\om\oti\id)(\al_F^\om(f')(f'^\perp\oti F))
-\ta_\om(f')\oti F
\\
&=
\ta_\om(f)\oti (\id\oti\vph)({}_F\De_K(K^\perp)).
\end{align*}
\end{proof}
Then by (\ref{eq:D1-g-a}) and the previous claim,
the following holds:
\begin{align}
(\ps\oti\vph\oti\vph)(D_1^*D_1)
&=
b_{E}\ta_\om(f)
(\ph\oti\vph)(g_F^* (1\oti h_F)g_F)
\notag\\
&=
b_{E}\ta_\om(f)
(\ph\oti\vph)(1\oti h_F)
\quad\mbox{by } h_F\in Z(L^\infty(\bhG))
\notag\\
&<
b_{E}\ta_\om(f)\de|F|_\vph|K|_\vph
\notag\\
&\leq
\de|F|_\vph\ta_\om(f')
\quad\mbox{by }(\ref{eq:f}).
\label{eq:D1D1}
\end{align}

Next we estimate
$(\ps\oti\vph\oti\vph)(D_2^*D_2)$ as follows.
\begin{align*}
&(\ps\oti\vph\oti\vph)(D_2^*D_2)\\
&=
(\ps\oti\vph\oti\vph)
(
g_F^*(\mu_{\ga_K(f)}\oti F\oti K)
\ga_F(E(f'^\perp\oti1))
(\mu_{\ga_K(f)}^*\oti F\oti K)g_F)
\\
&=
(\ps\oti\vph)
(
g_F^*(\mu_{\ga_K(f)}\oti F)
\ga_F((\id\oti\vph)(E)f'^\perp)
(\mu_{\ga_K(f)}^*\oti F)g_F)
\\
&\leq
(\ps\oti\vph)
(
g_F^*(\mu_{\ga_K(f)}\oti F)
\ga_F(f'^\perp)
(\mu_{\ga_K(f)}^*\oti F)g_F)
\\
&
=(\ps\oti\vph)
(\mu_{\ga_K(f)}\mu_{\ga_K(f)}^*\oti F)
-
(\ps\oti\vph)
(
g_F^*(\mu_{\ga_K(f)}\oti F)
\ga_F(f')
(\mu_{\ga_K(f)}^*\oti F)g_F).
\end{align*}
By $\mu_{\ga_K(f)}\mu_{\ga_K(f)}^*=(\id\oti\vph)(\ga_K(f))=f'$,
we have
\[
(\ps\oti\vph)
(\mu_{\ga_K(f)}\mu_{\ga_K(f)}^*\oti F)
=|F|_\vph\ta_\om(f')
=|F|_\vph|K|_\vph\ta_\om(f).
\]
We compute the last term as follows:
\begin{align*}
&(\mu_{\ga_K(f)}\oti F)\ga_F(f')
(\mu_{\ga_K(f)}^*\oti F)
\\
&=
(\id\oti\id\oti\vph)
\left(
(\id\oti{}_F\De_K)\big{(}
(\mu_{\ga_K(f)}^*\oti1)\ga(f)(\mu_{\ga_K(f)}^*\oti1)
\big{)}
\right)
\\
&=
(\id\oti\id\oti\vph)
\left(
(\id\oti{}_F\De_K)\big{(}
(\mu_{\ga_K(f)}^*\oti1)\ga_K(f)(\mu_{\ga_K(f)}^*\oti1)
\big{)}
\right)
\\
&=
(\id\oti\id\oti\vph)
\left(
(\id\oti{}_F\De_K)
(a^*v\ga_K(f)v^*a)
\right)
\quad\mbox{by }(\ref{eq:Emu})
\\
&=(\id\oti\id\oti\vph)
\left(
(\id\oti{}_F\De_K)(a^*\al_K^{\om}(f)a)
\right),
\end{align*}
and the following holds:
\[
(\ps\oti\vph)
(g_F^*(\mu_{\ga_K(f)}\oti F)
\ga_F(f')
(\mu_{\ga_K(f)}^*\oti F)g_F)
=
\ta_\om(f)(\vph\oti\vph)({}_F\De_K(a^*a_K)).
\]
Then we have
\begin{align}
(\ps\oti\vph\oti\vph)(D_2^*D_2)
&\leq
\ta_\om(f)|F|_\vph|K|_\vph
-
\ta_\om(f)(\vph\oti\vph)({}_F\De_K(a^*a_K))
\notag\\
&=
\ta_\om(f)(\vph\oti\vph)({}_F\De_{K^\perp}(a^*a_K))
\quad\mbox{by }(\ref{eq:aa})
\notag\\
&=
\ta_\om(f)(\vph\oti\vph)({}_F\De_{K^\perp}(a^*a(1\oti 1_{\mL})))
\notag\\
&\leq
\ta_\om(f)(\vph\oti\vph)({}_F\De(a^*a(1\oti 1_{\mL})))
\notag\\
&=
\ta_\om(f)\vph(1_{\mL})|F|_\vph
\quad\mbox{by }(\ref{eq:aa})
\notag\\
&<
\de^{1/2}
|F|_\vph\ta_\om(f')
\quad\mbox{by Lemma \ref{lem:L}}.
\label{eq:D2D2}
\end{align}

Hence by (\ref{eq:D12}),
(\ref{eq:D1D1}) and (\ref{eq:D2D2}),
we have
\begin{equation}\label{eq:D0}
|(\ps\oti\vph\oti\vph)(D)|<\de^{3/4}|F|_\vph\ta_\om(f').
\end{equation}

Finally from (\ref{eq:c-E}), (\ref{eq:A0}),
(\ref{eq:B-c-E}), (\ref{eq:psCF})
and (\ref{eq:D0}),
we obtain
\begin{align*}
|c_{E'}^g|-|c_E^g|
&\leq
|(\ps\oti\vph\oti\vph)(A)|
+
|(\ps\oti\vph\oti\vph)(B)|
+
|(\ps\oti\vph\oti\vph)(C)|
\\
&\quad
+
|(\ps\oti\vph\oti\vph)(D)|
-|c_E^g|
\\
&<
\de|F|_\vph\ta_\om(f')
-\ta_\om(f')|c_E^g|
+
\de^{1/2}|F|_\vph\ta_\om(f')
+
\de^{3/4}|F|_\vph\ta_\om(f')
\\
&<3\de^{1/2}|F|_\vph\ta_\om(f')
\\
&<
3\de^{1/4}|F|_\vph(b_{E'}-b_E)
\quad\mbox{by }(\ref{eq:fb}).
\end{align*}
\end{proof}

\begin{thm}[Rohlin type theorem]\label{thm:Rohlin}
\index{Rohlin type theorem}
There exist a projection $\ovl{E}\in M_\om\oti\lhG$
with (E.\ref{item:E-K}), (E.\ref{item:ga-E}),
(E.\ref{item:ga-avE}), (E.\ref{item:EvaavE}),
(E.\ref{item:vph-E})
and (E.\ref{item:ga-taE}),
and a projection
$p\in S'\cap M_\om$
such that the projection $E:=\ovl{E}+p\oti e_\btr$ satisfies
the following:
\begin{itemize}
\item
$a_E\leq 5\de^{1/4}$;
\item
$b_E=1$ and $\ta_\om(p)<\de^{1/4}$;
\item
$|c_E^g|\leq9\de^{1/8}|F|_\vph$ for all $g\in\mG$.
\end{itemize}
\end{thm}
\begin{proof}
We introduce the subset $\meS$ in $\meJ$
which is the collection of $E\in\meJ$ with
$a_E\leq3\de^{1/4}b_E$ and $\max_{g\in\mG}|c_E^g|\leq3\de^{1/4}|F|_\vph b_E$.
The order on $\meS$ is given by
$E\prec E'$ if and only if $E=E'$ or
all the inequalities of Lemma \ref{lem:abc} hold.
Then $\meS$ is proved to be inductive as \cite[Theorem 5.9]{MT1},
and we can take a maximal element $\ovl{E}$
in the ordered set $\meS$.
We set $p:=1-(\id\oti\vph)(E)$
and $E=\ovl{E}+p\oti e_\btr$.
Then it is trivial that $b_E=1$.
If $\ta_\om(p)\geq \de^{1/4}$,
then we can take $E'\in\meJ$ as in Lemma \ref{lem:abc}.
Then it is easy to see that $E'$ is an element of $\meS$
with $E\prec E'$,
and this is a contradiction.
Thus $\ta_\om(p)<\de^{1/4}$.
The inequality on $a_E$ is shown
in the same way as \cite[Theorem 5.9]{MT1}.
We will estimate $c_E^g$.
Since $\mu_E=\mu_{\ovl{E}}+p$, we have
\begin{align*}
&
(\ga_F(E)-(\id\oti{}_F\De_K)(E))(\mu_E^*\oti F\oti K)
\\
&=
(\ga_F(\ovl{E})-(\id\oti{}_F\De_K)(\ovl{E}))(\mu_{\ovl{E}}^*\oti F\oti K)
\\
&\quad+
(\ga_F(\ovl{E})-(\id\oti{}_F\De_K)(\ovl{E}))(p\oti F\oti K)
\\
&\quad+
(\ga_F(p)\oti e_\btr-p\oti{}_F\De_K(e_\btr))((\mu_{\ovl{E}}^*+p)\oti F\oti K)
\\
&=
(\ga_F(\ovl{E})-(\id\oti{}_F\De_K)(\ovl{E}))(\mu_{\ovl{E}}^*\oti F\oti K)
\\
&\quad+
\ga_F(\ovl{E})(p\oti F\oti K)
\\
&\quad+
(\ga_F(p)(\mu_{\ovl{E}}^*\oti 1))\oti e_\btr
+
(\ga_F(p)(p\oti 1))\oti e_{\btr}
-
p\oti{}_F\De_K(e_\btr).
\end{align*}
Thus we have
\begin{align*}
c_E^g
&=
c_{\ovl{E}}^g
+
(\ps\oti\vph\oti\vph)(g_F^*v_F\ga_F(v\ovl{E})(p\oti F\oti K)g_F)
\\
&\quad+
(\ps\oti\vph\oti\vph)(g_F^*v_F\ga_F(v(p\oti1))(\mu_{\ovl{E}}^*\oti F\oti e_\btr)g_F)
\\
&\quad+
(\ps\oti\vph\oti\vph)(g_F^*v_F\ga_F(v(p\oti1))(p\oti F\oti e_\btr)g_F)
\\
&\quad-
(\ps\oti\vph\oti\vph)(g_F^*v_F\ga_F(v)(p\oti {}_F\De_K(e_\btr))g_F).
\end{align*}
We estimate these five terms.
Since $\ovl{E}\in\meJ$,
we have $c_{\ovl{E}}^g\leq 3\de^{1/4}|F|_\vph b_{\ovl{E}}$.
Next, we have the following four estimates:
\begin{align*}
&|(\ps\oti\vph\oti\vph)(g_F^*v_F\ga_F(v\ovl{E})(p\oti F\oti K)g_F)|
\\
&\leq
|F|_\vph
\|v_F\ga_F(v)\|
|\ga_F(v\ovl{E})(p\oti F\oti K)|_{\ps\oti\vph\oti\vph}
\quad\mbox{by Lemma }\ref{lem:g-ps-vph}
\\
&=
|F|_\vph
|(\ga_F(v\ovl{E})-(\id\oti{}_F\De_K)(\ovl{E}))
(p\oti F\oti K)|_{\ps\oti\vph\oti\vph}
\\
&\leq
|F|_\vph
|\ga_F(v\ovl{E})-(\id\oti{}_F\De_K)(\ovl{E})|_{\ps\oti\vph\oti\vph}
\\
&<
3\de^{1/2}|F|_\vph b_{\ovl{E}},
\end{align*}
\begin{align*}
&
|(\ps\oti\vph\oti\vph)(g_F^*v_F\ga_F(v(p\oti1))
(\mu_{\ovl{E}}^*\oti F\oti e_\btr)g_F)|
\\
&=
|(\ps\oti\vph\oti\vph)(g_F^*v_F(\ga_F(p)\oti e_\btr)
(\mu_{\ovl{E}}^*\oti F\oti e_\btr)g_F)|
\\
&=
|(\ps\oti\vph)(g_F^*v_F\ga_F(p)
(\mu_{\ovl{E}}^*\oti F)g_F)|
\\
&\leq
(\ps\oti\vph)(g_F^*v_F\ga_F(p)v_F^*g_F)^{1/2}
(\ps\oti\vph)(g_F^*
(\mu_{\ovl{E}}\mu_{\ovl{E}}^*\oti F)g_F)^{1/2}
\\
&=
(\ps\oti\vph)(g_F^*\al^{\om}_F(p)g_F)^{1/2}
(\ps\oti\vph)(g_F^*((1-p)\oti F)g_F)^{1/2}
\\
&=
\ta_\om(p)^{1/2}|F|_\vph^{1/2}
\cdot \ta_\om(1-p)^{1/2}|F|_\vph^{1/2}
\\
&\leq
\ta_\om(p)^{1/2}|F|_\vph,
\end{align*}
\begin{align*}
&|(\ps\oti\vph\oti\vph)(g_F^*v_F\ga_F(v(p\oti1))
(p\oti F\oti e_\btr)g_F)|
\\
&=
|(\ps\oti\vph)(g_F^*v_F\ga_F(p)(p\oti F)g_F)|
\\
&\leq
(\ps\oti\vph)(g_F^*v_F\ga_F(p)v_F^*g_F)^{1/2}
(\ps\oti\vph)(g_F^*(p\oti F)g_F)^{1/2}
\\
&\leq
(\ps\oti\vph)(g_F^*\al_F^{\om}(p)g_F)^{1/2}
\ta_\om(p)^{1/2}|F|_\vph^{1/2}
\\
&=
\ta_\om(p)|F|_\vph
\end{align*}
and
\begin{align*}
&|(\ps\oti\vph\oti\vph)(g_F^*v_F\ga_F(v)
(p\oti {}_F\De_K(e_\btr))g_F)|
\\
&=
\ta_\om(p)
|(\ps\oti\vph\oti\vph)(g_F^*v_F\ga_F(v)
(1\oti {}_F\De_K(e_\btr))g_F)|
\\
&\leq
\ta_\om(p)
\|v_F\ga_F(v)\|
|1\oti {}_F\De_K(e_\btr)|_{\ps\oti\vph\oti\vph}
\\
&\leq
\ta_\om(p)|F|_\vph^2.
\end{align*}
Since $\de^{1/4}|F|_\vph<1$,
we obtain
\[
|c_E^g|<|F|_\vph
(
3\de^{1/4}b_{\ovl{E}}+3\de^{1/2}b_{\ovl{E}}
+\ta_\om(p)^{1/2}+2\ta_\om(p)|F|_\vph)
<9\de^{1/8}|F|_\vph.
\]
\end{proof}

\subsection{Almost vanishing theorem of approximate 1-cocycles}
Recall that we have treated a cocycle action $(\al,u)$
and an action $\ga=\Ad v\circ\al^\om$.
Let us introduce the following functions:
for $\rho\in\IG$ and $g\in\mG$,
\begin{align*}
&f_1(\al,u;\rho,g)
:=
|K|_\vph
|(\ph\oti\vph_\rho)(g^* (\Ph_\rho^{\al\th^{-1}}(a_\rho-1)\oti1_\rho)g)|,
\\
&f_2(\al,u;\rho,g)
:=
|K|_\vph|(\ph\oti\vph)(g_F^* (\Ph_\rho^{\al\th^{-1}}(a_\rho-1)\oti F)g_F)|
\\
&f_3(\al,u;\rho,g)
\\
&
:=
|K|_\vph
\sum_{S,T}
|(\ps\oti\vph\oti\vph)
(g_{12}^*(1\oti S) \Ph_\rho^\al((1\oti S^*)|u-1|^2(1\oti T))
(1\oti T^*)g_{12})|,
\\
&f_4(\al,u;g)
:=
|K|_\vph\|(u-1)(1\oti (F\vee K)\oti \ovl{K})\|_{g(\ps\oti\vph)g^*\oti\vph}^2,
\\
&f_5(\al,u;g)
:=
|(\ps\oti\vph\oti\vph)
(g_F^*(1-u(\id\oti{}_F\De_K)(a_K))g_F)|^2,
\\
&f_6(\al,u;g)
:=
|(\ps\oti\vph\oti\vph)
(g_F^*(a_F-1)(1\oti{}_F\De_{K}(e_\btr))g_F)|^2,
\\
&f_7(\al,u;g)
:=
(\ph\oti\vph\oti\vph)(g_F^*\al_F(|a_K-1\oti K|^2)g_F),
\end{align*}
where $S,T$ run through $\ONB(\rho,\pi\si)$
for $\pi\in\mF,\si\in\mK$.
Now we denote by $\ka(\al,u)$ the maximum
of those seven functions $f_k$
for $\pi\in\mF\cup\mK,\rho\in \mK\cup\ovl{\mK},g\in\mG$.
When we want to specify $F$ and $K$,
we write $\ka(\al,u;F,K,\mG)$ for $\ka(\al,u)$.

We will prove some inequalities which will be used
in the proof of almost vanishing of the approximate
1-cocycle $v$.
\begin{lem}\label{lem:Eva1}
The following inequality holds:
\[
|(\ps\oti\vph)(g^*Ev^*(a-1)vEg)|
\leq
\ka(\al,u)|K|_\vph^{-1}.
\]
\end{lem}
\begin{proof}
The property (E.\ref{item:ga-E}) yields
$E_{\rho_{\el,m}}E_{\rho_{j,k}}=d(\rho)^{-1}\de_{m,j}E_{\rho_{\el,k}}$,
and we have
\begin{align*}
(\ps\oti\vph)(g^*Ev^*avEg)
&=
\sum_{\rho\in\mK}
\sum_{i,j,k,\el,m}
d(\rho)\ps((g^*)_{\rho_{i,j}}E_{\rho_{j,k}}(v^*av)_{\rho_{k,\el}}E_{\rho_{\el,m}} g_{\rho_{m,i}})
\\
&=
\sum_{\rho\in\mK}
\sum_{i,j,k,\el,m}
d(\rho)\ps((g^*)_{\rho_{i,j}}(v^*av)_{\rho_{k,\el}}E_{\rho_{\el,m}}E_{\rho_{j,k}} g_{\rho_{m,i}})
\\
&=
\sum_{\rho\in\mK}
\sum_{i,j,k,\el,m}
d(\rho)d(\rho)^{-1}\ps((g^*)_{\rho_{i,j}}(v^*av)_{\rho_{k,\el}}E_{\rho_{\el,k}} g_{\rho_{j,i}})
\\
&=
\sum_{\rho\in\mK}
\sum_{i,j,k,\el,m}
\ps((g^*)_{\rho_{i,j}}(\id\oti\Tr_\rho)(v^*avE) g_{\rho_{j,i}})
\\
&=
\sum_{\rho\in\mK}
\sum_{i,j,k,\el,m}
d(\rho)\ph((g^*)_{\rho_{i,j}}(\ta^\om\oti\tr_\rho)(v^*avE)
g_{\rho_{j,i}}).
\end{align*}
By Theorem \ref{thm:app-cocycle},
we have
\[
(\ta^\om\oti\tr_\rho)(v^*avE)=\Ph_\rho^{\al\th^{-1}}((\ta^\om\oti\id)(avEv^*)).
\]
This implies the following:
\[
(\ps\oti\vph)(g^*Ev^*avEg)
=
\sum_{\rho\in\mK}
\sum_{i,j}
d(\rho)\ph((g^*)_{\rho_{i,j}}\Ph_\rho^{\al\th^{-1}}((\ta^\om\oti\id)(avEv^*))g_{\rho_{j,i}}).
\]
From the property (E.\ref{item:ga-taE}) for $\ovl{E}$,
we have
\begin{align*}
(\ps\oti\vph)(g^*Ev^*avEg)
&=
\sum_{\rho\in\mK}
\sum_{i,j}
d(\rho)\ph((g^*)_{\rho_{i,j}}\Ph_\rho^{\al\th^{-1}}(a(\ta^\om\oti\id)(v\ovl{E}v^*)))g_{\rho_{j,i}})
\\
&\quad
+
\sum_{\rho\in\mK}
\sum_{i,j}
d(\rho)\ph((g^*)_{\rho_{i,j}}
\th_\rho(\Ph_\rho^{\al\th^{-1}}(av(\ta_\om(p)\oti e_\btr)v^*))g_{\rho_{j,i}})
\\
&=
\sum_{\rho\in\mK}
\Big{(}
\sum_{i,j}
d(\rho)b_{\ovl{E}}|K|_\vph^{-1}
\ph((g^*)_{\rho_{i,j}}
\Ph_\rho^{\al\th^{-1}}(a)g_{\rho_{j,i}})
+
\ta_\om(p)\de_{\rho,\btr}\Big{)}
\\
&=
\sum_{\rho\in\mK}
b_{\ovl{E}}|K|_\vph^{-1}
(\ph\oti\vph_\rho)(g^* (\Ph_\rho^{\al\th^{-1}}(a)\oti1_\rho)g)
+
\ta_\om(p).
\end{align*}
Since $(\ps\oti\vph)(g^*Eg)=b_{\ovl{E}}+\ta_\om(p)$,
we have
\begin{align*}
|(\ps\oti\vph)(g^*Ev^*(a-1)vEg)|
&\leq
\sum_{\rho\in\mK}
|K|_\vph^{-1}
|(\ph\oti\vph_\rho)
(g^*(\Ph_\rho^{\al\th^{-1}}(a-1)\oti1_\rho)g)|
\\
&\leq
\sum_{\rho\in\mK}
|K|_\vph^{-2}
f_1(\al,u;\rho,g)
\leq
\ka(\al,u)|K|_\vph^{-1}.
\end{align*}
\end{proof}

\begin{lem}\label{lem:Eva2}
One has the following inequality:
\[
|(\ps\oti\vph_F)(g^*((\id\oti\vph)(Ev^*(1-a)vE)\oti F)g)|
\leq
\ka(\al,u)|K|_\vph^{-1}.
\]
\end{lem}
\begin{proof}
Employing Theorem \ref{thm:app-cocycle}, we have
\begin{align*}
&(\ps\oti\vph_F)(g^*((\id\oti\vph)(Ev^*(1-a)vE)\oti F)g)
\\
&=
(\ps\oti\vph\oti\vph)(g_F^*(Ev^*(1-a)vE)_{13}g_F)
\\
&=
(\ps\oti\vph\oti\vph)(g_F^*(v^*(1-a)vE)_{13}g_F)
\\
&=
\sum_{\rho\in\mK}
(\ps\oti\vph_F)(g^*
(\Ph_\rho^{\al\th^{-1}}((1-a)(\ta^\om\oti\id)(vEv^*))
\oti F)g)
\\
&=
\sum_{\rho\in\mK}
(\ps\oti\vph_F)(g^*
(\Ph_\rho^{\al\th^{-1}}((1-a)(\ta^\om\oti\id)(v\ovl{E}v^*))
\oti F)g)
\\
&\quad+
\sum_{\rho\in\mK}
(\ps\oti\vph_F)(g^*
(\Ph_\rho^{\al\th^{-1}}((1-a)(\ta^\om\oti\id)(v(p\oti
e_\btr)v^*))
\oti F)g)
\\
&=
b_{\ovl{E}}|K|_\vph^{-1}
\sum_{\rho\in\mK}
(\ph\oti\vph_F)(g^*
(\Ph_\rho^{\al\th^{-1}}(1-a_\rho)
\oti F)g)
\quad\mbox{by }a_\btr=1\oti e_\btr.
\end{align*}
Thus
\begin{align*}
&|(\ps\oti\vph_F)(g^*((\id\oti\vph)(Ev^*(1-a)vE)\oti F)g)|
\\
&\leq
b_{\ovl{E}}|K|_\vph^{-1}
\sum_{\rho\in\mK}
f_2(\al,u;\rho,g)|K|_\vph^{-1}
\leq
\ka(\al,u)|K|_\vph^{-1}.
\end{align*}
This implies the desired inequality.
\end{proof}

\begin{lem}\label{lem:a1vE}
The following inequalities hold:
\begin{enumerate}
\item
$\|(a^*-1\oti K)vE\|_{g(\ps\oti\vph)g^*}
\leq\sqrt{2}\ka(\al,u)^{1/2}\|K\|_\vph^{-1/2}$.
\item
$\|\ga_F((a^*-1\oti K)vE)\|_{g(\ps\oti\vph)g^*\oti\vph}
\leq\sqrt{2}\ka(\al,u)^{1/2}\|K\|_\vph^{-1/2}$.
\end{enumerate}
\end{lem}
\begin{proof}
(1).
By (E.\ref{item:EvaavE}) and Lemma \ref{lem:Eva1},
we have
\begin{align*}
&\|(a^*-1\oti K)vE\|_{g(\ps\oti\vph)g^*}^2
\\
&=
\|(a^*-1\oti K)vE\|_{g(\ps\oti\vph)g^*}^2
\\
&=
(\ps\oti\vph)(g^*(Ev^*aa^*vE+E)g)
-
2\Re(\ps\oti\vph)(g^*Ev^*avEg)
\\
&=
2(\ps\oti\vph)(g^*Eg)
-
2\Re(\ps\oti\vph)(g^*Ev^*avEg)
\\
&=
2\Re(\ps\oti\vph)(g^*Ev^*(1-a)vEg)
\\
&\leq
2\ka(\al,u)|K|_\vph^{-1}.
\end{align*}

(2).
By (E.\ref{item:EvaavE}) and Lemma \ref{lem:ta-ga},
we have
\begin{align*}
&\|\ga_F((a^*-1\oti K)vE)\|_{g(\ps\oti\vph)g^*\oti\vph}^2
\\
&=
(\ps\oti\vph\oti\vph)
(g_{12}^*\ga_F(|(a^*-1\oti K)vE|^2)g_{12})
\\
&=
(\ph\oti\vph\oti\vph)
(g_{12}^*(\ta^\om\oti\id)(\ga_F(|(a^*-1\oti K)vE|^2))g_{12})
\\
&=
(\ph\oti\vph_F\oti\vph)
(g_{12}^*(\ta^\om\oti\id)(|(a^*-1\oti K)vE|^2)_{13}g_{12})
\\
&=
(\ps\oti\vph_F)
\left(g((\id\oti\vph)(|(a^*-1\oti K)vE|^2)\oti F)g\right)
\\
&=
2\Re(\ps\oti\vph_F)
(g((\id\oti\vph)(Ev^*(1-a)vE)\oti F)g).
\end{align*}
Then by the previous lemma,
we have the desired inequality.
\end{proof}

\begin{lem}\label{lem:u1}
One has
\[
\|(u-1\oti1\oti1)
(\id\oti\De)(vE)(1\oti F\oti K)\|_{g(\ps\oti\vph)g^*\oti\vph}
\leq\sqrt{2}\ka(\al,u)^{1/2}\|K\|_\vph^{-1}.
\]
\end{lem}
\begin{proof}
Set the element $X:=|u-1\oti1\oti1|^2(1\oti F\oti K)$
that is contained in $M\oti \lhG\oti\lhG$.
Using the equality
$E(E_{\rho_{i,j}}\oti1)=d(\rho)^{-1}E(1\oti e_{\rho_{j,i}})$,
we have
\begin{align*}
&\|(u-1\oti1\oti1)
(\id\oti\De)(vE)(1\oti F\oti K)\|_{g(\ps\oti\vph)g^*\oti\vph}^2
\\
&=
(\ps\oti\vph\oti\vph)
\left(
g_{12}^*(\id\oti\De)(Ev^*)X(\id\oti\De)(vE)g_{12}
\right)
\\
&=
\sum_{\rho\in\mK}
\sum_{i,j}
d(\rho)^{-1}
(\ps\oti\vph\oti\vph)
\left(
g_{12}^*
(\id\oti\De)(1\oti e_{\rho_{i,j}}v^*)X(\id\oti\De)(vE(E_{\rho_{i,j}}\oti1))
g_{12}
\right)
\\
&=
\sum_{\rho\in\mK}
\sum_{i,j}
d(\rho)^{-1}
(\ps\oti\vph\oti\vph)
\left(
g_{12}^*
(\id\oti\De)(1\oti e_{\rho_{i,j}}v^*)X(\id\oti\De)(vE(1\oti e_{\rho_{j,i}}))
g_{12}
\right).
\end{align*}
Let us introduce the following:
\[
Y:=
\sum_{i,j}(\id\oti\De)(1\oti e_{\rho_{i,j}}v^*)X
(\id\oti\De)(vE(1\oti e_{\rho_{j,i}})).
\]
We decompose the coproducts in terms of intertwiners as follows:
\begin{align*}
Y&=
\sum_{S,T}
\sum_{i,j}(1\oti S)(1\oti e_{\rho_{i,j}})v^*
(1\oti S^*)X
(1\oti T)
vE(1\oti e_{\rho_{j,i}})(1\oti T^*)
\\
&=
\sum_{S,T}
(1\oti S)
\left(
(\id\oti \Tr_\rho)
(v_\rho^*(1\oti S^*)X
(1\oti T)v_\rho E_\rho)
\oti1_\rho\right)(1\oti T^*).
\end{align*}
By Theorem \ref{thm:app-cocycle}, we have
\begin{align*}
&(\ta^\om\oti\id\oti\id)(Y)
\\
&=
d(\rho)
\sum_{S,T}
(1\oti S)
\ta^\om\circ\Ph_\rho^{\al\th^{-1}}
\left(
(1\oti S^*)X
(1\oti T)(\ta^\om\oti\id)(vEv^*)
\right)(1\oti T^*)
\\
&=
d(\rho)
\sum_{S,T}
(1\oti S)
\Ph_\rho^{\al\th^{-1}}
\left(
(1\oti S^*)X
(1\oti T)(\ta^\om\oti\id)(v\ovl{E}v^*)
\right)(1\oti T^*)
\\
&\quad+
d(\rho)
\sum_{S,T}
(1\oti S)
\Ph_\rho^{\al\th^{-1}}
\left(
(1\oti S^*)X
(1\oti T)(\ta^\om\oti\id)(p\oti e_\btr)
\right)(1\oti T^*)
\\
&=
d(\rho)b_{\ovl{E}}|K|_\vph^{-1}
\sum_{S,T}
(1\oti S)
\Ph_\rho^{\al\th^{-1}}
\left(
(1\oti S^*)X
(1\oti T)
\right)(1\oti T^*)
+
\de_{\rho,\btr}\ta_\om(p)X.
\end{align*}
Thus we have
\begin{align*}
&\|(u_{F,K}-1\oti F\oti
K)(\id\oti\De)(vE)\|_{g(\ps\oti\vph)g^*\oti\vph}^2
\\
&=
\sum_{\rho\in\mK}
|(\ph\oti\vph_F\oti\vph)(g_F^*(\ta^\om\oti\id\oti\id)(Y)g_F)|
\\
&\leq
\sum_{\rho\in\mK}
\left(
d(\rho)|K|_\vph^{-2}f_3(\al,u;\rho,g)d(\rho)^{-1}
\right)
+
\ta_\om(p)
|(\ph\oti\vph_F\oti\vph)(g_F^*Xg_F)|d(\rho)^{-1}
\\
&\leq
\ka(\al,u)|K|_\vph^{-1}
+\ta_\om(p)f_4(\al,u;g)|K|_\vph^{-1}
\\
&\leq
2\ka(\al,u)|K|_\vph^{-1}.
\end{align*}
\end{proof}

Recall the Shapiro unitary $\mu:=(\id\oti\vph)(a^*vE)$.

\begin{thm}\label{thm:almvan}
Let $E$ be a projection as in Theorem \ref{thm:Rohlin}.
Then the unitary $v$ is close to
the $\ga$-coboundary $(\mu\oti1)\ga(\mu^*)$ in the following
sense:
\begin{equation}
\|v_F\ga_F(\mu)-\mu\oti F\|_{g(\ps\oti\vph)g^*}^\sharp
<(5\de^{1/16}+3\ka(\al,u)^{1/4})\|F\|_\vph
\quad\mbox{for all }g\in\mG.
\label{eq:ga-v-mu}
\end{equation}
\end{thm}
\begin{proof}
The statement follows from the two claims below.
However, Claim 2 is not used in this paper,
and readers can skip it.
\addtocounter{claim}{-2}
\begin{claim}
\[
\|v_F\ga_F(\mu)-\mu\oti F\|_{g(\ps\oti\vph)g^*}^2
<
(14\de^{1/4}+9\ka(\al,u)^{1/2})\|F\|_\vph
\quad\mbox{for all }g\in\mG.
\]
\end{claim}
\begin{proof}
We set $A$, $B$, $C$ and $D$ defined by
\begin{align*}
A&:=
(\id\oti\id\oti\vph)((\id\oti{}_F\De)((a^*-1\oti K)vE))
\\
B&:=
(\id\oti\id\oti\vph)((\id\oti{}_F\De_{K^\perp})(vE))
\\
C&:=
(\id\oti\id\oti\vph)((\id\oti{}_F\De_K)(vE)-v_F\ga_F(vE))
\\
D&:=
-(\id\oti\id\oti\vph)(v_F\ga_F((a^*-1\oti K)vE)).
\end{align*}
Then we have
\[
\mu\oti F-v_F\ga_F(\mu)
=
A+B+C+D.
\]
Thus we obtain
\begin{align}
&\|v_F\ga_F(\mu)-\mu\oti F\|_{g(\ps\oti\vph)g^*}^2
\notag\\
&=
2|F|_\vph-2\Re(\ps\oti\vph)(g^*(\mu^*\oti F)v_F\ga_F(\mu)g)
\notag\\
&=
2\Re(\ps\oti\vph)(g^*(\mu^*\oti F)(\mu\oti F-v_F\ga_F(\mu))g)
\notag\\
&=
2\Re (\ps\oti\vph)(g^*(\mu^*\oti F)Ag)
+2\Re (\ps\oti\vph)(g^*(\mu^*\oti F)Bg)
\notag\\
&\quad+
2\Re (\ps\oti\vph)(g^*(\mu^*\oti F)Cg)
+2\Re (\ps\oti\vph)(g^*(\mu^*\oti F)Dg).
\label{v-ga-mu-F}
\end{align}
We will estimate these four terms.
Since $E$ centralizes $\ps\oti\vph$,
the first term is estimated as follows:
\begin{align}
&
|(\ps\oti\vph)(g^*(\mu^*\oti F)Ag)|
\notag\\
&=
|(\ps\oti\vph)
\left(
g^*(\mu^*\oti F) \cdot((\id\oti\vph)((a^*-1\oti K)vE)\oti F)g
\right)|
\notag\\
&=
|(\ps\oti\vph\oti\vph)
\left(
g_F^*(\mu^*\oti F\oti K)(a^*-1\oti K)vE)_{13}
g_F\right)|
\notag\\
&=
|(\ps\oti\vph\oti\vph)
\left(
g_F^*E_{13}(\mu^*\oti F\oti K)(a^*-1\oti K)vE)_{13}
g_F\right)|
\notag\\
&\leq
\|(\mu\oti F\oti K)E_{13}\|_{g(\ps\oti\vph)g^*\oti\vph}
\|(a^*-1\oti K)vE\|_{g(\ps\oti\vph)g^*\oti\vph}
\notag\\
&\leq
\sqrt{2}\ka(\al,u)^{1/2}\|F\|_\vph\|K\|_\vph^{-1}
\quad\mbox{by Lemma }\ref{lem:a1vE}\ (1).
\label{eq:A}
\end{align}

On the second,
by
$(\ta^\om\oti\id)(E)=b_{\ovl{E}}|K|_\vph^{-1}K+\ta_\om(p)e_\btr$,
we have
\begin{align}
&|(\ps\oti\vph)(g^*(\mu^*\oti F)Bg)|
\notag\\
&=
|(\ps\oti\vph\oti\vph)(g_F^*(\mu^*\oti F\oti K^\perp)
(\id\oti{}_F\De{}_{K^\perp})(vE)g_F)|
\notag\\
&\leq
\|(\mu^*\oti F\oti K^\perp)
(\id\oti{}_F\De{}_{K^\perp})(v)\|
\cdot
|F|_\vph
|(\id\oti{}_F\De{}_{K^\perp})(E)|_{\ps\oti\vph\oti\vph}
\quad\mbox{by Lemma }\ref{lem:g-ps-vph}
\notag\\
&\leq
b_{\ovl{E}}|K|_\vph^{-1}|F|_\vph
(\ps\oti\vph)(g_F^*(1\oti(\id\oti\vph)({}_F\De_{K^\perp}(K)))g_F)
\notag\\
&\quad+
\ta_\om(p)|F|_\vph
(\ps\oti\vph)(g_F^*(1\oti(\id\oti\vph)({}_F\De_{K^\perp}(e_\btr)))g_F)
\notag\\
&<
b_{\ovl{E}}|K|_\vph^{-1}\cdot\de|F|_\vph^2|K|_\vph
+
\ta_\om(p)|F|_\vph^2
\notag\\
&<2\de^{1/2}|F|_\vph^2.
\label{eq:B}
\end{align}

The third term is estimated as follows:
\begin{align}
&|(\ps\oti\vph)(g^*(\mu^*\oti F)Cg)|
\notag\\
&=
|(\ps\oti\vph\oti\vph)
\left(g_F^*(\mu^*\oti F\oti1)
((\id\oti{}_F\De_K)(vE)-v_F\ga_F(vE))g_F
\right)|
\notag\\
&\leq
|(\ps\oti\vph\oti\vph)(g_F^*(\mu^*\oti F\oti K)
(1-u)(\id\oti{}_F\De_K)(vE)g_F)|
\notag\\
&\quad
+
|(\ps\oti\vph\oti\vph)
\left(g_F^*(\mu^*\oti F\oti K)
u(\id\oti{}_F\De_K)(v)
\cdot((\id\oti{}_F\De_K)(E)-\ga_F(E))
g_F\right)|
\notag\\
&\leq
\|\mu\oti F\oti K\|_{g(\ps\oti\vph)g^*\oti\vph}
\cdot
\|(u-1\oti F\oti K)(\id\oti{}_F\De_K)(vE)\|_{g(\ps\oti\vph)g^*\oti\vph}
\notag\\
&\quad+
\|(\mu^*\oti F\oti1)
u(\id\oti{}_F\De)(v)\|
\\
&\quad
\cdot
|F|_\vph
|(\id\oti{}_F\De)(E)-\ga_F(E)|_{\ps\oti\vph\oti\vph}
\quad\mbox{by Lemma \ref{lem:g-ps-vph}}
\notag\\
&\leq
\sqrt{2}\ka(\al,u)^{1/2}\|F\|_\vph+5\de^{1/2}|F|_\vph^2
\quad\mbox{by Theorem }\ref{thm:Rohlin},\ \mbox{Lemma }\ref{lem:u1}.
\label{eq:C}
\end{align}

On the last term, we have
\begin{align}
&|(\ps\oti\vph)(g_F^*(\mu^*\oti F)Dg_F)|
\notag\\
&=
|(\ps\oti\vph\oti\vph)(g_F^*(\mu^*\oti F\oti K)
v_F\ga_F((a^*-1\oti K)vE)g_F)|
\notag\\
&\leq
\|v_F^*(\mu\oti F\oti K)\|_{g(\ps\oti\vph)g^*\oti\vph}
\cdot
\|\ga_F((a^*-1\oti K)vE)\|_{g(\ps\oti\vph)g^*\oti\vph}
\notag\\
&<
\sqrt{2}\ka(\al,u)^{1/2}\|F\|_\vph.
\quad\mbox{by Lemma }\ref{lem:a1vE}\ (2)
\label{eq:D}
\end{align}

Thus by (\ref{v-ga-mu-F}),
(\ref{eq:A}), (\ref{eq:B}), (\ref{eq:C}) and (\ref{eq:D}),
we obtain
\begin{align*}
\|v_F\ga_F(\mu)-\mu\oti F\|_{g(\ps\oti\vph)g^*\oti\vph}^2
&<2(\sqrt{2}(1+2\|F\|_\vph)\ka(\al,u)^{1/2}
+7\de^{1/2}|F|_\vph^2)
\\
&<
(9\ka(\al,u)^{1/2}+14\de^{1/4})|F|_\vph.
\end{align*}
\end{proof}

Next we show the following:
\begin{claim}
\[
\|(v_F\ga_F(\mu))^*-\mu^*\oti F\|_{g(\ps\oti\vph)g^*}^2
<
(20\de^{1/8}+6\ka(\al,u)^{1/2})\|F\|_\vph
\quad\mbox{for all }g\in\mG.
\]
\end{claim}
\begin{proof}
By simple computation, we have
\begin{align*}
&\|(v_F\ga_F(\mu))^*-\mu^*\oti F\|_{g(\ps\oti\vph)g^*}^2
\\
&=
2|F|_\vph
-2\Re
(\ps\oti\vph)(g_F^*v_F\ga_F(\mu)(\mu^*\oti F)g_F)
\\
&=
2|F|_\vph
-2\Re
(\ps\oti\vph\oti\vph)
(g_F^*v_F\ga_F(a^*)\ga_F(vE)(\mu^*\oti F\oti K)g_F)
\\
&=
2|F|_\vph
-2\Re
(\ps\oti\vph\oti\vph)
(g_F^*v_F\ga_F(a^*-1)\ga_F(vE)(\mu^*\oti F\oti K)g_F)
\\
&\quad
-2\Re
(\ps\oti\vph\oti\vph)
(g_F^*v_F\ga_F(vE)(\mu^*\oti F\oti K)g_F)
\\
&=
2|F|_\vph
-2\Re
(\ps\oti\vph\oti\vph)
(g_F^*v_F\ga_F(a^*-1)\ga_F(vE)(\mu^*\oti F\oti K)g_F)
\\
&\quad
-2\Re c_E^g
-2\Re
(\ps\oti\vph\oti\vph)
(g_F^*v_F\ga_F(v)(\id\oti{}_F\De_K)(E)(\mu^*\oti F\oti K)g_F)
\\
&=
2|F|_\vph
-2\Re
(\ps\oti\vph\oti\vph)
(g_F^*v_F\ga_F(a^*-1)\ga_F(vE)(\mu^*\oti F\oti K)g_F)
\\
&\quad
-2\Re c_E^g
-2\Re
(\ps\oti\vph\oti\vph)
(g_F^*u(\id\oti{}_F\De_K)(vEv^*a)g_F)
\quad\mbox{by (\ref{eq:Emu})}.
\end{align*}
Since
\begin{align*}
&|F|_\vph-(\ps\oti\vph\oti\vph)
(g_F^*u(\id\oti{}_F\De_K)(vEv^*a)g_F)
\\
&=
b_{\ovl{E}}|K|_\vph^{-1}
(\ps\oti\vph\oti\vph)
(g_F^*(1\oti F\oti K-u(\id\oti{}_F\De_K)(a_K))g_F)
\\
&\quad+
(\ps\oti\vph\oti\vph)
(g_F^*u(p\oti{}_F\De_{K}(e_\btr))g_F)
\\
&=
b_{\ovl{E}}|K|_\vph^{-1}
(\ps\oti\vph\oti\vph)
(g_F^*(1\oti F\oti K-u(\id\oti{}_F\De_K)(a_K))g_F)
\\
&\quad+
\ta_\om(p)(\ps\oti\vph\oti\vph)
(g_F^*a_F(1\oti{}_F\De_{K}(e_\btr))g_F),
\end{align*}
we have
\begin{align*}
&\|(v_F\ga_F(\mu))^*-\mu^*\oti F\|_{g(\ps\oti\vph)g^*}^2
\\
&\leq
2b_{\ovl{E}}|K|_\vph^{-1}
|(\ps\oti\vph\oti\vph)
(g_F^*(1-u(\id\oti{}_F\De_K)(a_K))g_F)|
\\
&\quad+
2\ta_\om(p)|(\ps\oti\vph\oti\vph)
(g_F^*a_F(1\oti{}_F\De_{K}(e_\btr))g_F)|
\\
&\quad+
2|
(\ps\oti\vph\oti\vph)
(g_F^*v_F\ga_F(a^*-1)\ga_F(vE)(\mu^*\oti F\oti K)g_F)|
\\
&\quad
+2|c_E^g|
\\
&\leq
2b_{\ovl{E}}|K|_\vph^{-1}f_5(\al,u)^{1/2}
\\
&\quad+
2\de^{1/2}|(\ps\oti\vph\oti\vph)
(g_F^*a_F(1\oti{}_F\De_{K}(e_\btr))g_F)|
\\
&\quad+
2|
(\ps\oti\vph\oti\vph)
(g_F^*v_F\ga_F(a^*-1)\ga_F(vE)(\mu^*\oti F\oti K)g_F)|
\\
&\quad
+18\de^{1/8}|F|_\vph
\\
&\leq
2b_{\ovl{E}}|K|_\vph^{-1}\ka(\al,u)^{1/2}+18\de^{1/8}
\\
&\quad+
2\de^{1/2}|(\ps\oti\vph\oti\vph)
(g_F^*(1\oti{}_F\De_K(e_\btr))g_F)|
\\
&\quad+
2\de^{1/2}|(\ps\oti\vph\oti\vph)
(g_F^*(a_F-1)(1\oti{}_F\De_{K}(e_\btr))g_F)|
\\
&\quad+
2|
(\ps\oti\vph\oti\vph)
(g_F^*v_F\ga_F(a^*-1)\ga_F(vE)(\mu^*\oti F\oti K)g_F)|
\\
&\leq
2|K|_\vph^{-1}\ka(\al,u)^{1/2}+18\de^{1/8}|F|_\vph+
2\de^{1/2}|F|_\vph
+
2\de^{1/2}f_6(\al,u)^{1/2}
\\
&\quad+
2|
(\ps\oti\vph\oti\vph)
(g_F^*v_F\ga_F(a^*-1)\ga_F(vE)(\mu^*\oti F\oti K)g_F)|.
\end{align*}
Hence we have to prove the last term is small.
Indeed, this is verified as
\begin{align*}
&|(\ps\oti\vph\oti\vph)
(g_F^*v_F\ga_F(a_K^*-1\oti K)\ga_F(vE)(\mu^*\oti F\oti K)g_F)|
\\
&\leq
(\ps\oti\vph\oti\vph)
(g_F^*v_F\ga_F(|a_K-1\oti K|^2)v_F^*g_F)^{1/2}
\\
&\quad\cdot
(\ps\oti\vph\oti\vph)
(g_F^*(\mu^*\oti F\oti K)\ga_F(E)(\mu^*\oti F\oti K)g_F)^{1/2}
\\
&=(\ps\oti\vph\oti\vph)
(g_F^*\al_F(|a_K-1\oti K|^2)g_F)^{1/2}
\\
&\quad\cdot
(\ps\oti\vph)
(g_F^*(\mu^*\oti F)\ga_F((\id\oti\vph)(E))(\mu^*\oti F)g_F)^{1/2}
\\
&<
f_7(\al,u;g)^{1/2}|F|_\vph^{1/2}.
\end{align*}
Therefore we have
\begin{align*}
&\|(v_F\ga_F(\mu))^*-\mu^*\oti F\|_{g(\ps\oti\vph)g^*}^2
\\
&<
2|K|_\vph^{-1}\ka(\al,u)^{1/2}+18\de^{1/8}+
2\de^{1/2}|F|_\vph
+
2\de^{1/2}\ka(\al,u)^{1/2}
+2\ka(\al,u)^{1/2}|F|_\vph^{1/2}
\\
&=
(20\de^{1/4}+6\ka(\al,u)^{1/2})|F|_\vph.
\end{align*}
Thus we have proved Claim 2.
\end{proof}
By Claim 1 and 2, we obtain (\ref{eq:ga-v-mu}).
\end{proof}

Let $(\al^\nu,u^\nu)$ be the perturbation of $(\al,u)$
by $(v^\nu)^*$, where $(v^\nu)_\nu$ is a representing
unitary sequence of $v$,
that is,
\[
\al^\nu=\Ad (v^\nu)^*\circ\al,
\quad
u^\nu=(v^\nu)_{12}^*\al((v^\nu)^*)u(\id\oti\De)(v^\nu).
\]
Then $u^\nu$ and the diagonal $a_\pi^\nu$
converges to $1$
as $\nu\to\om$ in the strong* topology.

\begin{lem}\label{lem:kaalu}
For a fixed $F,K,\mG$,
one has $\ka(\al^\nu,u^\nu;F,K,\mG)\to0$ as $\nu\to\om$.
\end{lem}
\begin{proof}
It is easy to see that the values of the functions
$f_1$, $f_5$, $f_6$ and $f_7$ converge to 0 as $\nu\to\om$.
Note that $\Ph_\rho^{\al^\nu}=\Ph_\rho^\al\circ\Ad v^\nu$.
Since $u^\nu-1\to0$ as $\nu\to\om$,
so does $v^\nu(1\oti S^*)(u^\nu-1)(1\oti T)(v^\nu)^*$
by definition of $M^\om$.
Hence $f_2$ converges to 0.
Likewise, we can show the remaining
$f_3$ and $f_4$ are also converging to 0.
Therefore, we have $\ka(\al^\nu,u^\nu)\to0$
as $\nu\to\om$.
\end{proof}

\subsection{2-cohomology vanishing}
We prove the 2-cohomology vanishing theorem.
Let $F_n,K_n$ be given as \cite[p. 537]{MT1} but with
$(5\de_n^{1/16}+3\de_{n+1}^{1/4})\|F_n\|_\vph<\de_{n-1}/2$.

\begin{thm}\label{thm:2cohvan}
\index{2-cohomology vanishing theorem}
Let $(\al,u)$
be a centrally free cocycle action of $\bhG$ on a McDuff factor
$M$ such that Assumption \ref{ass:app-cent} is fulfilled.
Then the 2-cocycle $u$ is a coboundary.
Moreover, assume that $\ka(\al,u;F_{n+1},K_{n+1},\mG)<\de_{n+1}$
for some $n\geq2$.
Then there exists a unitary $w\in M\oti\lhG$ such that
\begin{enumerate}
\item 
$(w\oti1)\al(w)u(\id\oti\De)(w^*)=1$,

\item
$\|w_{F_n}-1\oti F_n\|_{g(\ph\oti\vph)g^*}<\de_{n-2}$ for all $g\in\mG$.
\end{enumerate}
\end{thm}
\begin{proof}
We show it is possible to construct a family of
finite subsets
$\mG_m\subs M\oti \lhG$,
cocycle actions $\{(\al^{m},u^{m})\}_{m\geq n}$
and unitaries $\{w^{m}\}_{m\geq n}$
such that
\begin{enumerate}[(1,$m$)]
\item
$\displaystyle\mG_m
=\{1\}\cup \bigcup_{g\in\mG}\{w^{m-1}w^{m-2}\cdots w^1 g\}\cup\mG$;

\item
$\ka(\al^m,u^m;F_m,K_m,\mG_m)<\de_{m+1}$;

\item
$\|w_{F_m}^m-1\oti F_m\|_{g(\ph\oti\vph)g^*}^\sharp<\de_{m-1}$
for $g\in\mG_m$;

\item
$\al^{m+1}:=\Ad w^{m}\circ\al^{m}$,
$u^{m+1}=(w^{m}\oti1)\al^{m}(w^m)u^m (\id\oti\De)(w^{m*})$.
\end{enumerate}

We suppose that the construction
of $(\al^{m},u^{m})$ and $w^{m-1}$
has been done.
Take $v^m\in M^\om\oti\lhG$
associated with $(\al^{m},u^{m})$ as before.
Then by Theorem \ref{thm:almvan} for $(\al^{m},u^{m})$,
we obtain a unitary $\mu_m\in M^\om$ with
\[
\|\ga_{F_{m}}^m(\mu_m^*)(v_{F_m}^m)^*
-\mu_m^*\oti F_{m}\|_{g(\ps\oti\vph)g^*}
<
(5\de_m^{1/16}+3\ka(\al^m,u^m;F_m,K_m,\mG_m)^{1/4})\|F_m\|_\vph
\]
for all $g\in\mG_m$,
where $\ga^m=\Ad (v^m)^*\circ (\al^{m})^\om$.
We set
$\wdt{w}:=(\mu_m\oti1)(v^m)^*(\al^\nu)^\om(\mu_m^*)$,
and then we have
\begin{align*}
\|\wdt{w}_{F_{m}}-1\oti F_{m}\|_{g(\ps\oti\vph)g^*}
&=
\|(\mu_{m}\oti1)\ga_{F_{m}}(\mu_{m}^*)(v_{F_{m}}^m)^*
-1\oti F_{m}\|_{g(\ps\oti\vph)g^*}
\\
&<\de_{m-1},
\end{align*}
and
\[
1=
\wdt{w}_{12}(\al^{m})^\om(\wdt{w})u^m(\id\oti\De)(\wdt{w}^*).
\]
Then in a representing unitary sequence of $\wdt{w}$,
we can find $w^m$ with the desired properties
employing Lemma \ref{lem:kaalu},
and the induction is done.

We let $\ovl{w}^m:=w^mw^{m-1}\cdots w^n$.
Then it is trivial that $(\al^m,u^m)$ is the perturbation
of $(\al,u)$ by $\ovl{w}^m$.
Furthermore, $\{\ovl{w}^m\}_{m\geq n}$
is a strong*-Cauchy sequence and converges to the limit $w$,
which is a solution of $w_{12}\al(w)u(\id\oti\De)(w^*)=1$.
Indeed, if we take $m_0\in\N$ with $m_0\geq n$
and $\pi\in\mF_{m_0}$ for fixed $\pi\in\IG$,
then for $m\geq m_0$ and $g\in\mG$,
\begin{align*}
\|\ovl{w}_{F_{m_0}}^{m+1}-\ovl{w}_{F_{m_0}}^{m}\|_{g(\ps\oti \vph_\pi)g^*}
&=
\|w_{F_{m_0}}^{m+1}-1\oti F_{m_0}\|_{\ovl{w}_\pi^{m}g(\ps\oti \vph_\pi)(\ovl{w}_\pi^{m}g)^*}
\\
&<\de_m,
\end{align*}
and
\[
\|(\ovl{w}_{F_{m_0}}^{m+1})^*-(\ovl{w}_{F_{m_0}}^{m})^*\|_{g(\ps\oti \vph_\pi)g^*}
=
\|(w_{F_{m_0}}^{m+1})^*-1\oti F_{m_0}\|_{g(\ps\oti \vph_\pi)g^*}
<\de_m.
\]
Hence
$\|\ovl{w}_{F_{m_0}}^{m+1}-\ovl{w}_{F_{m_0}}^{m}\|_{g(\ps\oti\vph)g^*}^\sharp<\de_m$,
and $\{\ovl{w}_\pi^{m+1}\}$ is a Cauchy sequence for each $\pi\in\IG$.
Moreover we have
\begin{align*}
\|\ovl{w}_{F_n}^{m}-1\oti F_n\|_{g(\ps\oti\vph)g^*}^\sharp
&<
\|\ovl{w}_{F_n}^{m}-\ovl{w}_{F_n}^{m-1}\|_{g(\ps\oti\vph)g^*}^\sharp
+
\|\ovl{w}_{F_n}^{m-1}-\ovl{w}_{F_n}^{m-2}\|_{g(\ps\oti\vph)g^*}^\sharp
\\
&\quad
+\cdots+
\|\ovl{w}_{F_n}^{n+1}-w_{F_n}^{n}\|_{g(\ps\oti\vph)g^*}^\sharp
+
\|w_{F_n}^{n}-1\oti F_n\|_{g(\ps\oti\vph)g^*}^\sharp
\\
&<\de_{m-1}+\de_{m-2}+\cdots+\de_{n-1}
\\
&<\de_{n-1}(1+1/2+\cdots)=2\de_{n-1}<\de_{n-2}.
\end{align*}
\end{proof}

\begin{lem}\label{lem:CTinj}
Let $(\al,u)$ be
a centrally free cocycle action of $\bhG$ on an injective factor $M$.
Then the following statements hold:
\begin{enumerate}
\item
$(\al,u)$ has the Connes--Takesaki module;
\item
$(\al,u)$ satisfies Assumption \ref{ass:app-cent}.
\end{enumerate}
\end{lem}
\begin{proof}
(1).
Thanks to \cite[Theorem A.6 (2)]{MT3}, $\al_\pi$ is non-modular for all $\pi\neq\btr$,
that is, the canonical extension $\tal$ on $\tM$
is a free action.
Thus $\tal$ preserves the center $Z(\tM)$ by \cite[Lemma 2.9]{MT1}.

(2).
Recall a homomorphic section $s\col \Aut_\th(Z(\tM))\to \Aut(M)$
for the Connes--Takesaki module map stated in Example \ref{ex:modulo-auto}.
Then the composed map $\th_\pi:=s(\mo(\al_\pi))$
gives $\mo(\al_\pi\th_\pi^{-1})=\id$.
Thus, from \cite[Theorem A.6 (1), Proposition A.10]{MT3},
it turns out that
$\al_\pi\th_\pi^{-1}$ is approximately inner.
\end{proof}

\begin{cor}
Let $(\al,u)$ be
a centrally free cocycle action of $\bhG$ on an injective factor $M$.
Then the 2-cocycle $u$ is a coboundary.
Moreover,
assume for a fixed $n\geq2$,
the inequality
$\ka(\al,u;F_{n+1},K_{n+1},\mG)<\de_{n+1}$ holds.
Then there exists a unitary $w\in M\oti\lhG$ such that
\begin{enumerate}
\item 
$(w\oti1)\al(w)u(\id\oti\De)(w^*)=1$,

\item
$\|w_{F_n}-1\oti F_n\|_{g(\ph\oti\vph)g^*}^\sharp<\de_{n-2}$
for all $g\in\mG$.
\end{enumerate}
\end{cor}

\subsection{Intertwining cocycles}

Next we study an intertwining cocycle.
Let $(\al,u)$ be a cocycle action of $\bhG$ on $M$.
For $\pi\in\IG$, we define the map
$\al_\pi\in B(M_*, (M\oti B(H_\pi))_*)$
by $\al_\pi(\chi)=\chi\circ\Ph_\pi^\al$.
Then a perturbed cocycle action $(\al^v,u')$ by
$v$ yields
$\al_\pi^v(\chi)=v_\pi\cdot \al_\pi(\chi)\cdot v_\pi^*$.

\begin{lem}\label{lem:coc-per}
Let $M$ be a von Neumann algebra such that $M_\om$ is of type II$_1$.
Let $\al,\be$ be actions of $\bhG$ on $M$.
Suppose that there exists a homomorphism
$\th\col\IG\ra\Aut(M)$ such that $\al_\pi\th_\pi^{-1}$ and
$\be_\pi\th_\pi^{-1}$ are approximately inner for each $\pi\in\IG$.
Then there exists an $\al^\om$-cocycle $W\in M^\om\oti\lhG$
whose unitary representing sequence $(W^\nu)_\nu$
satisfies the following norm convergence:
\[
\lim_{\nu\to\om}\Ad W^\nu(\al_\pi(\chi))
=\be_\pi(\chi)
\quad\mbox{for all }\chi\in M_*,\pi\in\IG.
\]
In particular, we have $\be=\Ad W\circ\al$ on $M$.
\end{lem}
\begin{proof}
Thanks to Theorem \ref{thm:app-cocycle},
we can take an $\al^\om$-cocycle $U^*\in M^\om\oti\lhG$
such that $U^*\al_\pi(x)U=\th_\pi(x)\oti1$ for $x\in M$.
Take a unitary $V\in M^\om\oti\lhG$ such that $\Ad V_\pi^\nu$ converges
to $\be_\pi\th_\pi^{-1}$ for $\pi\in\IG$.
Then the map $M_\om\ni x\mapsto
\th_\opi(V_\pi^*)\th_\opi(x)\th_\opi(V_\pi)$
is a cocycle action of $\bhG^\opp$ on $M_\om$.
By the 2-cohomology vanishing \cite[Lemma 4.3]{MT1},
we may and do assume that
$\th_\opi(V_\pi^*)$ is $\th_\opi$-cocycle,
and $V_\pi$ is a $\th$-cocycle.

Let $N$ be a von Neumann algebra
generated by $M$ and the matrix elements of $V$.
Take a slow indexation map $\Ps\col N\ra M^\om$
associated with the semiliftable action $\Ad U^*\circ \al^\om$
which converges to $\th$ on $M$ in the pointwise strong* topology,
and set $W:=(\Ps\oti\id)(V) U^*$.
We check that $\Ph_\pi^\al\circ (\Ad W^\nu)^*$ converges
to $\Ph_\pi^\be$ for each $\pi\in\IG$ as $\nu\to\om$
in the point-wise norm topology of $B(M_*,(M\oti B(H_\pi))_*)$.
Let $\chi\in M_*$.
Then we have
\begin{align*}
\chi\circ\Ph_\pi^\al\circ (\Ad W^\nu)^*
&=
\chi\circ\Ph_\pi^\al\circ \Ad U^\nu (V^{k(\nu)})^*
\\
&\sim
(\chi\th_\pi^{-1}\oti\tr_\pi)\circ \Ad (V^{k(\nu)})^*,
\end{align*}
where $k(\nu)$ is the number defined in \cite[p. 36]{Oc}.
It is easy to see that if $\nu\to\om$, then $k(\nu)\to\infty$.
Therefore, the final term above converges to $\chi\circ\Ph_\pi^\be$.
Also we can verify the cocycle identity,
\begin{align*}
W_{\pi}\al_\pi^\om(W_\rho)
&=(\Ps\oti\id)(V_\pi)
U_\pi^*\al_\pi^\om((\Ps\oti\id)(V_\rho))U_\pi
\cdot
U_\pi^*\al_\pi^\om(U_\rho^*)
\\
&=
(\Ps\oti\id)(V_\pi)
(\Ps\oti\id)((\th_\pi\oti\id)(V_\rho))
(\id\oti{}_\pi\De_\rho)(U^*)
\\
&=
(\Ps\oti\id)((\id\oti{}_\pi\De_\rho)(V))
(\id\oti{}_\pi\De_\rho)(U^*)
\\
&=
(\id\oti{}_\pi\De_\rho)(W).
\end{align*}
\end{proof}

\begin{cor}\label{cor:coc-per}
Let $M$ be a McDuff factor.
Let $\al,\be$ be centrally free actions of $\bhG$ on $M$.
Suppose that there exists a homomorphism
$\th\col\IG\ra\Aut(M)$ such that the maps
$\al_\pi\circ\th_\pi^{-1}$
and $\be_\pi\circ\th_\pi^{-1}$ are approximately inner for all
$\pi\in\IG$.
Then for any $\vep>0$, finite sets $\mF\subs\IG$ and $\Ps\subs M_*$,
there exists an $\al$-cocycle $v$ such that
\[
\|\be_\pi(\chi)-\Ad v_\pi(\al_\pi(\chi))\|<\vep
\quad\mbox{for all }\pi\in\mF,\chi\in \Ps.
\]
\end{cor}
\begin{proof}
By the previous lemma,
we obtain an $\al^\om$-cocycle $W$ whose
unitary representing sequence
$W^\nu$ satisfies
$\dsp\lim_{\nu\to\om}W^\nu\al_\pi(\chi)(W^\nu)^*=\be_\pi(\chi)$
for all $\pi\in\mF$ and $\chi\in\Ps$.
The cocycle identity $W$ implies
$u^\nu:=W^\nu\al(W^\nu)(\id\oti\De)((W^\nu)^*)$
converges to 1 in the strong* topology
as $\nu\to\om$.
Applying Theorem \ref{thm:2cohvan}
to the cocycle action $(\Ad W^\nu\circ\al,u^\nu)$,
we have a unitary sequence $v^\nu$ such that
$\hat{v}^\nu:=v^\nu W^\nu$ is an $\al$-cocycle
and $v^\nu$ converges to 1 in the strong* topology
as $\nu\to\om$.
Then the required inequality holds if we set
$v:=\hat{v}^\nu$ to an enough large $\nu$.
\end{proof}

Recall the state $\om_\rho\col B(H_\rho)\oti B(H_\orho)\ra \C$
that is defined by $\om_\rho(x)=T_{\rho,\orho}^* xT_{\rho,\orho}$.
\begin{lem}\label{lem:omrho}
Let $\rho\in\IG$.
For $X_\orho\in M\oti B(H_\orho)$ and $Y_\rho\in M\oti B(H_\rho)$,
one has the following equality:
\[
(\id\oti\om_\rho)(\al_\rho(X_\orho)Y_\rho)
=(\Ph_\orho^\al\oti\vph_\rho)
(X_\orho(1\oti{}_\orho\De_\rho(e_\btr))\al_\orho(Y_\rho)).
\]
\end{lem}
\begin{proof}
Indeed, by definition of $\Ph_\orho^\al$,
we have
\begin{align*}
&(\Ph_\orho^\al\oti\id_\rho)(X_\orho(1\oti\De(e_\btr))\al_\orho(Y_\rho))
\\
&=
(1\oti T_{\rho,\orho}^*\oti1_\rho)
(\al_\rho(X_\orho)\oti1_\rho)
(1\oti1_\rho\oti{}_{\orho}\De_\rho(e_\btr))
(1\oti T_{\rho,\orho}\oti1_\rho)
Y_\rho
\\
&=
d(\rho)^{-1}
(1\oti T_{\rho,\orho}^*\oti1_\rho)
(\al_\rho(X_\orho)\oti1_\rho)
(1\oti 1_\rho\oti T_{\orho,\rho})
Y_\rho,
\end{align*}
and
\begin{align*}
&(\Ph_\orho^\al\oti\vph_\rho)
(X_\orho(1\oti{}_\orho\De_\rho(e_\btr))\al_\orho(Y_\rho))
\\
&=
d(\rho)^2
(1\oti T_{\rho,\orho}^*)
(\Ph_\orho^\al\oti\id_\rho)(X_\orho(1\oti\De(e_\btr))\al_\orho(Y_\rho))
(1\oti T_{\rho,\orho})
\\
&=
(1\oti T_{\rho,\orho}^*)\al_\rho(X_\orho)Y_\rho(1\oti T_{\rho,\orho})
\\
&=
(\id\oti\om_\rho)(\al_\rho(X_\orho)Y_\rho).
\end{align*}
\end{proof}

We take an $\al^\om$-cocycle $v^*$ such that $\ga:=\Ad v^* \circ\al$
gives an action on $M^\om$ preserving $M_\om$ as before.
Let us introduce the set $\meJ$ that is a collection
of a projection $E\in M^\om\oti\lhG$ such that
\begin{enumerate}[(E.1)]
\item
$E=E(1\oti K)$;
\item
$v^*Ev\in M_\om\oti\lhG$;

\item
In the decomposition of
\[
E=\sum_{\rho\in\mK}d(\rho)^{-1}f_{\rho_{i,j}}^\al\oti e_{\rho_{i,j}},
\]
the family $\{f_{\rho_{i,j}}^\al\}_{i,j\in I_\rho}$ is a system of
matrix units and they are orthogonal with respect to $\rho\in\mK$;

\item
In the decomposition of
\[
vEv^*=\sum_{\rho\in\mK}d(\rho)^{-1}f_{\rho_{i,j}}^\be\oti e_{\rho_{i,j}},
\]
the family $\{f_{\rho_{i,j}}^\be\}_{i,j\in I_\rho}$ is a system of
matrix units and they are orthogonal with respect to $\rho\in\mK$;

\item
$(\id\oti\vph_\rho)(E)=(\id\oti\vph_\rho)(vEv^*)$;

\item
$(\ta^\om\oti\id)(E)\in\C K$;
\label{item:taE}

\item
Let $\iota_\orho\col M\oti B(H_\orho)\ra M\oti B(H_\orho)\oti B(H_\rho)$
be the map sending $x$ to $x\oti1$.
Then
\[
\lim_{\nu\to\om}
\|[(\chi\circ\Ph_\orho^\al\oti\vph_\rho),
(1\oti{}_\orho\De_\rho(e_\btr))\al_\orho(E_\rho^\nu)]\circ\iota_\orho\|
=0
\quad
\mbox{for all }\chi\in\Ps,
\]
where the norm is taken in the predual of $M\oti B(H_\orho)$.
\label{item:limom}
\end{enumerate}

We note that the last condition does not depend on
a choice of a representing sequence of $E$.
Define the functions $b,c,d$ on $\meJ$ as follows:
\begin{align}
b_E^g
&=|E|_{g(\ps\oti\vph)g^*}=|E|_{\ps\oti\vph}=b_E,
\notag
\\
c_E^g
&=
(\ps\oti\vph\oti\vph)(g_{12}^*
(\mu_E^*\oti1\oti1)(\id\oti{}_F\De)(v)(\al_F^\om(E)-(\id\oti{}_F\De)(E))
g_{12}),
\label{eq:cg}\\
d_E^g
&=
(\ps\oti\vph\oti\vph)
\left(g_{12}^*
(\id\oti{}_F\De)(v)(\al_F^\om(E)-(\id\oti{}_F\De)(E))(\mu_E^*\oti1\oti1)
g_{12}\right),
\notag
\end{align}
where $\mu_E$ denotes
the partial isometry $(\id\oti\vph)(vE)$.
Note that $\mu_{E}^*\mu_{E}=(\id\oti\vph)(E)=\mu_{E}\mu_{E}^*$.

\begin{lem}
Assume that $E\in \meJ$ satisfies $b_E<1-\de^{1/2}$.
Then there exists $E'\in\meJ$ satisfying the following
for all $g\in\mG$:
\begin{enumerate}
\item
$0<(\de^{1/2}/2)|v^*E'v-v^*Ev|_{\ps\oti\vph}<b_{E'}-b_E$,
\item
$|c_{E'}^g|-|c_E^g|\leq5\de^{1/2}|F|_\vph(b_{E'}-b_E)$,
\item
$|d_{E'}^g|0-|d_E^g|\leq5\de^{1/2}|F|_\vph(b_{E'}-b_E)$.
\end{enumerate}
\end{lem}
\begin{proof}
We let $\mS:=\mF\cdot\mK$.
We take a non-zero projection
$e\in M_\om$ commuting with $E$, $v$, $U$
and $\al_\rho(e)$
for all $\rho\in \ovl{\mS}\cdot\mS\setminus\{\btr\}$.
Then by the Fast Reindexation Trick (\cite[Lemma 3.10]{MT1}),
we set $f:=\Ps(e)$
and $f':=(\id\oti\vph)(\al_K^\om(f))$.
\[
E':=E(f'^\perp\oti1)+\al_K^\om(f).
\]
The inequality $\de^{1/2}|f'|_\ps< b_{E'}-b_E$
is proved as \cite[Lemma 5.11]{MT1}.
In particular, $E'\neq E$.

We show that $E'$ is in fact a member of $\meJ$.
We can prove that $E'$ satisfies the conditions from (E.1) to
(E.\ref{item:taE}) as \cite[Lemma 5.11]{MT1}.
We have to show the remaining (E.\ref{item:limom}), that is,
\begin{equation}
\lim_{\nu\to\om}
\|[(\chi\circ\Ph_\orho^\al\oti\vph_\rho),
(1\oti{}_\orho\De_\rho(e_\btr))\al_\orho((E_\rho')^\nu)]\circ\iota_\orho\|=0
\label{eq:limch}
\end{equation}
for all $\chi\in\Ps$ and $\rho\in\mK$.
Take representing sequences $E^\nu$, $f'^\nu$ and $f^\nu$ for $E$,
$f'$ and $f$ so that they are projections, respectively.
Set $E'^\nu:=E^\nu((1-f'^\nu))+\al_K(f^\nu)$.
Since
\[
(1\oti{}_\orho\De_\rho(e_\btr))\al_\orho((E_\rho')^\nu)
=
(1\oti{}_\orho\De_\rho(e_\btr))\al_\orho(E_\rho^\nu)\al_\orho(1-f'^\nu)
+
f^\nu\oti{}_\orho\De_\rho(e_\btr),
\]
we have the following for $x_\orho\in M\oti B(H_\orho)$:
\begin{align*}
&
\left|
[(\chi\circ\Ph_\orho^\al\oti\vph_\rho),
(1\oti{}_\orho\De_\rho(e_\btr))\al_\orho((E_\rho')^\nu)](x_\orho)
\right|
\\
&\leq
\left|
[(\chi\circ\Ph_\orho^\al\oti\vph_\rho),
(1\oti{}_\orho\De_\rho(e_\btr))\al_\orho(E_\rho^\nu)]\cdot
(\al_\orho(1-f'^\nu)x_\orho)
\right|
\\
&\quad+
\left|
[(\chi\circ\Ph_\orho^\al\oti\vph_\rho),
\al_\orho(1-f'^\nu)]
(x_\orho(1\oti{}_\orho\De_\rho(e_\btr))\al_\orho(E_\rho^\nu))
\right|
\\
&\quad+
\left|
[(\chi\circ\Ph_\orho^\al\oti\vph_\rho),
f^\nu\oti{}_\orho\De_\rho(e_\btr)](x_\orho)
\right|
\\
&\leq
\left\|
[(\chi\circ\Ph_\orho^\al\oti\vph_\rho),
(1\oti{}_\orho\De_\rho(e_\btr))\al_\orho(E_\rho^\nu)]
\circ\iota_\orho\right\|
\|x_\orho\|
\\
&\quad+
\left|([\chi,f'^\nu]\oti\om_\rho)(\al_\rho(x_\orho)E_\rho^\nu)
\right|
\quad(\mbox{by Lemma } \ref{lem:omrho})
\\
&\quad+
\left|
[\chi\circ\Ph_\orho^\al,f^\nu\oti1_\orho] (x_\orho)
\right|
\\
&\leq
\left\|
[(\chi\circ\Ph_\orho^\al\oti\vph_\rho),
(1\oti{}_\orho\De_\rho(e_\btr))\al_\orho(E_\rho^\nu)]\circ\iota_\orho\right\|
\|x_\orho\|
\\
&\quad
+
\left\|
[\chi,f'^\nu]
\right\|\|x_\orho\|
+
\left\|
[\chi\circ\Ph_\orho^\al,f^\nu\oti1_\orho]
\right\|\|x_\orho\|.
\end{align*}
This implies (\ref{eq:limch}),
and hence $E'\in\meJ$.

We show the inequality in (1).
Since $U^*(E'-E)U=-U^*EU(f'\oti1)+\ga_K(f)$,
we have
\begin{align*}
|U^*(E'-E)U|_{\ps\oti\vph}
&\leq
|U^*EU(f'\oti1)|_{\ps\oti\vph}+
|\ga_K(f)|_{\ps\oti\vph}
\\
&=
\ta_\om(f')(\ps\oti\vph)(U^*EU)
+
\ta_\om(f)|K|_\vph
\\
&=
\ta_\om(f')((\ps\oti\vph)(E)+1)
\\
&\leq
2|f'|_\ps.
\end{align*}
Thus $0<(\de^{1/2}/2)|U^*(E'-E)U|_{\ps\oti\vph}\leq\de^{1/2}|f'|_\ps<b_{E'}-b_E$.

We will check the remaining (2) and (3).
Let us introduce the following:
\begin{align*}
X_1&:=
(\mu_{E}^*\oti F\oti1)(\id\oti{}_F\De)(v)
\cdot
(\al_F^\om(E)-(\id\oti{}_F\De)(E))((f'^\perp\oti1)\al_F^\om(f'^\perp)\oti1),
\\
X_2&:=
(\mu_{\al_K^\om(f)}^*\oti F\oti1)(\id\oti{}_F\De)(v)
\cdot
(\al_F^\om(E)-(\id\oti{}_F\De)(E))(\al_F^\om(f'^\perp)\oti1),
\\
X_3&:=
(\mu_{E}^*\oti F\oti1)(\id\oti{}_F\De)(v)
\cdot
(\id\oti{}_F\De)(E)
((f'^\perp\oti1)\al_F^\om(f'^\perp)\oti1-f'^\perp\oti1\oti1),
\\
X_4&:=
(\mu_{\al_K^\om(f)}^*\oti F\oti1)(\id\oti{}_F\De)(v)
\cdot
(\id\oti{}_F\De)(E)(\al_F^\om(f'^\perp)\oti1),
\\
X_5&:=
(\mu_{E}^*\oti F\oti1)(\id\oti{}_F\De)(v)
\cdot
(\id\oti{}_F\De_K)(\al_{K^\perp}^\om(f)),
\\
X_6&:=
(\mu_{\al_K^\om(f)}^*\oti F\oti1)(\id\oti{}_F\De)(v)
\cdot
((\id\oti{}_F\De_K)(\al_K^\om(f))-(\id\oti{}_F\De)(\al_K^\om(f))).
\end{align*}
Note that
we have $\al^\om(f)(f'\oti 1_\mS)=\al_K^\om(f)$ by definition of $f'$,
and
\[
(\id\oti{}_F\De_K)(\al^\om(f))(f'^\perp\oti F\oti K)
=(\id\oti{}_F\De_K)(\al_{K^\perp}^\om(f)).
\]
Using $\mu_{E'}=\mu_E f'^\perp+\mu_{\al_K(f)}$
and $(f'\oti1)\al_{K}(f)=\al_K(f)$,
we have
\[
(\mu_{E'}^*\oti F\oti1)(\id\oti{}_F\De)(v)
(\al_F^\om(E')-(\id\oti{}_F\De)(E'))
=X_1+X_2+X_3+X_4+X_5+X_6.
\]
Then by (\ref{eq:cg}), we get
\begin{equation}
\label{eq:cg'}
c_{E'}^g
=
\sum_{k=1}^6(\ps\oti\vph\oti\vph)(g_{12}^*X_kg_{12}).
\end{equation}

We obtain
\begin{align*}
X_{24}&:=X_2+X_4=
(\mu_{\al_K^\om(f)}^*\oti F\oti1)(\id\oti{}_F\De)(v)
\cdot
\al_F^\om(E)((f'\oti1)\al_F^\om(f'^\perp)\oti1)
\\
&=
(\mu_{\al_K^\om(f)}^*\oti F\oti1)v_F
\cdot
\al_F^\om(vE)((f'\oti1)\al_F^\om(f'^\perp)\oti1),
\\
X_6&=
-
(\mu_{\al_K^\om(f)}^*\oti F\oti1)(\id\oti{}_F\De)(v)
\cdot
(\id\oti{}_F\De_{K^\perp})(\al_K^\om(f)).
\end{align*}

By Claim \ref{claim:alf} in the proof of Lemma \ref{lem:abc},
we have
\begin{align*}
(\ta^\om\oti\id)((f'^\perp\oti1)\al_F(f'^\perp))
&=
(\ta^\om\oti\id)(\al_F(f'^\perp))
-
(\ta^\om\oti\id)((f'\oti1)\al_F(f'^\perp))
\\
&=
\ta_\om(f'^\perp)\oti F
-
\ta_\om(f)\oti
(\id\oti\vph)({}_F\De_K(K^\perp)).
\end{align*}
By the equality above,
the $\ta^\om$-splitting of $f'$ and the centrality
of $(\id\oti\vph)({}_F\De_K(K^\perp))$,
we have
\begin{align}
&|(\ps\oti\vph\oti\vph)(g_{12}^*X_1g_{12})|
\notag\\
&\leq
\ta_\om(f'^\perp)|c_E^g|
\notag\\
&\quad+\ta_\om(f)
\|g_{12}^*
(\mu_{E}^*\oti 1_\pi\oti1)(\id\oti{}_\pi\De)(v)
\cdot
(\al_\pi^\om(E)-(\id\oti{}_\pi\De)(E))g_{12})
\|
\notag\\
&\quad\quad\cdot
(\vph\oti\vph)({}_F\De_K(K^\perp))
\notag\\
&<
\ta_\om(f'^\perp)|c_{E}|
+
\ta_\om(f)\cdot\de|F|_\vph|K|_\vph
\notag\\
&=
\ta_\om(f'^\perp)|c_{E}|
+
\de\ta_\om(f')|F|_\vph.
\label{eq:X1}
\end{align}


Next we estimate $(\ps\oti\vph\oti\vph)(g_{12}^*X_{24}g_{12})$
as follows.
By Claim \ref{claim:alf},
\begin{align*}
(\mu_{\al_K^\om(f)}^*\oti F)\be_F^\om(f'^\perp)
&=
(\mu_{\al_K^\om(f)}^*\oti F)(f'\oti 1)\be_F^\om(f'^\perp)
\\
&=
(\mu_{\al_K^\om(f)}^*\oti F)
(\id\oti\id\oti\vph)((\id\oti{}_F\De_{K^\perp})(\be_K^\om(f)))
\\
&=
(\id\oti\id\oti\vph)((\id\oti{}_F\De_{K^\perp})
((\mu_{\al_K^\om(f)}^*\oti1)\be_K^\om(f)))
\\
&=
(\id\oti\id\oti\vph)((\id\oti{}_F\De_{K^\perp})
(v_K^*\be_K^\om(f))).
\quad\mbox{by }(\ref{eq:Emu})
\end{align*}
By this computation and $[\mu_E, f'^\perp]=0$, we have
\begin{align}
&|(\ps\oti\vph\oti\vph)(g_{12}^*X_{24}g_{12})|
\notag\\
&=
|(\ps\oti\vph)(g^*
(\mu_{\al_K^\om(f)}^*\oti F)v_F\al_F^\om(\mu_E)
(f'\oti1)\al_F^\om(f'^\perp)
g)|
\notag\\
&=
|(\ps\oti\vph)(g^*
(\mu_{\al_K^\om(f)}^*\oti F)
\be_F^\om(f'^\perp)v_F\al_F^\om(\mu_E)
g)|
\notag\\
&=
|(\ps\oti\vph\oti\vph)
(g_{12}^*
(\id\oti{}_F\De_{K^\perp})(v_K^*\be_K^\om(f))
v_F\al_F^\om(\mu_E)
g_{12})|
\notag\\
&=
\ta_\om(f)
|(\ps\oti\vph\oti\vph)
(g_{12}^*
(\id\oti{}_F\De_{K^\perp})(v_K^*)
v_F\al_F^\om(\mu_E)
g_{12})|
\quad
\mbox{by $\ta^\om$-splitting}
\notag\\
&\leq
\ta_\om(f)
\|(\id\oti{}_F\De_{K^\perp})(v_K)g_{12}\|_{\ps\oti\vph\oti\vph}
\cdot
\|(1\oti{}_F\De_{K^\perp}(K))v_F\al_F^\om(\mu_E)g_{12}\|_{\ps\oti\vph\oti\vph}
\notag\\
&\leq
\ta_\om(f)
(\ps\oti\vph\oti\vph)(1\oti {}_F\De_{K^\perp}(K))^{1/2}
\cdot
(\ps\oti\vph\oti\vph)(1\oti {}_F\De_{K^\perp}(K))^{1/2}
\notag\\
&<
\de\ta_\om(f')|F|_\vph,
\label{eq:X24}
\end{align}
where we have used the fact that $(\id\oti\vph)({}_F\De_{K^\perp}(K))$
is central in $\lhG$.

For $X_3$, we have
\begin{align*}
&(\ta^\om\oti\id)((f'^\perp\oti1)\al_F^\om(f'^\perp)-f'^\perp\oti1)
\\
&=
\ta_\om(f'^\perp)\oti F
-(\ta^\om\oti\id)((f'\oti1)\al_F^\om(f'^\perp))
-\ta_\om(f'^\perp)\oti F
\\
&=
-(\ta^\om\oti\id)((f'\oti1)\al_F^\om(f'^\perp)),
\end{align*}
and again by Claim \ref{claim:alf},
\begin{align}
&|(\ps\oti\vph\oti\vph)(g_{12}^*X_3g_{12})|
\notag\\
&=
\left|
(\ps\oti\vph\oti\vph)(g_{12}^*
(\mu_E^*\oti F)
\cdot
(\mu_E\oti F)
\cdot
(\ta^\om\oti\id)((f'\oti1)\al_F^\om(f'^\perp))
g_{12})
\right|
\notag\\
&=
\left|
(\ps\oti\vph\oti\vph)(
(\mu_E^*\mu_E\oti F)
\cdot
(\ta^\om\oti\id)((f'\oti1)\al_F^\om(f'^\perp)))
\right|
\notag\\
&=
b_E\cdot\ta_\om(f)(\vph_F\oti\vph_{K^\perp})(\De(K))
\notag\\
&<
\de \ta_\om(f') b_E|F|_\vph.
\label{eq:X3}
\end{align}
An estimate of $X_5$ is given as
\begin{align}
&|(\ps\oti\vph\oti\vph)(g_{12}^*X_5g_{12})|
\notag\\
&=
\ta_\om(f)
|(\ps\oti\vph\oti\vph)(g_{12}^*
(\ta^\om\oti\id)
\left((\mu_E^*\oti F\oti1)
(\id\oti{}_F\De_{K^\perp})(v_K)\right)
g_{12})|
\notag\\
&=
\ta_\om(f)
|(\ps\oti\vph\oti\vph)(g_{12}^*
(\ta^\om(\mu_E^*)\oti F\oti1)
(\id\oti{}_F\De_{K^\perp})(v_K)
g_{12})|
\notag\\
&\leq
\ta_\om(f)
\|(1\oti{}_F\De_{K^\perp}(K))
(\ta^\om(\mu_E)\oti F\oti1)g_{12}\|_{\ps\oti\vph\oti\vph}
\notag\\
&\quad\cdot
\|(\id\oti{}_F\De_{K^\perp})(v_K)g_{12}\|_{\ps\oti\vph\oti\vph}
\notag\\
&\leq
\ta_\om(f)
(\ps\oti\vph\oti\vph)(1\oti{}_F\De_{K^\perp}(K))^{1/2}
\cdot
(\ps\oti\vph\oti\vph)(1\oti{}_F\De_{K^\perp}(K))^{1/2}
\notag\\
&<
\de\ta_\om(f')|F|_\vph.
\label{eq:X5}
\end{align}

Finally for $X_6$, we have
\begin{align}
&|(\ps\oti\vph\oti\vph)(g_{12}^*X_6g_{12})|
\notag\\
&=|(\ps\oti\vph\oti\vph)(g_{12}^*
(\mu_{\al_K^\om(f)}^*\oti F\oti1)
(\id\oti{}_F\De_{K^\perp})(v\al_K^\om(f))
g_{12})|
\notag\\
&=
\ta_\om(f)
|(\ps\oti\vph\oti\vph)(g_{12}^*
(\mu_{\al_K^\om(f)}^*\oti F\oti1)
(\id\oti{}_F\De_{K^\perp})(v_K)
g_{12})|
\notag\\
&\leq
\ta_\om(f)
\|(1\oti{}_F\De_{K^\perp}(K))
(\mu_{\al_K^\om(f)}\oti F\oti1)
g_{12}
\|_{\ps\oti\vph\oti\vph}
\notag\\
&\quad\cdot
\|(\id\oti{}_F\De_{K^\perp})(v_K)g_{12}\|_{\ps\oti\vph\oti\vph}
\notag\\
&\leq
\ta_\om(f)
(\ps\oti\vph\oti\vph)(1\oti{}_F\De_{K^\perp}(K))^{1/2}
\cdot
(\ps\oti\vph\oti\vph)(1\oti{}_F\De_{K^\perp}(K))^{1/2}
\notag\\
&<
\de\ta_\om(f')|F|_\vph.
\label{eq:X6}
\end{align}

From (\ref{eq:cg'}),
(\ref{eq:X1}), (\ref{eq:X24}), (\ref{eq:X3}),
(\ref{eq:X5}) and (\ref{eq:X6}),
we have obtained
\[
|c_{E'}^g|-|c_E^g|
=
-\ta_\om(f')|c_E^g|
+
5\de\ta_\om(f')|F|_\vph
<
5\de^{1/2}|F|_\vph(b_{E'}-b_E).
\]
Likewise,
we can prove that $|d_{E'}^g|-|d_E^g|<5\de^{1/2}|F|_\vph(b_{E'}-b_E)$.
\end{proof}

\begin{thm}\label{thm:coc-van}
Let $\al,\be,\th$ be as in Corollary \ref{cor:coc-per}.
Let $F\in \lhG$ be a central projection
and $K$ an $(F,\de)$-invariant central projection.
Let $\mG$ be a finite subset of $U(M\oti \lhG)$.
If $\al$ and $\be$
satisfy the following inequality
for a positive number $\vep$ and a finite subset $\Ps\subs M_*$,
\[
\|\be_\orho(\chi)-\al_\orho(\chi)\|<\vep/|K|_\vph
\quad\mbox{for all }\rho\in\mK,\chi\in \Ps,
\]
then there exists a unitary $w\in M$ such that
\[
\|v_F\al_F^\om(w)-w\oti F\|_{g(\ps\oti\vph)g^*}^\sharp
<5\de^{1/8}\|F\|_\vph,
\quad
\|[w,\chi]\|<\vep
\quad\mbox{for all }g\in\mG,\chi\in \Ps.
\]
\end{thm}
\begin{proof}
Let $\meS$ be the set of families $E\in\meJ$
satisfying $|c_E^g|\leq 5\de^{1/2}|F|_\vph b_E$
and $|d_E^g|\leq 5\de^{1/2}|F|_\vph b_E$
for all $g\in\mG$.
In $\meS$, the order $E\leq E'$ is given
so that $E'$ satisfies (1), (2) and (3) in the previous lemma.
Then we see that $\meS$ is inductive,
and we can take a maximal element $\ovl{E}$ thanks to Zorn's lemma.
Let $p:=1-(\id\oti\vph)(\ovl{E})$ and $E:=E+p\oti e_\btr$.
We estimate $c_E$ and $d_E$ as follows.
Let
\begin{align*}
Z_1&:=
(\mu_{\ovl{E}}^*\oti1\oti1)
(\id\oti{}_F\De)(v)(\al^\om(\ovl{E})-(\id\oti{}_F\De)(\ovl{E})),
\\
Z_2&:=
(p\oti1\oti1)
(\id\oti{}_F\De)(v)(\al^\om(\ovl{E})-(\id\oti{}_F\De)(\ovl{E})),
\\
Z_3&:=
(\mu_{\ovl{E}}^*\oti1\oti1)
(\id\oti{}_F\De)(v)(\al^\om(p)\oti e_\btr-p\oti{}_F\De(e_\btr)),
\\
Z_4&:=
(p\oti1\oti1)
(\id\oti{}_F\De)(v)(\al^\om(p)\oti e_\btr-p\oti{}_F\De(e_\btr)),
\end{align*}
and then
\[
c_{E}=(\ps\oti\vph\oti\vph)(g_{12}^*(Z_1+Z_2+Z_3+Z_4)g_{12}).
\]

For $Z_1$, we have the following
\begin{equation}
|(\ps\oti\vph\oti\vph)(g_{12}^*Z_1g_{12})|
=|c_{\ovl{E}}|
<6\de^{1/2}|F|_\vph.
\label{eq:Z1}
\end{equation}

Next, for $Z_2$,
\begin{align*}
Z_2
&=
(p\oti1\oti1)
(\id\oti{}_F\De)(v)(\al^\om(\ovl{E})-(\id\oti{}_F\De)(\ovl{E}))
\\
&=
(p\oti1\oti1)v_F\al_F^\om(v\ovl{E})
-
(p\oti1\oti1)(\id\oti{}_F\De)(v\ovl{E}))
\\
&=
(p\oti1\oti1)v_F\al_F^\om(v\ovl{E}),
\end{align*}
and we get
\begin{align}
|(\ps\oti\vph\oti\vph)(g_{12}^*Z_2g_{12})|
&\leq
|(\ps\oti\vph\oti\vph)(g_{12}^*(p\oti1\oti1)v_F\al_F^\om(v\ovl{E})g_{12})|
\notag\\
&=
|(\ps\oti\vph)(g(p\oti1)v_F\al_F^\om(\mu_{\ovl{E}})g)|
\notag\\
&\leq
\ta_\om(p)\|gv_F\al_F^\om(\mu_{\ovl{E}})g\||F|_\vph
\notag\\
&\leq
\ta_\om(p)|F|_\vph
\notag\\
&<
\de^{1/2}|F|_\vph,
\label{eq:Z2}
\end{align}
where we have used the fact
that $p\oti1$ is centralizing $g(\ps\oti\vph)g^*$.

Since
$Z_3=
(\mu_{\ovl{E}}^*\oti1\oti1)
(v_F\al_F^\om(p)\oti e_\btr-p\oti{}_F\De(e_\btr))$
and $\mu_{\ovl{E}}^*p=0$,
we have
\begin{align}
|(\ps\oti\vph\oti\vph)(g_{12}^*Z_3g_{12})|
&=
|(\ps\oti\vph)
(g^*(\mu_{\ovl{E}}^*\oti1)
(v_F\al^\om(p)-p\oti F)g)|
\notag\\
&=
|(\ps\oti\vph)
(g^*(\mu_{\ovl{E}}^*\oti1)
v_F\al^\om(p)g)|
\notag\\
&\leq
\|v_F^*(\mu_{\ovl{E}}\oti1)g\|_{\ps\oti\vph}
\cdot
\|\al_F^\om(p)g)\|_{\ps\oti\vph}
\notag\\
&\leq
\|F\|_\vph\cdot\ta_\om(p)^{1/2}\|F\|_\vph
\notag\\
&<
\de^{1/4}|F|_\vph.
\label{eq:Z3}
\end{align}

For the last $Z_4$, we have the following computation:
\begin{align}
|(\ps\oti\vph\oti\vph)(g_{12}^*Z_4g_{12})|
&=
|(\ps\oti\vph)(g^*
(p\oti1\oti1)(v_F\al_F^\om(p)\oti e_\btr-p\oti{}_F\De(e_\btr))
g)|
\notag\\
&\leq
|(\ps\oti\vph)(g^*(p\oti1)v_F\al_F^\om(p)g)|
+
|(\ps\oti\vph)(g^*(p\oti F)g)|
\notag\\
&\leq
\|v_F^*(p\oti1)g\|_{\ps\oti\vph}
\|\al_F^\om(p)g\|_{\ps\oti\vph}
+
\ta_\om(p)|F|_\vph
\notag\\
&\leq
2\ta_\om(p)|F|_\vph
\notag\\
&<2\de^{1/2}|F|_\vph.
\label{eq:Z4}
\end{align}

Therefore, from (\ref{eq:Z1}), (\ref{eq:Z2}), (\ref{eq:Z3})
and (\ref{eq:Z4}),
we have
\[
|c_E|<10\de^{1/4}|F|_\vph.
\]
Likewise, we get
\[
|d_E|<10\de^{1/4}|F|_\vph.
\]

Now we set the Shapiro unitary $\mu:=\mu_E=(\id\oti\vph)(vE)$.
Then we have
\begin{align*}
&\|v_F\al_F^\om(\mu)-\mu\oti F\|_{g(\ps\oti\vph)g^*}^2
\\
&=
2|F|_\vph
-2\Re (\ps\oti\vph)(g^*(\mu^*\oti F)v_F\al_F^\om(\mu)g)
\\
&=
2|F|_\vph
-2\Re (\ps\oti\vph\oti\vph)
(g_{12}^*(\mu^*\oti F\oti1)
(\id\oti{}_F\De)(v)\al^\om(E)g_{12})
\\
&=
-2\Re c_E
\\
&\leq
2|c_E|<20\de^{1/4}|F|_\vph,
\end{align*}
and
\begin{align*}
&\|(v_F\al_F^\om(\mu))^*-\mu^*\oti F\|_{g(\ps\oti\vph)g^*}^2
\\
&=
2|F|_\vph
-2\Re (\ps\oti\vph)(g^*v_F\al_F^\om(\mu)(\mu^*\oti F)g)
\\
&=
2|F|_\vph
-2\Re (\ps\oti\vph\oti\vph)
(g_{12}^*
(\id\oti{}_F\De)(v)\al^\om(E)
(\mu^*\oti F\oti1)
g_{12})
\\
&=
-2\Re d_E
\\
&\leq
2|d_E|<20\de^{1/4}|F|_\vph.
\end{align*}
Thus we obtain
\begin{equation}
\|v_F\al_F^\om(\mu)-\mu\oti F\|_{g(\ps\oti\vph)g^*}^\sharp
\leq
\sqrt{20}\de^{1/8}\|F\|_\vph.
\label{eq:musharp}
\end{equation}

Next we prove the latter statement about commutativity.
Take representing sequences $E^\nu$ of $E$,
so that they are projections.
Then for $\nu\in\N$,
we have
\begin{align*}
(\id\oti\vph)(vE^\nu)
&=
(\id\oti\vph)(vE^\nu)
\\
&=
\sum_{\rho\in\mK}d(\rho)^2
(\id\oti\om_\rho)(v_\rho E_\rho^\nu)
\\
&=
\sum_{\rho\in\mK}d(\rho)^2
(\id\oti\om_\rho)(\al_\rho(v_\orho^*) E_\rho^\nu)
\\
&=
\sum_{\rho\in\mK}
(\Ph_\orho^\al\oti\vph_\rho)(v_\orho^*(1\oti {}_\orho\De_\rho(e_\btr))
\al_\orho(E_\rho^\nu))
\quad\mbox{by Lemma } \ref{lem:omrho}.
\end{align*}
Thus for $x\in M$ and $\chi\in\Ps$,
we obtain
\begin{align*}
[\chi,(\id\oti\vph)(vE^\nu)](x)
&=
\sum_{\rho\in\mK}
(\chi\circ\Ph_\orho^\al\oti\vph_\rho)
(v_\orho^*(1\oti {}_\orho\De_\rho(e_\btr))
\al_\orho(E_\rho^\nu)\al_\orho(x))
\\
&\quad
-
\sum_{\rho\in\mK}
(\chi\circ\Ph_\orho^\al\oti\vph_\rho)
(\al_\orho(x)v_\orho^*(1\oti {}_\orho\De_\rho(e_\btr))
\al_\orho(E_\rho^\nu))
\\
&=
\sum_{\rho\in\mK}
([\chi\circ\Ph_\orho^\al,v_\orho^*]\oti\vph_\rho)
((1\oti {}_\orho\De_\rho(e_\btr))
\al_\orho(E_\rho^\nu)\al_\orho(x))
\\
&\quad
+
\sum_{\rho\in\mK}
[(\chi\circ\Ph_\orho^\al\oti\vph_\rho),
((1\oti {}_\orho\De_\rho(e_\btr))
\al_\orho(E_\rho^\nu))]
(\al_\orho(x)v_\orho^*).
\end{align*}
Hence we have
\begin{align*}
&\|[\chi,(\id\oti\vph)(vE^\nu)]\|
\\
&\leq
\sum_{\rho\in\mK}
d(\rho)^2\|[\chi\circ\Ph_\orho^\al,v_\orho^*]\|
+
\sum_{\rho\in\mK}
\|
[(\chi\circ\Ph_\orho^\al\oti\vph_\rho),
(1\oti {}_\orho\De_\rho(e_\btr))
\al_\orho(E_\rho^\nu)\circ\iota_{\orho}]
\|
\\
&\leq
\sum_{\rho\in\mK}
d(\rho)^2\|[\be_\orho(\chi)-\al_\orho(\chi)]\|
+
\sum_{\rho\in\mK}
\|
[(\chi\circ\Ph_\orho^\al\oti\vph_\rho),
(1\oti {}_\orho\De_\rho(e_\btr))
\al_\orho(E_\rho^\nu)]\circ\iota_{\orho}
\|.
\end{align*}
By the property (E.\ref{item:limom}), we get
\[
\lim_{\nu\to\om}
\|[\chi,(\id\oti\vph)(vE^\nu)]\|
\leq
\sum_{\rho\in\mK}
d(\rho)^2\|[\be_\orho(\chi)-\al_\orho(\chi)]\|
<\vep.
\]

Let $\mu^\nu$ be a unitary representing sequence of $\mu$.
Then $\mu^\nu-(\id\oti\vph)(vE^\nu)\to0$
in the strong* topology as $\nu\to\om$.
Therefore the inequality above implies
\begin{equation}
\lim_{\nu\to\om}
\|[\chi,\mu^\nu]\|
=
\lim_{\nu\to\om}
\|[\chi,(\id\oti\vph)(vE^\nu)]\|
<\vep.
\label{eq:chimu}
\end{equation}
By (\ref{eq:musharp}) and (\ref{eq:chimu}),
we can take some $\nu\in\N$ so that
$w:=\mu^\nu$ has the desired properties.
\end{proof}

\subsection{Intertwining argument}
We will prove the cocycle conjugacy
of two actions
using Bratteli--Elliott--Evans--Kishimoto intertwining argument
\index{intertwining argument}
by repeated use of Corollary \ref{cor:coc-per} and Theorem \ref{thm:coc-van}.

\begin{thm}
\label{thm:free-class}
Let $\al,\be$ be centrally free actions of $\bhG$ on a McDuff factor $M$.
Suppose that there exists a homomorphism
$\th\col\IG\ra \Aut(M)$ such that
the maps $\al_\pi\circ\th_\pi^{-1}$ and $\be_\pi\circ\th_\pi^{-1}$
are approximately inner for all $\pi\in\IG$.
Then they are strongly cocycle conjugate, that is,
there exist an automorphism $\si\in\oInt(M)$ and an $\al$-cocycle
$v$ such that
\[
\Ad v\circ\al=(\si^{-1}\oti\id)\circ\be\circ\si.
\]
\end{thm}
\begin{proof}
Let $\vep_n:=1/4^n$.
We fix a faithful normal state $\ph$ as before.
Let $\{\Ps_n\}_{n=1}^\infty$ be an increasing sequence of
finite subsets of $M_*$ that is total in $M_*$.
Let $F_n,K_n$ and $\de_n$ be as before,
where we may take them as $F_0=e_\btr=K_0$ and $\de_0=1$.

Set $\Ph_0=\emptyset$,
$\mG_{-1}:=\{1\}=:\mG_0$,
$\ga^{(-1)}:=\be$, $\ga^{(0)}:=\al$, $u^0:=1=:u^{-1}$
and $\th_0:=\id_M=:\th_{-1}$.
We will inductively construct the following for $n\geq1$:
\begin{itemize}
\item
$\Ph_n$, a finite subset of $M_*$,
\item
$\mG_n$, a finite subset of $U(M)$,
\item
$\ga^{(n)}$, an action of $\bhG$ on $M$,
\item
$w_n\in U(M)$,
\item
$\th_n\in\Int(M)$,
\item
$u^n$, an $\Ad (w_n\oti1)\circ\ga^{(n-2)}\circ\Ad w_n^*$-cocycle,
\item
$\bar{u}^n$, a $(\th_n\oti\id)\circ\ga^{(\epsilon(n))}\circ\th_n^{-1}$-
cocycle, where $\epsilon$ is the quotient map from $\Z$ onto $\Z/2\Z$
that is identified with $\{0,-1\}$.
\end{itemize}
such that
\begin{enumerate}[($n$.1)]
\item
$\Ph_n=\Ps_n\cup \th_{n-2}(\Ps_n)\cup
\bigcup_{\pi\in\mF_n,i,j\in I_\pi}
\{\ph\bar{u}_{\pi_{i,j}}^{n-1},\bar{u}_{\pi_{i,j}}^{n-1}\ph\}
\cup\Ph_{n-1}$,
\label{en:Ph}

\item
$\|\ga_\rho^{(n)}(\chi)-\ga_\rho^{(n-1)}(\chi)\|<\vep_n/(2|K_{n-1}|_\vph|K_n|_\vph)$
for all $\rho\in\mK_n\cup\ovl{\mK_n}$ and $\chi\in\Ph_n$,
\label{en:ga-rho}

\item
$\|u_{F_n}^n-1\oti F_n\|_{g(\ph\oti\vph)g^*}<8\de_n^{1/8}\|F_n\|_\vph$
for all $g\in\mG_{n-2}$,
\label{en:u}

\item
$\|[w_n,\chi]\|<\vep_{n-1}$ for all $\chi\in\Ph_{n-1}$,
\label{en:w}

\item
$\bar{u}^n=u^n(w_n\oti1)\bar{u}^{n-2}(w_n^*\oti1)$,
\label{en:bu}

\item
$\ga^{(n)}=\Ad u^n\circ \Ad (w_n\oti1)\circ\ga^{(n-2)}\circ\Ad w_n^*$,
\label{en:ga}

\item
$\th_n=\Ad w_n\circ\th_{n-2}$,
\label{en:th}

\item
$\mG_n=\{1\}\cup u^n\mG_{n-2}\cup\{\bar{u}^n\}$.
\label{en:G}
\end{enumerate}

\noindent{\bf Step 1.}
Let $\Ph_1$ be as (1.\ref{en:Ph}).
By Corollary \ref{cor:coc-per},
we can take a $\ga^{(-1)}$-cocycle $v^1$ satisfying
$\|\Ad v_\rho^1(\ga_\rho^{(-1)}(\chi))-\ga_\rho^{(0)}(\chi)\|
<\vep_1/2|K_1|_\vph$
for all $\rho\in \mK_1\cup\ovl{\mK_1}$
and $\chi\in\Ph_1$.
Then by Theorem \ref{thm:coc-van},
we obtain a unitary $w_1\in U(M)$ with
$\|v_{F_1}^1\ga_{F_1}^{(-1)}(w_1)-w_1\oti F_1
\|_{\ph\oti\vph}^\sharp<5\de_1^{1/8}\|F_1\|_\vph$.
We put $u^1:=v^1\ga^{(-1)}(w_1)(w_1^*\oti 1)$
that is an
$\Ad(w_1\oti1)\circ\ga^{(-1)}\circ\Ad w_1^*$-cocycle.
Then we set $\bar{u}^1$ as (1.\ref{en:bu}),
and (1.\ref{en:u}) holds by definition of
the norm $\|\cdot\|_{\ph\oti\vph}^{\sharp}$.
We set $\ga^{(1)}$, $\th_1$ and $\mG_1$
as (1.\ref{en:ga}), (1.\ref{en:th}) and (1.\ref{en:G}).
Then (1.\ref{en:ga-rho}) holds and so does (1.\ref{en:w})
because $\Ph_0$ is empty.

\noindent{\bf Step 2.}
Let $\Ph_2$ be as (2.\ref{en:Ph}).
By Corollary \ref{cor:coc-per},
we can take a $\ga^{(0)}$-cocycle $v^2$ satisfying
$\|\Ad v_\rho^2(\ga_\rho^{(0)}(\chi))-\ga_\rho^{(1)}(\chi)\|
<\vep_2/(2|K_1|_\vph|K_2|_\vph)$
for all $\rho\in \mK_2\cup\ovl{\mK_2}$
and $\chi\in\Ph_2$.
The condition (1.\ref{en:ga-rho}) implies
$\|\Ad v_\rho^2(\ga_\rho^{(0)}(\chi))
-\ga_\rho^{(0)}(\chi)\|
<\vep_1/|K_1|_\vph$
for all $\rho\in \mK_1\cup\ovl{\mK_1}$
and $\chi\in\Ph_1$.
Then by Theorem \ref{thm:coc-van},
we get a unitary $w_2\in U(M)$ with
$\|v_{F_2}^2\ga_{F_2}^{(0)}(w_2)-w_2\oti F_2
\|_{\ph\oti\vph}^\sharp<5\de_2^{1/8}\|F_2\|_\vph$
and (2.\ref{en:w}).
We put $u^2:=v^2\ga^{(0)}(w_2)(w_2^*\oti 1)$
that is an
$\Ad(w_2\oti1)\circ\ga^{(0)}\circ\Ad w_2^*$-cocycle.
We set $\bar{u}^1$ as (2.\ref{en:bu}).
Then (2.\ref{en:u}) holds,
and we set $\ga^{(2)}$, $\th_2$ and $\mG_2$
as (2.\ref{en:ga}), (2.\ref{en:th}) and (2.\ref{en:G}).
Then (2.\ref{en:ga-rho}) holds.

\noindent{\bf Step $\boldsymbol{n}$.}
Suppose that we have done the Step $(n-1)$.
We set $\Ph_n$ as ($n$.\ref{en:Ph}).
By Corollary \ref{cor:coc-per},
we obtain a $\ga^{(n-2)}$-cocycle $v^n$ satisfying
$\|\Ad v_\rho^n(\ga_\rho^{(n-2)}(\chi))-\ga_\rho^{(n-1)}(\chi)\|
<\vep_n/2|K_n|_\vph$
for all $\rho\in \mK_n\cup\ovl{\mK_n}$
and $\chi\in\Ps_n$.
The condition ($n-1$.\ref{en:ga-rho}) implies
$\|\Ad v_\rho^n(\ga_\rho^{(n-2)}(\chi))
-\ga_\rho^{(n-2)}(\chi)\|
<\vep_{n-1}/|K_{n-1}|_\vph$
for all $\rho\in \mK_{n-1}\cup\ovl{\mK_{n-1}}$
and $\chi\in\Ps_{n-1}$.
Then by Theorem \ref{thm:coc-van},
we get a unitary $w_n\in U(M)$ with
$\|v_{F_n}^n\ga_{F_n}^{(n-2)}(w_n)-w_n\oti F_n
\|_{g(\ph\oti\vph)g^*}^\sharp
<5\de_n^{1/8}\|F_n\|_\vph$
for all $g\in\mG_{n-2}$
and ($n$.\ref{en:w}).
We put $u^n:=v^n\ga^{(n-2)}(w_n)(w_n^*\oti 1)$
that is an
$\Ad(w_n\oti1)\circ\ga^{(n-2)}\circ\Ad w_n^*$-cocycle.
Then ($n$.\ref{en:u}) holds,
and we set $\ga^{(n)}$, $\th_n$ and $\mG_n$
as ($n$.\ref{en:ga}), ($n$.\ref{en:th}) and ($n$.\ref{en:G}).
Then ($n$.\ref{en:ga-rho}) holds.
Thus we have finished the construction by induction.

We show the convergence of
$\dsp\lim_{m\to\infty}\th_{2m-1}$ and $\dsp\lim_{m\to\infty}\th_{2m}$
in $\Aut(M)$.
For fixed $n$, $\chi\in \Ph_n$ and $2m-1\geq n+1$,
we have
\[
\|\chi\circ\th_{2m+1}-\chi\circ\th_{2m-1}\|
=
\|\chi\circ\Ad w_{2m+1}-\chi\|
=
\|[\chi,w_{2m+1}]\|
<\vep_{2m},
\]
and
\[
\|\chi\circ\th_{2m+1}^{-1}-\chi\circ\th_{2m-1}^{-1}\|
=
\|
[\th_{2m-1}(\chi),w_{2m+1}]
\|
<\vep_{2m}.
\]
Thus $\{\th_{2m+1}\}_{m}$ is a Cauchy sequence
in the $u$-topology.
The convergence of $\{\th_{2m}\}_{m}$ is also shown
in a similar way.
Let $\bar{\th}_1:=\dsp\lim_{m\to\infty}\th_{2m+1}$
and $\bar{\th}_0:=\dsp\lim_{m\to\infty}\th_{2m}$.

It is easy to see that
$\bar{u}^{2m+1}$ and $\bar{u}^{2m}$ are
cocycles of
$(\th_{2m+1}\oti\id)\circ\ga^{(-1)}\circ\th_{2m+1}^{-1}$
and
$(\th_{2m}\oti\id)\circ\ga^{(0)}\circ\th_{2m}^{-1}$,
respectively.
We show that $\{\bar{u}^{2m+1}\}_m$ and $\{\bar{u}^{2m}\}$
are Cauchy sequences in the strong* topology.
Let $n$ be fixed and take $\chi\in \Ps_n$.
Suppose that $m$ is enough large to satisfy $2m+1\geq n+1$.
Note that we have the following inequalities:
\begin{align}
\|[w_{2m+1}\oti1,(\ph\oti\vph_{F_n})\bar{u}^{2m-1}]\|
&\leq
\sum_{\pi\in\mF_n}\sum_{i,j\in I_n}
\|[w_{2m+1}\oti1,\ph \bar{u}_{\pi_{i,j}}^{2m-1}\oti\vph_\pi e_{\pi_{i,j}}]\|
\notag\\
&\leq
\sum_{\pi\in\mF_n}\sum_{i,j\in I_n}
\|[w_{2m+1},\ph \bar{u}_{\pi_{i,j}}^{2m-1}]\|\|\vph_\pi e_{\pi_{i,j}}\|
\notag\\
&<
\vep_{2m}|F_n|_\vph,
\label{eq:w-ph}
\end{align}
and
\begin{align}
\|[w_{2m+1}\oti1,\bar{u}^{2m-1}(\ph\oti\vph_{F_n})]\|
&\leq
\sum_{\pi\in\mF_n}\sum_{i,j\in I_n}
\|[w_{2m+1}\oti1, \bar{u}_{\pi_{i,j}}^{2m-1}\ph\oti e_{\pi_{i,j}}\vph_\pi ]\|
\notag\\
&\leq
\sum_{\pi\in\mF_n}\sum_{i,j\in I_n}
\|[w_{2m+1}, \bar{u}_{\pi_{i,j}}^{2m-1}\ph]\|\|e_{\pi_{i,j}}\vph_\pi \|
\notag\\
&<
\vep_{2m}|F_n|_\vph.
\label{eq:w-ph2}
\end{align}
Then we have
\begin{align*}
&\|(\ph\oti\vph_{F_n})\bar{u}^{2m+1}
-(\ph\oti\vph_{F_n})\bar{u}^{2m-1}\|
\\
&\leq
\|(\ph\oti\vph_{F_n})(u^{2m+1}-1)
(w_{2m+1}\oti1)\bar{u}^{2m-1}(w_{2m+1}^*\oti1)\|
\\
&\quad+
\|(\ph\oti\vph_{F_n})
(w_{2m+1}\oti1)\bar{u}^{2m-1}(w_{2m+1}^*\oti1)
-(\ph\oti\vph_\pi)\bar{u}^{2m-1}\|
\\
&=
\|(\ph\oti\vph_{F_n})(u^{2m+1}-1)\|
\\
&\quad+
\|(\ph\oti\vph_{F_n})
(w_{2m+1}\oti1)\bar{u}^{2m-1}
-(\ph\oti\vph_\pi)\bar{u}^{2m-1}(w_{2m+1}\oti1)\|
\\
&<
\|(u_{F_n}^{2m+1})^*-1\oti F_n\|_{\ph\oti\vph}
+
\|[(\ph\oti\vph_{F_n}),
w_{2m+1}\oti1]\bar{u}^{2m-1}\|
\\
&\quad+
\|[w_{2m+1}\oti1,(\ph\oti\vph_{F_n})\bar{u}^{2m-1}]\|
\\
&<
8\de_{2m+1}^{1/8}\|F_{2m+1}\|_{\vph}
+
\|[w_{2m+1},\ph]\|\|\vph_{F_n}\|
+\vep_{2m}|F_n|_\vph
\quad\mbox{by }(\ref{eq:w-ph})\\
&<
8\de_{2m+1}^{1/8}\|F_{2m+1}\|_{\vph}
+2\vep_{2m}|F_n|_\vph,
\end{align*}
and
\begin{align*}
&\|\bar{u}^{2m+1}(\ph\oti\vph_{F_n})
-\bar{u}^{2m-1}(\ph\oti\vph_{F_n})\|
\\
&\leq
\|u^{2m+1}(w_{2m+1}\oti1)\bar{u}^{n-2}[w_n^*\oti1,\ph\oti\vph_{F_n}]\|
\\
&\quad+
\|u^{2m+1}
(w_{2m+1}\oti1)\bar{u}^{2m-1}\cdot(\ph\oti\vph_{F_n})\cdot(w_{2m+1}^*\oti1)
-u^{2m+1}\bar{u}^{2m-1}(\ph\oti\vph_{F_n})\|
\\
&\quad+
\|(u^{2m+1}-1)\bar{u}^{2m-1}(\ph\oti\vph_{F_n})\|
\\
&\leq
\|[w_{2m+1}^*,\ph]\||F_n|_\vph
+
\|[w_{2m+1}\oti1,\bar{u}^{2m-1}(\ph\oti\vph_{F_n})]\|
\\
&\quad+
\|u_{F_n}^{2m+1}
-1\oti F_n\|_{\bar{u}^{2m-1}\cdot(\ph\oti\vph_{F_n})\cdot(\bar{u}^{2m-1})^*}
\\
&<
\vep_{2m}|F_n|_\vph
+\vep_{2m}|F_n|_\vph
+8\de_{2m+1}^{1/8}\|F_{2m+1}\|_{\vph}
\quad\mbox{by }(\ref{eq:w-ph2})
\\
&<
8\de_{2m+1}^{1/8}\|F_{2m+1}\|_{\vph}
+2\vep_{2m}|F_n|_\vph.
\end{align*}
Therefore, $\{\bar{u}^{2m+1}\}_m$ is a Cauchy sequence
in the strong* topology.
Similarly we can show $\{\bar{u}^{2m}\}_m$ is also a Cauchy sequence.
Let $\hat{u}^1:=\dsp\lim_{m\to\infty}\bar{u}^{2m+1}$
and
$\hat{u}^0:=\dsp\lim_{m\to\infty}\bar{u}^{2m}$.
Then they are cocycles of
the actions
$(\bar{\th}_1\oti\id)\circ\be\circ\bar{\th}_1^{-1}$
and
$(\bar{\th}_0\oti\id)\circ\al\circ\bar{\th}_0^{-1}$,
respectively.

Since
we have $(n.\ref{en:ga})$,
$\ga^{(2m+1)}=\Ad \bar{u}^{2m+1}
\circ(\th_{2m+1}\oti\id)\circ\be\circ\th_{2m+1}^{-1}$
and
$\ga^{(2m)}=\Ad \bar{u}^{2m}
\circ(\th_{2m}\oti\id)\circ\be\circ\th_{2m}^{-1}$,
we get
\[
\Ad \hat{u}^0\circ(\bar{\th}_0\oti\id)\circ\al\circ\bar{\th}_0^{-1}
=
\Ad \hat{u}^1\circ(\bar{\th}_1\oti\id)\circ\be\circ\bar{\th}_1^{-1}.
\]
This proves the strong cocycle conjugacy of $\al$ and $\be$.
\end{proof}

\begin{cor}
\label{cor:cent-free-inj}
Let $\al$ and $\be$ be centrally free actions of $\bhG$
on an injective factor $M$.
If $\mo(\al_\pi)=\mo(\be_\pi)$ for all $\pi\in\IG$,
then $\al$ and $\be$ are cocycle conjugate.
\end{cor}
\begin{proof}
This immediately follows from Lemma \ref{lem:CTinj}
and the previous theorem.
\end{proof}

\begin{cor}[Model action]
Let $\al$ be a centrally free action of $\bhG$
on an injective factor $M$.
Let $\th\col\IG\ra \Aut(M)$ be a homomorphism
with $\mo(\th_\pi)=\mo(\al_\pi)$
and $\be$ a free action of $\bhG$ on the injective type II$_1$ factor
$\meR_0$.
Then $\al$ and $\th\oti \be$ are cocycle conjugate.
\end{cor}

\section{Related problems}

\subsection{1-cohomology of an ergodic flow}
Let $(F_t)_{t\in\R}$ be an ergodic flow
on a measure space $X$.
Let $c\col X\times\R\ra K$ be a cocycle whose target space
is a compact group $K$.
Then $c(t)$ is regarded as an element of $L^\infty(X)\oti L(K)$
satisfying
\[
c(s)(\th_s\oti\id)(c(t))=c(s+t),
\quad
(\id\oti\De)(c)=c_{12}c_{13},
\]
where $\th_s\in\Aut(L^\infty(X))$
is defined by $\th_s(f)(x)=f(F_{-s} x)$ for $f\in L^\infty(X)$ and $x\in X$.

\begin{prob}
Let $c,c'\col X\times\R\ra K$ be cocycles.
Suppose that they are conjugate as $\th\oti\id$-cocycles.
Does there exist a Borel function $a\col X\ra K$ such that
\[
c'(x,t)=a(x)c(x,t)a(F_{-t} x)^*,
\quad \mbox{a.e. }x\in X?
\]
\end{prob}

\subsection{Invariants for general actions}
We have classified actions with normal modular part.
To a classification of general actions,
we present the following conjecture.
\begin{conj}
Let $\al$ be an action of $\bhG$ on an injective factor $M$.
We introduce a $\bhG^\opp$-action $\ga^\al$
on $R_\al:=\tal(\tM)'\cap(\tM\rti_\tal\bhG)$
that is defined by
\[
\ga_\pi^\al(x)=(\la_\pi^\tal)^*(x\oti1)\la_\pi^\tal
\quad\mbox{for }x\in R_\al.
\]
Let $\th$ be the dual $\R$-action on $\tM$.
If there exists an isomorphism $\Ps\col R_\al\ra R_\be$
satisfying the following conditions:
\begin{itemize}
\item
$\Ps|_{Z(\tM)}=\id$;
\item
$\Ps\circ\th_t=\th_t\circ\Ps$;
\item
$(\Ps\oti\id)\circ\ga^\al=\ga^\be\circ\Ps$;
\item
$(\Ps\oti\id)\circ\wdh{\tal}=\wdh{\tbe}\circ\Ps$,
\end{itemize}
then $\al$ and $\be$ are strongly cocycle conjugate.
\end{conj}

Put $W_\rho:=U_\si^*\la_\si^\tal$
for $\rho$ belonging to the modular part of $\al$.
Then we can describe how $\ga_\pi$ acts on $zW_{\rho_{mn}}$
for $z\in Z(\tM)$ and $m,n\in I_\rho$
as follows:
\begin{align}
&\ga_{\pi_{ij}}(zW_{\rho_{mn}})
\notag\\
&=
\sum_{k,\el,p,q,r,s,\si,S}
S_{\opi_\el \rho_p \pi_k}^{\si_q}
\ovl{
S_{\opi_i \rho_n \pi_j}^{\si_r}
}
\tal_{\opi_{k\el}}(zU_{\rho_{pm}}^*)U_{\si_{qs}}W_{\si_{sr}}
\notag\\
&=
\sum_{\si,S}
\vep_{\rho_m}^*
d(\pi)^{1/2}
T_{\opi,\pi}^*
\tal_\opi((z\oti1_\rho)U_\rho^*)
S_{\opi\rho\pi}^\si U_\si W_{\si}(S_{\opi\rho\pi}^\si)^*
(\vep_{\opi_i}\oti\vep_{\rho_n}\oti\vep_{\pi_j}),
\label{eq:gazW}
\end{align}
where $\si\prec\opi\rho\pi$ belongs to the modular part of $\al$,
and $S$ is an element of $\ONB(\si,\opi\rho\pi)$.
For a discrete group,
the above is written like
\begin{align*}
\ga_g(zW_{m})
&=\tal_{g^{-1}}(zU_m^*)U_{g^{-1}mg}W_{g^{-1}mg}
\\
&=
\mo(\al_g)^{-1}(z\la(m,g))W_{g^{-1}mg}.
\end{align*}
Hence
the isomorphism $\Ps$ provides us with equal invariants,
that is, the Connes--Takesaki module,
the characteristic invariant and a modular cocycle,
and those actions are strongly cocycle conjugate
from the main result of \cite{KtST}.

When we put $\rho=\btr$ at the formula above,
the coefficient of $\ga_{\pi_{ij}}(zW_{\btr})$ at $W_\btr$
is equal to $\de_{i,j}\Ph_\pi^\al(z\oti1)$.
Therefore,
if the covariant systems $R_\al$ and $R_\be$ are isomorphic,
then we obtain $\Ph_\pi^\al(z\oti1)=\Ph_\pi^\be(z\oti1)$.
This observation leads us to the following problem.

\begin{prob}
Let $\rho,\si$ be endomorphisms with finite index
on an injective factor $M$.
Assume that $\ph_{\trho}=\ph_{\tsi}$ on $Z(\tM)$ and $d(\rho)=d(\si)$.
Does there exist a sequence of unitaries $\{u^\nu\}_\nu$ in $M$
such that $\dsp\ph_\rho=\lim_{\nu\to\infty}\ph_\si\circ\Ad u^\nu$?
\end{prob}

Let us consider
when the modular part of $\al$ generates a normal subcategory
of $\IG$.
In (\ref{eq:gazW}),
we have
\[
T_{\opi,\pi}^*\tal_\opi(U_\rho^*)S_{\opi\rho\pi}^\si U_\si
=
T_{\opi,\pi}^* a_{\opi,\pi}^\al (a_{\opi,\rho}^\al)^*
(a_{\opi\rho,\pi}^\al)^*
S_{\opi\rho\pi}^\si,
\]
where $a^\al$ denotes the invariant introduced
in Lemma \ref{lem:ad}.
Since $\Ps$ maps $Z(\tM)$ identically and
$\sum_\si S_{\opi\rho\pi}^\si(S_{\opi\rho\pi}^\si)^*=1$,
we have
\[
T_{\opi,\pi}^* a_{\opi,\pi}^\al (a_{\opi,\rho}^\al)^*
(a_{\opi\rho,\pi}^\al)^*
=
T_{\opi,\pi}^* a_{\opi,\pi}^\be (a_{\opi,\rho}^\be)^*
(a_{\opi\rho,\pi}^\be)^*.
\]
This equality looks weaker than $a^\al=a^\be$.
Therefore, we might say that
the invariant proposed above is more implicit than
the pair $(a^\al,c^\al)$.

\section{Appendix}

In this section,
we collect some basic facts about index theory
for inclusions of von Neumann algebras.
All von Neumann algebras treated in this section
have separable preduals,
and a conditional expectation is
assumed to be faithful and normal.

\subsection{Index theory for von Neumann algebras}
\label{sect:index}
Let $N\subs M$ be an inclusion of
von Neumann algebras.
We say that a conditional expectation $E\col M\ra N$
has \emph{finite probability index}
\index{index!finite probability --}
if there exists a positive constant $c\in\R$,
such that
the map $E-c\id_M$ on $M$ is completely positive.
We let $\Ind_p(E)$ be the minimum
of invertible $\mu\in Z(M)_+$ such that $E-\mu^{-1}\id_M$
is completely positive.
It is trivial that $1\leq \Ind_p(E)$.
The following result is elementary.

\begin{lem}\label{lem:Ap}
Let $\{p_i\}_{i\in I}$ be a finite partition of unity in
a von Neumann algebra $A$.
If $p_i Ap_i$ is finite and of type I for each $i\in I$,
then so is $A$.
\end{lem}

\begin{lem}\label{lem:Bab}
Let $B\subs A$ be an inclusion of von Neumann algebras
with a conditional expectation $E$ of finite probability index.
If $B\subs Z(A)$, then $A$ is finite and of type I.
\end{lem}
\begin{proof}
Regarding $B$ as the function algebra $L^\infty(X,\mu)$,
we obtain the disintegration of the inclusion $B\subs A$
as follows:
\[
\int_{X}^\oplus B_x\,d\mu(x)\subs \int_X^\oplus A_x\,d\mu(x).
\]
Since $E=\id$ on $Z(B)$,
we have a measurable field $\{E_x\}_x$ such that $E_x\col A_x\ra B_x=\C$
is a conditional expectation for each $x$.
It is easy to see that
$E_x$ inherits the Pimsner--Popa inequality of $E$.
Thus $A_x$ must be finite dimensional,
and $A$ is finite and of type I.
\end{proof}

\begin{prop}\label{prop:typeI}
Let $B\subs A$ be an inclusion of von Neumann algebras
with a conditional expectation $E$ of finite probability index.
If $B$ is finite and of type I,
then so is $A$.
\end{prop}
\begin{proof}
We may and do assume that $B$ is of type I$_n$ with $n\in\N$.
Indeed,
let $z_n\in B$ be the central projection
such that $Bz_n$ is of type I$_n$.
Then $\{z_n\}_n$ is a partition of unity.

For a fixed $n_0$, we will show
there exist finite $m$'s such that $z_m A z_{n_0}\neq\{0\}$.
Let $J_{n_0}$ be the set of such $m$
and $I$ a finite subset of $J_{n_0}$.
Take a partial isometry $v_m\in A$ such that
$v_m^* v_m\leq z_{n_0}$ and $v_m v_m^*\leq z_m$.
Then we put
$q:=|I|^{-1}\sum_{m,n\in I}e_{m,n}\oti v_m v_n^*$
that is a positive operator in $M_{|I|}(\C)\oti A$.

By the complete positivity of $E-c\id_M$,
we have
\[
\frac{1}{|I|}
\sum_{m,n\in I}e_{m,n}\oti E(v_m v_n^*)
=
(\id\oti E)(q)
\geq
c q.
\]
Multiplying $1\oti z_m$ from the both side in the above,
we obtain the inequality
$|I|^{-1}
e_{m,m}\oti E(v_m v_m^*)\geq c e_{m,m}\oti v_m v_m^*$.
This implies that $c^{-1}\geq|I|$.
Thus $|J_{n_0}|<c^{-1}$,
and by Lemma \ref{lem:Ap},
we can assume that $B$ is of type I$_n$.

Now let us suppose that $B$ is of type I$_n$.
Again by Lemma \ref{lem:Ap},
we may and do assume that $B$ is abelian.
Applying Lemma \ref{lem:Bab} to the inclusion $B\subs B'\cap A$,
we see $B'\cap A$ is finite and of type I.
We further reduce the inclusion $B\subs A$
by an abelian projection in $B'\cap A$,
we may and do assume $C:=B'\cap A$ is abelian.
Again by Lemma \ref{lem:Bab},
we have the disintegration over the measure space $\{X,\mu\}$
\[
\int_{X}^\oplus B_x\,d\mu(x)\subs \int_X^\oplus C_x\,d\mu(x)
\]
with a measurable field of conditional expectations $\{E_x\}_x$.
Since $B_x=\C$ and $C_x$ is commutative,
$C_x$ is isomorphic to $\C^n$ with $n\leq \Ind_p(E_x)\leq\Ind_p(E)$.
Then we can decompose the inclusion
$B\subs C$ to the finite direct sum of $Bz_n\subs Cz_n$
such that $z_n$ is a central projection in $B$ satisfying $\dim Cz_n=n$.
Thus by reducing,
we may and do assume that $B=C$ from Lemma \ref{lem:Ap}.
In particular, $Z(A)$ is contained in $B$.

Let us take a measure space $\{Y,\nu\}$ so that $Z(A)=L^\infty(Y,\nu)$.
Then by disintegration, we have
\[
\int_{Y}^\oplus B_y\,d\nu(y)\subs \int_Y^\oplus A_y\,d\nu(y).
\]
with a measurable field of conditional expectations $\{E_y\}_y$.
Since each $A_y$ is a factor,
it suffices to show the statement when $A$ is a factor.

Let $\{p_i\}_{i=1}^n$ be a partition of unity in $B$.
Since $A$ is a factor, there exists a family of
non-zero partial isometries
$v_i\in A$ such that $p_1\geq v_i^* v_i$ and $p_i\geq v_iv_i^*$.
We let $q:=n^{-1}\sum_{i,j}e_{ij}\oti v_iv_j^*$, where
$e_{ij}$ is a matrix unit of $M_n(\C)$.
Using $E(v_iv_j^*)=0$ for $i\neq j$, and
the complete positivity of $E-c\id_M$,
we have
\[
\frac{1}{n}\sum_{i=1}^n e_{ij}\oti E(v_iv_i^*)
=(\id\oti E)(q)
\geq c q.
\]
Since $E(v_iv_i^*)\leq p_i$, we obtain $n\leq c^{-1}$.
Thus $B$ is finite dimensional,
and we may and do assume that those $p_i$'s are minimal.
The conditional expectation $E\col p_i Ap_i\ra \C p_i$
has finite probability index, and $p_i Ap_i$ is a finite
factor of type I.
Then Lemma \ref{lem:Ap} implies $A$ is finite and of type I.
\end{proof}

Let $N\subs M$ be an inclusion of separable von Neumann algebras
with a conditional expectation $E$ of finite probability index.
Suppose that $Z(N)$ contains $Z(M)$.
If we realize $Z(M)$ as $L^\infty(X,\mu)$,
then we have the disintegration over $X$ as follows:
\[
\int_X^\oplus N_x\, d\mu(x)\subs \int_X^\oplus M_x\, d\mu(x).
\]
Since the restriction of $E$ on $Z(M)$ is trivial,
we have a measurable field $\{E_x\}_x$ such that $E_x\col M_x\ra N_x$
is a conditional expectation.
Note that each $M_x$ is a factor, and $\Ind_p(E_{x})\in [1,\infty]$.

\begin{lem}
The function $X\ni x\mapsto \Ind_p(E_x)$ is measurable.
\end{lem}
\begin{proof}
Let $a\geq1$ and $A:=\{x\in X\mid \Ind_p(E_x)\leq a\}$.
For $x\in X$,
we set $T_x^n:=\id\oti E_x-a^{-1}\id_{M_n(M_x)}$.
If we set $A_n:=\{x\in X\mid T_x^n
\mbox{ is positive}\}$,
then $A=\bigcap_n A_n$.
Thus it suffices to prove that $A_n$ is measurable for each $n$.
Let $\{b_m\}_{m=1}^\infty\in M_n(M)_+$ be a measurable field of operators
such that $\{b_m(x)\}_m$ are weakly dense in $M_n(M_x)_+$ for almost every $x$.
Let $\{\xi_k\}_{k=1}^\infty$ be a measurable field of vectors in $H$
such that $\{\xi_k(x)\}_{k=1}^\infty$ is dense in $\C^n\oti H_x$
for almost every $x$.
Then $x\in A_n$ if and only if
$\langle T_x^n (b_m)\xi_k(x), \xi_k(x)\rangle \geq0$ for all $k,m\in\N$.
Since the function $x\mapsto \langle T_x^n(b_m)\xi_k(x), \xi_k(x)\rangle$
is measurable, we see $A_n$ is measurable.
\end{proof}

\begin{thm}\label{thm:Indp}
The probability index $\Ind_p(E)$ has the following decomposition:
\[
\Ind_p(E)=\int_X^\oplus \Ind_p(E_x)\,d\mu(x).
\]
\end{thm}
\begin{proof}
We let $\la:=\Ind_p(E)$, which belongs to the extended positive part
of $Z(M)$.
Let $\{b_m\}_{m=1}^\infty\in M_n(M)_+$ be a measurable field of operators
such that $\{b_m(x)\}_m$ are weakly dense in $M_n(M_x)_+$ for almost every $x$.
The inequality $(\id\oti E)(b_m)\geq \la^{-1} b_m$ yields
$(\id\oti E_x)(b_m(x))\geq \la(x)^{-1} b_m(x)$ for almost every $x$.
Thus $E_x-\la(x)^{-1}\id_{M_x}$ is completely positive,
and $\Ind_p(E_x)\leq \la(x)$ for almost every $x$.

To prove the converse inequality,
we let $\la_n:=\la-1/n$ and $z_n:=1_{A_n}\in Z(M)$.
By the definition of $\la$,
$E-\la_n^{-1}\id_M$ is not completely positive on $Mz$.
Let $A_n:=\{x\in X\mid \la(x)-1/n>\Ind_p(E_x)\}$.
Then we see that $A_n$ is measurable employing the previous lemma.
However, when $x\in A_n$, we see $E_x-\la_n(x)^{-1}\id_{M_x}$
is completely positive.
Thus $E-\la_n^{-1}\id$ is completely positive on $Mz_n$,
which is a contradiction unless $z_n=0$, that is, $\mu(A_n)=0$.
Since the set $\{x\in X\mid \Ind_p(E_x)<\la(x)\}$ is the union
of $A_n$'s, we are done.
\end{proof}



Let us recall the index of a conditional expectation
defined by Kosaki in \cite{K}.
Let $E\col M\ra N$ be a conditional expectation,
and $E^{-1}\col N'\ra M'$ the dual operator valued weight
introduced in \cite{H1,H2}.
Then the Kosaki index $\Ind(E)$ is given by $E^{-1}(1)$,
which may be an unbounded operator affiliated with $Z(M)_+$.
Readers are referred to \cite{K,K2} for detail.

\begin{lem}\label{lem:finite-factor}
Let $N\subs M$ be an inclusion of von Neumann algebras
such that $M$ is a factor.
If $E$ is a conditional expectation from $M$ onto $N$
with finite probability index,
then $\Ind(E)$ is finite.
\end{lem}
\begin{proof}
By the similar technique in the proof of Proposition \ref{prop:typeI},
we see $Z(N)$ is finite dimensional.
Let $\{z_i\}_{i=1}^n$ be the partition of unity in $Z(N)$
such that they are minimal.
Then the reduced expectation $E_{z_i}\col z_iMz_i\ra Nz_i$
has finite probability index,
and $\Ind(E_{z_i})$ is finite because $Nz_i$ is a subfactor.
Although $N$ is not assumed to be a factor,
the proof of \cite[Proposition 4.2]{K} is applicable if we set $E(z_i)^{-1}=z_i$.
Then $\Ind(E_{z_i})=E^{-1}(z_i)$,
and we have
$\Ind(E)=\sum_i \Ind(E_{z_i})$.
Thus $\Ind(E)$ is finite.
\end{proof}

Let $N\subs M$ be an inclusion of von Neumann algebras
such that $Z(M)\subs Z(N)$.
Suppose that $E\col M\ra N$ is a conditional expectation.
Let $X\ni x\mapsto E_x$ be a measurable field of conditional
expectations as before.
Then we have the operator valued weight
$(E_x)^{-1}\col N_x'\ra M_x'$
associated with $E_x$.
Then we have a measurable field $X\ni x\mapsto E_x^{-1}$,
which means that the function
$x\mapsto \langle E_x^{-1}(a(x))\xi(x),\xi(x)\rangle$
is measurable for all $a\in N'$ and $\xi\in H$.
The field naturally makes an operator valued weight
\[
\int_X^\oplus (E_x)^{-1}\,d\mu(x).
\]
Let us present our proof of the following result due to Isola
\cite[Theorem 3.4]{Is} for readers' convenience.

\begin{thm}[Isola]
The following equality holds:
\begin{equation}
\label{eq:opv}
E^{-1}=\int_X^\oplus (E_x)^{-1}\,d\mu(x).
\end{equation}
\end{thm}
\begin{proof}
Let $e_N\col L^2(M)\ra L^2(N)$ be the Jones projection.
Then $N'$ is the weak closure of the $*$-subalgebra
$M'e_NM'$.
Since $e_N$ is commuting with $N$,
it is a diagonal operator.
Thus we have
\[
e_N=
\int_X^\oplus e_N(x)\,d\mu(x).
\]
The projection $e_N(x)$ is naturally regarded
as the Jones projection from $ L^2(M_x)$ onto $L^2(N_x)$.
Let $T\col N'\ra M'$ be the operator valued weight
in the right hand side of (\ref{eq:opv}).
Since $E^{-1}(e_N)=1$ and $(E_x)^{-1}(e_N(x))=1$ by definition,
we see $E^{-1}$ coincides with $T$ on $M'e_NM'$.

Let $\vph$ be a faithful normal state on $M'$.
By \cite[Theorem VIII.4.8]{Ta},
we obtain a measurable field of states $X\ni x\mapsto \vph_x\in
(M_x)_*$.
The map $X\ni x\mapsto \vph_x\circ (E_x)^{-1}$ gives
a measurable field of weights,
and
\[
\si_t^{\vph\circ T}=\int_X^\oplus \si_t^{\vph_x\circ (E_x)^{-1}}
\,d\mu(x).
\]
Since $\si_t^{\vph\circ E^{-1}}(e_N)=e_N$ and
$\si_t^{\vph_x\circ(E_x)^{-1}}(e_N(x))=e_N(x)$,
we see $\si_t^{\vph\circ E^{-1}}=\si_t^{\vph_x\circ(E_x)^{-1}}$ on
$M'e_NM'$.
Thanks to \cite[Proposition VIII.3.16]{Ta},
we have $\vph\circ E^{-1}=\vph\circ T$, and $E^{-1}=T$.
\end{proof}

\begin{cor}[Isola]\label{cor:exp-int}
Let $N\subs M$ be an inclusion of von Neumann algebras
such that $Z(M)\subs Z(N)$.
For a conditional expectation $E\col M\ra N$,
we have
\[
\Ind(E)=\int_X^\oplus \Ind(E_x)\,d\mu(x).
\]  
\end{cor}

\begin{lem}\label{lem: P-infinite}
Let $N,M,E$ be as before.
If $\Ind E=\infty$,
one has $\Ind_p(E)=\infty$.
\end{lem}
\begin{proof}
The previous result
implies
$\Ind(E_x)=\infty$ for almost every $x$.
Lemma \ref{lem:finite-factor} implies
that $\Ind_p(E_x)=\infty$ for almost every $x$.
Thus we obtain $\Ind_p(E)=\infty$ by Theorem \ref{thm:Indp}.
\end{proof}

\begin{thm}
Let $N\subs M$ be an inclusion of von Neumann algebras
with
a conditional expectation $E\col M\ra N$.
If $Z(M)\subs Z(N)$,
then the following holds:
\[
1\leq \Ind_p(E)\leq \Ind(E)\leq\Ind_p(E)^2.
\]
\end{thm}
\begin{proof}
Thanks to Theorem \ref{thm:Indp} and Corollary \ref{cor:exp-int},
we may and do assume that $M$ is a factor.
Then $\Ind(E)=\infty$ implies $\Ind_p(E)=\infty$
from Lemma \ref{lem: P-infinite}.
When $\Ind(E)<\infty$, then $\Ind_p(E)\leq\Ind(E)$
by the same computation in \cite{K2}.

We show the inequality $\Ind(E)\leq\Ind_p(E)^2$.
It turns out that $\dim Z(N)\leq\Ind_p(E)$ by the proof
of Proposition \ref{prop:typeI}.
Let $\{z_i\}_{i=1}^n$ be the partition of unity in $Z(N)$
such that $z_i$'s are minimal,
and $E_i\col z_iMz_i\ra Nz_i$ the conditional expectation
associated with $E$.
Then $\Ind(E_i)=\Ind_p(E_i)$ because $Nz_i$ is a subfactor.

We let $F(x):=\sum_{i}z_ixz_i$ that is a conditional
expectation from $M$ onto $P:=\sum_i z_i Mz_i$.
Then $E=E|_P\circ F$,
and we have $\Ind(E)=F^{-1}(E|_P^{-1}(1))$.

Note that the inclusion $N\stackrel{E|_P}{\subs} P$
has common center,
and that it is isomorphic to the direct sum of
$Nz_i\stackrel{E_i}{\subs}z_i Mz_i$.
Hence $E|_P^{-1}(1)=\sum_i \Ind(E_i)z_i$,
where $\Ind(E_i)\in\R$.
This implies $\Ind(E)=\sum_i \Ind(E_i)F^{-1}(z_i)$.

We will show that $F^{-1}(z_i)=1$.
The Jones projection $e_P$ for the inclusion
$P\subs M$ is given by $e_P=\sum_i z_i Jz_iJ$,
where $J$ is the modular conjugation on the left $M$-module $L^2(M)$.
Since $(\id_{B(\el_2)}\oti F)^{-1}=\id_{B(\el_2)}\oti F^{-1}$,
we may and do assume that $M$ is properly infinite,
and $z_i$ are infinite.
Then there exist partial isometries $\{v_i\}_{i=1}^n$ in $M$
such that $v_i^*v_i=z_1$ and $v_iv_i^*=z_i$.
It is easy to see that $Mp_iM=M$, and $Me_P M=M_1$.
Hence $F^{-1}$ is given by $F^{-1}(x)z_1=\sum_{i}v_i^*xv_i$.
This implies $F^{-1}(z_i)=1$.
Therefore, we have $\Ind(E)=\sum_i\Ind(E_i)=\sum_i\Ind_p(E_i)$.

By definition of $\Ind_p(E)$,
we have $E_i-\Ind_p(E)^{-1}\id_{z_i Mz_i}$
is completely positive on $z_iMz_i$.
Hence $\Ind_p(E_i)\leq \Ind_p(E)$.
Since $n\leq\Ind_p(E)$,
we have
\[
\Ind(E)=\sum_i\Ind_p(E_i)\leq n\Ind_p(E)\leq \Ind(E_p)^2.
\]
\end{proof}

The following result is an immediate consequence from
the previous theorem.
\begin{cor}\label{cor:dualexp}
Let $N\subs M$ be an inclusion of von Neumann algebras
such that $Z(M)\subs Z(N)$,
and $E\col M\ra N$ a conditional expectation
with finite probability index.
Then $\Ind(E)$ is a positive invertible element in $Z(M)$.
In particular,
the map $E^{-1}(1)^{-1} E^{-1}(\cdot)$
is a conditional expectation from $N'$ onto $M'$.
\end{cor}

\subsection{Index of an endomorphism}
\label{subsect:index}

\begin{defn}\label{defn:irr-co}
Let $\rho$ be an endomorphism on $M$.
We will say that
\begin{itemize}
\item
$\rho$ has
\emph{finite probability index}
\index{endomorphism!of finite probability index}
if the inclusion $\rho(M)\subs M$ does;
\item
$\rho$ is \emph{irreducible}
\index{endomorphism!irreducible--}
if
$\rho(M)'\cap M=\rho(Z(M))$;
\item
$\rho$ is \emph{co-irreducible}
\index{endomorphism!co-irreducible--}
if
$\rho(M)'\cap M=Z(M)$;
\item
$\rho$ \emph{preserves}
\index{endomorphism!preserving center}
the center $Z(M)$
if
$\rho(Z(M))=Z(M)$.
\end{itemize}
\end{defn}

Note that these are the properties as a sector.
The conjugate endomorphism $\orho$ is defined
by the following bimodules:
\[
{}_{M\,\rho}L^2(M)_M\cong {}_M L^2(M)_{\orho\,M},
\]
where ${}_M L^2(M)$ denotes the standard form of $M$.
Of course, $\orho$ is determined up to unitary perturbations
in $\End(M)$.
For an inclusion $N\subs M\subs B(L^2(M))$ such that $N$ is properly infinite,
we define the canonical endomorphism
\index{endomorphism!canonical--}
$\ga\col M\ra N$
by $\ga(x)=J_NJ_M xJ_M J_N$.
Then the unitary $J_NJ_M$ yields the following isomorphism:
\[
{}_{M}L^2(M)_N\cong {}_{M\,\ga} L^2(N)_{N}.
\]
If $N=\rho(M)$, then
\[
{}_{M}L^2(M)_{\rho\,M}\cong {}_{M\,\ga} L^2(\rho(M))_{\rho\,M}.
\]
Using the isomorphism $\rho^{-1}\col \rho(M)\ra M$,
we have
\[
{}_{M}L^2(M)_{\rho\,M}\cong {}_{M\,\rho^{-1}\ga} L^2(M)_{M}.
\]
This means $\orho=\rho^{-1}\ga$ in $\Sect(M)$.
Note that $\rho$ is irreducible if and
only if $\orho$ is co-irreducible.

For an inclusion $N\subs M$,
let us denote by $\cE(M,N)$
the set of all conditional expectations
from $M$ onto $N$.
The following result has already appeared in a situation
of a factor in \cite[Proposition 5.1]{L1}.

\begin{lem}\label{lem:E-isom}
Let $N\subs M$ be an inclusion of von Neumann algebras
such that $\cE(M,N)$ is a non-empty set.
Then there exists an injective map $w$ from $\cE(M,N)$
into $(\id_N,\ga|_N)_{\rm is}$ such that
$E(x)=w^*\ga(x)w$ for $x\in M$,
where $(\id_N,\ga|_N)_{\rm is}$ denotes
the set of isometries
in $(\id_N,\ga|_N)$,
and $\ga$ the canonical endomorphism from $M$ into $N$.
In particular, $(\id_N,\ga|_N)_{\rm is}$ is non-empty.
\end{lem}
\begin{proof}
We represent $N\subs M$ as the standard form on $L^2(M)$
with a bi-cyclic and separating vector $\Omega\in L^2(M)$.
The modular conjugations for $N$ and $M$
are denoted by $J_N$ and $J_M$, respectively.
By $\mathcal{P}_M$, we denote the natural cone of $M$.

Let $E\in \cE(M,N)$ and take $\xi_E\in\mP$
such that $\om_{\xi_E}=\om_\Om\circ E$,
where $\om_\eta$ denotes the vector functional
$\langle\cdot\eta,\eta\rangle$ for $\eta\in L^2(M)$.
Let $e_N$ be the Jones projection
associated with $E$ and $\xi$,
that is, it projects $L^2(M)$ onto $\ovl{N\xi}$.
Then the map $N\Om\ni x\Om\mapsto x\xi\in e_N L^2(M)$
defines the isometry of $N$-$N$ bimodules
as
$\nu_E\col {}_N L^2(N)_N\ra {}_N e_N L^2(M)_N$.
The unitary $\Ga:=J_N J_M$ yields the isomorphism
from ${}_N L^2(M)_N$ onto ${}_N {}_\ga L^2(N)_N$.
Thus the isometry $w_E:=\Ga \nu_E$ is contained in $(\id_N,\ga|_N)$.
Then the map $w\col \cE(M,N)\ra (\id_N,\ga|_N)$ is clearly injective.
\end{proof}

\begin{prob}
Is the map $w$ always bijective?
\end{prob}

\begin{lem}
Let $N\subs M$ be an inclusion of von Neumann algebras
such that $\cE(M,N)$ is non-empty.
Then the following statements are equivalent:
\begin{enumerate}
\item
The inclusion $N\subs M$ is irreducible;
\item
$\cE(M,N)$ is a singleton;
\item
The $U(Z(N))$-set $(\id_N,\ga|_N)_{\rm is}$
is generated by a single element $R$, that is,
$(\id_N,\ga|_N)_{\rm is}=U(Z(N)) R$.
\end{enumerate}
\end{lem}
\begin{proof}
(1)$\Rightarrow$(2).
Thanks to \cite[Th\'{e}or\`{e}me 5.3]{CD},
the map
$\cE(M,N)\ni E\ra E|_{N'\cap M}\in \cE(N'\cap M,Z(N))$
is bijective.
Since $N'\cap M=Z(N)$, the statement (2) follows.

(2)$\Rightarrow$(1).
The bijection stated above says $\cE(N'\cap M,Z(N))$ is a singleton.
Thus $N'\cap M$ must coincide with $Z(N)$
by \cite[Th\'{e}or\`{e}me 5.6]{CD}.

(2)$\Rightarrow$(3).
Let $E\col M\ra N$ be the unique expectation.
We set $R:=w_E\in (\id_N,\ga|_N)_{\rm is}$ as in the previous lemma.
Let $w\in (\id_N,\ga|_N)_{\rm is}$.
Then the map $F(x):=w^* \ga(x)w$ defines a projection from $M$
onto $N$.
We show the faithfulness of $F$.
Let $p$ be a projection in $M$ such that
$Mp=\{x\in M\mid F(x^*x)=0\}$.
The bimodule map property of $F$ implies $p\in N'\cap M$,
and $p\in Z(N)$ since $N\subs M$ is irreducible.
Then $0=F(p)=p$, and $F$ is faithful.
Thus $E=F$, and we can construct a partial isometry $V$ on $L^2(M)$
such that $V(\ga(x)R\Om)=\ga(x)w\Om$ for all $x\in M$.
Since $w,R\in (\id_N,\ga|_N)$,
we see $V\in \ga(M)'\cap N=Z(N)$.
Hence $w=VR$, and $V$ is unitary.

(3)$\Rightarrow$(2).
Let $E_1,E_2\in \cE(M,N)$.
The map $w$ constructed in Lemma \ref{lem:E-isom} sends them
to $(\id_N,\ga|_N)$.
Hence there exists a unitary $u\in Z(N)$
such that $w_{E_1}=u w_{E_2}$.
Then it is trivial that $E_1=E_2$.
\end{proof}

The following result is a dual version of the previous.

\begin{lem}\label{lem:co-irr-single}
Let $N\subs M$ be an inclusion of von Neumann algebras
such that $\cE(N,\ga(N))$ is non-empty.
Then the following statements are equivalent:
\begin{enumerate}
\item
The inclusion $N\subs M$ is co-irreducible;
\item
$\cE(N,\ga(N))$ is a singleton;
\item
The $U(Z(M))$-set $(\id_M,\ga)_{\rm is}$
is generated by a single element $S$, that is,
$(\id_M,\ga)_{\rm is}=U(Z(M)) S$.
\end{enumerate}
\end{lem}

Applying the previous two results,
we obtain the following result.

\begin{lem}\label{lem:vw-isom}
Let $\rho$ be an endomorphism on $M$.
\begin{enumerate}
\item
If $\rho$ is irreducible and $\cE(M,\rho(M))\neq\emptyset$,
then the $U(Z(M))$-set $(\id_M,\orho\rho)$ is transitive.
\item
If $\rho$ is co-irreducible and $\cE(M,\orho(M))\neq\emptyset$,
then the $U(Z(M))$-set $(\id_M,\rho\orho)$ is transitive.
\end{enumerate}
\end{lem}

It is not hard to prove the following result.

\begin{lem}
If $\rho\in\End(M)$ preserves $Z(M)$,
then so does $\orho$.
Moreover, if $\cE(M,\rho(M))\neq\emptyset$,
then $\rho\orho=\id=\orho\rho$ on $Z(M)$.
\end{lem}

\begin{lem}
If an irreducible endomorphism $\rho$
has finite probability index,
then $\cE(M,\orho(M))\neq\emptyset$.
In particular, $(\id,\rho\orho)$ contains an isometry.
\end{lem}
\begin{proof}
Thanks to Corollary \ref{cor:dualexp},
$\cE(M,\orho(M))$ is non-empty.
Then $(\id,\rho\orho)$ contains an isometry
from Lemma \ref{lem:E-isom} (1),
\end{proof}

Let $\rho$ be as above.
Take two isometries
$R_\rho\in (\id,\rho\orho)$ and $R_\orho\in (\id,\orho\rho)$.
Then $R_\rho^*\rho(R_\orho)\in (\rho,\rho)=\rho(Z(M))$
and $R_\orho^*\orho(R_\rho)\in(\orho,\orho)=Z(M)$.

\begin{prop}
Let $\rho$ be an irreducible endomorphism on $M$
with finite probability index.
Suppose that $\rho$ is preserving the center $Z(M)$.
Then the following hold:
\begin{enumerate}
\item $\orho$ has finite probability index;
\item $R_\rho^*\rho(R_\orho)$ and $R_\orho\orho(R_\rho)$
are invertible elements in $Z(M)$.
\end{enumerate}
\end{prop}
\begin{proof}
By our assumption, 
the (unique) conditional expectation
$M\ni x\mapsto\rho(R_\orho^* \orho(x)R_\orho)\in \rho(M)$
has finite probability index.
Thus there exists a constant $\la>0$
such that $\rho(R_\orho^* \orho(x)R_\orho)\geq \la x$
for all $x\in M$.
Then we have
$\rho(R_\orho^* \orho(R_\rho R_\rho^*)R_\orho)\geq \la R_\rho R_\rho^*$.
Since $R_\orho^* \orho(R_\rho)\in Z(M)$,
$\rho(R_\orho^* \orho(R_\rho))\in Z(M)$.
Hence
\[
\rho(R_\orho^* \orho(R_\rho R_\rho^*)R_\orho)
=
R_\rho^*\rho(R_\orho^* \orho(R_\rho R_\rho^*)R_\orho)R_\rho
\geq
\la.
\]
Hence $|R_\orho^* \orho(R_\rho)|$ is invertible,
and so is $R_\orho^* \orho(R_\rho)$ because of centrality.
Also, the conjugate endomorphism $\orho$
is preserving $Z(M)$,
we can show $R_\rho^* \rho(R_\orho)$ is invertible.
\end{proof}

\begin{thm}\label{thm:isom-conj}
Let $M$ be a properly infinite von Neumann algebra.
Let $\rho$ and $\si$ be an irreducible and a co-irreducible
endomorphism on $M$ with finite index,
respectively.
Then the following statements are equivalent:
\begin{enumerate}
\item
$\orho=\si$ in $\Sect(M)$;
\item
$(\id,\si\rho)$ contains an isometry.
\end{enumerate}
Moreover, the following (3) or (4) also imply (1).
\begin{enumerate}
\addtocounter{enumi}{2}
\item
$\rho$ preserves $Z(M)$ and $(\id,\rho\si)$ contains an isometry;
\item
There exist isometries $R_\rho\in(\id,\orho\rho)$,
$\ovl{R}_\rho\in (\id,\rho\orho)$ and $w\in(\id,\rho\si)$
such that
$\orho(\ovl{R}_\rho^*)R_\rho$ is invertible element in $Z(M)$,
and the conditional expectation from $M$ onto $N$
defined by $E_\si(x)=\si(w^*\rho(x)w)$ is faithful.
\end{enumerate}
\end{thm}
\begin{proof}
(1)$\Rightarrow$(2).
It is clear from Lemma \ref{lem:vw-isom} (1).

(2)$\Rightarrow$(1).
Let $v\in(\id,\si\rho)$ be an isometry.
Take an isometry $R_\rho\in(\id,\orho\rho)$.
Then the maps $M\ni x\mapsto \rho(v^*\si(x)v)\in\rho(M)$
and $M\ni x\mapsto \rho(R_\rho^*\si(x)R_\rho)\in\rho(M)$
are both conditional expectations,
and they are equal because $\cE(M,\rho(M))$ is a singleton.
Thus we can construct a partial isometry $V$ on $L^2(M)$ as
$V(\si(x)v\Om)=\orho(x)R_\rho\Om$ for $x\in M$.
Note that $V\in M$.
Indeed, for $x$ being analytic with respect to the modular group
$\si^\Om$, we have
\[
V(J_M \si_{i/2}^\Om(x)^*J_M\si(y)v\Om)
=
V(\si(y)vx\Om)
=
V(\si(y\rho(x)\Om)
=
\orho(y\rho(x))R_\rho\Om,
\]
which is equal to $J_M \si_{i/2}^\Om(x)^*J_M V(\si(y)v\Om)$.
Thus $V\in (J_M MJ_M)'=M$.

The support projection $V^*V$ is contained in $\si(M)'\cap M=Z(M)$.
However, the computation
$v V^*V\Om=V^*V v\Om=v\Om$
shows $vV^*V=v$, and $V^*V=1$.
Similarly, we get $VV^*=1$, that is, $V$ is unitary.
It is trivial that $V$ intertwines $\si$ and $\orho$.

Next we assume (3).
The ``if '' part is clear.
We show the ``only if'' part.
Assume $v\in (\id,\rho\si)$ is an isometry.
We set $u:=R_\orho^*\orho(v)$ that is in $(\orho,\si)$.
Then $u^*u,uu^*\in Z(M)$.
Let $E_\rho\col M\ra \rho(M)$ be the unique conditional expectation
with probability index $\la\in Z(M)_+$,
which is given by $E_\rho(x)=\rho(R_\rho^*\orho(x)R_\rho)$.
Note that $\orho$ is co-irreducible, and $\orho(Z(M))\subs Z(M)$.
Since $E_\rho(uu^*)\geq\la^{-1}vv^*$ and $\orho(uu^*)\in Z(M)$,
we have $\rho(uu^*)=v^*\rho(uu^*)v\geq \la$.
Hence $uu^*\geq\la$.
Since $uu^*$ and $u^*u$ are central,
the element $|u^*|^{-1}u\in(\orho,\si)$ is a unitary.
\end{proof}

\subsection{Cocycle crossed product}\label{subsect:cocycle-cross}
We summarize some well-known facts on cocycle crossed product.
Let $M$ be a von Neumann algebra,
and $L^2(M)$ the standard Hilbert space. 
Let $(\alpha_g,v(g,h))$ be a cocycle crossed action
\index{cocycle action}
\index{action!cocycle--}
of $G$ on $M$. 
Define
$\pi(x),\lambda_g\in B(L^2(M)\otimes \ell^2(G))$, $x\in M$, $g\in G$,
as follows
\[
\left(\pi(x)\xi\right)(g)
=\alpha_{g^{-1}}(x)\xi(g),
\quad
\left(\lambda_g\xi\right)(h)
=v(h^{-1},g)\xi(g^{-1}h),
\]
where we should not confuse $\alpha_{g^{-1}}$ and $\alpha_g^{-1}$.
Then
\[
\lambda_g\pi(x)\lambda_g^*=\pi(\alpha_g(x)),
\quad
\lambda_g\lambda_h=\pi(v(g,h))\lambda_{gh}.
\] 
By definition, the twisted crossed
product von Neumann algebra
\index{twisted crossed product von Neumann algebra}
is
\[
M\rtimes_{\alpha,v}G:=\pi(M)\vee\{\lambda_g\}_{g\in G}''.
\]

\begin{thm}\label{thm:ext}
Let $(\al,v)$ and $M$ as above.
Then the following holds:
\begin{enumerate}
\item
Suppose that an automorphism $\theta$ on $M$
induces an automorphism on $G$
(denoted by the same symbol $\theta$)
such that 
\[
\theta\circ\alpha_{\theta^{-1}(g)}\circ \theta^{-1}=\Ad
 u_g^\theta\circ \alpha_g(x),
\]
\[
u_g^\theta\alpha_g(u_h^\theta)v(g,h)u_{gh}^{\theta*}=\theta(v(\theta^{-1}(g),\theta^{-1}(h)))
\]
for some unitary $u_g^\theta\in M$.
Then $\th$ uniquely extends to $M\rtimes_{\alpha, v}G$
as $\tilde{\th}$ 
such that
$\tilde{\theta}(\lambda_{\theta^{-1}(g)})=\pi(u_g^\theta)\lambda_g$. \\
\item
If we have
$u_g^{\theta\sigma}=\theta(u_{\theta^{-1}(g)}^\sigma)u_g^\theta$
for  $\theta,\sigma\in \Aut(M)$,  then
$\tilde{\theta}\circ \tilde{\sigma}=\widetilde{\theta\circ \sigma}$
holds.
\end{enumerate}
\end{thm}
\begin{proof}
(1) Denote $u_g^\theta$ by $u_g$ for simplicity.
Let $U_\theta\in B(L^2(M))$ be the standard implementing
unitary of $\theta$.
Define a unitary $W_\theta\in B(L^2(M)\otimes \ell^2(G))$ by 
\[
\left(W_\theta \xi\right)
(g)=u_{g^{-1}}^*
U_\theta \xi(\theta^{-1}(g)),\,\,\, \xi(g)\in L^2(M)\otimes \ell^2(G).
\]
Then $W_\theta^*$ is given by 
\[
\left(W_\theta \xi\right)
(\theta^{-1}(g))=U_\theta^*u_{g^{-1}}\xi(g).
\]
We show $\Ad W_\theta$ gives a desired action.  
At first, we verify $\Ad W_\theta (\pi(a))=\pi(\theta(a))$ as follows.
\begin{align*}
 \left(W_\theta\pi(a)W_\theta^*\xi\right)(g)
&=
u_{g^{-1}}^*
U_\theta\left(\pi(a)W_\theta^*\xi\right)(\theta^{-1}(g)) \\
&=
u_{g^{-1}}^*
U_\theta\alpha_{\theta^{-1}(g^{-1})}(a)\left(W_\theta^*\xi\right)(\theta^{-1}(g)) \\
&=
u_{g^{-1}}^*
U_\theta\alpha_{\theta^{-1}(g^{-1})}(a)
U_\theta^*u_{g^{-1}}\xi(g)  \\
&=
u_{g^{-1}}^*
\theta\alpha_{\theta^{-1}(g^{-1})}(a)u_{g^{-1}}\xi(g) \\ 
&=
\alpha_{g^{-1}}(\theta(a))
\xi(g)  \\
&=
\left(\pi(\theta(a))\xi\right)(g).
\end{align*}

Next we show $\Ad W_\theta(\lambda_{\theta^{-1}(g)})=\pi(u_g)\lambda_g$ as follows. 
\begin{align*}
\left( W_\theta \lambda_{\theta^{-1}(g)} W_\theta^*\xi \right)(h)
&=u_{h^{-1}}^*
U_\theta\left(\lambda_g W_\theta^*\xi
			  \right)(\theta^{-1}(h))\\
&=
u_{h^{-1}}^*U_\theta\alpha_{\theta^{-1}(h^{-1})}\left(v(\theta^{-1}(h^{-1}),
\theta^{-1}(g))\right)
\left(W_\theta^*\xi \right)(\theta^{-1}(g^{-1}h))\\
&=
u_{h^{-1}}^*U_\theta\alpha_{\theta^{-1}(h^{-1})}\left(v(\theta^{-1}(h^{-1}),
\theta^{-1}(g))\right)
U_\theta^*u_{h^{-1}g}
\xi(g^{-1}h)\\
&=
u_{h^{-1}}^*\theta\alpha_{\theta^{-1}(h^{-1})}\left(v(\theta^{-1}(h^{-1}),
\theta^{-1}(g))\right)u_{h^{-1}g}
\xi(g^{-1}h)\\
&=
u_{h^{-1}}^*u_{h^{-1}}\alpha_{h^{-1}}(u_g)
v(h^{-1},g)
\xi(g^{-1}h)\\
&=
\alpha_{h^{-1}}(u_g)
\left(\lambda_g\xi\right)(h)\\
&=
\left(\pi(u_g)\lambda_g\xi\right)(h).
\end{align*}
This shows that $\tilde{\theta}=\Ad W_\theta$  preserves $M\rtimes_{\alpha, v}G$, and 
is a desired extension. The uniqueness of $\tilde{\theta}$ is obvious.\\
(2) By the assumption, $W_\theta W_\sigma=W_{\theta\sigma}$ holds.
\end{proof}

\subsection{Dual cocycle twisting of a Kac algebra}\label{subsect:cocycle-Kac}
Let $\om\in Z^2(\bhG)$.
Then the map $\De^\om:=\Ad \om\circ \De$ gives a new coproduct on $\lhG$.
We will show the trace $\varphi$ is still bi-invariant
with respect to $\Delta^\om$.
This fact is well-known for experts, but we give a proof
for readers' convenience.
Let  $v\in U(\lhG)$ be as in the proof of Lemma \ref{lem:normalize}. 
Since
$\om_{\pi,\opi}T_{\pi,\ovl{\pi}}=
(v_\pi\otimes 1)T_{\pi,\ovl{\pi}}$, 
and
$(\Tr_{\pi}\oti\id)(T_{\pi,\ovl{\pi}}T_{\pi,\ovl{\pi}}^*)=d(\pi)^{-1}$,
we have
\[
(\Tr_\pi\otimes \id_{\ovl{\pi}})
(\om_{\pi,\ovl{\pi}}^*
T_{\pi,\ovl{\pi}}T_{\pi,\ovl{\pi}}^*\om_{\pi,\ovl{\pi}})
=
(\Tr_\pi\otimes\id_{\ovl{\pi}})
((v_\pi\oti1)T_{\pi,\ovl{\pi}}T_{\pi,\ovl{\pi}}^*(v_\pi^*\oti1))
=\frac{1}{d(\pi)}.
\]
Note that the following map from $(\pi,\si\orho)$ into $(\si,\pi\rho)$
defined by
\[
S\mapsto
\sqrt{d(\si)d(\rho)d(\pi)^{-1}}(S^*\oti1_\rho)(1_\si\oti T_{\orho,\rho})
\]
is a conjugate unitary.
Thus for $x_\si\in B(H_\si)$, we obtain
\[
{}_\pi\De_\rho(x_\si)
=
\sum_{S\in\ONB(\pi,\si\orho)}
d(\si)d(\rho)d(\pi)^{-1}
(S^*\oti1_\rho)(x_\si\oti T_{\orho,\rho}T_{\orho,\rho}^*)(S\oti1_\rho).
\]
Hence for any $x\in\lhG$,
\begin{align*}
&d(\pi)(\Tr_\pi \otimes \id_\rho)({}_\pi\Delta_{\rho}^\om(x))
\\
&=
\sum_{\si\prec\pi\rho}
\sum_{S\in\ONB(\pi,\si\orho)}
d(\si)d(\rho)
(\Tr_\pi \otimes \id_\rho)
\big{(}
\om_{\pi,\rho}
(S^*\oti1_\rho)(x_\si\oti T_{\orho,\rho}T_{\orho,\rho}^*)(S\oti1_\rho)
\om_{\pi,\rho}^*
\big{)}
\\
&=
\sum_{\si\prec\pi\rho}
\sum_{S\in\ONB(\pi,\si\orho)}
d(\si)d(\rho)
(\Tr_\pi \otimes \id_\rho)
\big{(}
(S^*\oti1_\rho)
\om_{\si\orho,\rho}(x_\si\oti T_{\orho,\rho}T_{\orho,\rho}^*)
\om_{\si\orho,\rho}^*
(S\oti1_\rho)
\big{)}
\\
&=
\sum_{\si\prec\pi\rho}
\sum_{S\in\ONB(\pi,\si\orho)}
d(\si)d(\rho)
(\Tr_{\si\orho} \otimes \id_\rho)
\big{(}
\om_{\si\orho,\rho}(x_\si\oti T_{\orho,\rho}T_{\orho,\rho}^*)
\om_{\si\orho,\rho}^*
(SS^*\oti1_\rho)
\big{)}.
\end{align*}
Then summing up with $\pi\in\IG$
and changing the summation to $\sum_{\si\in\IG}\sum_{\pi\prec\si\orho}$,
we obtain
\begin{align*}
&\sum_{\pi\in\IG}
d(\pi)(\Tr_\pi \otimes \id_\rho)(\Delta^\om_{\pi,\rho}(x))
\\
&=
\sum_{\si\in\IG}
d(\si)d(\rho)
(\Tr_{\si\orho} \otimes \id_\rho)
\big{(}
\om_{\si\orho,\rho}\om_{\si,\orho\rho}^*
(x_\si\oti T_{\orho,\rho}T_{\orho,\rho}^*)
\om_{\si,\orho\rho}
\om_{\si\orho,\rho}^*
\big{)}
\\
&=
\sum_{\si\in\IG}
d(\si)d(\rho)
(\Tr_{\si\orho} \otimes \id_\rho)
\big{(}
\om_{\si,\orho}^*\om_{\orho,\rho}
(x_\si\oti T_{\orho,\rho}T_{\orho,\rho}^*)
\om_{\orho,\rho}^*\om_{\si,\orho}
\big{)}
\\
&=
\sum_{\si\in\IG}
d(\si)d(\rho)
(\Tr_{\si\orho} \otimes \id_\rho)
\big{(}
\om_{\si,\orho}^*
(x_\si\oti T_{\orho,\rho}T_{\orho,\rho}^*)
\om_{\si\orho,\rho}
\big{)}
\quad\mbox{(since $\om$ is normalized)}
\\
&=
\sum_{\si\in\IG}
d(\si)d(\rho)
(\Tr_{\si\orho} \otimes \id_\rho)
(x_\si\oti T_{\orho,\rho}T_{\orho,\rho}^*)
\\
&=
\sum_{\si\in\IG}
d(\si)
\Tr_\si(x_\si).
\end{align*}
Namely, we have $(\vph\oti\id)(\De^\om(x))=\vph(x)$.
We can prove the right invariance of $\vph$ in a similar way.
We call $\bhG_\om:=(\lhG,\De^\om,\vph)$
an \emph{$\om$-twisted Kac algebra}.
\index{Kac algebra!$\om$-twisted--}
The dual is a compact Kac algebra,
and we denote it by $\bG_\om$.

The following result is essentially same as \cite[Lemma 26]{Wass-cptII}.

\begin{lem}\label{lem:skew-symm}
Let $\om$ be a co-commutative cocycle.
Then $\beta=\om^*F(\om)$ is a skew symmetric bicharacter,
\index{skew symmetric bicharacter}
where $F$ denotes the flip of the tensors.
Namely,
we have
\[
(\Delta\otimes \id)(\beta)=\beta_{23}\beta_{13},
\quad
(\id\otimes \Delta)(\beta)=\beta_{12}\beta_{13},
\]
\[
F(\be)=\be^*,
\quad
(\id\oti\ka)(\be)=(\ka\oti\id)(\be)=\be^*.
\]
\end{lem}
\begin{proof}
If we show the first and second equality,
then the remaining assertions are an immediate consequence.
Since $\om^*$  is a 2-cocycle for $\Delta^\om$,
we have
$(\De^\om\oti\id)(\om)\om_{12}=(\id\oti\De^\om)(\om)\om_{23}$.
Flipping the first and second tensors,
we obtain the following since $F\circ \De^\om=\De^\om$:
\[
(\De^\om\oti\id)(\om)\om_{21}=(\id\oti\De^\om)(\om)_{213}\om_{13}.
\]
From the cocycle identity,
the left hand side is equal to
$(\id\oti\De^\om)(\om)\om_{23}\om_{12}^*\om_{21}$.
The right equals
$(\id\oti\De^\om)(F(\om))_{132}\om_{13}$, respectively.
We further flip the second and third tensors.
Then
\[
(\id\oti\De^\om)(\om)\om_{32}\om_{13}^*\om_{31}
=
(\De^\om\oti\id)(F(\om))\om_{12}.
\]
We again use the cocycle identity in the left.
Then
\[
(\De^\om\oti\id)(\om)\om_{12}\om_{23}^*\om_{32}\om_{13}^*\om_{31}
=
(\De^\om\oti\id)(F(\om))\om_{12}.
\]
This implies $(\De\oti\id)(\be)=\be_{23}\be_{13}$.
Similarly, we have $(\id\oti\De)(\be)=\be_{12}\be_{13}$.
\end{proof}

Let $\om\in Z_c^2(\bhG)$ and $\be_\om:=\om^*\om_{21}$ as before.
Suppose now that $\be$ is non-degenerate.
Thus the linear span of $(\id\oti\ph)(\be_\om)$, $\ph\in\lhG_*$
is dense in $\lhG$.
We put a $*$-algebra structure on $\lhG_*$ as follows:
\[
\ph\cdot\ps:=(\ph\oti\ps)\circ\De,
\quad
\ph^*:=\ovl{\ph}\circ\ka.
\]
The map $\ph\mapsto (\id\oti\ph)(\be_\om)$
gives an injective $*$-homomorphism
from $\lhG_*$ into $\lhG$.
Thus we have
an injective $*$-homomorphism
$(\id\oti\ph)(W_\bG)\mapsto (\id\oti\ph)(\be_\om)$
from $A:=\{(\id\oti\ph)(W_\bG)\mid\ph\in\lhG_*\}$
into $\lhG$,
where
$W_\bG$ denotes the left regular representation of $\bG$.
If $\bG$ is amenable,
then the map extends to $\Ph\col\CG\ra \lhG$
with dense range.
It is straightforward to see that
$\Ph$ is a Hopf $*$-algebra homomorphism,
but $\Ph$ does not extend to $\lG$ as a normal map
unless $\bG$ is finite
since $\lG$ has a normal invariant state.

\begin{lem}
Let $\bhG$ be a finite dimensional Kac algebra, and $\varphi$
be the normalized Haar state on $\bhG$. 
Let $\om\in Z^2_c(\bhG)$,
and $\beta=\om^*F(\om)$ be the bicharacter of $\om$.
Then $\beta$ is non-degenerate
if and only if $(\id\otimes \varphi)(\beta) =e_{\mathbf{1}}$. 
\end{lem}
\begin{proof}
First assume $(\id\otimes
\varphi)(\beta)=e_{\mathbf{1}}$.
Take $x\in \lhG$ such that $(x\otimes1)\beta=x\otimes 1$.
Then we obtain $x(\id\otimes \varphi)(\beta)=x$, and
$x=xe_{\mathbf{1}}$ by assumption. Hence $\beta$ is non-degenerate.
Conversely assume $\beta$ is non-degenerate.
Then
\begin{align*}
 \left((\id\otimes \varphi)(\beta)\otimes 1\right)\beta
&=
(\id\otimes \varphi \otimes \id)(\beta_{12}\beta_{13})\\
&=
(\id\otimes \vph \otimes \id)((\id\otimes \Delta)(\beta))\\
&=
(\id\otimes \varphi)(\beta)\otimes 1.
\end{align*}
Since $\beta$ is non-degenerate,
we have $(\id\otimes\varphi)(\beta)=e_{\mathbf{1}}$.
\end{proof}


\begin{thebibliography}{11}
\bibitem{AH}
Ando, H.,
Haagerup, U.,
Ultraproducts of von Neumann algebras,
{\it J. Funct. Anal.} {\bf 266} (2014), 6842--6913.

\bibitem{CD}
Anantharaman-Delaroche, C., Combes, F.,
Groupe modulaire d'une esp\'{e}rance conditionnele
dans une alg\`{e}bre de von Neumann,
{\it Bull. Soc. Math. France} {\bf 103} (1975), 385--426.

\bibitem{BaSk}
Baaj, S., G. Skandalis, G.,
Unitaires multiplicatifs et dualit\'{e} pour les produits
crois\'{e}s de C$^*$-alg\`{e}bres,
{\it Ann. Sci. \'{E}cole Norm. Sup.} (4) $\boldsymbol{26}$ (1993),
425--488.

\bibitem{Ba-gen}
Banica, T.,
Symmetries of a generic coaction,
{\it Math. Ann.} {\bf 314} (1999), 763--780. 

\bibitem{BaBi}
Banica, T., Bichon, J.,
Quantum groups acting on 4 points,
{\it J. Reine Angew. Math.} {\bf 626} (2009), 75--114.

\bibitem{BeCoTu}
B\'{e}dos, E., Conti, R., Tuset, L.,
On amenability and co-amenability of algebraic quantum groups 
and their corepresentations,
{\it Canad. J. Math.} \textbf{57} (2005), 17--60.

\bibitem{BeMuTu1}
B\'{e}dos, E., Murphy, G., Tuset, L.,
Co-amenability of compact quantum groups,
{\it J. Geom. Phys.} \textbf{40} (2001), 130--153. 

\bibitem{BeMuTu2}
B\'{e}dos, E., Murphy, G., Tuset, L.,
Amenability and coamenability of algebraic quantum groups,
{\it Int. J. Math. Math. Sci.} \textbf{31} (2002), 577--601.

\bibitem{BRV}
Bichon, J., De Rijdt, A., Vaes, S.,
Ergodic coactions with large multiplicity 
and monoidal equivalence of quantum Groups, 
{\it Comm. Math. Phys.} \textbf{262} (2006), 703--728.

\bibitem{C}Connes, A.,
Une classification des facteurs de type ${\rm III}$, 
{\it Ann. Sci. \'{E}cole Norm. Sup.} (4) \textbf{6} (1973), 133--252.

\bibitem{Con-auto}
Connes, A.,
Outer conjugacy classes of automorphisms of factors,
{\it Ann. Sci. Eco. Norm. Sup.} (4) {\bf 8} (1975), 383--419.

\bibitem{Co-peri}
Connes, A.,
Periodic automorphisms of the hyperfinite factor of type II$_1$,
{\it Acta Sci. Math} {\bf 39} (1977), 39--66.

\bibitem{CT}
Connes, A., Takesaki, M., 
The flow of weights on factors of type III,
{\it T\^{o}hoku Math. J.} (2)  {\bf 29}  (1977), 473--575. 

\bibitem{DR}
Doplicher, S.,
Roberts, J. E.,
A new duality theory for compact groups,
{\it Invent. Math.} {\bf 98} (1989), 157--218.

\bibitem{ES}Enock, M., Schwartz, J.-M.,
{\it Kac algebras and duality of locally compact groups},
Springer-Verlag, Berlin (1992).

\bibitem{FT}Falcone, T., Takesaki, M.,
The non-commutative flow of weights on a von Neumann algebra,
{\it J. Funct. Anal.} \textbf{182} (2001), 170--206.

\bibitem{H1}
Haagerup, U.,
Operator-valued weights in von Neumann algebras. I,
{\it J. Funct. Anal.} {\bf 32} (1979), 175--206.

\bibitem{H2}
Haagerup, U.,
Operator-valued weights in von Neumann algebras. II,
{\it J. Funct. Anal.} {\bf 33} (1979), 339--361.

\bibitem{Ham}
Hamachi, T.,
The normalizer group of an ergodic automorphism of type III
and the commutant of an ergodic flow,
{\it J. Funct. Anal.} {\bf 40} (1981), 387--403.

\bibitem{Hi}Hiai, F.,
Minimizing indices of conditional expectations onto a subfactor,
{\it Publ. Res. Inst. Math. Sci.} \textbf{24} (1988), 673--678.

\bibitem{Is}
Isola, T.,
Modular structure of the crossed product by a compact group dual,
{\it J. Operator Theory} {\bf 33} (1995), 3--31. 

\bibitem{Iz-f}
Izumi, M.,
Application of fusion rules to classification of subfactors,
{\it Publ. Res. Inst. Math. Sci.} \textbf{27} (1991), 953--994.

\bibitem{Iz}
Izumi, M.,
Canonical extension of endomorphisms of type III factors,
{\it Amer. J. Math.} {\bf 125} (2003), 1--56. 

\bibitem{IKf}
Izumi, M., Kosaki, H.,
Finite-dimensional Kac algebras arising from certain group actions
on a factor,
{\it Internat. Math. Res. Notices} (1996), 357--370. 

\bibitem{IK}
Izumi, M., Kosaki, H.,
{\it Kac algebras arising from composition of subfactors:
general theory and classification},
Mem. Amer. Math. Soc. {\bf 158} (2002), 198pp.

\bibitem{IzKo-coh}
Izumi, M., Kosaki, H.,
On a subfactor analogue of the second cohomology,
{\it Rev. Math. Phys.} {\bf 14} (2002), 733--757.

\bibitem{ILP}
Izumi, M., Longo, R., Popa, S.,
A Galois correspondence for compact groups of automorphisms of
von Neumann algebras with a generalization to Kac algebras,
{\it J. Funct. Anal.} {\bf 155} (1998), 25--63.

\bibitem{J-act}
Jones, V. F.~R.,
{\it Actions of finite groups on the hyperfinite type {II}$_1$ factor},
Mem. Amer. Math. Soc. {\bf 237} (1980).

\bibitem{JT}
Jones, V. F. R., Takesaki, M.,
Actions of compact abelian groups on semifinite injective factors,
{\it Acta. Math.} \textbf{153} (1984), 213--258. 

\bibitem{KP}
Kac, G. I.,
Paljutkin, V. G.,
Finite ring groups,
{\it Trudy Moskov. Mat. Ob\v{s}\v{c}.} \textbf{15} (1966), 224--261. 


\bibitem{KtST}
Katayama, Y., Sutherland, C.~E., and Takesaki, M.,
The characteristic square of a factor
and the cocycle conjugacy of discrete group actions on factors,
{\it Invent. Math.} {\bf 132} (1998), 331--380.

\bibitem{KtT-outerI}
Katayama, Y., Takesaki, M.,
{\it Outer actions of a discrete amenable group
on approximately finite dimensional factors, {I}, {G}eneral theory},
Operator algebras, quantization, and noncommutative geometry, Contemp. Math.,
  vol. 365, Amer. Math. Soc., Providence, RI,  (2004), 181--237.

\bibitem{KST}
Kawahigashi, Y.,
Sutherland, C. E.,
Takesaki, M.,
The structure of the automorphism group of an injective factor
and the cocycle conjugacy of discrete abelian group actions,
{\it Acta Math.} {\bf 169} (1992), 105--130.

\bibitem{Kw-Tak}
Kawahigashi, Y., Takesaki, M.,
Compact abelian group actions on injective factors,
{\it J. Funct. Anal.} {\bf 105} (1992), 112--128.

\bibitem{Kl}
Kleppner, A.,
Multipliers on abelian groups,
{\it Math. Ann.} {\bf 158} (1965), 11--34.

\bibitem{KS}Korogodski, L. I., Y. S. Soibelman, Y. S.,
\textit{Algebras of functions on quantum groups. Part I}, 
Mathematical Surveys and Monographs \textbf{56}, 
American Mathematical Society, Providence, RI, 1998. x+150 pp.


\bibitem{K}
Kosaki, H.,
Extension of Jones' theory on index to arbitrary factors
{\it J. Funct. Anal.} {\bf 66} (1986), 123--140.

\bibitem{K2}
Kosaki, H.,
{\it Type III factors and index theory},
Lecture Notes Series 43.
Seoul National University, Research Institute of Mathematics
(1998), ii+96 pp.

\bibitem{LK}
Baggett, L., Kleppner, A.,
Multiplier representations of abelian groups.
{\it J. Funct. Anal.} {\bf 14} (1973), 299--324. 

\bibitem{L1}
Longo, R.,
Index of subfactors and statistics of quantum fields. I.
{\it Comm. Math. Phys.} {\bf 126} (1989), 217--247.

\bibitem{L2}
Longo, R.,
Index of subfactors and statistics of quantum fields. II.
Correspondences, braid group statistics and Jones polynomial.
{\it Comm. Math. Phys.} {\bf 130} (1990), 285--309.

\bibitem{M}
Masuda, Toshihiko,
Unified approach to classification of actions
of amenable discrete groups on injective factors,
{\it J. Reine Angew. Math.} {\bf 683} (2013), 1--47.

\bibitem{MT1}
Masuda, T., Tomatsu, R.,
Classification of minimal actions of a compact Kac algebra with
amenable dual,
{\it Comm. Math. Phys.} {\bf 274}
(2007), 487--551.

\bibitem{MT2}
Masuda, T., Tomatsu, R.,
Approximate innerness and central triviality of endomorphisms,
{\it Adv. Math.} {\bf 220} (2009), 1075--1134.

\bibitem{MT3}
Masuda, T., Tomatsu, R.,
Classification of minimal actions of a compact Kac algebra
with amenable dual on injective factors of type III,
{\it J. Funct. Anal.} {\bf 258} (2010), 1965--2025.

\bibitem{MT-Roh}
Masuda, T., Tomatsu, R.,
Rohlin flows on von Neumann algebras,
to appear in Mem. Amer. Math. Soc.

\bibitem{Mu}
M\"uger, M.,
Abstract duality for symmetric tensor $*$-categories
(Appendix to Halvorson, H.:
``Algebraic Quantum Field Theory''),
{\it Handbook of the Philosophy of Physics}, p. 865-922. North Holland, 2007.
arXiv:math-ph/0602036.

\bibitem{NeTu}
Neshveyev, S., Tuset, L.,
On second cohomology of duals of compact groups,
Internat. J. Math. {\bf 22} (2011), 1231--1260.

\bibitem{NT}
Nakagami, Y., Takesaki, M.,
{\it Duality for crossed products of von Neumann algebras},
Lecture Notes in Mathematics, 731. Springer, Berlin, 1979. ix+139 pp.

\bibitem{Oc}
Ocneanu, A.,
{\it Actions of discrete amenable groups on von Neumann algebras},
Lecture Notes in Mathematics, 1138. Springer-Verlag, Berlin,
1985. iv+115 pp.

\bibitem{OPT}
Olesen, D., Pedersen, G. K.,
Takesaki, M.,
Ergodic actions of compact abelian groups.
{\it J. Operator Theory} {\bf 3} (1980), 237--269. 

\bibitem{Ro}Roberts, J. E.,
Crossed product of von Neumann algebras by group dual,
{\it Symposia Mathematica} \textbf{XX} (1976), 335--363.

\bibitem{Se-flow}
Sekine, Y.,
Flow of weights of the crossed products of type {III} factors by
discrete groups,
{\it Publ. Res. Inst. Math. Sci.} {\bf 26} (1990), 655--666.

\bibitem{Su-Tak-act}
Sutherland, C.~E., Takesaki, M.,
Actions of discrete amenable groups
on injective factors of type {III}$_\lambda$, $\lambda \ne 1$,
{\it Pacific. J. Math.} {\bf 137} (1989), 405--444.

\bibitem{ST}
Sutherland C. E., Takesaki, M.,
{\it Right inverse of the module of approximately finite-dimensional
factors of type III and approximately finite ergodic principal
measured groupoids},
Operator algebras and their applications, II
(Waterloo, ON, 1994/1995),
149--159, Fields Inst. Commun., 20, Amer. Math. Soc., Providence, RI, 1998. 

\bibitem{Ta}
Takesaki, M.,
{\it Theory of operator algebras. I, II, III},
Encyclopaedia of Mathematical Sciences, 124, 125, 127.
Operator Algebras and Non-commutative Geometry, 8.
Springer-Verlag, Berlin, 2003.

\bibitem{To1}Tomatsu, R.,
Amenable discrete quantum groups,
{\it J. Math. Soc. Japan} \textbf{58} (2006), 949--964.

\bibitem{T}
Tomatsu, R.,
A paving theorem for amenable discrete Kac algebras,
{\it Internat. J. Math.} {\bf 17} (2006), 905--919.

\bibitem{T2}Tomatsu, R.,
A characterization of right coideals of quotient type 
and its application to classification of Poisson boundaries,
{\it Comm. Math. Phys.} \textbf{275} (2007), 271--296.

\bibitem{Wan-sym}
Wang, S.,
Quantum symmetry groups of finite spaces,
{\it Comm. Math. Phys.} {\bf 195} (1998), 195--211.

\bibitem{Wass-cptI}
Wassermann, A.
Ergodic actions of compact groups on operator algebras. I.
General theory,
{\it Ann. of Math.} (2) {\bf 130} (1989), 273--319. 

\bibitem{Wass-cptII}
Wassermann, A.,
Ergodic actions of compact groups on operator algebras. II.
Classification of full multiplicity ergodic actions,
{\it Can. J. Math.} {\bf 40} (1988), 1482--1527.

\bibitem{Wass-cptIII}
Wassermann, A.,
Ergodic actions of compact groups on operator algebras. III.
Classification for ${\rm SU}(2)$,
{\it Invent. Math.} {\bf 93} (1988), 309--354.

\bibitem{Wass-coact}
Wassermann, A.,
{\it Coactions and Yang-Baxter equations for ergodic actions
and subfactors},
Operator algebras and applications, Vol. 2, 203--236,
London Math. Soc. Lecture Note Ser., 136,
Cambridge Univ. Press, Cambridge, 1988.


\bibitem{Wor}Woronowicz, S. L.,
Compact quantum groups,
{\it Sym\'{e}tries quantiques} (Les Houches, 1995),
845--884, North-Holland, Amsterdam, (1998).

\bibitem{Y}Yamanouchi, T.,
Canonical extension of actions of locally compact quantum groups.
{\it J. Funct. Anal.} {\bf 201} (2003), 522--560. 

\bibitem{Z}
Zimmer, R.,
Extensions of ergodic group actions,
{\it Illinois J. Math.} {\bf 20} (1976), 373--409.

\end{thebibliography}
\end{document}